\pdfoutput=1
\documentclass{amsart}
\usepackage[colorlinks=true,
pdfstartview=FitV, linkcolor=cyan, citecolor=magenta,
urlcolor=blue]{hyperref}
\usepackage{amsthm,amsmath,amsfonts,latexsym,amssymb}
\usepackage[normalem]{ulem}
\usepackage{hyperref}
\usepackage{enumitem}
\usepackage{mathrsfs}
\usepackage[latin1]{inputenc}
\usepackage[T1]{fontenc}
\usepackage{ae,aecompl}
\usepackage{braket}
\usepackage{comment}
\usepackage{color}
\usepackage{graphicx}
\usepackage{subfig}
\theoremstyle{plain}

\newtheorem*{stirling'sformula}{Stirling's Formula}
\newtheorem{thm}{Theorem}[section]
\newtheorem{prop}[thm]{Proposition}
\newtheorem{lem}[thm]{Lemma}
\newtheorem{cor}[thm]{Corollary}

\theoremstyle{definition}
\newtheorem{definition}[thm]{Definition}

\theoremstyle{remark}
\newtheorem{notation}[thm]{Notation}
\newtheorem{remark}[thm]{Remark}

\numberwithin{equation}{section}

\newcommand{\Amc}{\mathcal{A}} 
\newcommand{\Bmc}{\mathcal{B}} 
\newcommand{\Cmc}{\mathcal{C}} 
\newcommand{\Dmc}{\mathcal{D}} 
\newcommand{\Emc}{\mathcal{E}} 
\newcommand{\Fmc}{\mathcal{F}} 
 
\newcommand{\Hmc}{\mathcal{H}} 
 
\newcommand{\Jmc}{\mathcal{J}}

\newcommand{\Mmc}{\mathcal{M}}

\newcommand{\Pmc}{\mathcal{P}}

\newcommand{\Smc}{\mathcal{S}} 
\newcommand{\Tmc}{\mathcal{T}}

\newcommand{\Wmc}{\mathcal{W}} 
\newcommand{\Xmc}{\mathcal{X}} 
\newcommand{\Ymc}{\mathcal{Y}} 
\newcommand{\Zmc}{\mathcal{Z}}

\newcommand{\Fbbb}{\mathbb{F}} 
\newcommand{\Gbbb}{\mathbb{G}}

\newcommand{\Pbbb}{\mathbb{P}} 
 
\newcommand{\Rbbb}{\mathbb{R}} 
 
\newcommand{\Tbbb}{\mathbb{T}}  
 
\newcommand{\Vbbb}{\mathbb{V}}

\newcommand{\Zbbb}{\mathbb{Z}}

\newcommand{\lmf}{\mathfrak{l}}

\newcommand{\smf}{\mathfrak{s}}

\newcommand{\Itd}{\widetilde{I}}

\newcommand{\Std}{\widetilde{S}}

\DeclareMathAlphabet{\pazocal}{OMS}{zplm}{m}{n}

\DeclareMathOperator{\sgn}{sgn}

\DeclareMathOperator{\PSL}{\mathrm{PSL}}
\DeclareMathOperator{\PGL}{\mathrm{PGL}}
\DeclareMathOperator{\Span}{\mathrm{Span}}
\DeclareMathOperator{\tr}{\mathrm{tr}}

\DeclareMathOperator{\End}{\mathrm{End}}

\DeclareMathOperator{\Hit}{\mathrm{Hit}}
\DeclareMathOperator{\id}{\mathrm{id}}

\DeclareMathOperator{\Ad}{\mathrm{Ad}}

\DeclareMathOperator{\Hom}{\mathrm{Hom}}

\begin{document}

\title{The Goldman Symplectic form on the $\PGL(V)$-Hitchin component}

\author{Zhe Sun}
\address{School of Mathematical Sciences, University of Science and Technology of China, 96 Jinzhai Road, 230026 Hefei, Anhui, China}
\email{sunz@ustc.edu.cn}

\author{Tengren Zhang}
\address{Mathematics Department, National University of Singapore, 10 Lower Kent Ridge Road, Singapore 119076}
\email{matzt@nus.edu.sg}
\thanks{ZS was partially supported by the Luxembourg National Research Fund(FNR) AFR bilateral grant COALAS 11802479-2. TZ was partially supported by the National Science Foundation under agreements DMS-1536017, and by the NUS-MOE grant R-146-000-270-133 and A-8000458-00-00. The authors acknowledge support from U.S. National Science Foundation grants DMS 1107452, 1107263, 1107367 ``RNMS: GEometric structures And Representation varieties" (the GEAR Network).}

\begin{abstract}
This article is the second of a pair of articles about the Goldman symplectic form on $\Hit_V(S)$, the $\PGL(V)$-Hitchin component of a closed, connected, oriented, hyperbolic surface $S$. We show that any ideal triangulation $\Tmc$ on $S$ and any compatible bridge system $\Jmc$ determine a symplectic trivialization of the tangent bundle to the $\PGL(V)$-Hitchin component of $S$. Using this, we prove that a large class of vector fields defined in the companion paper \cite{SunWienhardZhang}, called the $(\Tmc,\Jmc)$-parallel vector fields, are Hamiltonian. We also show that if $(\Tmc,\Jmc)$ is subordinate to a pants decomposition $\Pmc$ of $S$, then the special $(\Tmc,\Jmc)$-parallel vector fields constructed in the companion paper \cite{SunWienhardZhang} give a symplectic basis of the tangent space to $\Hit_V(S)$ at every point in $\Hit_V(S)$. This is then used to prove that the explicit global coordinate system defined in the companion paper \cite{SunWienhardZhang} is a global Darboux coordinate system for the $\PGL(V)$-Hitchin component.
\end{abstract}

\maketitle

\tableofcontents

%%%%%%%%%%%%%%%%%%%%%%%%%%%%%%%%%%%%%%%%%%%%%%%%%%%%%%%%%%%%%%%%%%%%%%%%%%%%%%%%%%%%%%%%%%%%%%%%%%%%
\section{Introduction} 
%%%%%%%%%%%%%%%%%%%%%%%%%%%%%%%%%%%%%%%%%%%%%%%%%%%%%%%%%%%%%%%%%%%%%%%%%%%%%%%%%%%%%%%%%%%%%%%%%%%%

Let $S$ be a closed, oriented, connected, hyperbolic surface, and let $\Gamma$ denote the fundamental group of $S$. The \emph{Teichm\"uller space} of $S$, denoted $\Tmc(S)$, is the deformation space  of hyperbolic structures on $S$. By considering the holonomy representations of these hyperbolic structures, one can identify $\Tmc(S)$ as the connected component
\[\left\{[\rho]\in\Xmc\big(\Gamma,\PGL(2,\Rbbb)\big):\rho\text{ is discrete and faithful}\right\}\]
of $\Xmc\big(\Gamma,\PGL(2,\Rbbb)\big):=\Hom\big(\Gamma,\PGL(2,\Rbbb)\big)/\PGL(2,\Rbbb)$. 

This description of $\Tmc(S)$ as representations admits a natural generalization where we replace $\PGL(2,\Rbbb)$ with a higher rank, semisimple Lie group $G$ that is real split. These generalizations of $\Tmc(S)$ are known as the $G$-\emph{Hitchin components}, and are central objects of investigation in the study of higher Teichm\"uller theory. In this article, we consider the setting when $G=\PGL(V)$, where $V$ is an $n$-dimensional real vector space (see Section \ref{sec:Hitchin}). The $\PGL(V)$-Hitchin component, denoted $\Hit_V(S)$, was first studied by Hitchin \cite{Hitchin}, who showed that $\Hit_V(S)$ is real-analytically diffeomorphic to $\Rbbb^{(n^2-1)(2g-2)}$. Later, Labourie \cite{Labourie2006} proved that every representation in $\Hit_V(S)$ is Borel-Anosov, which implies in particular that these representations are all quasi-isometric embeddings from $\Gamma$ (with any word metric) to $\PGL(V)$ (with any left-invariant Riemannian metric).

By specializing a general construction due to Goldman \cite{Goldmansymplectic}, one can equip $\Hit_V(S)$ with a natural symplectic structure (see Section \ref{sec:Goldman symplectic form}), known as the \emph{Goldman symplectic form}. Goldman also proved \cite{Goldmaninvariantfunctions} that when $V=\Rbbb^2$, this symplectic form is a multiple of the more classical Weil-Petersson symplectic form on $\Tmc(S)$. The Goldman symplectic form on $\Hit_V(S)$ has been an object of interest in the last few years, and was previously studied by many other authors, including \cite{Bridgeman}, \cite{LabourieWentworth}, \cite{Labourie_swapping}, \cite{Nie_Thesis}, \cite{Sun}, \cite{Sun1}.

Perhaps the most famous result regarding this symplectic form on $\Tmc(S)$ is a theorem of Wolpert \cite{Wolpert1},\cite{Wolpert2}, which states that for any pants decomposition $\Pmc$ on $S$, the Fenchel-Nielsen coordinate functions associated to $\Pmc$ are global Darboux coordinates for the Weil-Petersson symplectic form. In other words, the Hamiltonian vector fields of the Fenchel-Nielsen coordinate functions corresponding to $\Pmc$ are commuting vector fields that give a symplectic basis of $T_{[\rho]}\Tmc(S)$ for every $[\rho]\in\Tmc(S)$. This was partially generalized to $\Hit_V(S)$ when $\dim(V)=3$ by Kim \cite{kim1999} and Choi-Jung-Kim \cite{ChoiJungKim}, and when $\dim(V)=4$ by H.T. Jung (in preparation). The key tool used to prove these results is known as Fox calculus, which is a technique that allows one to compute the Goldman Poisson pairing between certain types of functions on $\Hit_V(S)$. However, the computations required to implement Fox calculus become exceedingly complicated as $\dim(V)$ gets large.

This article is the second of a pair of articles about the Goldman symplectic form on $\Hit_V(S)$, where the objective is to generalize Wolpert's theorem to $\Hit_V(S)$ for general $V$. 

\subsection{Statement of results} In the companion article,  Sun, Wienhard, and Zhang \cite{SunWienhardZhang} showed that that given any ideal triangulation $\Tmc$ on $S$ and any compatible bridge system $\Jmc$ (see Section \ref{sec:topological constructions} for definitions), one can construct a family of real-analytic vector fields on $\Hit_V(S)$ using the geometry of the representations in $\Hit_V(S)$. These are called the \emph{$(\Tmc,\Jmc)$-parallel vector fields} (\cite[Section 5.1]{SunWienhardZhang}), and satisfy the following properties:
\begin{itemize}
\item Any pair of $(\Tmc,\Jmc)$-parallel vector fields commute.
\item For any $[\rho]\in\Hit_V(S)$ and any $v\in T_{[\rho]}\Hit_V(S)$, there is a unique $(\Tmc,\Jmc)$-parallel vector field $\Xmc$ such that $\Xmc([\rho])=v$.
\end{itemize}
Special cases of these flows were previously constructed by Wienhard-Zhang \cite{WienhardZhang}. 

Furthermore, by modifying Bonahon and Dreyer's \cite{BonahonDreyer1} parameterization  of $\Hit_V(S)$, Sun, Wienhard, and Zhang \cite[Theorem 1.1]{SunWienhardZhang} also constructed a real-analytic diffeomorphism 
\[\Omega=\Omega_{\Tmc,\Jmc}:\Hit_V(S)\to \mathscr{C},\]
where $\mathscr{C}$ is a convex polyhedral cone of dimension $(n^2-1)(2g-2)$ in a real vector space $\mathscr{W}$, such that the $(\Tmc,\Jmc)$-parallel vector fields on $\Hit_V(S)$ are exactly the vector fields that are identified via the \emph{$(\Tmc,\Jmc)$-trivialization}
\[{\rm d}\Omega:T\Hit_V(S)\to T\mathscr{C}\cong\mathscr{C}\times\mathscr{W}\] 
to the constant vector fields on $\mathscr{C}$. 

%More informally, given a choice of some topological data on $S$ (the choice of $\Tmc$ and $\Jmc$), the geometry of the representations in $\Hit_V(S)$ determine a vector space of commuting flows on $\Hit_V(S)$ that is naturally in bijection with $T_{[\rho]}\Hit_V(S)$ for every $[\rho]\in\Hit_V(S)$. In particular, this determines a trivialization of $T\Hit_V(S)$. 

The first goal of this article is to prove that ${\rm d}\Omega$ is a symplectic trivialization with respect to the Goldman symplectic form.

\begin{thm}[Theorem \ref{thm:constant}]\label{thm:motive}
Let $\Tmc$ be an ideal triangulation on $S$ and let $\Jmc$ be a compatible bridge system. There is a symplectic bilinear form $\omega_{\mathscr{W}}$ on the vector space $\mathscr{W}$ for which the $(\Tmc,\Jmc)$-trivialization ${\rm d}\Omega$ is symplectic. Equivalently, if $\Xmc'$ and $\Xmc$ are $(\Tmc,\Jmc)$-parallel vector fields on $\Hit_V(S)$, then the function $\Hit_V(S)\to\Rbbb$ given by 
\[[\rho]\mapsto\omega_{[\rho]}(\Xmc'([\rho]),\Xmc([\rho]))\] 
is constant. 
\end{thm}

Theorem \ref{thm:motive} yields the following corollary.

\begin{cor}[Corollary \ref{cor:Hamiltonian}]\label{cor:intro}
Any $(\Tmc,\Jmc)$-parallel vector field is a Hamiltonian vector field on $\Hit_V(S)$.
\end{cor}

%As a consequence, we have the following corollary.

%\begin{cor}[Corollary \ref{cor:integrable}]
%$\Hit_V(S)$ is a complete integrable system.
%\end{cor}

When $(\Tmc,\Jmc)$ is subordinate to some pants decomposition $\Pmc$ of $S$, Sun, Wienhard, and Zhang \cite{SunWienhardZhang} specified a particular collection of $(\Tmc,\Jmc)$-parallel vector fields on $\Hit_V(S)$, called the \emph{special $(\Tmc,\Jmc)$-parallel vector fields}, see Definition~\ref{def:special}. These consist of 
\begin{itemize}
\item $\frac{(n-1)(n-2)}{2}$ \emph{eruption fields} for each pair of pants of $\Pmc$,
\item $\frac{(n-1)(n-2)}{2}$ \emph{hexagon fields} for each pair of pants of $\Pmc$,
\item $n-1$ \emph{twist fields} for each simple closed curve in $\Pmc$,
\item $n-1$ \emph{length fields} for each simple closed curve in $\Pmc$.
\end{itemize}
In total, there are $(n^2-1)(2g-2)$ special $(\Tmc,\Jmc)$-tangent parallel vector fields. If we specialize to the case when $\dim(V)=2$, i.e. when $\Hit_V(S)=\Tmc(S)$, then there are no eruption and hexagon fields, but there is one twist field and one length field for each simple closed curve in $\Pmc$, and the twist field integrates to the Fenchel-Nielsen twist flow. 

The second goal of this article is to use Theorem \ref{thm:motive} to prove that the special $(\Tmc,\Jmc)$-parallel vector fields form a symplectic basis at every point in $\Hit_V(S)$.

\begin{thm}[Corollary~\ref{cor:symplectic}]\label{thm:intro2}
Let $(\Tmc,\Jmc)$ be subordinate to a pants decomposition $\Pmc$ of $S$. For any $[\rho]\in\Hit_V(S)$, the set 
\[\{\Xmc([\rho])\in T_{[\rho]}\Hit_V(S):\Xmc\text{ is a special }(\Tmc,\Jmc)\text{-parallel vector field}\}\]
is a symplectic basis of $T_{[\rho]}\Hit_V(S)$.
\end{thm}

\begin{remark}\label{rem: mistake 1}
There is a small mistake (a missing factor of $2$) in the description of the length vector fields in the companion paper \cite{SunWienhardZhang}. See Definition \ref{def:length} for the correct description of the length vector fields used in Theorem \ref{thm:intro2} to hold, and see Remark \ref{rem: correction} for a description of the difference between the length vector fields described in \cite{SunWienhardZhang} and the length vector fields described here.
\end{remark}

In the companion paper \cite{SunWienhardZhang}, the authors also constructed, for every choice of $(\Tmc,\Jmc)$ that is subordinate to a pants decomposition of $S$, a global coordinate system on $\Hit_V(S)$ by associating to every special $(\Tmc,\Jmc)$-parallel vector field $\Xmc$ a coordinate function $H(\Xmc)$ of this coordinate system, which is explicitly given by projective invariants that can be read off from the Frenet curves associated to the $\PGL(V)$-Hitchin representations. As a consequence of Theorem~\ref{thm:intro2}, we have the following corollary.

\begin{cor}\label{cor: intro}
The global coordinate system on $\Hit_V(S)$ defined in the companion paper \cite[Corollary 8.15]{SunWienhardZhang} is a global Darboux coordinate system on $\Hit_V(S)$.
\end{cor}

%\begin{remark}
%Because of the mistake in the companion paper described in Remark \ref{rem: mistake 1}, one has to accordingly modify (by a factor of $2$) the coordinate functions associated to the ? vector fields in order for Corollary \ref{cor: intro} to hold. See Remark \ref{} for the corrected version.
%\end{remark}

When $\dim(V)=2$, this coordinate system recovers the Fenchel-Nielsen coordinates on $\Tmc(S)$. This thus generalizes Wolpert's theorem to $\Hit_V(S)$.

\subsection{Proof strategy for Theorem \ref{thm:motive}} Behind the proof of Theorem \ref{thm:motive} is a new method to compute the Goldman symplectic form on $\Hit_V(S)$, which is similar in flavor to the techniques used by Bonahon-Sozen \cite{BonahonSozen} on $\Tmc(S)$. This method uses the following theorem obtained by combining the work of Labourie \cite{Labourie2006} and Guichard \cite{Guichard}.

\begin{thm} [Guichard, Labourie]\label{thm:tool}
Let $\Fmc(V)$ denote the space of complete flags in $V$. There is canonical bijection
\[\Hit_V(S)\simeq\left\{\xi:\partial\Gamma\to\Fmc(V)\bigg|\begin{array}{l}
\xi\text{ is Frenet and }\rho\text{-equivariant }\\
\text{for some }\rho\in\Hom\big(\Gamma,\PGL(V)\big)\end{array}\right\}\bigg/\PGL(V)\]
\end{thm}

In other words, instead of thinking of $\Hit_V(S)$ as a space of conjugacy classes of representations, one can think of $\Hit_V(S)$ as a space of projective classes of Frenet curves (see Definition~\ref{def:Frenet}) satisfying an equivariance property.

We now briefly describe this new computational method. Let $[\rho]\in\Hit_V(S)$ and let $X\in T_{[\rho]}\Hit_V(S)$. Fix a choice of $(\Tmc,\Jmc)$ and a choice of representative $\rho\in[\rho]$. First, we use the Frenet curves point of view for Hitchin representations to explicitly define a cocycle 
\[\nu_X=\nu_{X,\rho,\Tmc,\Jmc}\in C^1(S,\smf\lmf(V)_{\Ad\circ\rho})\] 
called the $(\rho,\Tmc,\Jmc)$-\emph{tangent cocycle} of $X$, see Sections \ref{sec:tangent cocycle} and \ref{sec:tangent cocycle general}. We then prove that the assignment 
\[\Psi=\Psi_{\rho,\Tmc,\Jmc}:T_{[\rho]}\Hit_V(S)\to C^1(S,\smf\lmf(V)_{\Ad\circ\rho})\] 
given by $X\mapsto \nu_X$ is a linear injection, so its image $\mathscr{T}=\mathscr{T}(\rho,\Tmc,\Jmc)$ is a linear subspace of $C^1(S,\smf\lmf(V)_{\Ad\circ\rho})$ of dimension $(n^2-1)(2g-2)$. We then show, see Proposition \ref{prop:tangentcocyclebij}, that the induced map 
\[T_{[\rho]}\Hit_V(S)\to H^1(S,\smf\lmf(V)_{\Ad\circ\rho})\]
given by $X\mapsto[\Psi(X)]$ agrees precisely with the linear isomorphism $T_{[\rho]}\Hit_V(S)\simeq H^1(S,\smf\lmf(V)_{\Ad\circ\rho})$ introduced by Weil \cite{Weil}, which Goldman uses to define the Goldman symplectic form. In particular, each cohomology class in $H^1(S,\smf\lmf(V)_{\Ad\circ\rho})$ admits a unique cocycle representative in $\mathscr{T}$.

Then, we explicitly define another linear isomorphism
\[\Xi:=\Xi_{\rho,\Tmc,\Jmc}:\mathscr{T}\to\mathscr{W},\]
see Section \ref{sec: Xi} and Proposition \ref{prop: Xi injective}. One can think of $\Xi$ as abstracting certain ``coefficient data" from every $(\rho,\Tmc,\Jmc)$-tangent cocycle. We then prove, see Theorem~\ref{thm: symplectic trivialization}, that for any point $[\rho]\in\Hit_V(S)$ and any representative $\rho\in[\rho]$, the map 
\[\Xi\circ\Psi:T_{[\rho]}\Hit_V(S)\to \mathscr{W}\cong T_{\Omega([\rho])} \mathscr{C}\]
agrees with ${\rm d}\Omega_{[\rho]}$. In particular, $\Xi\circ\Psi$ does not depend on the choice of $\rho\in[\rho]$, and the map $[\rho]\mapsto \Xi\circ\Psi(\Xmc([\rho]))$ is constant for any $(\Tmc,\Jmc)$-parallel vector field $\Xmc$ on $\Hit_V(S)$.

Using the explicit descriptions of the $\Psi$ and $\Xi$, as well as a well-chosen triangulation of $S$ that is compatible with $(\Tmc,\Jmc)$, see Section \ref{sec:triangulation}, we may apply the formula (see \eqref{eqn:formula}) for the Goldman symplectic form to show that for any pair of $(\Tmc,\Jmc)$-parallel vector fields $\Xmc$ and $\Xmc'$ on $\Hit_V(S)$ and any $[\rho]\in\Hit_V(S)$, the Goldman symplectic pairing at $[\rho]$ between $\Xmc'([\rho])$ and $\Xmc([\rho])$ depends only on $\Xi\circ\Psi(\Xmc([\rho]))$ and $\Xi\circ\Psi(\Xmc'([\rho]))$, which are vectors in $\mathscr{W}$ that do not depend on $[\rho]$. Theorem~\ref{thm:motive} follows from this.\\

%Informally, by making some choices, we found a systematic, explicit, and concrete description of every tangent vector in $T_{[\rho]}\Hit_V(S)$ as a $\Tmc$-admissible labelling at $\rho$. This description allows us to compute the Goldman symplectic pairing on $T_{[\rho]}\Hit_V(S)$ by using the cup product formula from simplicial cohomology. Executing this computation involves choosing a triangulation of $S$ that is well-adapted to the pair $(\Tmc,\Jmc)$, see Section \ref{sec:triangulation} for more details.

%More precisely, for every $L\in\mathscr{A}(\rho,\Tmc)$, we define the \emph{coefficients} of $L$ (see Definition \ref{def:special}). These are a collection of real numbers that determine $L$ once we are given the $\rho$-equivariant Frenet curve. We then used Theorem \ref{thm:intro iso} to compute that for any $L_1,L_2\in\Amc(\rho,r,\Tmc)$, $\omega([\mu_{L_1}],[\mu_{L_2}])$ depends only on the coefficients of $L_1$ and $L_2$. Theorem \ref{thm:motive} then follows from the observation that the coefficients of the admissible labellings corresponding to tangent vectors to any $(\Tmc,\Jmc)$-parallel flow are constant on $\Hit_V(S)$. For more details, see the proof of Theorem \ref{thm:constant}.

The rest of this paper is organized as follows. In Section \ref{sec: background}, we discuss the relationship between Frenet curves and Hitchin representations, as well as some well-known projective invariants such as the cross ratio and triple ratio. Then in Section~\ref{sec:topological constructions}, we formally define ideal triangulations, barriers systems, and bridge systems. Using this, we define the $(\rho,\Tmc,\Jmc)$-tangent cocycles and the map $\Psi=\Psi_{\rho,\Tmc,\Jmc}$ in Section \ref{sec: tangent cocycles}, and we define the map $\Xi=\Xi_{\rho,\Tmc,\Jmc}$ in Section \ref{sec: coeff and lab}. Then in Section \ref{sec:trivialization} and Section \ref{sec:symplectic basis}, we prove Theorem \ref{thm:motive}, and Theorem \ref{thm:intro2} respectively.

%In Section 7, we use $\Phi_\rho$ to construct a symplectic trivialization of $T\Hit_V(S)$, and prove Theorem \ref{thm:motive}. Then in Section 6, we define the special admissible labellings and prove Theorem \ref{thm:intro2}.

%Some of the proofs in this article require long but elementary computations, which are completely written up in the appendices attached. Also, one should note that the results in this article depend only on the first five sections (but not Section 6) of the companion article \cite{SunWienhardZhang}.

{\bf Acknowledgements:} The authors thank Sara Maloni, Sam Ballas, Frederic Palesi, Fran\c{c}ois Labourie, and Francis Bonahon for many useful conversations that helped them develop their understanding of this subject. They especially thank Bill Goldman for asking the question that motivated this project. They also thank Anna Wienhard, their collaborator in the companion paper, for suggesting that they implement their computational methods explicitly.

This project started when the first author visited the second author at the California Institute of Technology in the spring of 2016. The bulk of the work in this project was done when both authors visited the Institute for Mathematical Sciences in the National University of Singapore during the summer of 2016, and also during the second author's visit to the Yau Mathematical Sciences Center in Tsinghua University during the spring and summer of 2017. The authors are grateful towards these institutions for their hospitality. The first author would also like to thank Yi Huang for some financial support for this project.

%%%%%%%%%%%%%%%%%%%%%%%%%%%%%%%%%%%%%%%%%%%%%%%%%%%%%%%%%%%%%%%%%%%%%%%%%%%%%%%%%%%%%%%%%%%%%%%%%%%%
\section{Features of Hitchin representations and the Hitchin component}\label{sec: background}
%%%%%%%%%%%%%%%%%%%%%%%%%%%%%%%%%%%%%%%%%%%%%%%%%%%%%%%%%%%%%%%%%%%%%%%%%%%%%%%%%%%%%%%%%%%%%%%%%%%%
In this section, we recall the definition of $\PGL(V)$-Hitchin representations and the $\PGL(V)$-Hitchin component. We first explain results by Labourie  \cite{Labourie2006} and Guichard \cite{Guichard} relating $\PGL(V)$-Hitchin representations to Frenet curves, as well as results due to Fock and Goncharov \cite{FockGoncharov} regarding the positive nature of $\PGL(V)$-Hitchin representations. Then, we  use a regularity result of Bridgeman, Canary, Labourie, and Sambarino \cite{BCLS} to give an explicit description of the real-analytic structure on the $\PGL(V)$-Hitchin component. Finally, we recall Weil's \cite{Weil} cohomological description of the tangent space to the $\PGL(V)$-Hitchin component, as well as the definition of the Goldman symplectic form \cite{Goldmansymplectic} on the $\PGL(V)$-Hitchin component.

%%%%%%%%%%%%%%%%%%%%%%%%%%%%%%%%%%%%%%%%%%%%%%%%%%
\subsection{Frenet curves and $\PGL(V)$-Hitchin representations}\label{sec:Hitchin}
%%%%%%%%%%%%%%%%%%%%%%%%%%%%%%%%%%%%%%%%%%%%%%%%%%
Recall that $V$ denotes an $n$-dimensional real vector space, $S$ denotes a closed, connected, oriented topological surface of genus $g\geq 2$, $\Gamma$ denotes the fundamental group of $S$, and $\partial\Gamma$ denotes the Gromov boundary of $\Gamma$. It is well-known that $\partial\Gamma$ is topologically a circle, and the orientation on $S$ induces a (clockwise) cyclic ordering on $\partial\Gamma$. 

\begin{definition}\label{def:Hitchin component}
A representation $\rho:\Gamma\to\PGL(V)$ is a $\PGL(V)$-\emph{Hitchin representation} if it can be continuously deformed in $\Hom(\Gamma,\PGL(V))$ to a faithful representation whose image is a discrete subgroup that lies in the image of an irreducible representation from $\PSL_2(\Rbbb)$ to $\PGL(V)$.
\end{definition}

\begin{remark}
Since $\PSL_2(\Rbbb)$ is connected, the image of any $\PGL(V)$-Hitchin representation necessarily lies in the identity component of $\PGL(V)$, namely $\PSL(V)$. As such, these are also sometimes called $\PSL(V)$-Hitchin representations.
\end{remark}

Let $\widetilde{\Hit}_V(S)$ denotes the set of Hitchin representations in $\Hom(\Gamma,\PGL(V))$. Since $\PGL(V)$ is a real algebraic group, $\Hom(\Gamma,\PGL(V))$ is naturally a real algebraic variety. It is clear that in the the compact-open topology on $\Hom(\Gamma,\PGL(V))$, $\widetilde{\Hit}_V(S)$ is a union of connected components of $\Hom(\Gamma,\PGL(V))$. Furthermore, since the condition for being a $\PGL(V)$-Hitchin representation is invariant under conjugation by elements in $\PGL(V)$, we can consider the quotient $\Hit_V(S):=\widetilde{\Hit}_V(S)/\PGL(V)$. This was first studied by Hitchin \cite{Hitchin}, who used Higgs bundle techniques to prove that $\Hit_V(S)$ is homeomorphic to a cell of dimension $(2g-2)(n^2-1)$. In particular, $\Hit_V(S)$ is connected, and is commonly known as the $\PGL(V)$-\emph{Hitchin component}.

\begin{remark}
The quotient $\widetilde{\Hit}_V(S)/\PSL(V)$ has two connected components when $n$ is even, but is connected when $n$ is odd. Since we want to work with one Hitchin component, it is more convenient to take the quotient $\widetilde{\Hit}_V(S)/\PGL(V)$.
\end{remark}

Note that $\Hit_{\Rbbb^2}(S)$ is exactly the space of discrete and faithful representations from $\Gamma$ to $\PSL_2(\Rbbb)$ considered up to conjugation in $\PGL_2(\Rbbb)$. Every such conjugacy class is the holonomy of a unique hyperbolic structure on $S$, so $\Hit_{\Rbbb^2}(S)$ can be naturally identified with the deformation space of hyperbolic structures on $S$, also known as the Fricke-Teichm\"uller space. This phenomena was generalized by Choi-Goldman \cite{ChoiGoldman}, who proved that the $\PGL(3,\Rbbb)$-Hitchin representations are exactly holonomies of convex real projective structures on $S$. However, the geometric properties of $\PGL(V)$-Hitchin representations in general were poorly understood until a breakthrough by Labourie \cite{Labourie2006}, who used Frenet curves to study them. We will now give a brief account of Labourie's result.

A \emph{(complete) flag} $F$ in $V$ is a nested sequence of subspaces 
\[F^{(1)}\subset F^{(2)}\subset\dots\subset F^{(n-1)}\subset V\] 
so that $\dim F^{(i)}=i$. Let $\Fmc(V)$ denote the space of flags in $V$. The space $\Fmc(V)$ naturally has the structure of a real-analytic manifold, and $\PGL(V)$ acts transitively and real-analytically on $\Fmc(V)$. We say that a $k$-tuple of flags $(F_1,\dots,F_k)$ in $\Fmc(V)^k$ is \emph{generic} if $F_1^{(i_1)}+\dots+F_k^{(i_k)}=V$ for all non-negative integers $i_1,\dots,i_k$ that sum to $n$. Let $\Fmc(V)^{[k]}$ denote the space of generic $k$-tuples of flags. Note that this is an open set in $\Fmc(V)^k$, and so naturally inherits a real-analytic structure which descends to a real-analytic structure on $\Fmc(V)^{[k]}/\PGL(V)$. 

\begin{remark}\label{rem:normalization}
We will repeatedly use the following two observations. 
\begin{enumerate}
\item Let $(F,G), (F',G')$ be generic pairs of flags in $\Fmc(V)$, and let $p,p'\in\Pbbb(V)$ so that 
\[p+F^{(i)}+G^{(n-i-1)}=V=p'+F'^{(i)}+G'^{(n-i-1)}\] 
for all $i=0,\dots,n-1$. Then there is a unique projective transformation $g\in\PGL(V)$ such that $g\cdot F=F'$, $g\cdot G=G'$ and $g\cdot p=p'$.
\item Let $F,G,G'$ be flags in $\Fmc(V)$ such that $(F,G)$ and  $(F,G')$ are generic. Then there is a unique unipotent projective transformation $u\in\PGL(V)$ such that $u\cdot F=F$ and $u\cdot G=G'$.
\end{enumerate}
\end{remark}

\begin{definition}\label{def:Frenet}
A continuous map $\xi:S^1\to\Fmc(V)$ is \emph{Frenet} if the following hold for all $k=1,\dots,n$:
\begin{itemize}
\item If $(x_1,\dots,x_k)$ is a pairwise distinct $k$-tuple of points in $S^1$, then 
\[\big(\xi(x_1),\dots,\xi(x_k)\big)\] 
is a generic $k$-tuple of flags in $\Fmc(V)$.
\item Let $\{(x_{1,i},\dots,x_{k,i})\}_{i=1}^\infty$ be a sequence of pairwise distinct $k$-tuples in $S^1$, such that $\lim_{i\to\infty}x_{j,i}=x\in S^1$ for all $j=1,\dots,k$. Then for any positive integers $n_1,\dots,n_k$ so that $n_1+\dots+n_k=m\leq n$, we have
\[\lim_{i\to\infty}\xi(x_{1,i})^{(n_1)}+\dots+\xi(x_{k,i})^{(n_k)}=\xi^{(m)}(x).\]
\end{itemize}
\end{definition}

\begin{remark}\label{rem:Frenet normalization}
As a consequence of Remark \ref{rem:normalization}(1), given any two Frenet curves $\xi,\xi':S^1\to\Fmc(V)$ and any triple of pairwise distinct points $x_1,x_2,x_3\in S^1$, there is a unique projective transformation $g\in\PGL(V)$ such that 
\[g\cdot\left(\xi(x_1),\xi(x_2),\xi^{(1)}(x_3)\right)=\left(\xi'(x_1),\xi'(x_2),{\xi'}^{(1)}(x_3)\right).\]
In particular, if $\xi$ and $\xi'$ are projectively equivalent, then $g\cdot\xi=\xi'$.
\end{remark}

Since $\partial\Gamma$ is a topological circle, it makes sense to say whether a continuous map $\xi:\partial\Gamma\to\Fmc(V)$ is Frenet. Labourie proved that every $\PGL(V)$-Hitchin representation $\rho:\Gamma\to\PGL(V)$ preserves a unique $\rho$-equivariant Frenet map $\xi_\rho:\partial\Gamma\to\Fmc(V)$. Using this, he deduces that every $\PGL(V)$-Hitchin representation is an irreducible quasi-isometric embedding (with respect to any word metric on $\Gamma$ and any left-invariant Riemannian metric on $\PGL(V)$). He also conjectured that for any representation $\rho:\Gamma\to\PGL(V)$, the existence of such a $\rho$-equivariant Frenet map implies that $\rho$ is a $\PGL(V)$-Hitchin representation. This was later proven by Guichard, thus establishing the following theorem.

\begin{thm}\cite[Theorem 1]{Guichard}, \cite[Theorem 1.4]{Labourie2006} \label{thm:Labourie,Guichard}
A representation $\rho:\Gamma\to\PGL(V)$ is a $\PGL(V)$-\emph{Hitchin representation} if and only if there exists a $\rho$-equivariant Frenet map $\xi_\rho:\partial\Gamma\to\Fmc(V)$. If such a $\rho$-equivariant Frenet map exists, then it is unique. Furthermore, for any pair $\rho,\rho'$ of $\PGL(V)$-Hitchin representations, $\xi_\rho=\xi_{\rho'}$ if and only if $\rho=\rho'$.
\end{thm}

In particular, this establishes a natural bijection between $\Hit_V(S)$ and the set of projective classes of Frenet curves $\xi:\partial\Gamma\to\Fmc(V)$ that are $\rho$-equivariant for some representation $\rho:\Gamma\to\PGL(V)$. At about the same time, Fock and Goncharov \cite{FockGoncharov} also gave a more algebraic version of the criterion in Theorem \ref{thm:Labourie,Guichard}, where they replace the Frenet property with the notion of a ``positive map". 

As a consequence of Theorem \ref{thm:Labourie,Guichard}, Labourie deduces the following.

\begin{cor}\cite[Proposition 3.2]{Labourie2006}\label{cor:Labourie}
Let $\rho:\Gamma\to\PGL(V)$ be a $\PGL(V)$-Hitchin representations with $\xi_\rho:\partial\Gamma\to\Fmc(V)$ as its $\rho$-equivariant map. Then for any $\gamma\in\Gamma$ that is not the identity, $\rho(\gamma)$ has $\xi_\rho(\gamma^-)$ and $\xi_\rho(\gamma^+)$ as its attracting and repelling fixed point in $\Fmc(V)$, where $\gamma^-$ and $\gamma^+$ are the attracting and repelling fixed point of $\gamma$ in $\partial\Gamma$. Furthermore, any (some) linear representative of $\rho(\gamma)$ is diagonalizable over $\Rbbb$ with pairwise distinct eigenvalues, all of which have the same sign.
\end{cor}

%%%%%%%%%%%%%%%%%%%%%%%%%%%%%%%%%%%%%%%%%%%%%%%%%%
\subsection{Positivity of $\PGL(V)$-Hitchin representations}\label{sec:positivity}
%%%%%%%%%%%%%%%%%%%%%%%%%%%%%%%%%%%%%%%%%%%%%%%%%%
The description of Hitchin representations via Frenet curves allows one to prove a striking positivity feature for Hitchin representations, which was first observed by Fock and Goncharov \cite{FockGoncharov}. To describe this, we will define two well-known projective invariants associated to transverse pairs and triples of flags in $\Fmc(V)$.

Let $F_1,F_2,F_3$ be a triple of flags in $\Fmc(V)$. For $m=1,2,3$, choose a basis $\{f_{m,1},\dots,f_{m,n}\}$ of $V$ such that $F_m^{(j)}=\Span_\Rbbb(f_{m,1},\dots,f_{m,j})$ for all $j=1,\dots,n-1$. For each triple $\mathbf i=(i_1,i_2,i_3)$ of non-negative integers that are weakly between $0$ and $n-1$, and that sum to $n$, define 
\[F_1^{(i_1)}\wedge F_2^{(i_2)}\wedge F_3^{(i_3)}:=\det(f_{1,1},\dots,f_{1,i_1},f_{2,1},\dots,f_{2,i_2},f_{3,1},\dots,f_{3,i_3}).\]
If $(F_1,F_2,F_3)$ is a generic triple of flags in $\Fmc(V)$, then $F_1^{(i_1)}\wedge F_2^{(i_2)}\wedge F_3^{(i_3)}\neq 0$. 

\begin{definition}\label{def:ratio} \ %Let $\Amc$ denote the set of pairs of positive integers that sum to $n$, and let $\Bmc$ denote the set of triples of positive integers that sum to $n$.
\begin{enumerate} 
\item Let $F_1,G_1,F_2,G_2$ be flags in $\Fmc(V)$ so that $(F_1,G_1,F_2)$ and $(F_1,F_2,G_2)$ are triples of generic flags in $\Fmc(V)$. For any pair $\mathbf k=(k_1,k_2)$ of positive integers that sum to $n$, define the \emph{cross ratio} by
\[C^\mathbf k(F_1,G_1,F_2,G_2):=\frac{F_1^{(k_1)}\wedge F_2^{(k_2-1)}\wedge G_2^{(1)}}{F_1^{(k_1)}\wedge F_2^{(k_2-1)}\wedge G_1^{(1)}}\cdot\frac{F_1^{(k_1-1)}\wedge F_2^{(k_2)}\wedge G_1^{(1)}}{F_1^{(k_1-1)}\wedge F_2^{(k_2)}\wedge G_2^{(1)}}.\]
\item Let $(F_1,F_2,F_3)$ be a triple of generic flags in $\Fmc(V)$. For any triple $\mathbf i=(i_1,i_2,i_3)$ of positive integers that sum to $n$, define the \emph{triple ratio}
\begin{align*}
T^\mathbf i(F_1,F_2,F_3):=\frac{F_1^{(i_1+1)}\wedge F_2^{(i_2)}\wedge F_3^{(i_3-1)}}{F_1^{(i_1+1)}\wedge F_2^{(i_2-1)}\wedge F_3^{(i_3)}}\cdot&\frac{F_1^{(i_1-1)}\wedge F_2^{(i_2+1)}\wedge F_3^{(i_3)}}{F_1^{(i_1)}\wedge F_2^{(i_2+1)}\wedge F_3^{(i_3-1)}}\\
&\cdot\frac{F_1^{(i_1)}\wedge F_2^{(i_2-1)}\wedge F_3^{(i_3+1)}}{F_1^{(i_1-1)}\wedge F_2^{(i_2)}\wedge F_3^{(i_3+1)}}.
\end{align*}
\end{enumerate}
\end{definition}

The definition of $F_1^{(i_1)}\wedge F_2^{(i_2)}\wedge F_3^{(i_3)}$ depends on the bases $\{f_{m,1},\dots,f_{m,n}\}$ for $m=1,2,3$, but one can easily verify that the cross ratio and triple ratio do not. Another important but easily verified property of the cross ratio and triple ratio is their projective invariance, i.e. for any $g\in\PGL(V)$, 
\begin{itemize}
\item $C^\mathbf k(F_1,G_1,F_2,G_2)=C^\mathbf k(g\cdot F_1,g\cdot G_1,g\cdot F_2,g\cdot G_2)$ for all $\mathbf k$, and 
\item $T^\mathbf i(F_1,F_2,F_3)=T^\mathbf i(g\cdot F_1,g\cdot G_2,g\cdot F_3)$ for all $\mathbf i$.
\end{itemize}
It is also easy to check the following symmetries from the definition:
\begin{itemize}
\item For any pair of positive integers $\mathbf k=(k_1,k_2)$ that sum to $n$, let $\bar{\mathbf k}:=(k_2,k_1)$. Then 
\begin{align}\label{eqn: cross ratio symmetry}
C^\mathbf{k} (F_1,G_1,F_2,G_2)=C^{\bar{\mathbf{k}}}(F_2,G_2,F_1,G_1)=\frac{1}{C^{\mathbf k}(F_1,G_2,F_2,G_1)}.
\end{align}
\item For any triple $\mathbf i=(i_1,i_2,i_3)$ of positive integers that sum to $n$, let $\mathbf i_+:=(i_2,i_3,i_1)$, let $\mathbf i_-:=(i_3,i_1,i_2)$, and let $\bar{\mathbf i}:=(i_3,i_2,i_1)$, where the arithmetic in the subscripts are done modulo $3$. Then
\begin{align}\label{eqn: triple ratio symmetry}
T^{\mathbf i}(F_1,F_2,F_3)=T^{\mathbf i_+}(F_2,F_3,F_1)=T^{\mathbf i_-}(F_3,F_1,F_2)=\frac{1}{T^{\bar{\mathbf i}}(F_3,F_2,F_1)}.
\end{align}
\end{itemize}

For any pairwise distinct triple $\mathbf x=(x_1,x_2,x_3)$ in $S^1$ and any Frenet curve $\xi:S^1\to\Fmc(V)$, $\xi(\mathbf x)$ is a generic triple of flags in $\Fmc(V)$. Thus, we can evaluate cross ratios and triple ratios along the image of $\xi$. The following theorem is due to Fock-Goncharov \cite{FockGoncharov} (also see \cite[Appendix B]{Labourie-McShane} or \cite[Proposition 2.5.7]{Zhang_thesis}). 

\begin{thm}\label{thm:Zhang}
Let $\xi:S^1\to\Fmc(V)$ be a Frenet curve, and equip $S^1$ with the clockwise cyclic order.
\begin{enumerate}
\item For any pairwise distinct triple of points $\mathbf x=(x_1,x_2,x_3)$ in $S^1$, and for any triple $\mathbf i=(i_1,i_2,i_3)$ of positive integers that sum to $n$, we have
\[T^\mathbf i\big(\xi(\mathbf x)\big)>0.\]
\item For any pairwise distinct quadruple of points $r_1,s_1,r_2,s_2\in S^1$ such that $r_1<s_1<r_2<s_2<r_1$ in the cyclic ordering on $S^1$, and for any pair $\mathbf k=(k_1,k_2)$ of positive integers that sum to $n$, we have
\[C^\mathbf k\big(\xi(r_1),\xi(s_1),\xi(r_2),\xi(s_2)\big)<0.\]
\end{enumerate}
\end{thm}

%%%%%%%%%%%%%%%%%%%%%%%%%%%%%%%%%%%%%%%%%%%%%%%%%%
\subsection{Real-analytic structure of the $\PGL(V)$-Hitchin representations}\label{sec:real-analytic}
%%%%%%%%%%%%%%%%%%%%%%%%%%%%%%%%%%%%%%%%%%%%%%%%%%

It is a straightforward consequence of Theorem \ref{thm:Labourie,Guichard} that the centralizer of every $\PGL(V)$-Hitchin representation is trivial, so by Goldman \cite[Proposition~1.2]{Goldmansymplectic}, $\widetilde{\Hit}_V(S)$ lies in the smooth locus of the real algebraic variety $\Hom(\Gamma,\PGL(V))$. As such, $\widetilde{\Hit}_V(S)$ is naturally equipped with a real-analytic structure. This real-analytic structure has the defining property that for every $\gamma\in\Gamma$, the map 
\[\widetilde{\Hit}_V(S)\to G\]
given by $\rho\mapsto\rho(\gamma)$ is real-analytic. 

Theorem \ref{thm:Labourie,Guichard} also implies that every $\PGL(V)$-Hitchin representation is irreducible, see \cite[Proposition 14]{Guichard}. Thus, by Johnson-Millson \cite[Section 1]{Johnson1987} (also see Goldman \cite[Section~1.3]{Goldmansymplectic}), the $\PGL(V)$-action on $\widetilde{\Hit}_V(S)$ by conjugation is real-analytic and proper, so $\Hit_V(S)$ has the structure of a real-analytic manifold with the defining property that the quotient map $\pi:\widetilde{\Hit}_V(S)\to\Hit_V(S)$ is real-analytic.  With this real-analytic structure, Hitchin's proof \cite{Hitchin} in fact shows that $\Hit_V(S)$ is real-analytically diffeomorphic to $\Rbbb^{(2g-2)(n^2-1)}$ (also see an independent proof due to Bonahon-Dreyer \cite{BonahonDreyer1}).

We will now describe the real-analytic structure on $\Hit_V(S)$ explicitly by specifying an explicit real-analytic embedded submanifold of $\Hmc\subset\widetilde{\Hit}_V(S)$ (depending on some choices), such that $\pi|_\Hmc:\Hmc\to\Hit_V(S)$ restricted to $\Hmc$ is a real-analytic diffeomorphism onto $\Hit_V(S)$. This will also be used later in the description of the tangent cocycle in Section~\ref{sec: tangent}.

Our description of the real-analytic structure on $\Hit_V(S)$ relies on the specialization of a result of Bridgeman, Canary, Labourie, and Sambarino \cite{BCLS} to $\PGL(V)$-Hitchin representations. %Recall that for every $\PGL(V)$-Hitchin representation $\rho:\Gamma\to\PGL(V)$, $\xi_\rho:\partial\Gamma\to\Fmc(V)$ denotes the $\rho$-equivariant Frenet curve. 

\begin{thm}\label{thm: BCLS} \cite[Theorem 6.1]{BCLS} For every $x\in\partial\Gamma$, the map
\[\widetilde{\Hit}_V(S)\to\Fmc(V)\]
given by $\rho\mapsto\xi_\rho(x)$ is real-analytic. 
\end{thm}

We say that a triple $(F,G,p)\in\Fmc(V)^2\times\Pbbb(V)$ is a \emph{normalizing triple} if $(F,G)$ is a transverse pair of flags and 
\[p+F^{(i)}+G^{(n-i-1)}=V\]
for all $i=0,\dots,n-1$. For any normalizing triple $(F,G,p)$ and any triple ${\bf x}=(x_1,x_2,x_3)$ of pairwise distinct points in $\partial\Gamma$, let 
\begin{align}\label{eqn: submanifold}
\Hmc((F,G,p),{\bf x}):=\left\{\rho\in\widetilde{\Hit}_V(S):\left(\xi_\rho(x_1),\xi_\rho(x_2),\xi_\rho^{(1)}(x_3)\right)=(F,G,p)\right\}.
\end{align}
For any $g\in\PGL(V)$, let $c_g:\PGL(V)\to\PGL(V)$ be the map given by $h\mapsto ghg^{-1}$. The following is a corollary of Theorem \ref{thm: BCLS}.

\begin{cor}\label{cor: BCLS1}Fix a normalizing triple $(F,G,p)$ and any triple ${\bf x}=(x_1,x_2,x_3)$ of pairwise distinct points in $\partial\Gamma$, and let $\Hmc:=\Hmc((F,G,p),{\bf x})$.
\begin{enumerate}
\item $\Hmc\subset\widetilde{\Hit}_V(S)$ is an embedded real-analytic submanifold.
\item The map $\Itd:\Hmc\times \PGL(V)\to\widetilde{\Hit}_V(S)$ given by $\Itd:(\rho,g)\mapsto c_g\circ\rho$ is a real-analytic diffeomorphism. In particular, the obvious quotient map $\pi:\widetilde{\Hit}_V(S)\to\Hit_V(S)$ restricts to a real-analytic diffeomorphism $\pi|_{\Hmc}:\Hmc\to\Hit_V(S)$.
\end{enumerate}
\end{cor}

\begin{proof}
Define the surjective map
\begin{equation*}\label{eqn: smooth f}
f:\widetilde{\Hit}_V(S)\to\PGL(V)
\end{equation*}
with the defining property that $f(\rho)\cdot(F,G,p)=\left(\xi_\rho(x_1),\xi_\rho(x_2),\xi_\rho^{(1)}(x_3)\right)$. It follows from Theorem~\ref{thm: BCLS} that $f$ is real-analytic. 

(1) Since $\Hmc=f^{-1}(\id)$, by the constant rank level set theorem, it suffices to show that $f$ is regular, i.e. ${\rm d}f_\rho:T_{\rho}\widetilde{\Hit}_V(S)\to T_{f(\rho)}\PGL(V)$ is surjective for every $\rho\in\widetilde{\Hit}_V(S)$. Let $X\in T_{f(\rho)}\PGL(V)$, and let $t\mapsto g_t$ be a smooth path in $\PGL(V)$ such that $g_0=f(\rho)$ and $\frac{d}{dt}|_{t=0}g_t=X$. Note that $t\mapsto \rho_t:=c_{g_tg_0^{-1}}\circ\rho$ is a smooth path in $\widetilde{\Hit}_V(S)$, and $\xi_{\rho_t}=g_tg_0^{-1}\cdot\xi_\rho$ for all $t$. Thus, 
\begin{align*}
f(\rho_t)\cdot(F,G,p)&=\left(\xi_{\rho_t}(x_1),\xi_{\rho_t}(x_2),\xi_{\rho_t}^{(1)}(x_3)\right)=g_tg_0^{-1}\cdot\left(\xi_\rho(x_1),\xi_\rho(x_2),\xi_\rho^{(1)}(x_3)\right)\\
&=g_t\cdot (F,G,p),
\end{align*}
so $f(\rho_t)=g_t$. This means that ${\rm d}f_\rho(\frac{d}{dt}|_{t=0}\rho_t)=\frac{d}{dt}|_{t=0}f(\rho_t)=X$, which proves that $f$ is regular.

(2) Observe that $\Itd^{-1}$ is given by $\Itd^{-1}(\rho):=(c_{f(\rho)^{-1}}\circ\rho,f(\rho))$. Since $f$ is real-analytic, it follows that $\Itd$ is a real-analytic diffeomorphism. For the second claim, note that $\Itd$ is $\PGL(V)$-equivariant, where the $\PGL(V)$ action on $\Hmc\times\PGL(V)$ is given by the trivial action on the first factor and left multiplication in the second factor, while the $\PGL(V)$ action on $\widetilde{\Hit}_V(S)$ is by conjugation. Thus, $\Itd$ descends to a real-analytic diffeomorphism
\begin{align}\label{eqn: I}
I:\Hmc\to\Hit_V(S).
\end{align}
Observe that $I$ associates every $\PGL(V)$-Hitchin representation in $\Hmc$ to its conjugacy class in $\Hit_V(S)$, so $I=\pi|_{\Hmc}$. 
\end{proof}

Since $\Hmc((F,G,p),{\bf x})$ is a real-analytic embedded submanifold of $\widetilde{\Hit}_V(S)$, the following corollary is immediate from Theorem \ref{thm: BCLS}.

\begin{cor}\label{cor: smooth}Fix a normalizing triple $(F,G,p)$ and any triple ${\bf x}=(x_1,x_2,x_3)$ of pairwise distinct points in $\partial\Gamma$, and let $\Hmc:=\Hmc((F,G,p),{\bf x})$.
\begin{enumerate}
\item For any triple ${\bf y}=(y_1,y_2,y_3)$ of pairwise distinct points in $\partial\Gamma$ and $\rho\in\Hmc$, let $g(\rho)\in\PGL(V)$ be defined by
\[g(\rho)\cdot(F,G,p)=\left(\xi_\rho(y_1),\xi_\rho(y_2),\xi_\rho^{(1)}(y_3)\right).\] 
Then the map $g:\Hmc\to\PGL(V)$ is real-analytic.
\item For any point $r$ in $\partial\Gamma$ distinct from $x_1$  and $\rho\in\Hmc$, let $u(\rho)\in\PGL(V)$ be the unique unipotent projective transformation defined by
\[u(\rho)\cdot \left(F,G\right)=\left(\xi_\rho(x_1),\xi_\rho(r)\right).\]
Then the map $u:\Hmc\to\PGL(V)$ is real-analytic.
\end{enumerate}
\end{cor}

%\begin{proof}
%Since $t\mapsto[\rho_t]$ is a smooth path in $\Hit_V(S)$, by definition, there exists a smooth path $t\mapsto\rho'_t$ in $\widetilde\Hit_V(S)$ such that $\rho'_t$ is a representative of the conjugacy class $[\rho_t]$ for all $t$. By conjugating, we may assume that $\rho_0'=\rho_0$. Let $\xi'_t$ denote the $\rho'_t$-equivariant Frenet curve. Then by Theorem \ref{thm: BCLS}, for every $x\in\partial\Gamma$, the map $t\mapsto\xi'_t(x)$ is a smooth path in $\Fmc(V)$. It follows that if $a(t)$ and $b(t)$ are the elements in $\PGL(V)$ defined by
%\[a(t)\cdot\left(\xi_0(x_1),\xi_0(x_2),\xi_0^{(1)}(x_3)\right)=\left(\xi'_t(x_1),\xi'_t(x_2),{\xi'_t}^{(1)}(x_3)\right)\] 
%and
%\[b(t)\cdot\left(\xi_0(y_1),\xi_0(y_2),\xi_0^{(1)}(y_3)\right)=\left(\xi'_t(y_1),\xi'_t(y_2),{\xi'_t}^{(1)}(y_3)\right)\]
%(see Remark~\ref{rem:normalization}(1)), then $t\mapsto a(t)$ and $t\mapsto b(t)$ are smooth paths in $\PGL(V)$. By Remark \ref{rem:Frenet normalization}, $a(t)^{-1}\cdot\xi_t'=\xi_t$, so the first claim of the lemma holds.

%Note that $g_{{\bf x},{\bf y}}(t)$ in (1) is well-defined by Remark~\ref{rem:normalization}(1). Then (1) holds because $g_{{\bf x},{\bf y}}(t)=a(t)^{-1}b(t)$, while (2) is a consequence of (1) and the observation that $\rho_t(\gamma)=g_{{\bf x},\gamma\cdot{\bf x}}(t)\rho_0(\gamma)$ for all $\gamma\in\Gamma$.
%\end{proof}

%%%%%%%%%%%%%%%%%%%%%%%%%%%%%%%%%%%%%%%%%%%%%%%%%%
\subsection{The tangent space to $\Hit_V(S)$}\label{sec: tangent}
%%%%%%%%%%%%%%%%%%%%%%%%%%%%%%%%%%%%%%%%%%%%%%%%%%
We now recall a remarkable observation by Weil \cite{Weil} (also see Goldman  \cite[Section 1]{Goldmansymplectic}) that gives a chomological description of the tangent space to $\Hit_V(S)$ at any point $[\rho]\in\Hit_V(S)$. This is essential in order to define the Goldman symplectic form. Although this chomological interpretation holds in much greater generality, we will only describe this for $\Hit_V(S)$.

Let $\mathfrak{sl}(V)$ denote the Lie algebra of $\PGL(V)$ and let $\rho:\Gamma\to\PGL(V)$ be a Hitchin representation. We denote the first (group) cohomology of $\Gamma$ and the first (singular) cohomology of $S$, both with coefficients in $\mathfrak{sl}(V)$ twisted by $\Ad\circ\rho$, by 
\[H^1(\Gamma,\smf\lmf(V)_{\Ad\circ\rho})\,\text{ and }\,H^1(S,\smf\lmf(V)_{\Ad\circ\rho})\] 
respectively. We also denote the corresponding real vector spaces of $1$-cocycles by 
\[C^1(\Gamma,\smf\lmf(V)_{\Ad\circ\rho})\,\text{ and }\,C^1(S,\smf\lmf(V)_{\Ad\circ\rho})\] 
respectively. 

Weil observed that there is a canonical isomorphism
\begin{equation}\label{eqn: Weil 1}
T_{\rho}\widetilde{\Hit}_V(S)\simeq C^1(\Gamma,\smf\lmf(V)_{\Ad\circ\rho}).
\end{equation}
This isomorphism can be described explicitly: Given a tangent vector $X\in T_{\rho}\widetilde{\Hit}_V(S)$, let $t\mapsto\rho_t$ be a smooth path in $\widetilde{\Hit}_V(S)$ such that $\rho_0=\rho$ and $\frac{d}{dt}|_{t=0}\rho_t=X$. Then for any $\gamma\in\Gamma$, $t\mapsto\rho_t(\gamma)\rho(\gamma)^{-1}$ is a smooth path in $\PGL(V)$. Since $\rho_t$ is a homomorphism for all $t$, one can verify that the map $\mu_X:\Gamma\to\smf\lmf(V)$ defined by 
\[\mu_X:\gamma\mapsto\frac{d}{dt}\bigg|_{t=0}\rho_t(\gamma)\rho(\gamma)^{-1}\in\mathfrak{sl}(V)\] 
is a $1$-cocycle in $C^1(\Gamma,\smf\lmf(V)_{\Ad\circ\rho})$, i.e. 
\[\mu_X(\gamma_1\gamma_2)=\mu_X(\gamma_1)+\Ad\circ\rho(\gamma_1)\cdot\mu_X(\gamma_2)\] 
for all $\gamma_1,\gamma_2\in\Gamma$. Then the map $T_{\rho}\widetilde{\Hit}_V(S)\to C^1(\Gamma,\smf\lmf(V)_{\Ad\circ\rho})$ given by $X\mapsto\mu_X$ is the required isomorphism. 

Furthermore, via the isomorphism \eqref{eqn: Weil 1}, the coboundaries in $C^1(\Gamma,\smf\lmf(V)_{\Ad\circ\rho})$, i.e. $1$-cocycles of the form 
\[\gamma\mapsto A-\Ad\circ\rho(\gamma)\cdot A\] 
for some $A\in\mathfrak{sl}(V)$, are identified with the tangent vectors in $T_{\rho}\widetilde{\Hit}_V(S)$ to paths of the form $t\mapsto c_{g_t}\circ\rho$ for some smooth path $t\mapsto g_t$ in $\PGL(V)$ with $g_0=\id$. As such, \eqref{eqn: Weil 1} descends to an isomorphism 
\begin{equation}\label{eqn: Weil 2}
T_{[\rho]} \Hit_V(S)\simeq H^1(\Gamma,\smf\lmf(V)_{\Ad\circ\rho}).
\end{equation}

%Fix a triple ${\bf x}=(x_1,x_2,x_3)$ of pairwise distinct points in $\partial\Gamma$, and let 
%\[\Hmc_{\rho,{\bf x}}:=\left\{\rho'\in\widetilde{\Hit}_V(S):\left(\xi_{\rho'}(x_1),\xi_{\rho'}(x_1),\xi_{\rho'}^1(x_1)\right)=\left(\xi_{\rho}(x_1),\xi_{\rho}(x_1),\xi_{\rho}^1(x_1)\right)\right\}.\] 
%Note that $\rho\in\Hmc_{\rho,{\bf x}}$. By Corollary \ref{cor: BCLS1}, $\Hmc_{\rho,{\bf x}}$ is an embedded submanifold of $\widetilde{\Hit}_V(S)$, and if $\pi:\widetilde{\Hit}_V(S)\to\Hit_V(S)$ is the obvious quotient map, then
%\[d\pi|_{T_\rho\Hmc}:T_\rho\Hmc\to T_{[\rho]}\Hit_V(S)\]
%is an isomorphism. In other words, the choice of ${\bf x}$ determines a cocycle representative in $T_\rho\Hmc\subset T_\rho\widetilde\Hit_V(S)\simeq C^1(\Gamma,\mathfrak{sl}(V)_{\Ad\circ\rho})$ for each cohomology class in $T_{[\rho]}\Hit_V(S)\simeq H^1(\Gamma,\smf\lmf(V)_{\Ad\circ\rho})$.

At the same time, there is a well-known isomorphism
\begin{equation}\label{eqn: Weil 3}
H^1(\Gamma,\smf\lmf(V)_{\Ad\circ\rho})\simeq H^1(S,\smf\lmf(V)_{\Ad\circ\rho})
\end{equation}
which can also be described explicitly. Let $\nu\in C^1(S,\smf\lmf(V)_{\Ad\circ\rho})$ be a $1$-cocycle in the cohomology class $[\nu]\in H^1(S,\smf\lmf(V)_{\Ad\circ\rho})$, and let $\widetilde{\nu}: C_1(\Std,\Zbbb)\to\smf\lmf(V)$ be its $\Ad\circ\rho$-equivariant lift, where $\widetilde{S}$ is the universal cover of $S$. Pick any point $o\in\Std$. For any $\gamma\in\Gamma$, let $h_\gamma:[0,1]\to\Std$ be a $1$-simplex such that $h_\gamma(0)=o$ and $h_\gamma(1)=\gamma\cdot o$. Then define $\bar{\nu}:\Gamma\to\smf\lmf(V)$ by $\bar{\nu}:\gamma\mapsto\widetilde{\nu}(h_\gamma)$. One can verify that $\bar{\nu}$ is a well-defined $1$-cocycle in $C^1(\Gamma,\smf\lmf(V)_{\Ad\circ\rho})$, and that the map $H^1(S,\smf\lmf(V)_{\Ad\circ\rho})\to H^1(\Gamma,\smf\lmf(V)_{\Ad\circ\rho})$ given by $[\nu]\mapsto[\bar{\nu}]$ is an isomorphism that does not depend on the choice of $o$.

The composite of the isomorphisms \eqref{eqn: Weil 2} and \eqref{eqn: Weil 3} give an isomorphism 
\begin{align}\label{eqn: F}
F_\rho:T_{[\rho]}\Hit_V(S)\to H^1(S,\smf\lmf(V)_{\Ad\circ\rho}).
\end{align}
This is the required cohomological description of $T_{[\rho]}\Hit_V(S)$.

%%%%%%%%%%%%%%%%%%%%%%%%%%%%%%%%%%%%%%%%%%%%%%%%%%
\subsection{The Goldman symplectic form}\label{sec:Goldman symplectic form}
%%%%%%%%%%%%%%%%%%%%%%%%%%%%%%%%%%%%%%%%%%%%%%%%%%

Recall that the Lie algebra $\smf\lmf(V)$ of $\PGL(V)$ is naturally identified with the traceless endomorphisms of $V$. Hence, one can define the \emph{trace form}
\[B:\smf\lmf(V)\times\smf\lmf(V)\to\Rbbb\]
by $B(X,Y):=\tr(XY)$. This pairing is well-known to be bilinear, symmetric, non-degenerate, and $\Ad(\PGL(V))$-invariant, i.e. is invariant under conjugation by group elements in $\PGL(V)$. Using the trace form, we apply a general construction by Goldman \cite{Goldmansymplectic} to define a symplectic form on $\Hit_V(S)$, called the \emph{Goldman symplectic form}. 

\begin{definition} \label{def:Goldman}
Let $[\rho]\in\Hit_V(S)$.
\begin{enumerate}
\item For any  be a representative $\rho$ in $[\rho]$, define the pairing 
\[\omega_\rho:H^1(S,\smf\lmf(V)_{\Ad\circ\rho})\times H^1(S,\smf\lmf(V)_{\Ad\circ\rho})\to H^2(S,\Rbbb)\to\Rbbb\]
where the first map is the cup product on cohomology while pairing the $\smf\lmf(V)_{\Ad\circ\rho}$ coefficients using the trace form, and the second map is taking the cap product with the fundamental class $[S]\in H_2(S,\Rbbb)$ (given by the orientation on $S$). 
\item The \emph{Goldman symplectic form} is the pull-back pairing 
\[\omega_{[\rho]}:=F_\rho^*(\omega_\rho):T_{[\rho]}\Hit_V(S)\times T_{[\rho]}\Hit_V(S)\to\Rbbb.\]
\end{enumerate}
\end{definition}

Using the fact that the trace form $B$ is $\Ad(\PGL(V))$-invariant, one can verify that $\omega_\rho$ is well-defined and $\omega_{[\rho]}$ does not depend on the choice of representative $\rho$ in $[\rho]$. Also, Goldman \cite[Theorem 1.7]{Goldmansymplectic} proved that $\omega$ is a symplectic form on $\Hit_V(S)$, i.e if we view $\omega:[\rho]\mapsto\omega_{[\rho]}$ as a $2$-form on $\Hit_V(S)$, then $d\omega=0$.

\begin{remark}From the work of Atiyah-Bott \cite{AtiyahBott}, one also has a description of this symplectic form in the language of flat bundles over $S$. In the case when $\dim(V)=2$, Goldman  \cite[Proposition 2.5]{Goldmansymplectic} showed that this symplectic form agrees with (a multiple of) the Weil-Petersson symplectic form on the Teichm\"uller space of $S$. 
\end{remark}

One can use singular homology and cohomology to compute the Goldman symplectic form explicitly. To do so, one needs an explicit $2$-cycle that represents the fundamental class $[S]$ in $H_2(S,\mathbb{R})$. A standard way to specify such a $2$-cycle is to choose a finite triangulation $\Tbbb$ of $S$ such that the three vertices of every triangle in $\Tbbb$ are pairwise distinct, and to choose an enumeration $v_1,\dots,v_k$ of the vertices of $\Tbbb$. Then for any triangle $\widehat{\delta}$ in $\Tbbb$, let $v_a$, $v_b$, and $v_c$ be the vertices of $\widehat\delta$ such that $a<b<c$, let $\sigma_{\widehat\delta}$ denote the $2$-simplex $[v_a,v_b,v_c]$ in $S$, let $e_{\widehat\delta,1}$ denote the $1$-simplex $[v_a,v_b]$, and let $e_{\widehat\delta,2}$ denote the $1$-simplex $[v_b,v_c]$. One can then verify \cite[Chapter 1]{munkres2018elements} that the fundamental class $[S]$ in $H_2(S,\mathbb{R})$ is represented by the $2$-cycle
\begin{equation*}
\sum_{\widehat\delta\in\Tbbb}\sgn(\widehat\delta)\cdot\sigma_{\widehat\delta},
\end{equation*}
where $\sgn(\widehat\delta)$ is $1$ if the enumeration of the vertices of $\widehat\delta$ is increasing in the anti-clockwise order (in the orientation on $S$), and is $-1$ otherwise. 

Now, for each $X\in T_{[\rho]}\Hit_V(S)$, choose a $1$-cocycle representative 
\[\nu_X\in C^1(S,\smf\lmf(V)_{\Ad\circ\rho})\] 
of the cohomology class $F_{\rho}(X)$, and let $\widetilde{\nu}_X:C_1(\Std,\Zbbb)\to\smf\lmf(V)$ be the $\Ad\circ\rho$-equivariant lift of $\nu_X$. For each triangle $\widehat\delta$ in $\mathbb{T}$, choose a lift $\delta$ to $\widetilde S$ of $\widehat\delta$, and let $\Tbbb_{\rm lift}$ denote the set of all such lifts, one for each triangle in $\Tbbb$. Then for each triangle $\delta$ in $\Tbbb_{\rm lift}$, let $e_{\delta,1}$ and $e_{\delta,2}$ denote the oriented edges of $\delta$ such that $\pi_S(e_{\delta,1})=e_{\widehat\delta,1}$ and $\pi_S(e_{\delta,2})=e_{\widehat\delta,2}$, where $\pi_S:\widetilde S\to S$ is the covering map. In light of the $2$-cycle given above that represents the fundamental class in $H_2(S,\mathbb{R})$, we see that the Goldman symplectic form is computed by
\begin{equation}\label{eqn:formula}
\omega_{[\rho]}\big(X_1,X_2\big)=\sum_{\delta\in\Tbbb_{\rm lift}}\sgn(\delta)\cdot\tr\big(\widetilde{\nu}_{X_1}(e_{\delta,1})\cdot\widetilde{\nu}_{X_2}(e_{\delta,2})\big),
\end{equation}
for any pair of vectors $X_1,X_2\in T_{[\rho]}\Hit_V(S)$, where $\sgn(\delta)$ is $1$ if the enumeration of the vertices of $\pi_S(\delta)$ is increasing in the anti-clockwise order, and is $-1$ otherwise. 

%%%%%%%%%%%%%%%%%%%%%%%%%%%%%%%%%%%%%%%%%%%%%%%%%%%%%%%%%%%%%%%%%%%%%%%%%%%%%%%%%%%%%%%%%%%%%%%%%%%%
\section{Topological constructions}\label{sec:topological constructions}
%%%%%%%%%%%%%%%%%%%%%%%%%%%%%%%%%%%%%%%%%%%%%%%%%%%%%%%%%%%%%%%%%%%%%%%%%%%%%%%%%%%%%%%%%%%%%%%%%%%%

To compute the Goldman symplectic pairing using \eqref{eqn:formula}, one needs to choose, for each $X\in T_{[\rho]}\Hit_V(S)$, a cocycle representative in $C^1(S,\smf\lmf(V)_{\Ad\circ\rho})$ for the cohomology class $F_\rho(X)\in H^1(S,\smf\lmf(V)_{\Ad\circ\rho})$ (the map $F_\rho$ was defined by \eqref{eqn: F}). In general, there is no canonical way to choose such cocycle representatives. Thus, in order to determine a cocycle representative for each $F_{\rho}(X)$, we will have to make some topological choices on $S$. In this section we will describe these topological choices. 

Fix once and for all a hyperbolic metric on $S$. We will recall the notion of an (oriented and unoriented) ideal triangulation $\Tmc$ on $S$. Using this, we then define the (oriented and unoriented) barrier systems on $S$ associated to $\Tmc$, as well as the bridge systems on $S$ compatible with $\Tmc$. Even though we chose a hyperbolic metric on $S$, this is for purely expository reasons; all the notions we define here can be described using only the topology of $S$. However, we do not give these topological descriptions because they are not needed for our purposes.

%%%%%%%%%%%%%%%%%%%%%%%%%%%%%%%%%%%%%%%%%%%%%%%%%%
\subsection{Ideal triangulations on $S$}\label{sec:ideal triangulations}
%%%%%%%%%%%%%%%%%%%%%%%%%%%%%%%%%%%%%%%%%%%%%%%%%%
First, we recall the notion of an ideal triangulation on $S$. For more details, see \cite{Thurstonbook, Cassonbook, Pennerbook, Bo01}. 

An \emph{ideal triangulation} on $S$ is a collection $\Tmc$ of FINITELY many pairwise non-intersecting simple geodesics in $S$, whose union is a closed subset of $S$ that cuts $S$ into finitely many open sets, each of which is a hyperbolic ideal triangle. We refer to the geodesics in $\Tmc$ as the \emph{edges} in $\Tmc$, and refer to the open sets in the complement of $\Tmc$ as the \emph{ideal triangles} of $\Tmc$.

There are two possible kinds of edges in $\Tmc$, the bi-infinite ones and the closed ones. The bi-infinite edges are isolated in $S$, i.e. for every bi-infinite edge $e$, there is an open set $U\subset S$ containing $e$, such that $U$ does not intersect any other edge in $\Tmc$. However, they are not closed as subsets in $S$; they accumulate onto closed geodesics in $\Tmc$ in both directions. On the other hand, the closed edges in $S$ are closed as subsets of $S$, but they are not isolated; in fact, for every closed edge $e$ and every collar neighborhood $U\subset S$ containing $e$, each connected component of $U\setminus e$ intersects an isolated edge in $\Tmc$. So that these notions still make sense in $\widetilde{S}$, we will refer to the bi-infinite edges in $S$ as \emph{isolated edges}, and the closed edges in $S$ as \emph{non-isolated edges}. Define $\widetilde{\Tmc}$ to be the set of geodesics in $\widetilde{S}$ that are mapped via the covering map $\pi_S:\widetilde{S}\to S$ to a geodesic in $\Tmc$. We say that an edge in $\widetilde{\Tmc}$ is \emph{isolated} (resp. \emph{non-isolated}) if its projection to $S$ is isolated (resp. non-isolated). We also refer to the open sets in the complement of $\widetilde{\Tmc}$ as the \emph{ideal triangles} of $\widetilde{\Tmc}$.

Note that the three edges of every ideal triangle of $\Tmc$ are necessarily isolated edges. As a consequence of the Gauss-Bonnet theorem, if $g$ is the genus of $S$, then there are $4g-4$ ideal triangles given by $\Tmc$, $6g-6$ isolated edges in $\Tmc$, and (weakly) between $1$ and $3g-3$ non-isolated edges in $\Tmc$.

Given an ideal triangulation $\Tmc$ on $S$ (resp. $\widetilde{\Tmc}$ on $\widetilde{S}$), we define the \emph{associated oriented ideal triangulation}, denoted $\Tmc^o$ (resp. $\widetilde{\Tmc}^o$), to be the set of oriented geodesics in $S$ (resp. $\widetilde{S}$) whose underlying non-oriented geodesic lies in $\Tmc$ (resp. $\widetilde{\Tmc}$).

%%%%%%%%%%%%%%%%%%%%%%%%%%%%%%%%%%%%%%%%%%%%%%%%%%
\subsection{Barrier systems}\label{sec:barriers}
%%%%%%%%%%%%%%%%%%%%%%%%%%%%%%%%%%%%%%%%%%%%%%%%%%
Next, we define the notion of a barrier system associated to an ideal triangulation $\Tmc$ on $S$. Note that for each ideal triangle $T$ of $\widetilde\Tmc$ and each vertex $x$ of $T$, there is a unique geodesic ray in $\widetilde S$ from the centroid of $T$ to $x$. Let $\widetilde\Mmc=\widetilde{\Mmc}_{\widetilde\Tmc}$ denote the collection of all such geodesic rays, and define $\Mmc$ to be the set of geodesic rays in $S$ that lift to a geodesic ray in $\widetilde{\Mmc}$. We will refer to the geodesic rays in $\Mmc$ (resp. $\widetilde{\Mmc}$) as \emph{non-edge barriers associated to $\Tmc$ (resp. $\widetilde{\Tmc}$)}, see Figure \ref{fig:topological}.

\begin{figure}[ht]
\centering
\includegraphics[scale=0.7]{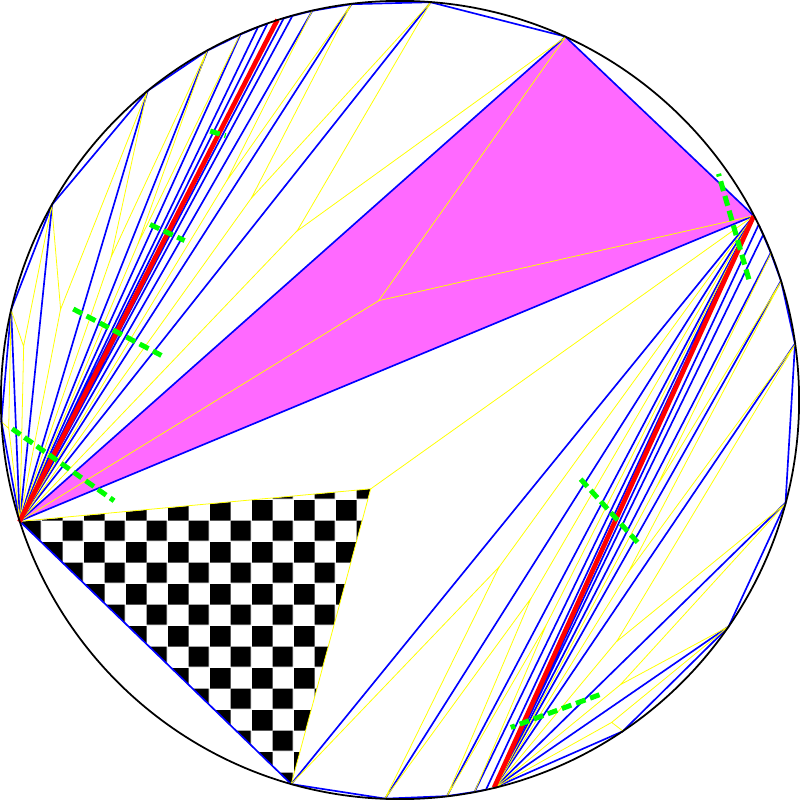}
\small
\caption{This is a drawing in the Klein model of $\widetilde{S}$ of a barrier system $\widetilde{\Dmc}$ associated to an ideal triangulation $\widetilde{\Tmc}$ and a bridge system $\widetilde{\Jmc}$ compatible to $\widetilde{\Tmc}$. The non-isolated edges in $\widetilde{\Tmc}$ are drawn in thick red, the isolated edges in $\widetilde{\Tmc}$ are drawn in blue, and the non-edge barriers in $\widetilde{\Mmc}$ are drawn in thin yellow. The bridges in $\widetilde{\Jmc}$ are realized by geodesic segments drawn in dotted green. One ideal triangle of $\widetilde{\Tmc}$ is shaded violet, and one fragmented ideal triangle is shaded in checkerboard.}\label{fig:topological}
\end{figure}

\begin{definition}\label{def:barrier} 
\begin{enumerate}
\item The \emph{barrier system} in $\widetilde{S}$ (resp. $S$) associated to $\widetilde{\Tmc}$ (resp. $\Tmc$)  is the set $\widetilde{\Dmc}:=\widetilde{\Tmc}\cup\widetilde{\Mmc}$ (resp. $\Dmc:=\Tmc\cup\Mmc$). The elements in $\Dmc$ and $\widetilde{\Dmc}$ are called \emph{barriers}.
\item The \emph{oriented barrier system} associated to $\Tmc$ (resp. $\widetilde{\Tmc}$) is the set  $\Tmc^o\cup\Mmc$ (resp. $\widetilde{\Tmc}^o\cup\widetilde{\Mmc}$), where $\Tmc^o$ (resp. $\widetilde{\Tmc}^o$) is the oriented ideal triangulation associated to $\Tmc$ (resp. $\widetilde{\Tmc}$), see Section \ref{sec:ideal triangulations}.
\end{enumerate}
\end{definition}

%%%%%%%%%%%%%%%%%%%%%%%%%%%%%%%%%%%%%%%%%%%%%%%%%%
\subsection{Bridge systems}\label{sec:bridges}
%%%%%%%%%%%%%%%%%%%%%%%%%%%%%%%%%%%%%%%%%%%%%%%%%%

Finally, we define the notion of a bridge system compatible with an ideal triangulation $\Tmc$ of $S$. To describe this, we use the following notion.

\begin{definition} \label{def:fragmented}
A \emph{fragmented ideal triangle} of $\widetilde{\Dmc}$ is a connected component of $\widetilde S\setminus\widetilde{\Dmc}$, see Figure~\ref{fig:topological}. We also refer to a connected component of $S\setminus\Dmc$ as a \emph{fragmented ideal triangle} of $\Dmc$. 
\end{definition}

Note that every fragmented ideal triangle $\mathfrak{T}$ of $\widetilde{\Dmc}$ lies in a unique ideal triangle $T$ of $\widetilde{\Tmc}$. Furthermore, the two vertices of $\mathfrak{T}$ that lie in $\partial\widetilde{S}$ are vertices of $T$, and the vertex of $\mathfrak{T}$ that lies in $\widetilde{S}$ is the centroid of $T$. 

Given a non-isolated edge $e$, we say that a fragmented ideal triangle $\mathfrak{T}$ of $\widetilde{\Dmc}$ \emph{faces} $e$ if 
\begin{itemize}
\item one of the vertices of $\mathfrak{T}$ is an endpoint of $e$, and 
\item the geodesic between the two vertices of $\mathfrak{T}$ that lie in $\partial\widetilde S$ separates $\widetilde S$ into two connected components, one of which contains $e$ and the other contains the third vertex of $\mathfrak{T}$.
\end{itemize}
With this, we define the notion of a bridge.

\begin{definition} \label{def:bridge} 
\begin{enumerate}
\item For any non-isolated edge $e$ in $\widetilde{\Tmc}$, a \emph{bridge} across $e$ is a geodesic segment $J$ that intersects $e$ transversely, and whose endpoints both lie in the interiors of fragmented ideal triangles that face $e$, see Figure \ref{fig:topological}.

\item A \emph{bridge system compatible with $\widetilde{\Tmc}$} is a minimal, $\Gamma$-invariant collection of bridges $\widetilde{\Jmc}$ such that every non-isolated edge in $\widetilde{\Tmc}$ has at least (equivalently, exactly) one bridge in $\widetilde{\Jmc}$ across it. If  $\widetilde{\Jmc}$ is a bridge system compatible with $\widetilde{\Tmc}$, then $\Jmc:=\widetilde{\Jmc}/\Gamma$ is a \emph{bridge system compatible with $\Tmc$}.
\end{enumerate}
\end{definition}

While the barrier system $\Dmc$ associated to $\Tmc$ is determined by $\Tmc$, there are infinitely many bridge systems that are compatible with $\Tmc$.

%%%%%%%%%%%%%%%%%%%%%%%%%%%%%%%%%%%%%%%%%%%%%%%%%%%%%%%%%%%%%%%%%%%%%%%%%%%%%%%%%%%%%%%%%%%%%%%%%%%%
\section{Tangent cocycles for the Hitchin component}\label{sec: tangent cocycles}
%%%%%%%%%%%%%%%%%%%%%%%%%%%%%%%%%%%%%%%%%%%%%%%%%%%%%%%%%%%%%%%%%%%%%%%%%%%%%%%%%%%%%%%%%%%%%%%%%%%%

Recall that we have fixed once and for all a hyperbolic metric on $S$. Given a $\PGL(V)$-Hitchin representation $\rho:\Gamma\to\PGL(V)$, an ideal triangulation $\Tmc$ on $S$, and a compatible bridge system $\Jmc$, the goal of this section is to define a natural injective linear map
\begin{equation*}%\label{eqn: Psi}
\Psi=\Psi_{\rho,\Tmc,\Jmc}:T_{[\rho]}\Hit_V(S)\to C^1(S,\smf\lmf(V)_{\Ad\circ\rho})
\end{equation*}
with the property that for every $X\in T_{[\rho]}\Hit_V(S)$, the cocycle $\Psi(X)$ is a representative of the cohomology class $F_\rho(X)\in H^1(S,\mathfrak{sl}(V)_{\Ad\circ\rho})$ (the map $F_\rho$ was defined by \eqref{eqn: F}). For every $X\in T_{[\rho]}\Hit_V(S)$, we will refer to  
\[\nu_X=\nu_{X,\rho,\Tmc,\Jmc}:=\Psi(X)\in C^1(S,\smf\lmf(V)_{\Ad\circ\rho})\] 
as the \emph{$(\rho,\Tmc,\Jmc)$-tangent cocycle} associated to $X$, and denote the set of $(\rho,\Tmc,\Jmc)$-tangent cocycles (i.e. the image of $\Psi$) by 
\[\mathscr{T}(\rho,\Tmc,\Jmc)\subset C^1(S,\mathfrak{sl}(V)_{\Ad\circ\rho}).\] 
These $(\rho,\Tmc,\Jmc)$-tangent cocycles are the cocycle representatives that we will use in the formula \eqref{eqn:formula} to calculate the Goldman symplectic pairing.

%%%%%%%%%%%%%%%%%%%%%%%%%%%%%%%%%%%%%%%%%%%%%%%%%%
\subsection{Preliminaries to define the tangent cocycle}\label{sec: preliminaries cocycle}
%%%%%%%%%%%%%%%%%%%%%%%%%%%%%%%%%%%%%%%%%%%%%%%%%%
First, we set up some notation and prove some preliminary results needed to define the tangent cocycle.

Fix a Hitchin representation $\rho:\Gamma\to\PGL(V)$, and let $\xi_\rho$ denote the $\rho$-equivariant Frenet curve. For any triple of pairwise distinct points ${\bf x}=(x_1,x_2,x_3)$ in $\partial\Gamma$, let 
\begin{align}\label{eqn: submanifold}
\Hmc_{\rho,{\bf x}}:=\left\{\sigma\in\widetilde{\Hit}_V(S):\left(\xi_\sigma(x_1),\xi_\sigma(x_2),\xi_\sigma^1(x_3)\right)=\left(\xi_{\rho}(x_1),\xi_{\rho}(x_2),\xi_{\rho}^1(x_3)\right)\right\}.
\end{align} 
Note that $\rho\in\Hmc_{\rho,{\bf x}}$. By Corollary \ref{cor: BCLS1}, $\Hmc_{\rho,{\bf x}}$ is an embedded real-analytic submanifold of $\widetilde{\Hit}_V(S)$, and if $\pi:\widetilde{\Hit}_V(S)\to\Hit_V(S)$ is the obvious quotient map, then
\[\pi|_{\Hmc_{\rho,{\bf x}}}:\Hmc_{\rho,{\bf x}}\to \Hit_V(S)\]
is a real-analytic diffeomorphism. In particular,
\[{\rm d}(\pi|_{\Hmc_{\rho,{\bf x}}}^{-1})_{[\rho]}:T_{[\rho]}\Hit_V(S)\to T_{\rho}\Hmc_{\rho,{\bf x}}\] 
is a linear isomorphism. 

For any triples ${\bf x}=(x_1,x_2,x_3)$ and ${\bf y}=(y_1,y_2,y_3)$ of pairwise distinct points in $\partial\Gamma$, we may define the map
\begin{align}\label{eqn: g def}
g_{{\bf x},{\bf y}}:\Hmc_{\rho,{\bf x}}\to\PGL(V)
\end{align}
by $g_{{\bf x},{\bf y}}(\sigma)\cdot\left(\xi_\rho(y_1),\xi_\rho(y_2),\xi_\rho^{(1)}(y_3)\right)=\left(\xi_{\sigma}(y_1),\xi_{\sigma}(y_2),\xi_{\sigma}^{(1)}(y_3)\right)$. Similarly, for any triple   ${\bf x}=(x_1,x_2,x_3)$ of pairwise distinct points in $\partial\Gamma$ and any $q\in\partial\Gamma$ distinct from $x_1$, we may define the map
\begin{align}\label{eqn: u def}
u_{{\bf x},q}:\Hmc_{\rho,{\bf x}}\to\PGL(V)
\end{align}
with the defining property that $u_{{\bf x},q}$ is unipotent and $u_{{\bf x},q}(\sigma)\cdot\left(\xi_\rho(x_1),\xi_\rho(q)\right)=\left(\xi_{\sigma}(x_1),\xi_{\sigma}(q)\right)$. It follows from Corollary \ref{cor: smooth} that $g_{{\bf x},{\bf y}}$ and $u_{{\bf x},q}$ are real-analytic, and it is clear that $g_{{\bf x},{\bf y}}(\rho)=\id=u_{{\bf x},q}(\rho)$. The following lemma gives some basic properties of the map $g_{{\bf x},{\bf y}}$ and $u_{{\bf x},q}$. 

\begin{lem}\label{lem: g properties}
Let ${\bf x}=(x_1,x_2,x_3)$, ${\bf y}=(y_1,y_2,y_3)$ and ${\bf z}=(z_1,z_2,z_3)$ be triples of pairwise distinct points in $\partial\Gamma$, let $q\in\partial\Gamma$ be distinct from $x_1$, and let $\gamma\in\Gamma$. 
\begin{enumerate}
\item  
If $\sigma\in\Hmc_{\rho,{\bf x}}$ and $\sigma'\in\Hmc_{\rho,{\bf y}}$ are conjugate, then
\[g_{{\bf x},{\bf z}}(\sigma)=g_{{\bf x},{\bf y}}(\sigma)g_{{\bf y},{\bf z}}(\sigma').\]
In particular, $g_{{\bf x},{\bf x}}(\sigma)=\id$ and $g_{{\bf x},{\bf y}}(\sigma)=g_{{\bf y},{\bf x}}(\sigma')^{-1}$.
\item If $\sigma\in\Hmc_{\rho,{\bf x}}$ and $\sigma'\in\Hmc_{\rho,\gamma\cdot {\bf x}}$ are conjugate, then 
\[g_{\gamma\cdot {\bf x},\gamma\cdot {\bf y}}(\sigma')=\rho(\gamma)g_{{\bf x},{\bf y}}(\sigma) \rho(\gamma)^{-1}.\]
\item If $\sigma\in\Hmc_{\rho,{\bf x}}$ and $\sigma'\in\Hmc_{\rho,\gamma\cdot {\bf x}}$ are conjugate, then 
\[u_{\gamma\cdot {\bf x},\gamma\cdot q}(\sigma')=\rho(\gamma) u_{{\bf x},q}(\sigma) \rho(\gamma)^{-1}.\]
%\item $g_{{\bf x},{\bf y}}=g_{(x_2,x_1,x_3),{\bf y}}=g_{{\bf x},(y_2,y_1,y_3)}$.
\end{enumerate}
\end{lem}

\begin{proof}
First, we prove (1). Since $\left(\xi_{\rho}(y_1),\xi_{\rho}(y_2),\xi_{\rho}^{(1)}(y_3)\right)=\left(\xi_{\sigma'}(y_1),\xi_{\sigma'}(y_2),\xi_{\sigma'}^{(1)}(y_3)\right)$, it follows from the definition of $g_{{\bf x},{\bf y}}(\sigma)$ that
\[\left(\xi_{\sigma}(y_1),\xi_{\sigma}(y_2),\xi_{\sigma}^{(1)}(y_3)\right)=g_{{\bf x},{\bf y}}(\sigma)\cdot\left(\xi_{\sigma'}(y_1),\xi_{\sigma'}(y_1),\xi_{\sigma'}^{(1)}(y_1)\right),\]
which implies that $\xi_{\sigma}=g_{{\bf x},{\bf y}}(\sigma)\cdot\xi_{\sigma'}$ because $\sigma$ is conjugate to $\sigma'$. 
Thus,
\begin{eqnarray*}
g_{{\bf x},{\bf z}}(\sigma)\cdot\left(\xi_\rho(z_1),\xi_\rho(z_2),\xi_\rho^{(1)}(z_3)\right)&=&\left(\xi_{\sigma}(z_1), \xi_{\sigma}(z_2),\xi_{\sigma}^{(1)}(z_3) \right)\\
&=&g_{{\bf x},{\bf y}}(\sigma)\cdot\left(\xi_{\sigma'}(z_1), \xi_{\sigma'}(z_2),\xi_{\sigma'}^{(1)}(z_3) \right)\\
&=&g_{{\bf x},{\bf y}}(\sigma)g_{{\bf y},{\bf z}}(\sigma')\cdot\left( \xi_{\rho}(z_1),  \xi_{\rho}(z_2),\xi_{\rho}^{(1)}(z_3)\right),
\end{eqnarray*}
so (1) holds.

To prove (2) and (3), observe that 
\[\xi_{\sigma'}(\gamma\cdot{\bf x})=\xi_{\rho}(\gamma\cdot{\bf x})=\rho(\gamma)\cdot\xi_\rho({\bf x})=\rho(\gamma)\cdot\xi_{\sigma}({\bf x})=\rho(\gamma)\sigma(\gamma)^{-1}\cdot\xi_{\sigma}(\gamma\cdot{\bf x}).\]
Since $\sigma$ is conjugate to $\sigma'$, this implies that $\xi_{\sigma'}=\rho(\gamma)\sigma(\gamma)^{-1}\cdot\xi_{\sigma}$. Thus,
\begin{eqnarray*}
&&g_{\gamma\cdot{\bf x},\gamma\cdot{\bf y}}(\sigma')\cdot\left( \xi_{\rho}(\gamma\cdot y_1),  \xi_{\rho}(\gamma\cdot y_2),\xi_{\rho}^{(1)}(\gamma\cdot y_3)\right)\\
&=&\left( \xi_{\sigma'}(\gamma\cdot y_1),  \xi_{\sigma'}(\gamma\cdot y_2),\xi_{\sigma'}^{(1)}(\gamma\cdot y_3)\right)\\
%&=&\rho(\gamma)\rho_t(\gamma)^{-1}\cdot\left( \xi_{\rho_t}(\gamma\cdot y_1),  \xi_{\rho_t}(\gamma\cdot y_2),\xi_{\rho_t}^{(1)}(\gamma\cdot y_3)\right)\\
&=&\rho(\gamma)\cdot\left( \xi_{\sigma}(y_1),  \xi_{\sigma}(y_2),\xi_{\sigma}^{(1)}(y_3)\right)\\
&=&\rho(\gamma)g_{{\bf x},{\bf y}}(\sigma)\cdot\left( \xi_{\rho}(y_1),  \xi_{\rho}(y_2),\xi_{\rho}^{(1)}(y_3)\right)\\
&=&\rho(\gamma)g_{{\bf x},{\bf y}}(\sigma)\rho(\gamma)^{-1}\cdot\left( \xi_{\rho}(\gamma\cdot y_1),  \xi_{\rho}(\gamma\cdot y_2),\xi_{\rho}^{(1)}(\gamma\cdot y_3)\right),
\end{eqnarray*}
so (2) holds. Similarly,
\begin{eqnarray*}
u_{\gamma\cdot{\bf x},\gamma\cdot q}(\sigma')\cdot\left( \xi_{\rho}(\gamma\cdot x_1),  \xi_{\rho}(\gamma\cdot q)\right)&=&\left( \xi_{\sigma'}(\gamma\cdot x_1),  \xi_{\sigma'}(\gamma\cdot q)\right)\\
%&=&\rho(\gamma)\rho_t(\gamma)^{-1}\cdot\left( \xi_{\rho_t}(\gamma\cdot y_1),  \xi_{\rho_t}(\gamma\cdot y_2),\xi_{\rho_t}^{(1)}(\gamma\cdot y_3)\right)\\
&=&\rho(\gamma)\cdot\left( \xi_{\sigma}(x_1),  \xi_{\sigma}(q)\right)\\
&=&\rho(\gamma)u_{{\bf x},q}(\sigma)\cdot\left( \xi_{\rho}(x_1),  \xi_{\rho}(q)\right)\\
&=&\rho(\gamma)u_{{\bf x},q}(\sigma)\rho(\gamma)^{-1}\cdot\left( \xi_{\rho}(\gamma\cdot x_1),  \xi_{\rho}(\gamma\cdot q)\right),
\end{eqnarray*}
so (3) holds.
\end{proof}

Now, suppose that $r_1$ and $r_2$ are distinct points in $\partial\Gamma$. Let ${\bf x}=(x_1,x_2,x_3)$ and ${\bf x}'=(x_1',x_2',x_3')$ be triples of pairwise distinct points such that both $x_1$ and $x_1'$ lie in $\{r_1,r_2\}$, and let $q$ and $q'$ denote the points in $\{r_1,r_2\}$ that are not $x_1$ and $x_1'$ respectively. Then define 
\begin{align}\label{eqn: d def}
d_{{\bf x},{\bf x}',q,q'}:\Hmc_{\rho,{\bf x}}\to\PGL(V)
\end{align}
by $d_{{\bf x},{\bf x}',q,q'}(\sigma):=u_{{\bf x},q}(\sigma)^{-1} g_{{\bf x},{\bf x}'}(\sigma) u_{{\bf x}',q'}\left(\sigma'\right)$, where $\sigma'$ is the representation in $\Hmc_{\rho,{\bf x}'}$ that is conjugate to $\sigma$, i.e.
\[\sigma':=\pi|_{\Hmc_{\rho,{\bf x}'}}^{-1}\circ\pi|_{\Hmc_{\rho,{\bf x}}}(\sigma).\] 
By Corollary \ref{cor: smooth}, $d_{{\bf x},{\bf x}',q,q'}$ is real-analytic, and it is clear that $d_{{\bf x},{\bf x}',q,q'}(\rho)=\id$. The next lemma gives some basic properties of $d_{{\bf x,}{\bf x}',q,q'}$.

\begin{lem}\label{lem: d properties}
Let $r_1$ and $r_2$ be distinct points in $\partial\Gamma$, let ${\bf x}=(x_1,x_2,x_3)$ and ${\bf x}'=(x_1,x_2,x_3)$ be triples of pairwise distinct points such that both $x_1$ and $x_1'$ lie in $\{r_1,r_2\}$, and let $q$ and $q'$ be the points in $\{r_1,r_2\}$ that are not $x_1$ and $x_1'$ respectively.
\begin{enumerate}
\item For any $\sigma\in\Hmc_{\rho,{\bf x}}$, $d_{{\bf x},{\bf x'},q,q'}(\sigma)$ fixes both $\xi_\rho(r_1)$ and $\xi_\rho(r_2)$. In particular, $d_{{\bf x},{\bf x'},q,q'}(\sigma)^{\frac{1}{2}}$ is well-defined.
\item If $\sigma\in\Hmc_{\rho,{\bf x}}$ and $\sigma'\in\Hmc_{\rho,{\bf x}'}$ are conjugate, then
\[d_{{\bf x},{\bf x}',q,q'}(\sigma)^{-1}=d_{{\bf x}',{\bf x},q',q}(\sigma').\]
In particular, $g_{{\bf x},{\bf x}'}(\sigma)=\left(u_{{\bf x},q}(\sigma)d_{{\bf x},{\bf x}',q,q'}(\sigma)^{\frac{1}{2}}\right)\left(u_{{\bf x}',q'}\left(\sigma'\right)d_{{\bf x}',{\bf x},q',q}(\sigma')^{\frac{1}{2}}\right)^{-1}$.
\item If $\sigma\in\Hmc_{\rho,{\bf x}}$ and $\sigma'\in\Hmc_{\rho,\gamma\cdot {\bf x}}$ are conjugate, then
\[d_{\gamma\cdot {\bf x},\gamma\cdot {\bf x}',\gamma\cdot q,\gamma\cdot q'}(\sigma')=\rho(\gamma)d_{{\bf x},{\bf x}',q,q'}(\sigma)\rho(\gamma)^{-1}.\]
\end{enumerate}
\end{lem}

\begin{proof}
%Let $q$ and $q'$ be the points in $\{r_1,r_2\}$ that are not $x_1$ and $x_1'$ respectively.
(1) Let $\sigma'$ be the representation in $\Hmc_{\rho,{\bf x}'}$ that is conjugate to $\sigma$. Then
\[\left(\xi_\rho(x_1'),\xi_\rho(x_2'),\xi_\rho^{(1)}(x_3')\right)=\left(\xi_{\sigma'}(x_1'),\xi_{\sigma'}(x_2'),\xi_{\sigma'}^{(1)}(x_3')\right),\] 
so by the definition of $g_{{\bf x},{\bf x}'}(\sigma)$,
\begin{align*}\label{eqn: infinitesimal shearing}
d_{{\bf x},{\bf x}',q,q'}(\sigma) u_{{\bf x}',q'}\left(\sigma'\right)^{-1}&\cdot \left(\xi_{\sigma'}(x_1'),\xi_{\sigma'}(x_2'),\xi_{\sigma'}^{(1)}(x_3')\right)\\
&=u_{{\bf x},q}(\sigma)^{-1}\cdot \left(\xi_{\sigma}(x_1'),\xi_{\sigma}(x_2'),\xi_{\sigma}^{(1)}(x_3')\right).
\end{align*}
This implies that  $d_{{\bf x},{\bf x}',q,q'}(\sigma) u_{{\bf x}',q'}\left(\sigma'\right)^{-1}\cdot \xi_{\sigma'}=u_{{\bf x},q}(\sigma)^{-1}\cdot\xi_{\sigma}$ because $\sigma$ and $\sigma'$ are conjugate. By definition, for both $m=1,2$, 
\[u_{{\bf x}',q'}\left(\sigma'\right)^{-1}\cdot \xi_{\sigma'}(r_m)=\xi_\rho(r_m)=u_{{\bf x},q}(\sigma)^{-1}\cdot\xi_{\sigma}(r_m),\] 
so the first claim of (1) follows. Since $\xi_\rho(r_1)$ and $\xi_\rho(r_2)$ are transverse, this first claim implies that $d_{{\bf x},{\bf x}',q,q'}(\sigma)$ is diagonalizable, so the second claim of (1) holds.

(2) follows immediately from Lemma \ref{lem: g properties}(1), and (3) is a straight forward consequence of Lemma~\ref{lem: g properties}(2) and Lemma~\ref{lem: g properties}(3).
\end{proof}

%%%%%%%%%%%%%%%%%%%%%%%%%%%%%%%%%%%%%%%%%%%%%%%%%%
\subsection{Defining the tangent cocycle in special cases}\label{sec:tangent cocycle}
%%%%%%%%%%%%%%%%%%%%%%%%%%%%%%%%%%%%%%%%%%%%%%%%%%

Given $X\in T_{[\rho]}\Hit_V(S)$, we will specify the $(\rho,\Tmc,\Jmc)$-tangent cocycle associated to $X$, denoted $\nu_X\in C^1(S,\mathfrak{sl}(V)_{\Ad\circ\rho})$, by specifying its $\Ad\circ\rho$-equivariant lift $\widetilde{\nu}_X:C_1(\Std,\Zbbb)\to\smf\lmf(V)$. To do so, we will first define $\widetilde{\nu}_X(h)$ when the $1$-simplex $h:[0,1]\to\widetilde{S}$ is in the three special cases labelled as Cases 1 to 3 below. Then, in Section \ref{sec:tangent cocycle general}, we will use these special cases to define $\widetilde{\nu}_X(h)$ for general $h$.

Let $\widetilde{\Tmc}$ and $\widetilde{\Jmc}$ be the lifts to $\Std$ of $\Tmc$ and $\Jmc$ respectively, and let $\widetilde{\Dmc}$ be the barrier system associated to $\widetilde{\Tmc}$. For any $1$-simplex $h:[0,1]\to\widetilde{S}$, we let $\bar{h}:[0,1]\to\widetilde{S}$ denote the $1$-simplex defined by $\bar{h}(t):=h(1-t)$.

{\bf Case 1: The endpoints of the image of $h$ do not lie in the non-isolated edges of $\widetilde{\Tmc}$.} We will define $\widetilde{\nu}_X(h)$ in this case via three steps, each with increasing generality. The first step specifies $\widetilde{\nu}_X(h)$ is the generic situation when the endpoints of the image of $h$ lie in the interiors of fragmented ideal triangles of $\widetilde{\Dmc}$. The second and third steps specify $\widetilde{\nu}_X(h)$ in general via an averaging process.

\begin{figure}[ht]
\centering
\includegraphics[scale=0.5]{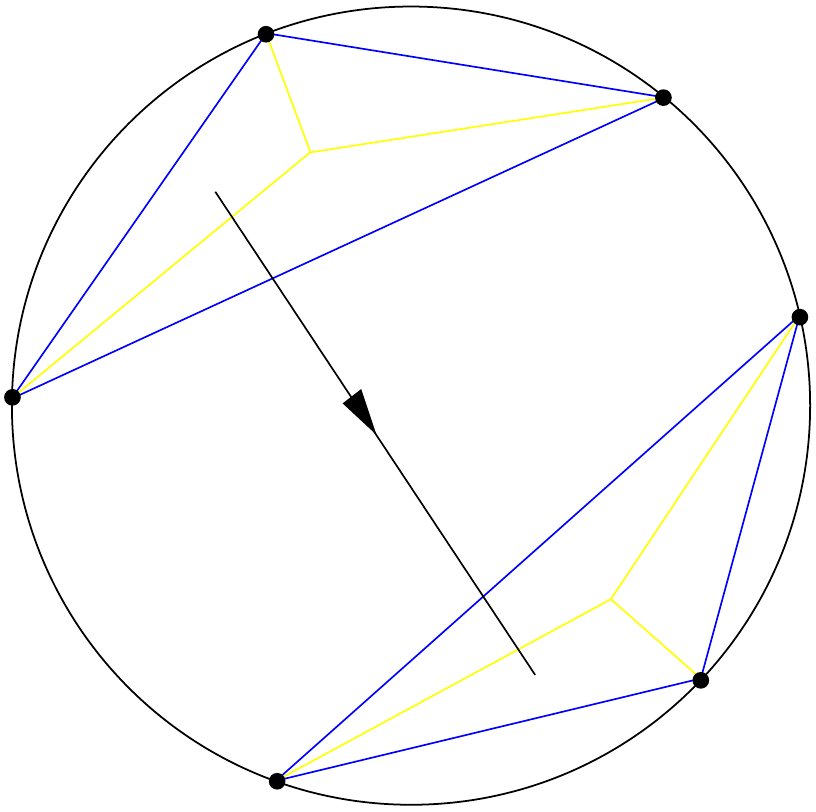}
\small
\put (-205, 100){$x_1$}
\put (-143, 190){$x_2$}
\put (-33, 172){$x_3$}
\put (-28, 24){$y_1$}
\put (-137, 1){$y_2$}
\put (-1, 117){$y_3$}
\put (-105, 97){$h$}
%\put (-143, 160){$T'_{h(0)}$}
%\put (-56, 35){$T'_{h(1)}$}
\caption{$1$-simplices considered in the first step of Case 1.}\label{fig:step1}
\end{figure}

For the first step, we define $\widetilde{\nu}_X(h)$ when the endpoints of the image of $h$ both lie in the interiors of fragmented ideal triangles of the barrier system $\widetilde{\Dmc}$. For such $h$, let ${\bf x}=(x_1,x_2,x_3)$ be the vertices of the ideal triangle of $\widetilde{\Tmc}$ that contains $h(0)$, such that $x_3$ is not a vertex of the fragmented ideal triangle of $\widetilde{\Dmc}$ that contains $h(0)$. Similarly, let ${\bf y}=(y_1,y_2,y_3)$ be the vertices of the ideal triangle of $\widetilde{\Tmc}$ that contains $h(1)$, such that $y_3$ is not a vertex of the fragmented ideal triangle of $\widetilde{\Dmc}$ that contains $h(1)$, see Figure \ref{fig:step1}. Then define 
\begin{equation}\label{eqn: tangent equation 1}
\widetilde{\nu}_X(h):={\rm d}\left(g_{{\bf x},{\bf y}}\circ\pi|_{\Hmc_{\rho,{\bf x}}}^{-1}\right)_{[\rho]}(X)\in\mathfrak{sl}(V),
\end{equation}
where $g_{{\bf x},{\bf y}}$ is the map defined by \eqref{eqn: g def}. Note that this is well-defined because switching the roles of $x_1$ and $x_2$ or switching the roles of $y_1$ and $y_2$ leave $g_{{\bf x},{\bf y}}$ invariant. Equivalently, if $t\mapsto \rho_t$ is a smooth path in $\Hmc_{\rho,{\bf x}}$ such that $\rho_0=\rho$ and ${\rm d}\pi\left(\frac{d}{dt}|_{t=0}\rho_t\right)=X$, then
\begin{equation}\label{eqn: tangent equation 1'}
\widetilde{\nu}_X(h)=\frac{d}{dt}\bigg|_{t=0}g_{{\bf x},{\bf y}}(\rho_t)\in\mathfrak{sl}(V).
\end{equation}

\begin{figure}[ht]
\centering
\includegraphics[scale=1]{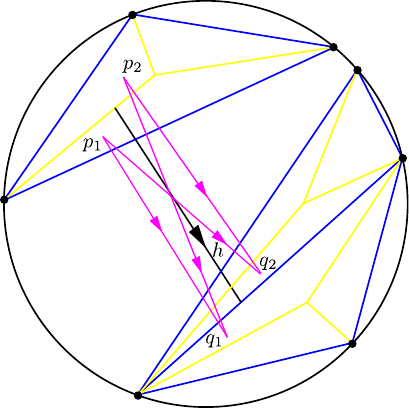}
\small
\caption{$1$-simplices considered in the second step of Case 1.}\label{fig:step2}
\end{figure}

For the second step, we define $\widetilde{\nu}_X(h)$ when both endpoints of the image of $h$ lie either in the interiors of fragmented ideal triangles of $\widetilde{\Dmc}$, in isolated edges of $\widetilde{\Tmc}$, or in the interiors of non-edge barriers associated to $\widetilde{\Tmc}$. For such $h$, let $h_1,\dots,h_4:[0,1]\to\widetilde{S}$ be $1$-simplices such that
\begin{itemize}
\item $h_1(0)=h_2(0)=:p_1$, $h_3(0)=h_4(0)=:p_2$, $h_1(1)=h_3(1)=:q_1$, $h_2(1)=h_4(1)=:q_2$,
\item if $h(0)$ lies in the interior of a fragmented ideal triangle, then $p_1=p_2=h(0)$,
\item if $h(0)$ lies in a isolated edge or in the interior of a non-edge barrier, then $p_1$ and $p_2$ are points in the interiors of the two fragmented ideal triangles that contain $h(0)$ in their closures,
\item if $h(1)$ lies in the interior of a fragmented ideal triangle, then $q_1=q_2=h(1)$,
\item if $h(1)$  lies in a isolated edge or in the interior of a non-edge barrier, then $q_1$ and $q_2$ are points in the interiors of the two fragmented ideal triangles that contain $h(1)$ in their closures,
\end{itemize}
see Figure \ref{fig:step2}. Note that for all $m=1,\dots,4$, the endpoints of the image of $h_m$ lie in the interiors of fragmented ideal triangles, so we have already defined $\widetilde{\nu}_X(h_m)$ in the previous step, see \eqref{eqn: tangent equation 1}. Thus, we may define
\begin{equation}\label{eqn: Case 1}\widetilde{\nu}_X(h):=\frac{1}{4}\sum_{m=1}^4\widetilde{\nu}_X(h_m).\end{equation}
Observe that $\widetilde{\nu}_X(h)$ does not depend on the choice of $h_1,\dots,h_4$. Also, if $h$ satisfies the conditions of the previous step, then the definition of $\widetilde{\nu}_X(h)$ given by \eqref{eqn: Case 1} agrees with the definition of $\widetilde{\nu}_X(h)$ given by \eqref{eqn: tangent equation 1}.

\begin{figure}[ht]
\centering
\includegraphics[scale=1]{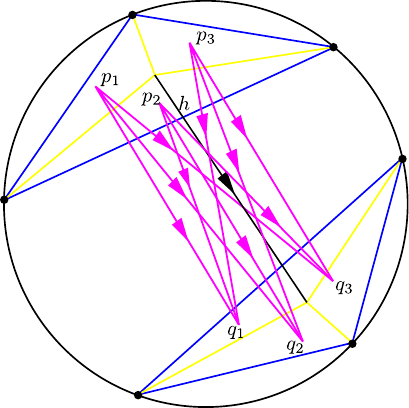}
\small
\caption{$1$-simplices considered in Step 3 when $h(0)$ and $h(1)$ are the centroids of ideal triangles of $\widetilde{\Tmc}$.}\label{fig:step3}
\end{figure}

Finally, for the third step, we define $\widetilde{\nu}(h)$ when both endpoints of the image of $h$ do not lie in the non-isolated edges of $\widetilde{\Tmc}$. For such $h$, let $h^1,\dots,h^9:[0,1]\to\widetilde{S}$ be $1$-simplices so that
\begin{itemize}
\item $h^1(0)=h^2(0)=h^3(0)=:p_1$, $h^4(0)=h^5(0)=h^6(0)=:p_2$, $h^7(0)=h^8(0)=h^9(0)=:p_3$, $h^1(1)=h^4(1)=h^7(1)=:q_1$, $h^2(1)=h^5(1)=h^8(1)=:q_2$, $h^3(1)=h^6(1)=h^9(1)=:q_3$,
\item if $h(0)$ lies in the interior of a fragmented ideal triangle, in an isolated edge, or in the interior of a non-edge barrier, then $p_1=p_2=p_3=h(0)$,
\item if $h(0)$ is the centroid of an ideal triangle of $\widetilde{\Tmc}$, then $p_1$, $p_2$, $p_3$ are points in the interiors of the three fragmented ideal triangles that contain $h(0)$ in its closure,
\item if $h(1)$ lies in the interior of a fragmented ideal triangle, in an isolated edge, or in the interior of a non-edge barrier, then $q_1=q_2=q_3=h(1)$,
\item if $h(1)$ is the centroid of an ideal triangle of $\widetilde{\Tmc}$, then $q_1$, $q_2$, $q_3$ are points in the interiors of the three fragmented ideal triangles that contain $h(1)$ in its closure,
\end{itemize}
see Figure \ref{fig:step3}. Note that for $m=1,\dots,9$, the endpoints of the image of $h^m$ both lie in either the interior of fragmented ideal triangles, in an isolated edge, or in the interior of a non-edge barrier, so we have already defined $\widetilde{\nu}_X(h^m)$ in the previous step. Thus, we can define
\begin{equation}\label{eqn: Case 2}\widetilde{\nu}_X(h):=\frac{1}{9}\sum_{m=1}^9\widetilde{\nu}(h^m).\end{equation}
Again observe that $\widetilde{\nu}_X(h)$ does not depend on the choice of $h^1,\dots,h^9$. Furthermore, if $h$ satisfies the conditions of the previous step, then the definition of $\widetilde{\nu}_X(h)$ given by \eqref{eqn: Case 2} specializes to the definition of $\widetilde{\nu}_X(h)$ given by \eqref{eqn: Case 1}. 

The following lemma gives some properties of $\widetilde{\nu}_X$ as defined in \eqref{eqn: tangent equation 1}. 

\begin{lem} \label{lem:endpoints}Let $h$ and $h'$ be $1$-simplices satisfying the conditions of Case 1. 
\begin{enumerate}
%\item If $h(0)$ and $h'(0)$ lie in the interior of the same fragmented ideal triangle, and $h(1)$ and $h'(1)$ lie in the interior of the same fragmented ideal triangle, then $\widetilde{\nu}_X(h)=\widetilde{\nu}_X(h')$. 
\item If $h(1)=h'(0)$ and $h''$ is the concatenation of $h$ and $h'$, then $\widetilde{\nu}_X(h)+\widetilde{\nu}_X(h')=\widetilde{\nu}_X(h'')$. In particular, $\widetilde{\nu}_X(h)=0$ if $h(0)=h(1)$, and $\widetilde{\nu}_X(h)=-\widetilde{\nu}_X(\bar{h})$.
\item For any $\gamma\in\Gamma$, $\widetilde{\nu}_X(\gamma\cdot h)=\Ad\circ\rho(\gamma)\cdot \widetilde{\nu}_X(h)$.
\end{enumerate}
\end{lem}

\begin{proof}
By the definition of $\widetilde{\nu}_X(h)$, it is sufficient to verify (1) and (2) when the endpoints of the image of $h$ and $h'$ lie in the interiors of fragmented ideal triangles in $\widetilde{\Dmc}$, in which case $\widetilde{\nu}_X(h)$ and $\widetilde{\nu}_X(h')$ are given by \eqref{eqn: tangent equation 1'}. For such $h$ and $h'$, (1) and (2) are straightforward consequences of Lemma \ref{lem: g properties}(1) and Lemma \ref{lem: g properties}(2) respectively. 
\end{proof}

{\bf Case 2: There is a bridge $J$ in $\widetilde\Jmc$ across a non-isolated edge $e$ in $\widetilde{\Tmc}$ such that one of the endpoints of the image of $h$ is $e\cap J$, while the other is an endpoint of $J$.} First, suppose that $h(1)=e\cap J$ and $h(0)$ is an endpoint of $J$. Let ${\bf x}=(x_1,x_2,x_3)$ be the vertices of the ideal triangle $T$ of $\widetilde{\Tmc}$ that contains $h(0)$, such that $x_1$ is an endpoint of $e$ and $x_3$ is not a vertex of the fragmented ideal triangle of $\widetilde{\Dmc}$ that contains $h(0)$, see Figure \ref{fig:cocycleclosededge}. Similarly, let ${\bf x}'=(x_1',x_2',x_3')$ be the vertices of the ideal triangle $T'$ of $\widetilde{\Tmc}$ that contains the other endpoint of $J$, such that $x_1'$ is an endpoint of $e$ and $x_3'$ is not a vertex of the fragmented ideal triangle of $\widetilde{\Dmc}$ that contains the other endpoint of $J$. Let $q$ and $q'$ be the endpoints of $e$ that are not $x_1$ and $x_1'$ respectively. Then define
\begin{equation}\label{eqn: tangent equation 3}
\widetilde{\nu}_X(h):={\rm d}\left(u_{{\bf x},q}\circ \pi|_{\Hmc_{\rho,{\bf x}}}^{-1}\right)_{[\rho]}(X)+\frac{1}{2}{\rm d}\left(d_{{\bf x},{\bf x}',q,q'}\circ \pi|_{\Hmc_{\rho,{\bf x}}}^{-1}\right)_{[\rho]}(X)\in\smf\lmf(V),
\end{equation}
where $u_{{\bf x},q}$ and $d_{{\bf x},{\bf x}',q,q'}$ are the maps defined by \eqref{eqn: u def} and \eqref{eqn: d def} respectively. Equivalently, if $t\mapsto \rho_t$ is a smooth path in $\Hmc_{\rho,{\bf x}}$ such that $\rho_0=\rho$ and ${\rm d}\pi\left(\frac{d}{dt}|_{t=0}\rho_t\right)=X$, then
\begin{equation}\label{eqn: tangent equation 3'}
\widetilde{\nu}_X(h):=\frac{d}{dt}\bigg|_{t=0}u_{{\bf x},q}(\rho_t)+\frac{1}{2}\cdot\frac{d}{dt}\bigg|_{t=0}d_{{\bf x},{\bf x}',q,q'}(\rho_t)\in\smf\lmf(V).
\end{equation}

\begin{figure}[ht]
\centering
\includegraphics[scale=0.8]{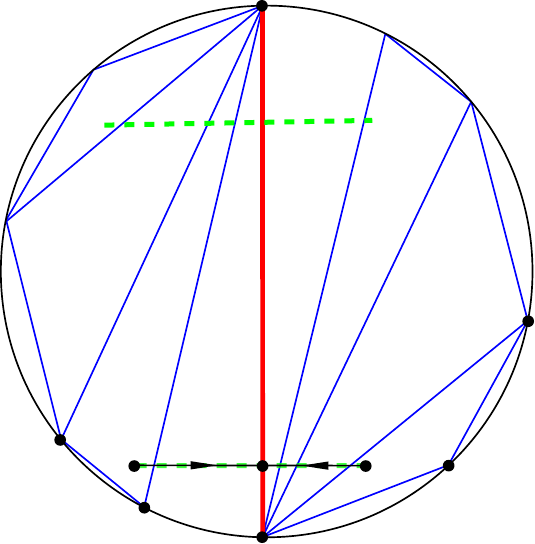}
\small
\put (-85, 34){$\bar{h}'$}
\put (-133, 34){$h$}
\put (-119, -6){$x_1'=q$}
\put (-119, 211){$x_1=q'$}
\put (1, 83){$x_2'$}
\put (-31, 24){$x_3'$}
\put (-160, 7){$x_2$}
\put (-192, 34){$x_3$}
\caption{Case 2.}\label{fig:cocycleclosededge}
\end{figure}

On the other hand, if $h(0)=e\cap J$ and $h(1)$ is an endpoint of $J$, then the $1$-simplex $\bar{h}$ has the property that $\bar{h}(1)=e\cap J$ and $\bar{h}(0)$ is an endpoint of $J$. Thus, $\widetilde{\nu}_X(\bar{h})$ was defined above, so we may define
\begin{equation}\label{eqn: tangent equation 4}
\widetilde{\nu}_X(h):=-\widetilde{\nu}_X(\bar{h}).
\end{equation}

 The following lemma gives the key properties of $\widetilde{\nu}_X(h)$ given by \eqref{eqn: tangent equation 3} and \eqref{eqn: tangent equation 4}.

\begin{lem}\label{lem:diag inverse}
Let $h,h':[0,1]\to\widetilde{S}$ be $1$-simplices satisfying the conditions of Case 2.
\begin{enumerate}
\item Suppose that there is some bridge $J$ in $\widetilde{\Jmc}$ across a non-isolated edge $e$ in $\widetilde{\Tmc}$ such that $h(0)$ and $h'(1)$ are the endpoints of $J$ (realized as a geodesic segment in $\widetilde{S}$) and $h(1)=h'(0)=e\cap J$. If $h''$ is the concatenation of $h$ and $h'$, then $\widetilde{\nu}_X(h)+\widetilde{\nu}_X(h')=\widetilde{\nu}_X(h'')$.
\item For any $\gamma\in\Gamma$, $\widetilde{\nu}_X(\gamma\cdot h)=\Ad\circ\rho(\gamma)\cdot \widetilde{\nu}_X(h)$.
\end{enumerate}
\end{lem}

\begin{proof}
To prove (1), first notice that $h''$ satisfies the conditions of Case 1. Then (see Figure \ref{fig:cocycleclosededge}) by Lemma \ref{lem: d properties}(2), 
\[\widetilde{\nu}_X(h)+\widetilde{\nu}_X(h')=\widetilde{\nu}_X(h)-\widetilde{\nu}_X(\bar{h}')=\widetilde{\nu}_X(h'').\]
(2) follows immediately from Lemma \ref{lem: g properties}(3) and Lemma \ref{lem: d properties}(3).
\end{proof}

\textbf{Case 3: The endpoints of the image of $h$ lie in the same non-isolated edge in $\widetilde{\Tmc}$.}  If $h(0)=h(1)$, define $\widetilde{\nu}_X(h)=0$. If $h(0)\neq h(1)$, let $e=\{r_1,r_2\}$ be the non-isolated edge that contains both $h(0)$ and $h(1)$, and let $\gamma\in\Gamma$ be the primitive group element whose axis is $e$, such that $\gamma$ translates along $e$ in the direction from $h(0)$ to $h(1)$. Let $n_1(h)$ be the number of $\langle\gamma\rangle$-translates of $J$ that intersect the interior of $h$, let $n_2(h)$ be the number of $\langle\gamma\rangle$-translates of $J$ that contain the endpoints of $h$, and let 
\[\#h:=n_1+\frac{n_2}{2}.\]

Choose any point $r_3\in\partial\Gamma$ that is not $r_2$ nor $r_2$, let ${\bf r}:=(r_1,r_2,r_3)$, and define the map
\[f_{\bf r}:\Hmc_{\rho,{\bf r}}\to\PGL(V)\]
by $f_{\bf r}:\sigma\mapsto \sigma(\gamma)\rho(\gamma)^{-1}$. Note that $f_{\bf r}$ is smooth and sends $\rho$ to $\id$. Furthermore, the map $f_{\bf r}\circ \pi|_{\Hmc_{\rho,{\bf r}}}^{-1}:\Hit_V(S)\to\PGL(V)$ depends only on $\gamma$ (and not on the choice of ${\bf r}$), so we may define
\begin{equation}\label{eqn: tangent equation 2}
\widetilde{\nu}_X(h):=\#h\cdot {\rm d}\left(f_{\bf r}\circ \pi|_{\Hmc_{\rho,{\bf r}}}^{-1}\right)_{[\rho]}(X)\in\smf\lmf(V).
\end{equation} 
Equivalently, if $t\mapsto \rho_t$ be a smooth path in $\Hmc_{\rho,{\bf r}}$ such that $\rho_0=\rho$ and ${\rm d}\pi\left(\frac{d}{dt}|_{t=0}\rho_t\right)=X$,
then
\begin{equation}\label{eqn: tangent equation 2'}
\widetilde{\nu}_X(h):=\#h\cdot\frac{d}{dt}\bigg|_{t=0}\rho_t(\gamma)\rho(\gamma)^{-1}\in\smf\lmf(V).
\end{equation} 

The following lemma follows easily from the definitions.

\begin{lem}\label{lem:Case 3.1}
Let $h,h':[0,1]\to\widetilde{S}$ be $1$-simplices satisfying the conditions of Case 3.
\begin{enumerate}
\item If $h(1)=h'(0)$ and $h''$ is the concatenation of $h$ and $h'$, then $\widetilde{\nu}_X(h)+\widetilde{\nu}_X(h')=\widetilde{\nu}_X(h'')$.
\item For any $\gamma\in\Gamma$, $\widetilde{\nu}_X(\gamma\cdot h)=\Ad\circ\rho(\gamma)\cdot \widetilde{\nu}_X(h)$.
\end{enumerate}
\end{lem}

Finally, we prove the prove the following lemma that relates Cases 1, 2, and 3.

\begin{lem}\label{lem:tangent cocycle 1}
Let $e$ be a non-isolated edge in $\widetilde{\Tmc}$, let $J$ be a bridge in $\widetilde{\Jmc}$ across $e$, and let $\gamma\in\Gamma$ be an element whose axis is $e$. Let $s,s':[0,1]\to \widetilde{S}$ be $1$-simplices such that $s(1)=s'(1)=e\cap J$, and $s(0)$ and $s'(0)$ are the two endpoints of $J$. Let $m,l,l':[0,1]\to\widetilde{S}
$ be $1$-simplices such that $m(0)=s(1)$, $m(1)=\gamma\cdot s(1)$, $l(0)=s(0)$, $l(1)=\gamma\cdot s(0)$, $l'(0)=s'(0)$, $l'(1)=\gamma\cdot s(0)$, see Figure~\ref{fig:cocycleproof1}. Then
\begin{enumerate}
\item $\widetilde{\nu}_X(s)+\widetilde{\nu}_X(m)=\widetilde{\nu}_X(l)+\widetilde{\nu}_X(\gamma\cdot s)$.
\item $\widetilde{\nu}_X(m)+\widetilde{\nu}_X(\gamma\cdot\bar{s})=\widetilde{\nu}_X(\bar{s})+\widetilde{\nu}_X(l)$.
\item $\widetilde{\nu}_X(m)+\widetilde{\nu}_X(\gamma\cdot \bar{s})=\widetilde{\nu}_X(\bar{s}')+\widetilde{\nu}_X(l')$.
\end{enumerate}
\end{lem}

\begin{figure}[ht]
\centering
\includegraphics[scale=0.8]{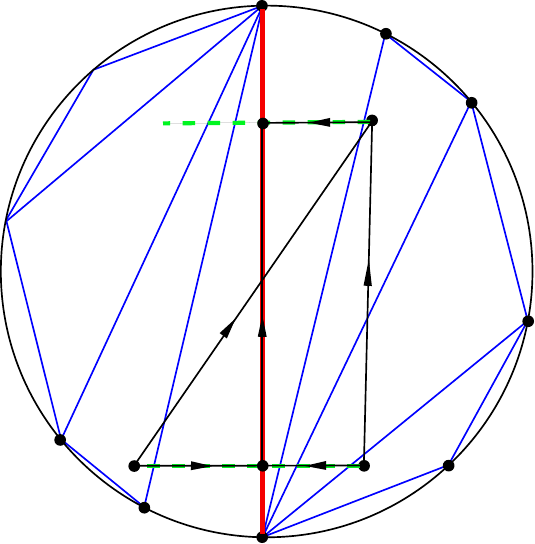}
\put (-90, 21){$s$}
\put (-133, 20){$s'$}
\put (-70, 100){$l$}
\put (-93, 167){$\gamma\cdot s$}
\put (-102, 100){$m$}
\put (-130, 75){$l'$}
\put (-140, -7){$x_1=\gamma\cdot x_1=q'$}
\put (-123, 211){$x_1'=q$}
\put (-55, 199){$\gamma\cdot x_2$}
\put (-22, 170){$\gamma\cdot x_3$}
\put (1, 82){$x_2$}
\put (-30, 25){$x_3$}
\put (-163, 5){$x_2'$}
\put (-193, 30){$x_3'$}
\caption{Lemma \ref{lem:tangent cocycle 1}.}\label{fig:cocycleproof1}
\end{figure}

\begin{proof}
First, we prove (1). Let ${\bf x}=(x_1,x_2,x_3)$ be the vertices of the ideal triangle of $\widetilde{\Tmc}$ that contains $h(0)$, so that $x_1$ is an endpoint of $e$ and $x_3$ is not a vertex of the fragmented ideal triangle of $\widetilde{\Dmc}$ that contains $h(0)$. Similarly, let ${\bf x}'=(x_1',x_2',x_3')$ be the vertices of the ideal triangle of $\widetilde{\Tmc}$ that contains the other endpoint of $J$, so that $x_1'$ is an endpoint of $e$ and $x_3'$ is not a vertex of the fragmented ideal triangle of $\widetilde{\Dmc}$ that contains the other endpoint of $J$. Let $q$ and $q'$ be the endpoints of $e$ that are not $x_1$ and $x_1'$ respectively.

Let $t\mapsto \rho_t$ be a smooth path in $\Hmc_{\rho,{\bf x}}$ such that $\rho_0=\rho$ and ${\rm d}\pi\left(\frac{d}{dt}|_{t=0}\rho_t\right)=X$. By \eqref{eqn: tangent equation 3'},
\[\widetilde{\nu}_X(s)=\frac{d}{dt}\bigg|_{t=0}u_{{\bf x},q}(\rho_t)+\frac{1}{2}\frac{d}{dt}\bigg|_{t=0}d_{{\bf x},{\bf x}',q,q'}(\rho_t).\]
Lemma \ref{lem: d properties}(1) implies that $d_{{\bf x},{\bf x}',q,q'}(\rho_t)$ commutes with $\rho(\gamma)$ for all $t$, so by Lemma \ref{lem:diag inverse}(2), 
\begin{eqnarray*}
\widetilde{\nu}_X(\gamma\cdot s)=\frac{d}{dt}\bigg|_{t=0} \rho(\gamma)u_{{\bf x},q}(\rho_t)\rho(\gamma)^{-1}+\frac{1}{2}\frac{d}{dt}\bigg|_{t=0}d_{{\bf x},{\bf x}',q,q'}(\rho_t).
\end{eqnarray*}
Since $\left(\xi_{\rho}(x_1),\xi_{\rho}(x_2),\xi_{\rho}^{(1)}(x_3)\right)=\left(\xi_{\rho_t}(x_1),\xi_{\rho_t}(x_2),\xi_{\rho_t}^{(1)}(x_3)\right)$ for all $t$, it follows that
\[\rho_t(\gamma)\rho(\gamma)^{-1}\cdot\left(\xi_{\rho}(\gamma\cdot x_1),\xi_{\rho}(\gamma\cdot x_2),\xi_{\rho}^{(1)}(\gamma\cdot x_3)\right)=\left(\xi_{\rho_t}(\gamma\cdot x_1),\xi_{\rho_t}(\gamma\cdot x_2),\xi_{\rho_t}^{(1)}(\gamma\cdot x_3)\right),\] 
so $\rho_t(\gamma)\rho(\gamma)^{-1}=g_{{\bf x},\gamma\cdot{\bf x}}(\rho_t)$. Then by \eqref{eqn: tangent equation 1'}
\[\widetilde{\nu}_X(l)=\frac{d}{dt}\bigg|_{t=0}\rho_t(\gamma)\rho(\gamma)^{-1}.\]
Finally, if we set $\rho_t':=c_{u_{{\bf x},q}(\rho_t)^{-1}}\circ\rho_t$, then $\rho_t'(\gamma)$ fixes both $\xi_\rho(x_1)$ and $\xi_\rho(q)$. Therefore, by \eqref{eqn: tangent equation 2'},
\begin{eqnarray*}
\widetilde{\nu}_X(m)&=&\frac{d}{dt}\bigg|_{t=0}u_{{\bf x},q}(\rho_t)^{-1}\rho_t(\gamma)u_{{\bf x},q}(\rho_t)\rho(\gamma)^{-1}\\
&=&-\frac{d}{dt}\bigg|_{t=0}u_{{\bf x},q}(\rho_t)+\frac{d}{dt}\bigg|_{t=0}\rho_t(\gamma)\rho(\gamma)^{-1}+\frac{d}{dt}\bigg|_{t=0}\rho(\gamma)u_{{\bf x},q}(\rho_t)\rho(\gamma)^{-1}.
\end{eqnarray*}
It now follows that (1) holds.

(2) follows immediately from (1). Since (2) holds, to prove (3), one only needs to verify that 
\[\widetilde{\nu}_X(\bar{s})+\widetilde{\nu}_X(l)=\widetilde{\nu}_X(\bar{s}')+\widetilde{\nu}_X(l').\]
By Lemma \ref{lem:endpoints}(1) and Lemma \ref{lem:diag inverse}(1), if $s''$ is the concatenation of $s'$ and $\bar{s}$, then
\[\widetilde{\nu}_X(s')+\widetilde{\nu}_X(\bar{s})+\widetilde{\nu}_X(l)=\widetilde{\nu}_X(s'')+\widetilde{\nu}_X(l)=\widetilde{\nu}_X(l'),\]
so the required equality holds.
\end{proof}

%%%%%%%%%%%%%%%%%%%%%%%%%%%%%%%%%%%%%%%%%%%%%%%%%%
\subsection{Defining the tangent cocycle in general}\label{sec:tangent cocycle general}
%%%%%%%%%%%%%%%%%%%%%%%%%%%%%%%%%%%%%%%%%%%%%%%%%%
Next, using Cases 1, 2, and 3 above, we define $\widetilde{\nu}_X(h)$ for general $h$. Given a general $1$-simplex $h:[0,1]\to\widetilde{S}$, let 
\[h_{--},\,\, h_-,\,\, \hat{h},\,\, h_+,\,\, h_{++}:[0,1]\to\widetilde{S}\]
be $1$-simplices such that the following holds, see Figure \ref{fig:cocyclecase4}:
\begin{itemize}
\item $h(0)=h_{--}(0)$, $h_{--}(1)=h_-(0)$, $h_-(1)=\hat{h}(0)$, $\hat{h}(1)=h_+(0)$, $h_+(1)=h_{++}(0)$, and $h_{++}(0)=h(1)$,
\item if $h(0)$ does not lie in a non-isolated edge of $\widetilde{\Tmc}$, then 
\[h(0)=h_{--}(0)=h_{--}(1)=h_-(0)=h_-(1)=\hat{h}(0),\]
\item if $h(0)$ lies in a non-isolated edge $e\in\widetilde{\Tmc}$, then
$h_{--}(1)=h_-(0)=J\cap e$ for some bridge $J\in\widetilde{\Jmc}$ across $e$, and $h_-(1)=\hat{h}(0)$ is an endpoint of $J$.
\item if $h(1)$ does not lie in a non-isolated edge of $\widetilde{\Tmc}$, then 
\[\hat{h}(1)=h_+(0)=h_+(1)=h_{++}(0)=h_{++}(0)=h(1).\]
\item if $h(1)$ lies in a non-isolated edge $e\in\widetilde{\Tmc}$, then
$h_{++}(0)=h_+(1)=J\cap e$ for some bridge $J\in\widetilde{\Jmc}$ across $e$, and $h_+(0)=\hat{h}(1)$ is an endpoint of $J$.
\end{itemize}

\begin{figure}[ht]
\centering
\includegraphics[scale=0.8]{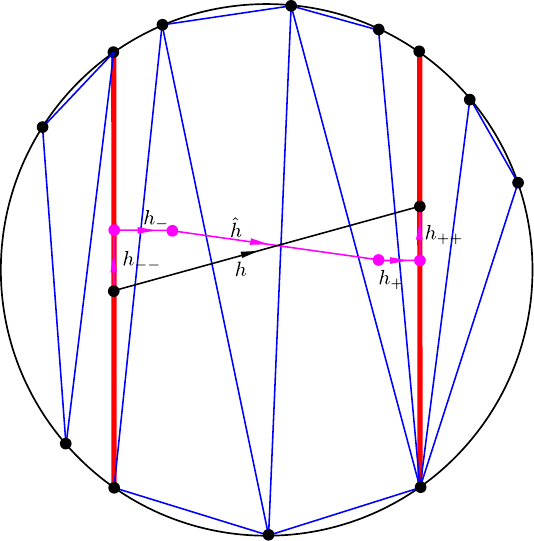}
\caption{Standard quintuple when $h(0)$ and $h(1)$ lie in $\mathfrak D$.}
\label{fig:cocyclecase4}
\end{figure}

\begin{definition}\label{def:standard5}
For any $1$-simplex $h:[0,1]\to\widetilde{S}$, we refer to 
\[(h_{--},\,\, h_-,\,\,\hat{h},\,\,  h_+,\,\, h_{++})\]
defined above as a \emph{standard quintuple} homotopic to $h$.
\end{definition}

Note that each of the five $1$-simplices in a standard quintuple of any $1$-simplex $h$ satisfy the conditions of either Case 1, Case 2, or Case 3. Thus, we may define 
\begin{equation}\label{eqn:welldefined}\widetilde{\nu}_X(h):=\widetilde{\nu}_X(h_{--})+\widetilde{\nu}_X(h_-)+\widetilde{\nu}_X(\hat{h})+\widetilde{\nu}_X(h_+)+\widetilde{\nu}_X(h_{++}),\end{equation}

\begin{thm}\label{thm:tangent cocycle}
For any $X\in T_{[\rho]}\Hit_V(S)$, the map $\widetilde{\nu}_X:C_1(\Std,\Zbbb)\to\smf\lmf(V)$ defined above is a well-defined, $\Ad\circ\rho$-equivariant cocycle. 
\end{thm}

\begin{proof}
First, we verify that $\widetilde{\nu}_X$ is well-defined, i.e. we check that 
\begin{enumerate}
\item $\widetilde{\nu}_X(h)$ as defined by (\ref{eqn:welldefined}) does not depend on the choice of standard quintuple homotopic to $h$, and 
\item $\widetilde{\nu}_X(h)$ as defined by (\ref{eqn:welldefined}) restricts to the definitions of $\widetilde{\nu}_X(h)$ given by 
\begin{itemize}
\item \eqref{eqn: Case 2} when $h$ satisfies the conditions of Case 1,
\item \eqref{eqn: tangent equation 3} or \eqref{eqn: tangent equation 4} when $h$ satisfies the conditions of Case 2,
\item \eqref{eqn: tangent equation 2} when $h$ satisfies the conditions of Case 3.
\end{itemize}
\end{enumerate}
We will only verify (1); once we know that (1) is true, the proof of (2) is a straightforward observation. 

To verify (1), pick any two standard quintuples
\[(h_{--},\,\, h_-,\,\, \hat{h},\,\, h_+,\,\, h_{++}),\,\,\text{ and }\,\, ( h_{--}',\,\, h_-',\,\,\hat{h}',\,\, h_+',\,\, h_{++}')\] 
homotopic to $h$. Let $k:[0,1]\to\widetilde{S}$ be a $1$-simplex such that $k(0)=\hat{h}'(0)$ and $k(1)=\hat{h}(1)$. It suffices to prove that
\begin{equation}\label{eqn: quin1}
\widetilde{\nu}_X(h_{--})+\widetilde{\nu}_X(h_-)+\widetilde{\nu}_X(\hat{h})=\widetilde{\nu}_X(h_{--}')+\widetilde{\nu}_X(h_-')+\widetilde{\nu}_X(k)
\end{equation}
and
\begin{equation}\label{eqn: quin2}
\widetilde{\nu}_X(\hat{h}')+\widetilde{\nu}_X(h_+')+\widetilde{\nu}_X(h_{++}')=\widetilde{\nu}_X(k)+\widetilde{\nu}_X(h_+)+\widetilde{\nu}_X(h_{++}).
\end{equation}
Indeed, \eqref{eqn: quin1} and \eqref{eqn: quin2} together imply that
\begin{align*}
\widetilde{\nu}_X(h_{--})+\widetilde{\nu}_X(h_-)+&\widetilde{\nu}_X(\hat{h})+\widetilde{\nu}_X(h_+)+\widetilde{\nu}_X(h_{++})\\
&=\widetilde{\nu}_X(h_{--}')+\widetilde{\nu}_X(h_-')+\widetilde{\nu}_X(\hat{h}')+\widetilde{\nu}_X(h_+')+\widetilde{\nu}_X(h_{++}').
\end{align*}
We will only give the proof of \eqref{eqn: quin1}; the proof of \eqref{eqn: quin2} is similar.

%\begin{equation}\label{eq:well-defined}
%\widetilde{\nu}(\hat{h})+\sum_{j=1}^2\sum_{m=0}^1\widetilde{\nu}(h_{m,j})=\widetilde{\nu}(\hat{h}')+\sum_{j=1}^2\sum_{m=0}^1\widetilde{\nu}(h_{m,j}').
%\end{equation}

\begin{figure}[ht]
\centering
\includegraphics[scale=0.37]{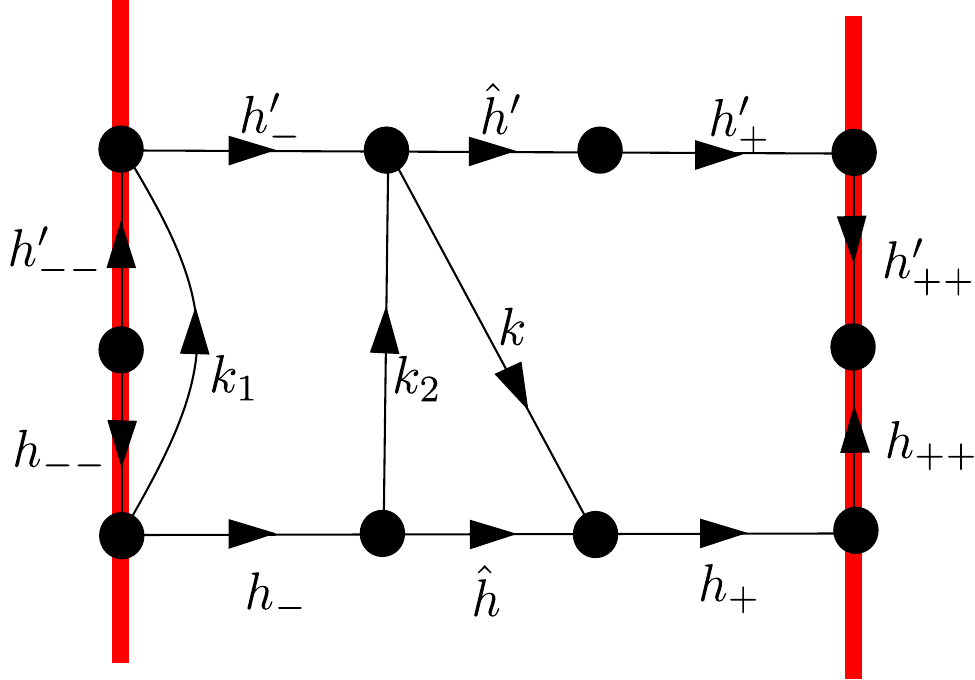}
\caption{$k_{0,1}$, $k_{0,2}$, $k_{1,1}$, and $k_{1,2}$.}\label{fig:doublequin}
\end{figure}

If $h(0)$ does not lie in a non-isolated edge of $\widetilde{\Tmc}$, then the endpoints of the images of $h_{--}$, $h_-$, $h_{--}'$ and $h_-'$, as well as the backward endpoints of $\hat{h}$ and $k$, are all $h(0)$. Hence, Lemma \ref{lem:endpoints}(1) implies that
\[\widetilde{\nu}_X(h_{--})=\widetilde{\nu}_X(h_-)=\widetilde{\nu}_X(h_{--}')=\widetilde{\nu}_X(h_-')=0\,\,\text{ and }\,\,\widetilde{\nu}_X(\hat{h})=\widetilde{\nu}_X(k),\]
so \eqref{eqn: quin1} holds. On the other hand, if $h(0)$ lies in a non-isolated edge $e$ of $\widetilde{\Tmc}$. Let $k_1,k_2:[0,1]\to\widetilde{S}$ be $1$-simplices such that 
\begin{itemize}
\item $k_1(0)=h_{--}(1)=h_-(0)$, 
\item $k_1(1)=h_{--}'(1)=h_-'(0)$, 
\item $k_2(0)=h_-(1)=\hat{h}(0)$,
\item $k_2(1)=h_-'(1)=\hat{h}'(0)=k(0)$, 
\end{itemize}
see Figure \ref{fig:doublequin}. Since the endpoints of $k_1$, $h_{--}$ and $h_{--}'$ all lie in the same non-isolated edge in $\widetilde{\Tmc}$, Lemma~\ref{lem:Case 3.1}(1) implies that 
\[\widetilde{\nu}_X(k_1)+\widetilde{\nu}_X(h_{--})=\widetilde{\nu}_X(h_{--}').\] 
Also, Lemma \ref{lem:endpoints}(1) implies that 
\[\widetilde{\nu}_X(k_2)+\widetilde{\nu}_X(k)=\widetilde{\nu}_X(\hat{h})\] 
and parts (2) and (3) of Lemma \ref{lem:tangent cocycle 1} imply that 
\[\widetilde{\nu}_X(k_1)+\widetilde{\nu}_X(h_-')=\widetilde{\nu}_X(h_-)+\widetilde{\nu}_X(k_2).\] 
Combining these together proves \eqref{eqn: quin1}.

Next, we will prove that $\widetilde{\nu}_X$ is a cocycle, i.e. 
\begin{equation}\label{eqn:cocycle proof} \widetilde{\nu}_X(h)+\widetilde{\nu}_X(h')=\widetilde{\nu}_X(h'').\end{equation}
for all $1$-simplices $h,h',h'':[0,1]\to\widetilde{S}$ such that $h(1)=h'(0)$ and $h''$ is the concatenation of $h$ with $h'$. To do so, choose standard quintuples
\[( h_{--},\,\,h_-,\,\,  \hat{h},\,\,h_+,\,\, h_{++}),\,\, (h_{--}',\,\,h_-',\,\,  \hat{h}',\,\, h_+',\,\, h_{++}'),\,\, (h_{--}'',\,\,h_-'',\,\,  \hat{h}'',\,\,  h_+'',\,\, h_{++}'')\] 
of $h$, $h'$, $h''$ respectively, such that $h''_{--}=h_{--}$, $h''_-=h_-$,  $h''_{++}=h'_{++}$, $h''_{+}=h'_{+}$, $\bar{h}_{++}=h'_{--}$, and $\bar{h}_+=h'_-$. With these choices, $\hat{h}''$ is a concatenation of $\hat{h}$ and $\hat{h}'$, so \eqref{eqn:cocycle proof} holds by Lemma~\ref{lem:endpoints}(1). The $\Ad\circ\rho$-equivariance of $\widetilde{\nu}$ is a consequence of Lemma \ref{lem:endpoints}(2), Lemma \ref{lem:diag inverse}(2), and  Lemma~\ref{lem:Case 3.1}(2).
\end{proof}

%%%%%%%%%%%%%%%%%%%%%%%%%%%%%%%%%%%%%%%%%%%%%%%%%%
\subsection{Properties of $\Psi$}\label{sec:tangent cocycle vector space}
%%%%%%%%%%%%%%%%%%%%%%%%%%%%%%%%%%%%%%%%%%%%%%%%%% 

By Theorem \ref{thm:tangent cocycle}, we may now define the required linear map 
\[\Psi=\Psi_{\rho,\Tmc,\Jmc}:T_{[\rho]}\Hit_V(S)\to C^1(S,\smf\lmf(V)_{\Ad\circ\rho})\] 
by $\Psi(X)=\nu_X$. The next proposition proves that $\Psi$ satisfies the properties stated at the start of this section.

\begin{prop}\label{prop:tangentcocyclebij}
The map $\Psi$ is a linear injection and satisfies $[\Psi(X)]=F_\rho(X)\in H^1(S,\mathfrak{sl}(V)_{\Ad\circ\rho})$ for all $X\in T_{[\rho]}\Hit_V(S)$. In particular, the set $\mathscr{T}(\rho,\Tmc,\Jmc)$ of $(\rho,\Tmc,\Jmc)$-tangent cocycles is a linear subspace of $C^1(S,\mathfrak{sl}(V)_{\Ad\circ\rho})$ whose dimension is $(2g-2)(n^2-1)$. 
\end{prop}

\begin{proof}
In Section \ref{sec: tangent}, we described the canonical isomorphism
\[H^1(\Gamma,\smf\lmf(V)_{\Ad\circ\rho})\simeq H^1(S,\smf\lmf(V)_{\Ad\circ\rho}).\] 
We will verify that $F_\rho(X)=[\Psi(X)]$ by showing that $F_\rho(X)$ and $[\Psi(X)]=[\nu_X]$ are identified via the above isomorphism with the same cohomology class in $H^1(\Gamma,\smf\lmf(V)_{\Ad\circ\rho})$. Choose any point $o$ that lies in the interior of a fragmented ideal triangle of the barrier system $\widetilde{\Dmc}$, and let ${\bf x}=(x_1,x_2,x_3)$ be the vertices of the ideal triangle of $\widetilde{\Tmc}$ that contains $o$, enumerated so that $x_3$ is not a vertex of the fragmented ideal triangle containing $o$. 

First, we describe the group cocycle that represents $[\nu_X]$. Let $\widetilde{\nu}_X:C_1(\Std,\Zbbb)\to\smf\lmf(V)$ be the $\Ad\circ\rho$-equivariant lift of the cocycle $\nu_X$. By the description of the isomorphism $H^1(\Gamma,\smf\lmf(V)_{\Ad\circ\rho})\simeq H^1(S,\smf\lmf(V)_{\Ad\circ\rho})$ in Section \ref{sec: tangent}, the map 
\[\bar{\nu}_X:\Gamma\to\mathfrak{sl}(V)\] 
given by $\bar{\nu}_X:\gamma\mapsto\widetilde{\nu}_X(h_\gamma)$, where $h_\gamma:[0,1]\to\widetilde{S}$ is a $1$-simplex with $h_\gamma(0)=o$ and $h_\gamma(1)=\gamma\cdot o$, is a group cocycle that represents $[\nu_X]$.

Next, we describe the group cocycle that represents $F_\rho(X)$. Let $t\mapsto\rho_t$ be a smooth path in $\Hmc_{\rho,{\bf x}}$ such that $\rho_0=\rho$, and 
\[{\rm d}\pi\left(\frac{d}{dt}\bigg|_{t=0}\rho_t\right)=X.\] 
Then by the description of the isomorphism
\[T_{[\rho]}\Hit_V(S)\simeq H^1(\Gamma,\smf\lmf(V)_{\Ad\circ\rho})\]
given in Section \ref{sec: tangent}, the map $\mu_X:\Gamma\to\mathfrak{sl}(V)$ given by $\mu_X:\gamma\mapsto \frac{d}{dt}\big|_{t=0}\rho_t(\gamma)\rho(\gamma)^{-1}$, is a group cocycle that represents $F_{\rho}(X)$.

To prove $F_\rho(X)=[\Psi(X)]$, it now suffices to show that $\mu_X=\bar{\nu}_X$. To see this, observe that 
\[\rho_t(\gamma)\rho(\gamma)^{-1}\cdot\left(\xi_\rho(\gamma\cdot x_1),\xi_\rho(\gamma\cdot x_2),\xi_\rho^{(1)}(\gamma\cdot x_3)\right)=\left(\xi_{\rho_t}(\gamma\cdot x_1),\xi_{\rho_t}(\gamma\cdot x_2),\xi_{\rho_t}^{(1)}(\gamma\cdot x_3))\right),\]
so $\rho_t(\gamma)\rho(\gamma)^{-1}=g_{{\bf x},\gamma\cdot{\bf x}}(\rho_t)$. Thus,
\[\bar{\nu}_X(\gamma)=\widetilde{\nu}_X(h_\gamma)=\frac{d}{dt}\bigg|_{t=0}g_{{\bf x},\gamma\cdot{\bf x}}(\rho_t)=\frac{d}{dt}\bigg|_{t=0}\rho_t(\gamma)\rho(\gamma)^{-1}=\mu_X(\gamma).\]

Next, we prove that $\Psi$ is an injective linear map. To check that $\Psi$ is injective, let $X_1$ and $X_2$ be vectors in $T_{[\rho]}\Hit_V(S)$. If $\nu_{X_1}=\nu_{X_2}$, then we have proven that $F_\rho(X_1)=[\nu_{X_1}]=[\nu_{X_2}]=F_\rho(X_2)$, which means that $X_1=X_2$ because $F_\rho:T_{[\rho]}\Hit_V(S)\to H^1(S,\smf\lmf(V)_{\Ad\circ\rho})$ is an isomorphism. Thus, $\Psi$ is injective. That fact that $\Psi$ is linear follows immediately from the observation that when $h$ satisfies the conditions of Cases 1, 2, or 3 in Section \ref{sec:tangent cocycle}, the map
\[T_{[\rho]}\Hit_V(S)\to\smf\lmf(V)\] 
given by $X\mapsto\widetilde{\nu}_X(h)$ (see \eqref{eqn: tangent equation 1}, \eqref{eqn: tangent equation 3}, and  \eqref{eqn: tangent equation 2}), is linear. 
\end{proof}

%\begin{remark}
%We show in Theorem \ref{thm:iso} that $\mathscr{T}$ is a linear subspace of $C^1(S,\smf\lmf(V)_{\Ad\circ\rho})$. It follows immediately from this that $\Psi$ is a linear isomorphism.
%\end{remark}

%%%%%%%%%%%%%%%%%%%%%%%%%%%%%%%%%%%%%%%%%%%%%%%%%%%%%%%%%%%%%%%%%%%%%%%%%%%%%%%%%%%%%%%%%%%%%%%%%%%%
\section{The $(\Tmc,\Jmc)$-trivialization of $T\Hit_V(S)$}\label{sec: coeff and lab}
%%%%%%%%%%%%%%%%%%%%%%%%%%%%%%%%%%%%%%%%%%%%%%%%%%%%%%%%%%%%%%%%%%%%%%%%%%%%%%%%%%%%%%%%%%%%%%%%%%%%

Fix an ideal triangulation $\Tmc$ on $S$ and a compatible bridge system $\Jmc$. Using the pair $(\Tmc,\Jmc)$, Bonahon and Dreyer \cite{BonahonDreyer1} gave an explicit, real-analytic parameterization of $\Hit_V(S)$ by an open convex cone $ \mathscr{C}= \mathscr{C}_{\Tmc}$ in a $(n^2-1)(2g-2)$-dimensional subspace $ \mathscr{W}= \mathscr{W}_{\Tmc}$ of $\Rbbb^{\Tmc^o\times\Amc}\times\Rbbb^{\Mmc\times\Bmc}$, where $\Amc$ denotes the set of pairs of positive integers that sum to $n$, and $\Bmc$ denotes the set of triples of positive integers that sum to $n$. (Recall that $\Tmc^o$ is the oriented ideal triangulation associated to $\Tmc$ and $\Mmc$ is the set of non-edge barriers in the barrier system $\Dmc$ associated to $\Tmc$.) In the companion paper \cite{SunWienhardZhang}, we defined a ($n^2-1)(2g-2)$-dimensional family of pairwise commuting vector fields on $\Hit_V(S)$, called the \emph{$(\Tmc,\Jmc)$-parallel vector fields}. We further showed that these can be described easily using a reparamaterization of Bonahon and Dreyer's parameterization of $\Hit_V(S)$ by $\mathscr{C}$.

\begin{thm}\cite[Theorem 1.1]{SunWienhardZhang}\label{thm: par}
There is a real-analytic diffeomorphism 
\[\Omega=\Omega_{\Tmc,\Jmc}:\Hit_V(S)\to \mathscr{C}\]
such that the $(\Tmc,\Jmc)$-parallel vector fields on $\Hit_V(S)$ are exactly the vector fields that are identified via 
\[{\rm d}\Omega:T\Hit_V(S)\to T\mathscr{C}\cong\mathscr{C}\times\mathscr{W}\] 
to the constant vector fields on $\mathscr{C}$. 
\end{thm}

The trivialization ${\rm d}\Omega$ is called the \emph{$(\Tmc,\Jmc)$-trivialization} of $T\Hit_V(S)$.

The goal of this section is to relate the $(\Tmc,\Jmc)$-trivialization to the tangent cocycles described in Section \ref{sec: tangent}. More precisely, fix $[\rho]\in\Hit_V(S)$ and a representative $\rho:\Gamma\to\PGL(V)$. Let $\xi_\rho:\partial\Gamma\to\Fmc(V)$ denote the $\rho$-equivariant Frenet curve. Recall that in Section \ref{sec: tangent cocycles}, we defined a linear embedding
\[\Psi=\Psi_{\rho,\Tmc,\Jmc}:T_{[\rho]}\Hit_V(S)\to C^1(S,\mathfrak{sl}(V)_{\Ad\circ\rho})\]
whose image $\mathscr{T}=\mathscr{T}(\rho,\Tmc,\Jmc)$ is the space of $(\rho,\Tmc,\Jmc)$-tangent cocycles. We will explicitly define a linear map
\[\Xi=\Xi_{\rho,\Tmc,\Jmc}:\mathscr{T}\to\Rbbb^{\Tmc^o\times\Amc}\times\Rbbb^{\Mmc\times\Bmc}\]
so that the following theorem holds.

\begin{thm}\label{thm: symplectic trivialization}
The map $\Xi$ is a linear embedding whose image is the subspace $ \mathscr{W}$. Furthermore, 
\[{\rm d}\Omega_{[\rho]}=\Xi\circ\Psi:T_{[\rho]}\Hit_V(S)\to T_{\Omega([\rho])} \mathscr{C}\cong  \mathscr{W}.\]
\end{thm}

Theorem \ref{thm: symplectic trivialization} will later be used in Section \ref{sec:trivialization} to show that the $(\Tmc,\Jmc)$-trivialization is symplectic.

The linear map $\Xi$ will be defined in several steps. Let $\nu$ be a $(\rho,\Tmc,\Jmc)$-tangent cocycle, and let $\widetilde{\nu}:C_1(\widetilde S,\Zbbb)\to\mathfrak{sl}(V))$ be its $\Ad\circ\rho$-equivariant lift. First, in Section \ref{sec:eruption}, we introduce the notion of an \emph{eruption endomorphism}, and use it to specify the set of possible values of $\widetilde{\nu}(h)$ when $h$ is a $1$-simplex that crosses exactly one non-edge barrier in the barrier system $\widetilde{\Dmc}$ associated to $\widetilde{\Tmc}$. Then, in Section \ref{sec:shearing}, we introduce the notion of a \emph{shearing endomorphism}, and use it to specify the set of possible values of $\widetilde{\nu}(h)$ when $h$ is $1$-simplex that crosses exactly one isolated edge in $\widetilde{\Tmc}$, and also the set of possible values of the ``diagonal part" of $\widetilde{\nu}(h)$ when $h$ is a $1$-simplex whose image is a bridge in $\widetilde{\Jmc}$. Using these, we define $\Xi$ in Section \ref{sec: Xi}  and prove the first statement in Theorem \ref{thm: symplectic trivialization}, i.e. that $\Xi$ is injective and its image $\mathscr{W}$. Finally, in Section \ref{sec:admissibility linear independence}, we recall the description of the map $\Omega$, and combine the definition of $\Xi$ with results from the companion paper \cite{SunWienhardZhang} to prove the second statement of Theorem \ref{thm: symplectic trivialization}.

%%%%%%%%%%%%%%%%%%%%%%%%%%%%%%%%%%%%%%%%%%%%%%%%%%
\subsection{Eruption endomorphisms}\label{sec:eruption}
%%%%%%%%%%%%%%%%%%%%%%%%%%%%%%%%%%%%%%%%%%%%%%%%%%
Given any non-edge barrier $\mathbf b$ in $\widetilde{\Dmc}$, let $x_1$, $x_2$, and $x_3$ be the vertices of the ideal triangle $T$ containing ${\bf b}$ so that $x_1<x_2<x_3<x_1$ and $x_1$ is the terminal endpoint of ${\bf b}$. We say that ${\bf x}:=(x_1,x_2,x_3)$ is the \emph{triple associated to ${\bf b}$}, see Figure \ref{fig:eruptionbases}. Let $\mathfrak{T}_2$ (resp. $\mathfrak{T}_3$) be the fragmented ideal triangle of $\widetilde{\Dmc}$ in $T$ that does not have $x_2$ (resp. $x_3$) as a vertex. We say that a $1$-simplex $h:[0,1]\to\widetilde S$ \emph{passes from the left to the right of ${\bf b}$} if $h(0)$ and $h(1)$ lie in the interiors of $\mathfrak{T}_2$ and $\mathfrak{T}_3$ respectively, see Figure \ref{fig:eruptionbases}.

\begin{figure}[ht]
\centering
\includegraphics[scale=0.6]{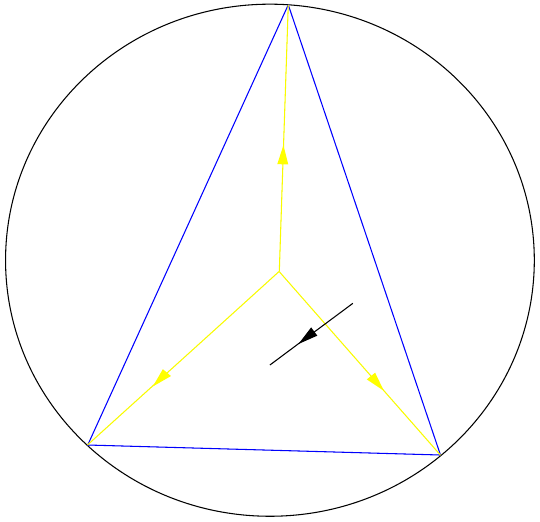}
\put (-70, 46){\tiny $h$}
\put (-55, 40){\tiny ${\bf b}$}
\put (-106, 39){\tiny ${\bf b}_+$}
\put (-85, 102){\tiny ${\bf b}_-$}
\put (-28, 16){\tiny $x_1$}
\put (-138, 19){\tiny $x_2$}
\put (-77, 151){\tiny $x_3$}
\put (-95, 75){\tiny ${\mathfrak{T}_1}$}
\put (-65, 75){\tiny ${\mathfrak{T}_2}$}
\put (-80, 35){\tiny ${\mathfrak{T}_3}$}
\caption{$h$ passes from the left to the right of ${\bf b}$.}\label{fig:eruptionbases}
\end{figure}

Let $\nu$ be a $(\rho,\Tmc,\Jmc)$-tangent cocycle, and let $\widetilde{\nu}:C_1(\widetilde S,\Zbbb)\to\mathfrak{sl}(V))$ be its $\Ad\circ\rho$-equivariant lift. We will now describe the possible values of $\widetilde{\nu}(h)$, where $h:[0,1]\to\widetilde S$. To do so, we use the notion of an eruption endomorphism.

\begin{definition}\label{def:eruption endomorphism}
Let $\mathbf i=(i_1,i_2,i_3)$ be a triple of positive integers that sum to $n$, and let $\mathbf F=(F_1,F_2,F_3)$ be a generic triple of flags in $\Fmc(V)$. The \emph{$\mathbf i$-eruption endomorphism with respect to $\mathbf F$} is the endomorphism $A^{\mathbf i}_{\mathbf F}\in\smf\lmf(V)$ with eigenspaces $F_1^{(i_1)}$ and $F_2^{(i_2)}+F_3^{(i_3)}$ corresponding to eigenvalues $\frac{i_2+i_3}{n}$ and $-\frac{i_1}{n}$ respectively.
\end{definition}

%\begin{remark}
%The eruption endomorphisms are motivated by results in the companion paper \cite{SunWienhardZhang}: Let $\mathbf x=(x_1,x_2,x_3)$ be a triple of pairwise distinct points in $S^1$, and let ${\bf i}=(i_1,i_2,i_3)$ be a triple of positive integers that sum to $n$. Set $\mathbf x_+:=(x_2,x_3,x_1)$, $\mathbf x_-:=(x_3,x_1,x_2)$, $\mathbf i_+:=(i_2,i_3,i_1)$, and $\mathbf i_-:=(i_3,i_1,i_2)$. Then for any Frenet map $\xi:S^1\to\Fmc(V)$, $\xi(\mathbf x)$ is a generic triple of flags, so we can define
%\[a^{\mathbf i_m}_{\xi(\mathbf x_m)}(t):=\exp\left(tA^{\mathbf i_m}_{\xi(\mathbf x_m)}\right)\in\PGL(V)\] 
%for all $t\in\Rbbb$. By \cite[Theorem 3.6]{SunWienhardZhang}, if $x_1<x_2<x_3<x_1$ in the clockwise order on $S^1$, then the map $\xi_t:S^1\to\Fmc(V)$ defined by
%\[\xi_t(p)=\left\{
%\begin{array}{ll}
%\xi(p)&\text{if }x_3\leq p\leq x_1;\\
%a^{\mathbf i_1}_{\xi(\mathbf x_1)}(t)\cdot \xi(p)&\text{if }x_1\leq p\leq x_2;\\
%a^{\mathbf i_3}_{\xi(\mathbf x_3)}(t)^{-1}\cdot \xi(p)&\text{if }x_2\leq p\leq x_3
%\end{array}\right.\]
%is a well-defined and Frenet. Also, if $\xi$ and $\xi'$ are projectively equivalent, then so are $\xi_t$ and $\xi_t'$ \cite[Proposition 3.13]{SunWienhardZhang}. In other words, the choice of ${\bf x}$ and ${\bf i}$ defines a flow on the space of projective classes of Frenet maps, called the \emph{$\mathbf i$-elementary eruption flow with respect to $\mathbf x$}, and $A^{\mathbf i_m}_{\xi(\mathbf x_m)}$ can be thought of as the ``derivative information" of this flow at the projective class $[\xi]$.
%\end{remark}

For any triple $\mathbf i=(i_1,i_2,i_3)$ of positive integers that sum to $n$, and any generic triple of flags $\mathbf F=(F_1,F_2,F_3)$ in $\Fmc(V)$, observe that 
\[A^{\mathbf i}_{\mathbf F}+A^{\mathbf i_+}_{\mathbf F_+}+A^{\mathbf i_-}_{\mathbf F_-}=0,\]
where $\mathbf i_+:=(i_2,i_3,i_1)$, $\mathbf i_-:=(i_3,i_1,i_2)$, $\mathbf F_+:=(F_2,F_3,F_1)$, and $\mathbf F_-:=(F_3,F_1,F_2)$. Also, if 
\[\{f_{m,1},\dots,f_{m,n}\}\] 
is a basis of $V$ such that $F_m^{(j)}=\Span_\Rbbb\{f_{m,1},\dots,f_{m,j}\}$ for all $m=1,2,3$ and $j=1,\dots, n$, then the genericity of $\mathbf F$ implies that
\begin{align}\label{eqn: basis}
\{f_{1,1},\dots,f_{1,i_1},f_{2,1},\dots,f_{2,i_2},f_{3,1},\dots,f_{3,i_3}\}
\end{align} 
is a basis of $V$. In this basis, we can write $A^{\mathbf i}_{\mathbf F}$ as the matrix
\[
A^{\mathbf i}_{\mathbf F}=\left[\begin{array}{cc}\frac{i_2+i_3}{n}\cdot\id_{i_1}&0\\0&-\frac{i_1}{n}\cdot\id_{i_2+i_3}\end{array}\right],
\]
where $\id_j$ is the $j\times j$ identity matrix. 

The next proposition states some important properties of $A^{\mathbf i}_{\mathbf F}$. 

\begin{prop}\label{prop:linearly independent triple}
Let $\mathbf F:=(F_1,F_2,F_3)$ be a generic triple of flags in $\Fmc(V)$. 
\begin{enumerate} 
\item For all triples ${\bf i}=(i_1,i_2,i_3)$ of positive integers that sum to $n$, $A^{\mathbf i}_{\mathbf F}\cdot F_1^{(j)}\subset F_1^{(j)}$ for all $j=1,\dots,n-1$, but $A^{\mathbf i}_{\mathbf F}\cdot F_2^{(i_2+1)}\not\subset F_2^{(i_2+1)}$ and $A^{\mathbf i}_{\mathbf F}\cdot F_3^{(i_3+1)}\not\subset F_3^{(i_3+1)}$.
%\item $\left\{A^{\mathbf i}_{\mathbf F}:i_1+i_2+i_3=n\right\}$ is linearly independent in $\smf\lmf(V)$.
\item Let $G_\mathbf F\subset\PGL(V)$ be the Lie subgroup defined by
\[G_\mathbf F:=\{g\in\PGL(V):g\cdot F_1=F_1,\,\,g\cdot F_2^{(1)}=F_2^{(1)},\,\,\text{ and }\,\,g\cdot F_3^{(1)}=F_3^{(1)}\}.\]
Then the set $\left\{A^{\mathbf i}_{\mathbf F}:i_1+i_2+i_3=n\right\}$ is a basis for the Lie algebra $\mathfrak g_\mathbf F$ of $G_\mathbf F$.
\end{enumerate}
\end{prop}

\begin{proof}
We first compute $A_{\mathbf F}^{\mathbf i}\cdot v$ for any vector $v\in V$. Observe that we may write $v$ uniquely as
\[ v=\sum_{j=1}^{i_1} a_{1,j}f_{1,j}+\sum_{j=1}^{i_2}a_{2,j}f_{2,j}+\sum_{j=1}^{i_3}a_{3,j}f_{3,j}\]
for some constants $a_{m,j}\in\Rbbb$, where $\{f_{1,1},\dots,f_{1,i_1},f_{2,1},\dots,f_{2,i_2},f_{3,1},\dots,f_{3,i_3}\}$ is the basis given by \eqref{eqn: basis}. Then
\begin{eqnarray}\label{eqn:basis move}
A_{\mathbf F}^{\mathbf i}\cdot v&=&\sum_{j=1}^{i_1} \frac{i_2+i_3}{n}a_{1,j}f_{1,j}-\sum_{j=1}^{i_2}\frac{i_1}{n}a_{2,j}f_{2,j}-\sum_{j=1}^{i_3}\frac{i_1}{n}a_{3,j}f_{3,j}\\
&=&\sum_{j=1}^{i_1} a_{1,j}f_{1,j}-\frac{i_1}{n}v.\nonumber
\end{eqnarray}
%In particular, $A_{\mathbf F}^{\mathbf i}\cdot v \in F_1^{(i_1)}+\Span_\Rbbb\{v\}$. 

First, we prove (1). It follows immediately from \eqref{eqn:basis move} implies that $A^{\bf i}_{\bf F}\cdot F_1^{(j)}\subset F_1^{(j)}$. On the other hand, if $v\in F_2^{(i_2+1)}\setminus F_2^{(i_2)}$, then the assumption that ${\bf F}$ is generic implies that $a_{1,i_1}\neq 0$, and that $F_1^{(i_1)}$ is transverse to $F_2^{(i_2+1)}$. Thus, by \eqref{eqn:basis move} that $A_{\mathbf F}^{\mathbf i}\cdot v\notin F_2^{(i_2+1)}$, which implies that $A^{\mathbf i}_{\mathbf F}\cdot F_2^{(i_2+1)}\not\subset F_2^{(i_2+1)}$. Similarly, $A^{\mathbf i}_{\mathbf F}\cdot F_3^{(i_3+1)}\not\subset F_3^{(i_3+1)}$.

To prove (2), first observe that $\Span_\Rbbb\left\{A^{\mathbf i}_{\mathbf F}:i_1+i_2+i_3=n\right\}\subset\mathfrak g_{\mathbf F}$:
Indeed, it follows from the definition of $A^\mathbf i_\mathbf F$ that $\exp(tA^\mathbf i_\mathbf F)$ fixes $F_2^{(1)}$ and $F_3^{(1)}$ for all $t\in\Rbbb$, and (1) implies that $\exp(tA^\mathbf i_\mathbf F)$ fixes $F_1$ for all $t\in\Rbbb$. It is also a straightforward calculation using the genericity of $\mathbf F$ that the dimension of $\mathfrak g_{\mathbf F}$ is also $\frac{(n-1)(n-2)}{2}$. Thus, it suffices to prove that the dimension of $\Span_\Rbbb\left\{A^{\mathbf i}_{\mathbf F}:i_1+i_2+i_3=n\right\}$ is $\frac{(n-1)(n-2)}{2}$, i.e. $\left\{A^{\mathbf i}_{\mathbf F}:i_1+i_2+i_3=n\right\}$ is a linearly independent collection of endomorphisms. 

Suppose for contradiction that there is a non-empty collection $\Bmc'$ of triples of positive integers that sum to $n$, such that
\begin{equation}\label{eqn:triple}
\sum_{\mathbf i\in\Bmc'}\alpha_{\mathbf i}A^{\mathbf i}_{\mathbf F}=0
\end{equation}
and $\alpha_{\mathbf i}\in\Rbbb$ are non-zero for all $\mathbf i=(i_1,i_2,i_3)\in\Bmc'$. Let $k:=\min\{i_2:\mathbf i\in\Bmc'\}$, let 
\[\Bmc'':=\{\mathbf i\in\Bmc':i_2=k\},\] 
and let $\mathbf j:=(j_1,j_2,j_3)\in\Bmc''$ be the triple whose first coordinate $j_1$ is the largest among all triples $\mathbf i\in\Bmc''$. By the definition of $A^\mathbf i_\mathbf F$, 
\[A^\mathbf i_\mathbf F\cdot f_{2,k+1}=-\frac{i_1}{n}f_{2,k+1}\in\Span_\Rbbb\{f_{2,k+1}\}\] 
for all $\mathbf i\in\Bmc'\setminus\Bmc''$. Also, (\ref{eqn:basis move}) implies that 
\[A^\mathbf i_\mathbf F\cdot f_{2,k+1}\in F_1^{(i_1)}+\Span_\Rbbb\{f_{2,k+1}\}\subset F_1^{(j_1-1)}+\Span_\Rbbb\{f_{2,k+1}\}\]
for all $\mathbf i\in\Bmc''\setminus\{\mathbf j\}$. As a consequence, we see that
\[\sum_{\mathbf i\in\Bmc'\setminus\{\mathbf j\}}\alpha_{\mathbf i}A^{\mathbf i}_{\mathbf F}\cdot f_{2,k+1}\in F_1^{(j_1-1)}+\Span_\Rbbb\{f_{2,k+1}\}.\]

On the other hand, if we write
\[f_{2,k+1}=f_{2,j_2+1}=\sum_{j=1}^{j_1} a_{1,j}f_{1,j}+\sum_{j=1}^{j_2}a_{2,j}f_{2,j}+\sum_{j=1}^{j_3}a_{3,j}f_{3,j}\]
for some $a_{m,j}\in\Rbbb$, then the genericity of $\mathbf F$ implies that $a_{1,j_1}\neq 0$. It then follows from (\ref{eqn:basis move}) that $A^\mathbf j_\mathbf F\cdot f_{2,k+1}$ lies in $F_1^{(j_1)}+\Span_\Rbbb\{f_{2,k+1}\}$, but not in $F_1^{(j_1-1)}+\Span_\Rbbb\{f_{2,k+1}\}$. Thus, 
\[\sum_{\mathbf i\in\Bmc'}\alpha_{\mathbf i}A^{\mathbf i}_{\mathbf F}\cdot f_{2,k+1}=\sum_{\mathbf i\in\Bmc'\setminus\{\mathbf j\}}\alpha_{\mathbf i}A^{\mathbf i}_{\mathbf F}\cdot f_{2,k+1}+A^{\mathbf j}_{\mathbf F}\cdot f_{2,k+1}\neq 0,\]
which contradicts (\ref{eqn:triple}).
\end{proof}

%Let $\rho:\Gamma\to\PGL(V)$ be a Hitchin representation, let $\Tmc$ be an ideal triangulation on $S$, let $\Jmc$ be a compatible bridge system, and let $S$ be $S$ equipped with a hyperbolic metric. 
We now relate eruption endomorphisms to $\widetilde{\nu}$.

\begin{prop}\label{prop:eruption justification}
Let ${\bf b}$ be a non-edge barrier in $\widetilde{\Dmc}$, let ${\bf x}=(x_1,x_2,x_3)$ be the triple associated to ${\bf b}$, and let $h:[0,1]\to\widetilde{S}$ be a $1$-simplex that passes from the left to the right of ${\bf b}$. Then 
for each triple ${\bf i}=(i_1,i_2,i_3)$ of positive integers that sum to $n$, there is a unique $\alpha^{{\bf i}}_{{\bf b}}=\alpha^{{\bf i}}_{{\bf b}}(\nu)\in\Rbbb$ such that
\[\widetilde{\nu}(h)=\sum_{i_1+i_2+i_3=n}\alpha^{\mathbf{i}}_{\mathbf{b}}\cdot A^{\mathbf{i}}_{\xi_\rho(\mathbf{x})}.\]
Furthermore, 
\begin{enumerate}
\item if ${\bf i}_+:=(i_2,i_3,i_1)$, ${\bf i}_-:=(i_3,i_1,i_2)$, and ${\bf b}_+$ (resp. ${\bf b}_-$) is the non-edge barrier in $\widetilde{\Dmc}$ with the same source as ${\bf b}$ but whose terminal endpoint is $x_2$ (resp. $x_3$), see Figure \ref{fig:eruptionbases}, then
\[\alpha^{{\bf i}}_{{\bf b}}=\alpha^{{\bf i}_+}_{{\bf b}_+}=\alpha^{{\bf i}_-}_{{\bf b}_-}.\]
\item if $\gamma\in\Gamma$, then
\[\alpha^{{\bf i}}_{{\bf b}}=\alpha^{{\bf i}}_{\gamma\cdot{\bf b}}.\]
\end{enumerate}
\end{prop}

\begin{proof}
Let $\Hmc_{\rho,{\bf x}}$ be the real-analytic submanifold of $\widetilde{\Hit}_V(S)$ defined by \eqref{eqn: submanifold}, let ${\bf x}_-:=(x_3,x_1,x_2)$ and ${\bf x}_+:=(x_2,x_3,x_1)$, and let $g_{{\bf x},{\bf x}_-}:\Hmc_{\rho,{\bf x}}\to\PGL(V)$ be the real-analytic map defined by \eqref{eqn: g def}.
%\begin{align*}
%g_{{\bf x},{\bf x}_3}:\Hmc_{\rho,{\bf x}}\to\PGL(V)
%\end{align*}
Observe from the definitions that $g_{{\bf x},{\bf x}_-}(\sigma)\in G_{\xi_\rho(\bf x)}$ for all $\sigma\in\Hmc_{\rho,{\bf x}}$ and $g_{{\bf x},{\bf x}_-}(\rho)=\id$. By \eqref{eqn: tangent equation 1'} and Lemma \ref{lem:endpoints}(1), $\widetilde{\nu}(h)=-\widetilde{\nu}(\bar{h})=\frac{d}{dt}|_{t=0}g_{{\bf x},{\bf x}_-}(\rho_t)$ for some smooth path $t\mapsto \rho_t$ in $\Hmc_{\rho,{\bf x}}$ with $\rho_0=\rho$. Proposition \ref{prop:linearly independent triple}(2) then implies the first claim of the proposition. 

To prove (1), choose a basis $\{f_1,\dots,f_n\}$ such that $f_j\in\xi_\rho^{(j)}(x_2)\cap\xi_\rho^{(n+1-j)}(x_3)$ for all $j=1,\dots,n$. For any $\mathbf i\in\Bmc$, Proposition \ref{prop:linearly independent triple}(1) implies that $A^{\mathbf i_+}_{\xi_\rho(\mathbf x_+)}$ and $A^{\mathbf i_-}_{\xi_\rho(\mathbf x_-)}$ are given by an upper triangular and lower triangular matrix in this basis respectively. Furthermore, neither $A^{\mathbf i_+}_{\xi_\rho(\mathbf x_+)}$ nor $A^{\mathbf i_-}_{\xi_\rho(\mathbf x_-)}$ is given by a diagonal matrix in this basis. This, together with Proposition~\ref{prop:linearly independent triple}(2), imply that 
\[\left\{A^{\mathbf i_+}_{\xi_\rho(\mathbf x_+)}:i_1+i_2+i_3=n\right\}\cup\left\{A^{\mathbf i_-}_{\xi_\rho(\mathbf x_-)}:i_1+i_2+i_3=n\right\}\]
is a linearly independent collection of endomorphisms in $\smf\lmf(V)$.

Now, suppose for the purpose of contradiction that $\alpha^{\mathbf i}_{\mathbf b}\neq \alpha^{\mathbf i_+}_{\mathbf b_+}$ or $\alpha^{\mathbf i}_{\mathbf b}\neq \alpha^{\mathbf i_-}_{\mathbf b_-}$ for some $\mathbf i\in\Bmc$. Let $h,h_+,h_-:[0,1]\to\widetilde{S}$ be $1$-simplices that pass from the left to the right of ${\bf b}$, ${\bf b}_+$, and ${\bf b}_-$ respectively, $h(1)=h_+(0)$, $h_+(1)=h_-(0)$, and $h_-(1)=h(0)$, see Figure \ref{fig:eruptionbases}. Then Proposition~\ref{prop:eruption justification} implies
\begin{eqnarray*}
0&=&\widetilde{\nu}(h)+\widetilde{\nu}(h_+)+\widetilde{\nu}(h_-)\\
&=&\sum_{\mathbf i\in\Bmc}\left(\alpha^{\mathbf i}_{\mathbf b}\cdot A^{\mathbf i}_{\xi_\rho(\mathbf x)}+\alpha^{\mathbf i_+}_{\mathbf b_+}\cdot A^{\mathbf i_+}_{\xi_\rho(\mathbf x_+)}+\alpha^{\mathbf i_-}_{\mathbf b_-}\cdot A^{\mathbf i_-}_{\xi_\rho(\mathbf x_-)}\right)\\
&=&\sum_{\mathbf i\in\Bmc}\left((\alpha^{\mathbf i_+}_{\mathbf b_+}-\alpha^{\mathbf i}_{\mathbf b})\cdot A^{\mathbf i_+}_{\xi_\rho(\mathbf x_+)}+(\alpha^{\mathbf i_-}_{\mathbf b_-}-\alpha^{\mathbf i}_{\mathbf b})\cdot A^{\mathbf i_-}_{\xi_\rho(\mathbf x_-)}\right)\\
\end{eqnarray*}
where the third equality is the observation that $A^{\mathbf i}_{\xi_\rho(\mathbf x)}+A^{\mathbf i_+}_{\xi_\rho(\mathbf x_+)}+A^{\mathbf i_-}_{\xi_\rho(\mathbf x_-)}=0$. This is impossible because
\[\left\{A^{\mathbf i_+}_{\xi_\rho(\mathbf x_+)}:i_1+i_2+i_3=n\right\}\cup\left\{A^{\mathbf i_-}_{\xi_\rho(\mathbf x_-)}:i_1+i_2+i_3=n\right\}\]
is a linearly independent collection.

(2) is an immediate consequence of the $\Ad\circ\rho$-equivariance of $\widetilde{\nu}$, and the observation that 
\[\Ad\circ\rho(\gamma)\cdot A^{\mathbf{i}}_{\xi_\rho(\mathbf{x})}=A^{\mathbf{i}}_{\xi_\rho(\gamma\cdot \mathbf{x})}\]
for all $\gamma\in\Gamma$.
\end{proof}

By Proposition \ref{prop:eruption justification}(2), for all triples of positive integers ${\bf i}$ that sum to $n$ and all non-edge barriers $\widehat{\bf b}$ in $\Dmc$, we may define 
\begin{equation}\label{eqn: coefficient 1}
\alpha^{\bf i}_{\widehat{\bf b}}(\nu):=\alpha^{\bf i}_{\bf b}(\nu)
\end{equation}
where ${\bf b}\in \widetilde{\Dmc}$ is some (equiv. any) lift of $\widehat{\bf b}$ to $\widetilde{S}$.

%%%%%%%%%%%%%%%%%%%%%%%%%%%%%%%%%%%%%%%%%%%%%%%%%%
\subsection{Shearing endomorphisms}\label{sec:shearing}
%%%%%%%%%%%%%%%%%%%%%%%%%%%%%%%%%%%%%%%%%%%%%%%%%%
Given any oriented edge ${\bf e}$ in $\widetilde{\Tmc}^o$, let $r_1$ and $r_2$ be the backward and forward endpoints of ${\bf e}$. We say that $\mathbf r:=(r_1,r_2)$ is the \emph{pair associated to ${\bf e}$}. If ${\bf e}$ is isolated, let $\mathfrak{T}_1$ (resp. $\mathfrak{T}_2$) be the fragmented ideal triangle of $\widetilde{\Dmc}$ that has ${\bf e}$ as an edge and lies to the left (resp. right) of ${\bf e}$. We say that a $1$-simplex $h:[0,1]\to\widetilde{S}$ \emph{passes from the left to the right of ${\bf e}$} if $h(0)$ and $h(1)$ lie in the interiors of $\mathfrak{T}_1$ and $\mathfrak{T}_2$ respectively, see Figure \ref{fig:shearingbases}. On the other hand, if ${\bf e}$ is non-isolated and $J$ is a bridge across ${\bf e}$, we say that a $1$-simplex $h:[0,1]\to\widetilde{S}$ \emph{passes from the left to the right of ${\bf e}$ along $J$} if $h(0)$ and $h(1)$ are the endpoints of $J$ that lie to the left and right of ${\bf e}$ respectively. 

\begin{figure}[ht]
\centering
\includegraphics[scale=0.6]{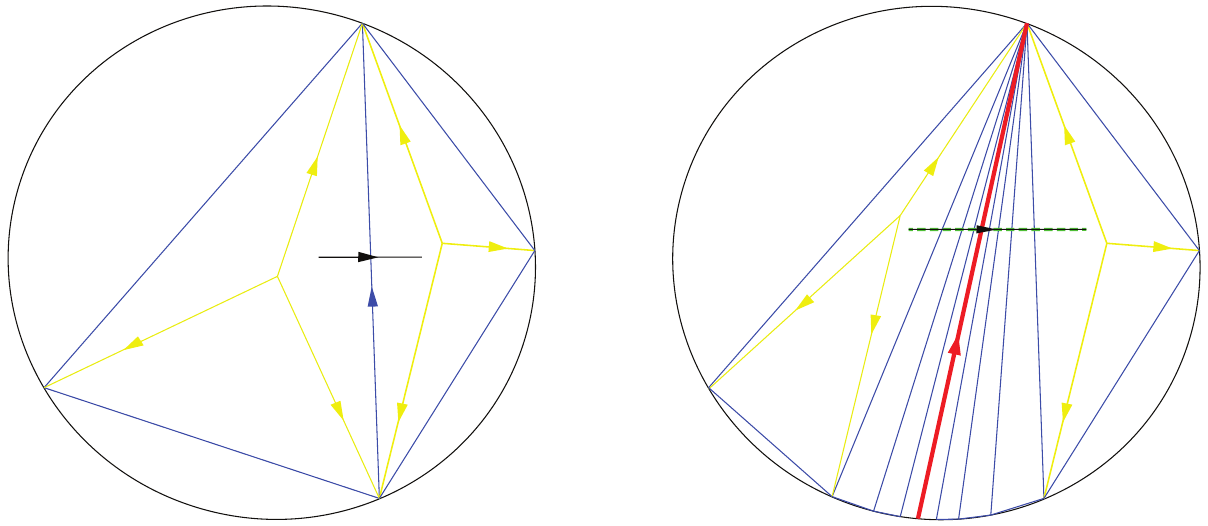}
\put (-247, 65){\tiny ${\bf e}$}
\put (-73, 45){\tiny ${\bf e}$}
\put (-255, 79){\tiny $h$}
\put (-45, 87){\tiny $h$}
\put (-45, 79){\tiny $J$}
\put (-241, 3){\tiny $r_1$}
\put (-245, 147){\tiny $r_2$}
\put (-255, 95){\tiny $\mathfrak T_1$}
\put (-240, 95){\tiny $\mathfrak T_2$}
\caption{$h$ passes from the left to the right of ${\bf e}$ when ${\bf e}$ is isolated (on the left) and when ${\bf e}$ is not isolated (on the right).}\label{fig:shearingbases}
\end{figure}

Again, let $\nu$ be a $(\rho,\Tmc,\Jmc)$-tangent cocycle, and let $\widetilde{\nu}:C_1(\widetilde S,\Zbbb)\to\mathfrak{sl}(V))$ be its $\Ad\circ\rho$-equivariant lift.  We will now describe the possible values of $\widetilde{\nu}(h)$ when $h$ passes from the left to the right of an isolated edge ${\bf e}$ in $\widetilde{\Tmc}^o$, and the possible values of the ``diagonal part" of $\widetilde{\nu}(h)$ with respect to $\xi_\rho({\bf r})$ when $h$ passes from the left to the right of an isolated edge ${\bf e}$ in $\widetilde{\Tmc}^o$ along  a bridge $J$ across ${\bf e}$. For that purpose, we define the shearing endomorphisms.

\begin{definition}\label{def:shearing endomorphism}
Let $\mathbf k:=(k_1,k_2)$ be a pair of positive integers that sum to $n$ and let $\mathbf E:=(E_1,E_2)$ be a generic pair of flags in $\Fmc(V)$. 
The $\mathbf k$-\emph{shearing endomorphism} with respect to $\mathbf E$ is the endomorphism $D^{\mathbf k}_{\mathbf E}\in\smf\lmf(V)$ with eigenspaces $E_1^{(k_1)}$ and $E_2^{(k_2)}$ corresponding to eigenvalues $\frac{k_2}{2n}$ and $-\frac{k_1}{2n}$ respectively. 
%that has $B^{i,n-i}_{F,G}$ as an eigenbasis, where $\frac{n-i}{n}$ is the eigenvalue of $f_1,\dots,f_i$ and $-\frac{i}{n}$ is the eigenvalue for $g_1,\dots,g_j$. 
\end{definition}

%\begin{remark}
%Again, the shearing endomorphisms are motivated by results in the companion paper \cite{SunWienhardZhang}: Let $\mathbf r=(r_1,r_2)$ be a distinct pair of points in $S^1$, and let ${\bf k}=(k_1,k_2)$ be a pair of positive integers that sum to $n$. For any $m=1,2$, let $\mathbf r_m:=(r_m,r_{m+1})$ and $\mathbf k_m:=(k_m,k_{m+1})$, where arithmetic in the subscripts are done modulo $2$. Then $\xi(\mathbf r_m)$ is a generic pair of flags for any Frenet map $\xi:S^1\to\Fmc(V)$, so we can define
%\[d^{\mathbf k_m}_{\xi(\mathbf r_m)}(t):=\exp\left(tD^{\mathbf k_m}_{\xi(\mathbf r_m)}\right)\in\PGL(V)\] 
%for all $t\in\Rbbb$. By \cite[Theorem 3.12]{SunWienhardZhang}, the map $\xi_t:S^1\to\Fmc(V)$ defined by
%\[\xi_t(p)=\left\{
%\begin{array}{ll}
%d^{\mathbf k_{2}}_{\xi(\mathbf r_{2})}(t)\cdot \xi(p)&\text{if }r_2\leq p\leq r_1;\\
%d^{\mathbf k_1}_{\xi(\mathbf r_1)}(t)\cdot \xi(p)&\text{if }r_1\leq p\leq r_{2},
%\end{array}\right. \]
%is well-defined and Frenet. Also, by \cite[Proposition 3.13]{SunWienhardZhang}, if $\xi$ and $\xi'$ are projectively equivalent, then so are $\xi_t$ and $\xi_t'$. In other words, the choice of ${\bf r}$ and ${\bf i}$ defines a flow on the space of projective classes of Frenet maps, called the \emph{$\mathbf k$-elementary shearing flow with respect to $\mathbf r$}, and $D^{\mathbf k_m}_{\xi(\mathbf r_m)}$ can be thought of as the ``derivative information" of this flow at the projective class $[\xi]$.
%\end{remark}

For any pair ${\bf k}=(k_1,k_2)$ of positive integers that sum to $n$ and any generic pair of flags $(E_1,E_2)$ in $\Fmc(V)$, observe that 
\[D^{\mathbf k}_{\mathbf E}=-D^{\bar{\mathbf k}}_{\bar{\mathbf E}},\]
where $\bar{\mathbf k}:=(k_2,k_1)$ and $\bar{\mathbf E}:=(E_2,E_1)$. Also, if 
\[\{f_1,\dots,f_n\}\] 
is a basis for $V$ such that $f_j\in E_1^{(j)}\cap E_2^{(n+1-j)}$ for all $j=1,\dots, n$, then in this basis, the endomorphism $D^{\mathbf k}_{\mathbf E}$ is written as the matrix
\[D^{\mathbf k}_{\mathbf E}=\left[\begin{array}{cc}\frac{k_2}{2n}\cdot\id_{k_1}&0\\0&-\frac{k_1}{2n}\cdot\id_{k_2}\end{array}\right].\]
Furthermore, it is clear that the following proposition holds.

\begin{prop}\label{prop:linearly independent pair}
Let $\mathbf E:=(E_1,E_2)$ be a generic pair of flags in $\Fmc(V)$. Then
\begin{enumerate} 
\item For all pairs $\mathbf k:=(k_1,k_2)$ of positive integers that sum to $n$, $D^{\mathbf k}_{\mathbf E}$ fixes both $E_1$ and $E_2$.
\item Let $G_\mathbf E\subset\PGL(V)$ be the subgroup defined by
\[G_\mathbf E:=\{g\in\PGL(V):g\cdot E_1=E_1\text{ and }g\cdot E_2=E_2\}.\]
Then $\left\{D^{\mathbf k}_{\mathbf E}:k_1+k_2=n\right\}$ is a basis of the Lie algebra $\mathfrak{g}_E$ of $G_E$.
\end{enumerate}
\end{prop}

The next pair of propositions relate shearing endomorphisms to $\widetilde{\nu}$. 

\begin{prop}\label{prop:shearing justification}
Let ${\bf e}$ be an oriented, isolated edge in $\widetilde{\Tmc}^o$, let $\mathbf r:=(r_1,r_2)$ be the pair associated to ${\bf e}$, and let $h$ be a $1$-simplex that passes from the left to the right of ${\bf e}$. Then for each pair ${\bf k}=(k_1,k_2)$ of positive integers that sum to $n$, there is a unique $\alpha^{{\bf k}}_{{\bf e}}=\alpha^{{\bf k}}_{{\bf e}}(\nu)\in\Rbbb$ such that
\[\widetilde{\nu}(h)=-\sum_{k_1+k_2=n}2\alpha^{\mathbf{k}}_{\mathbf{e}}\cdot D^{\mathbf{k}}_{\xi_\rho(\mathbf{r})}.\]
Furthermore, 
\begin{enumerate}
\item if $\bar{\bf k}:=(k_2,k_1)$ and $\bar{\bf e}$ is the edge in $\widetilde{\Tmc}^o$ whose backward and forward endpoints are $r_2$ and $r_1$ respectively, then
\[\alpha^{\bf k}_{\bf e}=\alpha^{\bar{\bf k}}_{\bar{\bf e}}.\]
\item if $\gamma\in\Gamma$, then $\alpha^{\bf k}_{\bf e}=\alpha^{\bf k}_{\gamma\cdot \bf e}$.
\end{enumerate}
\end{prop}

\begin{proof}
For both $m=1,2$, let $q_m$ be the vertex of the ideal triangle containing $\mathfrak{T}_m$ that is neither $r_1$ nor $r_2$. Let ${\bf x}:=(r_2,r_1,q_1)$, let ${\bf y}:=(r_1,r_2,q_2)$, let $\Hmc_{\rho,{\bf x}}$ be the real-analytic submanifold of $\widetilde{\Hit}_V(S)$ defined by \eqref{eqn: submanifold}, and let $g_{{\bf x},{\bf y}}:\Hmc_{\rho,{\bf x}}\to\PGL(V)$ be the real-analytic map defined by \eqref{eqn: g def}.
%\begin{align*}
%g_{{\bf x},{\bf x}_3}:\Hmc_{\rho,{\bf x}}\to\PGL(V)
%\end{align*}
Observe from the definitions that $g_{{\bf x},{\bf y}}(\sigma)\in G_{\bf E}$ for all $\sigma\in\Hmc_{\rho,{\bf x}}$ and $g_{{\bf x},{\bf y}}(\rho)=\id$. By \eqref{eqn: tangent equation 1'}, $\widetilde{\nu}(h)=\frac{d}{dt}|_{t=0}g_{{\bf x},{\bf y}}(\rho_t)$ for some smooth path $t\mapsto \rho_t$ in $\Hmc_{\rho,{\bf x}}$ with $\rho_0=\rho$, so the first claim follows from Proposition \ref{prop:linearly independent pair}(2).

(1) follows from the observations that $\widetilde{\nu}(h)=-\widetilde{\nu}(\bar{h})$ and $D^{\mathbf k}_{\xi_\rho(\mathbf r)}=-D^{\bar{\mathbf k}}_{\xi_\rho(\bar{\mathbf r})}$ for all ${\bf k}$, where $\bar{\bf r}:=(r_2,r_1)$, while (2) is an immediate consequence of the $\Ad\circ\rho$-equivariance of $\widetilde{\nu}$, and the observation that 
\[\Ad\circ\rho(\gamma)\cdot D^{\mathbf{k}}_{\xi_\rho(\mathbf{r})}=D^{\mathbf{k}}_{\xi_\rho(\gamma\cdot \mathbf{r})}\]
for all $\gamma\in\Gamma$.
\end{proof}

To state the next proposition, we need to define formally specify the ``diagonal part" of endomorphisms, we observe the following. Let ${\bf r}:=(r_1,r_2)$ be the pair associated to an oriented, non-isolated edge in $\widetilde{\Tmc}^o$. If $M\in\mathfrak{sl}(V)$ is an endomorphism for which there is some $m=1,2$ such that
\[M\cdot \xi_\rho^{(i)}(r_m)\subset \xi_\rho^{(i)}(r_m)\] 
for all $i\in\{1,\dots,n-1\}$, then we may decompose 
\begin{equation}\label{eqn: decompose}
M=M_{{\bf r},{\rm diag}}+M_{{\bf r},{\rm nil}},
\end{equation}
where $M_{{\bf r},{\rm diag}}\cdot \xi_\rho^{(i)}(r_1)\subset  \xi_\rho^{(i)}(r_1)$ and $M_{{\bf r},{\rm diag}}\cdot  \xi_\rho^{(i)}(r_2)\subset  \xi_\rho^{(i)}(r_2)$ for all $i\in\{1,\dots,n-1\}$, while $M_{{\bf r},{\rm nil}}$ is nilpotent and satisfies $M_{{\bf r},{\rm nil}}\cdot  \xi_\rho^{(i)}(r_m)\subset  \xi_\rho^{(i)}(r_m)$ for all $i\in\{1,\dots,n-1\}$. We refer to $M_{{\bf r},{\rm diag}}$ (resp. $M_{{\bf r},{\rm nil}}$) as the \emph{diagonal part (resp. nilpotent part) of $M$ with respect to ${\bf r}$}.

\begin{prop}\label{prop:shearing justification non-isolated}
Let ${\bf e}$ be an oriented, non-isolated edge in $\widetilde{\Tmc}^o$, let $\mathbf r:=(r_1,r_2)$ be the pair associated to ${\bf e}$, let $J$ be a bridge across ${\bf e}$, and let $h:[0,1]\to\widetilde{S}$ be a $1$-simplex such that $h(0)$ is the endpoint of $J$ that lies to the left of ${\bf e}$ and $h(1)$ is the intersection of ${\bf e}$ and $J$. If $\widetilde{\nu}(h)_{{\bf r},{\rm diag}}$ denotes the diagonal part of $\widetilde{\nu}(h)$ with respect to $\xi_\rho({\bf r})$, see \eqref{eqn: decompose}, then for each pair ${\bf k}=(k_1,k_2)$ of positive integers that sum to $n$, there is a unique $\alpha^{{\bf k}}_{{\bf e}}=\alpha^{{\bf k}}_{{\bf e}}(\nu)\in\Rbbb$ such that
\[\widetilde{\nu}(h)_{{\bf r},{\rm diag}}=-\sum_{k_1+k_2=n}\alpha^{\mathbf{k}}_{\mathbf{e}}\cdot D^{\mathbf{k}}_{\xi_\rho(\mathbf{r})}.\]
Furthermore, 
\begin{enumerate}
\item if $\bar{\bf k}:=(k_2,k_1)$ and $\bar{\bf e}$ is the edge in $\widetilde{\Tmc}^o$ whose backward and forward endpoints are $r_2$ and $r_1$ respectively, then
\[\alpha^{\bf k}_{\bf e}=\alpha^{\bar{\bf k}}_{\bar{\bf e}}.\]
\item if $\gamma\in\Gamma$, then $\alpha^{\bf k}_{\bf e}=\alpha^{\bf k}_{\gamma\cdot \bf e}$.
\end{enumerate}
\end{prop}

\begin{proof}
The first claim is a consequence of the definition of $\widetilde{\nu}(h)_{{\bf r},{\rm diag}}$ and Proposition \ref{prop:linearly independent pair}(2). To prove (1), let $h':[0,1]\to\widetilde{S}$ be a $1$-simplex such that $h'(0)$ is the endpoint of $J$ that lies to the left of $\bar{\bf e}$ and $h'(1)=h(1)$. By Lemma \ref{lem: d properties}(2) and the definition of $\widetilde{\nu}(h)$ and $\widetilde{\nu}(h')$, see \eqref{eqn: tangent equation 3}, $\displaystyle\widetilde{\nu}(h)_{{\bf r},{\rm diag}}=-\widetilde{\nu}(h')_{{\bf r},{\rm diag}}$. Thus
\begin{eqnarray*}
-\sum_{k_1+k_2=n}\alpha^{\mathbf{k}}_{\mathbf{e}}\cdot D^{\mathbf{k}}_{\xi_\rho(\mathbf{r})}&=&\widetilde{\nu}(h)_{{\bf r},{\rm diag}}=-\widetilde{\nu}(h')_{{\bf r},{\rm diag}}=-\widetilde{\nu}(h')_{\bar{\bf r},{\rm diag}}\\
&=&\sum_{k_1+k_2=n}\alpha^{\bar{\bf k}}_{\bar{\bf e}}\cdot D^{\bar{\bf k}}_{\xi_\rho(\bar{\bf r})}=-\sum_{k_1+k_2=n}\alpha^{\bar{\bf k}}_{\bar{\bf e}}\cdot D^{{\bf k}}_{\xi_\rho({\bf r})}
\end{eqnarray*}
(1) now follows from Proposition \ref{prop:linearly independent pair}(2). (2) is an immediate consequence of the $\Ad\circ\rho$-equivariance of $\widetilde{\nu}$, and the observation that 
\[\Ad\circ\rho(\gamma)\cdot D^{\mathbf{k}}_{\xi_\rho(\mathbf{r})}=D^{\mathbf{k}}_{\xi_\rho(\gamma\cdot \mathbf{r})}\]
for all $\gamma\in\Gamma$.
\end{proof}

By Proposition \ref{prop:shearing justification}(2) and Proposition \ref{prop:shearing justification non-isolated}(2), for all pairs of positive integers ${\bf k}$ that sum to $n$ and all oriented edges $\widehat{\bf e}$ in $\Tmc^o$, we may define 
\begin{equation}\label{eqn: coefficient 2}
\alpha^{\bf k}_{\widehat{\bf e}}(\nu):=\alpha^{\bf k}_{\bf e}(\nu).
\end{equation}
where ${\bf e}\in \widetilde{\Tmc}^o$ is some (equiv. any) lift $\widehat{\bf e}$ of to $\widetilde{S}$.

%%%%%%%%%%%%%%%%%%%%%%%%%%%%%%%%%%%%%%%%%%%%%%%%%%%%%%%%%
\subsection{The linear map $\Xi$}\label{sec: Xi}
%%%%%%%%%%%%%%%%%%%%%%%%%%%%%%%%%%%%%%%%%%%%%%%%%%%%%%%%%
We may now define the required linear map 
\begin{eqnarray}\label{eqn: Xi}
\Xi=\Xi_{\rho,\Tmc,\Jmc}:\mathscr{T}&\to&\Rbbb^{\Tmc^o\times\Amc}\times\Rbbb^{\Mmc\times\Bmc},\\
\nu&\mapsto&\left(\left(\alpha^{\bf k}_{\widehat{\bf e}}(\nu)\right)_{(\widehat{\bf e},{\bf k})\in \Tmc^o\times\Amc},\left(\alpha^{\bf i}_{\widehat{\bf b}}(\nu)\right)_{(\widehat{\bf b},{\bf i})\in \Mmc\times\Bmc}\right),\nonumber
\end{eqnarray}
where $\alpha^{\bf i}_{\widehat{\bf b}}(\nu)$ and $\alpha^{\bf k}_{\widehat{\bf e}}(\nu)$ are the real numbers defined by \eqref{eqn: coefficient 1} and \eqref{eqn: coefficient 2} respectively. To describe the image of $\Xi$, we use the following notation. 

\begin{notation}\label{not:admissible}
\begin{itemize}
\item For every oriented, non-isolated edge ${\bf e}$ in $\widetilde{\Tmc}^o$ and every bridge $J$ across ${\bf e}$, let ${\bf r}:=(r_1,r_2)$ be the pair associated to ${\bf e}$, i.e. $r_1$ and $r_2$ are the backward and forward endpoints of ${\bf e}$ respectively, and let $T$ and $T'$ be the ideal triangles of $\widetilde{\Tmc}$ that contain the endpoints of $J$ and respectively lie to the left and right of ${\bf e}$. Let $x$ (resp. $x'$) be the vertex of $T$ (resp. $T'$) that is an endpoint of ${\bf e}$, and let $\gamma=\gamma({\bf e})$ (resp. $\gamma'=\gamma'({\bf e})$) be the primitive element in $\Gamma$ that fixes both $r_1$ and $r_2$, and has $x$ (resp. $x'$) as its repelling fixed point. 

\item Let $\{T_j\}_{j\in\Zbbb}$ (resp. $\{T_j'\}_{j\in\Zbbb}$) be the bi-infinite sequence of pairwise distinct ideal triangles of $\widetilde{\Tmc}$ that lie to the left (resp. right) of ${\bf e}$ such that the following hold:
\begin{itemize}
\item $T=T_0$ (resp. $T=T_0'$),
\item $T_j$ (resp. $T_j'$) has $x$ (resp. $x'$) as a vertex for all $j$,
\item $T_j$ and $T_{j+1}$ (resp. $T_j'$ and $T_{j+1}'$)share an edge for all $j$,
\item there is some positive integer $H=H({\bf e})$ (resp. $H'=H'({\bf e})$) such that $T_{H+j}=\gamma\cdot T_j$ (resp. $T_{H'+j}'=\gamma'\cdot T_j'$) for all integers $j$. 
\end{itemize}
\item For each integer $j$, let ${\bf e}_j={\bf e}_j({\bf e},J)$ (resp. ${\bf e}_j'={\bf e}_j'({\bf e},J)$) denote the oriented common edge of $T_j$ and $T_{j+1}$ (resp. $T_j'$ and $T_{j+1}'$) whose forward endpoint is $x$ (resp. $x'$).
\item For each integer $j$, let ${\bf b}_j={\bf b}_j({\bf e},J)$ (resp. ${\bf b}_j'={\bf b}_j'({\bf e},J)$) denote the non-edge barrier in $T_j$ (resp. $T_j'$) whose forward endpoint is $x$ (resp. $x'$).
\item For each integer $j$, let $\mathfrak{T}_{j,-}$ and $\mathfrak T_{j,+}$ (resp. $\mathfrak{T}_{j,-}'$ and $\mathfrak T_{j,+}'$) denote the fragmented ideal triangles of $\widetilde{\Dmc}$ in $T_j$ (resp. $T_j'$) that respectively have ${\bf e}_j$ and ${\bf e}_{j+1}$ (resp. ${\bf e}_j'$ and ${\bf e}_{j+1}'$) as edges. 
%\item For every oriented edge $\widehat{\bf e}$ in $\Tmc^o$, let $\widehat{\bar{\bf e}}$ be the other oriented edge in $\Tmc^o$ so that $\widehat{\bar{\bf e}}$ and $\widehat{\bf e}$ project to the same edge in $\Tmc$.
\item For every non-isolated edge $\widehat{\bf e}$ of $\Tmc^o$, let ${\bf e}$ be a oriented, non-isolated edge of $\widetilde{\Tmc}$ that projects to $\widehat{\bf e}$ and let $J$ be a bridge across ${\bf e}$.  Then let $H(\widehat{\bf e}):=H({\bf e})$, let $H'(\widehat{\bf e}):=H'({\bf e})$, and for each integer $j\in\{1,\dots,H({\bf e})\}$ (resp. $j\in\{1,\dots,H'({\bf e})\}$), let $\widehat{\bf e}_j(\widehat{\bf e})$ and $\widehat{\bf b}_j(\widehat{\bf e})$ (resp. $\widehat{\bf e}'_j(\widehat{\bf e})$  and $\widehat{\bf b}'_j(\widehat{\bf e})$) be the projections in $\Dmc$ of ${\bf e}_j({\bf e},J)$ and ${\bf b}_j({\bf e},J)$ (resp. ${\bf e}_j'({\bf e},J)$ and ${\bf b}_j'({\bf e},J)$) respectively. Observe that $H(\widehat{\bf e})$, $H'(\widehat{\bf e})$, $\widehat{\bf e}_j(\widehat{\bf e})$, $\widehat{\bf b}_j(\widehat{\bf e})$, $\widehat{\bf e}'_j(\widehat{\bf e})$ and $\widehat{\bf b}'_j(\widehat{\bf e})$) do not depend on the choice of ${\bf e}$ and $J$.
\item For each integer $j$, let $l_j:[0,1]\to\widetilde{S}$ (resp. $l_j':[0,1]\to\widetilde{S}$) be a $1$-simplex such that $l_j(0)$ and $l_j(1)$ (resp. $l_j'(0)$ and $l_j'(1)$) lie in the interiors of $\mathfrak T_{j,-}$ and $\mathfrak T_{j,+}$ (resp. $\mathfrak T_{j,-}'$ and $\mathfrak T_{j,+}'$) respectively. Then let $m_j:[0,1]\to\widetilde S$ (resp. $m_j':[0,1]\to\widetilde S$) be a $1$-simplex such that $m_j(0)=l_j(1)$ and $m_j(1)=l_{j+1}(0)$ (resp. $m_j'(0)=l_j'(1)$ and $m_j'(1)=l_{j+1}'(0)$). 
%Similarly, let $l_j':[0,1]\to\widetilde{S}$ be a $1$-simplex such that $l_j'(0)$ and $l_j'(1)$ lie in the interiors of $\mathfrak T_{j,-}'$ and $\mathfrak T_{j,+}'$ respectively, and let $m_j':[0,1]\to\widetildeS$ be a $1$-simplex such that $m_j'(0)=l_j'(1)$ and $m_j'(1)=l_{j+1}'(0)$. 
\item Let $s,s',m,m':[0,1]\to\widetilde S$ be $1$-simplices such that $s(0)$ is the endpoint of $J$ that lies in $T$, $s'(0)$ is the endpoint of $J$ that lies in $T'$, $s(1)=s'(1)=e\cap J=m(0)=m'(0)$, $m(1)=\gamma\cdot m(0)$, and $m'(1)=\gamma'\cdot m'(0)$.
\end{itemize}
\end{notation}

\begin{figure}[ht]
\centering
\includegraphics[scale=0.8]{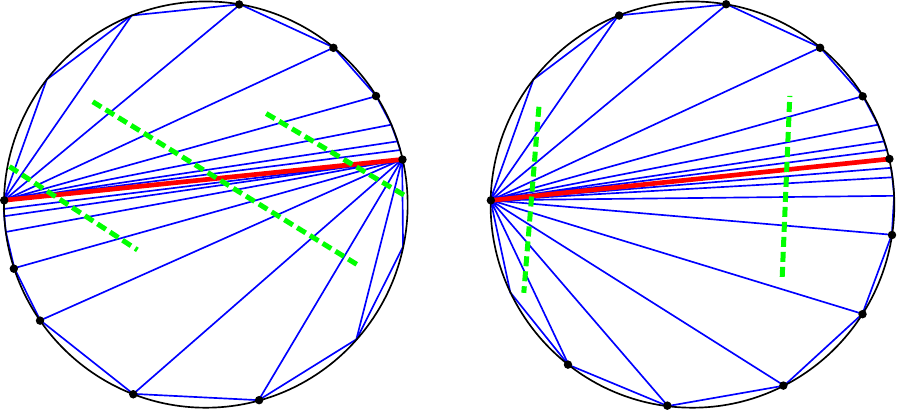}
\tiny
\put (-300, 132){$T_1'$}
\put (-270, 127){$T_2'$}
\put (-250, 118){$\gamma\cdot T_1'$}
\put (-323, 50){$\gamma\cdot T_1$}
\put (-290, 32){$T_3$}
\put (-255, 22){$T_2$}
\put (-223, 29){$T_1$}
\put (-140, 120){$T_1'$}
\put (-112, 135){$T_2'$}
\put (-82, 130){$T_3'$}
\put (-62, 120){$\gamma\cdot T_1'$}
\put (-148, 45){$T_1$}
\put (-131, 32){$T_2$}
\put (-105, 30){$T_3$}
\put (-72, 35){$T_4$}
\put (-60, 60){$\gamma\cdot T_1$}
\put (-354, 80){$r_2$}
\put (-188, 95){$r_1$}
\put (-166, 80){$r_2$}
\put (0, 96){$r_1$}
\caption{In both diagrams above, the non-isolated edge $e$ with endpoints $r_1$ and $r_2$ is drawn in thick red, and the $\langle\gamma_1\rangle=\langle\gamma_2\rangle$ orbit of the bridge $J$ is drawn in dotted green. On the left, $H=3$ and $H'=2$. On the right, $H=4$ and $H'=3$.}\label{fig:closededge}
\end{figure}

Since the images of $m$ and $m'$ lie in ${\bf e}$, $\widetilde{\nu}(m)\cdot\xi_\rho^{(i)}(r_k)\subset\xi_\rho^{(i)}(r_k)$ and $\widetilde{\nu}(m')\cdot\xi_\rho^{(i)}(r_k)\subset\xi_\rho^{(i)}(r_k)$ for both $k=1,2$ and all $i\in\{1,\dots,n-1\}$, see \eqref{eqn: tangent equation 2}.
Also, by definition, if $h=s$, $l_j$ or $m_j$ for some $j$, then 
\[\widetilde{\nu}(h)\cdot\xi^{(i)}(x)\subset\xi^{(i)}(x)\] 
for all $i\in\{1,\dots,n-1\}$. In particular, 
we may decompose 
\begin{align*}
\widetilde{\nu}(s)&=\widetilde{\nu}(s)_{{\bf r},{\rm diag}}+\widetilde{\nu}(s)_{{\bf r},{\rm nil}},\\
\widetilde{\nu}(l_j)&=\widetilde{\nu}(l_j)_{{\bf r},{\rm diag}}+\widetilde{\nu}(l_j)_{{\bf r},{\rm nil}},\\
\widetilde{\nu}(m_j)&=\widetilde{\nu}(m_j)_{{\bf r},{\rm diag}}+\widetilde{\nu}(m_j)_{{\bf r},{\rm nil}},
\end{align*} 
see \eqref{eqn: decompose}. Similarly, we may decompose 
\begin{align*}
\widetilde{\nu}(s')&=\widetilde{\nu}(s')_{{\bf r},{\rm diag}}+\widetilde{\nu}(s')_{{\bf r},{\rm nil}},\\
\widetilde{\nu}(l_j')&=\widetilde{\nu}(l_j')_{{\bf r},{\rm diag}}+\widetilde{\nu}(l_j')_{{\bf r},{\rm nil}},\\
\widetilde{\nu}(m_j')&=\widetilde{\nu}(m_j')_{{\bf r},{\rm diag}}+\widetilde{\nu}(m_j')_{{\bf r},{\rm nil}}.
\end{align*} 

\begin{prop}\label{prop: infinite box}\
\begin{enumerate}
\item $\displaystyle\widetilde{\nu}(m)=\sum_{j=1}^H\widetilde{\nu}(l_j)_{{\bf r},{\rm diag}}+\widetilde{\nu}(m_j)_{{\bf r},{\rm diag}}$. 
\item $\displaystyle\widetilde{\nu}(s)_{{\bf r},{\rm nil}}=\sum_{j=1}^\infty \widetilde{\nu}(l_j)_{{\bf r},{\rm nil}}+\widetilde{\nu}(m_j)_{{\bf r},{\rm nil}}$.
\end{enumerate}
%In particular, if we are given the pair of flags $\xi_\rho({\bf r})$, then the endomorphisms 
%\[\widetilde{\nu}(l_1),\dots,\widetilde{\nu}(l_H),\widetilde{\nu}(m_1),\dots,\widetilde{\nu}(m_H),\widetilde{\nu}(l_1'),\dots,\widetilde{\nu}(l_{H'}'),\widetilde{\nu}(m_1'),\dots,\widetilde{\nu}(m_{H'}'),\widetilde{\nu}(s)_{{\bf r},{\rm diag}}\]
%determine $\widetilde{\nu}(s)$, $\widetilde{\nu}(s')$, and $\widetilde{\nu}(m)$.
\end{prop}

\begin{proof}
Let $l:[0,1]\to\widetilde S$ be the concatenation $m_1\cdot l_1\cdot m_2\cdot l_2\cdot \ldots\cdot m_H\cdot l_H$. We may assume without loss of generality that $l(0)=s(0)$ and $l(1)=\gamma\cdot l(0)$. Then by the finite additivity of $\widetilde{\nu}$,
\begin{equation}\label{eqn: diagonal nilpotent}
\widetilde{\nu}(l)=\sum_{j=1}^H\left(\widetilde{\nu}(l_j)_{{\bf r},{\rm diag}}+\widetilde{\nu}(m_j)_{{\bf r},{\rm diag}}\right)+\sum_{j=1}^H\left(\widetilde{\nu}(l_j)_{{\bf r},{\rm nil}}+\widetilde{\nu}(m_j)_{{\bf r},{\rm nil}}\right).
\end{equation}
At the same time, by Lemma \ref{lem:tangent cocycle 1}(1),
\begin{equation}\label{eqn: diagonal nilpotent'}
\widetilde{\nu}(l)=\widetilde{\nu}(m)+\widetilde{\nu}(s)-\widetilde{\nu}(\gamma\cdot s)=\widetilde{\nu}(m)+\widetilde{\nu}(s)_{{\bf r},{\rm nil}}-\widetilde{\nu}(\gamma\cdot s)_{{\bf r},{\rm nil}},
\end{equation}
where the second equality holds because the $\Ad\circ\rho$-equivariance of $\widetilde{\nu}$ and the fact that $\gamma$ fixes $r_1$ and $r_2$ imply
\[\widetilde{\nu}(\gamma\cdot s)_{{\bf r},{\rm diag}}=\Ad\circ\rho(\gamma)\cdot \widetilde{\nu}(s)_{{\bf r},{\rm diag}}=\widetilde{\nu}(s)_{{\bf r},{\rm diag}}.\] 
By comparing the diagonal parts of \eqref{eqn: diagonal nilpotent} and \eqref{eqn: diagonal nilpotent'} we see that (1) holds.

Next, we prove (2). By Corollary \ref{cor:Labourie}, $\rho(\gamma)$ has $\xi_\rho(x)$ as a repelling fixed flag, and
\[0<K:=\max\left\{\frac{\lambda_a(\rho(\gamma))}{\lambda_b(\rho(\gamma))}:a>b\right\}<1,\] 
where $\lambda_1(\rho(\gamma))>\dots>\lambda_n(\rho(\gamma))$ are the eigenvalues of some (any) linear representative of $\rho(\gamma)$. Observe that for all positive integers $t$, 
\[\Ad\circ\rho(\gamma^t)(\widetilde\nu(l_j)_{{\bf r},{\rm nil}})=\widetilde\nu(\gamma^t\cdot l_j)_{{\bf r},{\rm nil}}=\widetilde\nu(l_{j+tH})_{{\bf r},{\rm nil}}.\]
Similarly, 
\[\Ad\circ\rho(\gamma^t)(\widetilde\nu(m_j)_{{\bf r},{\rm nil}})=\widetilde\nu(m_{j+tH})_{{\bf r},{\rm nil}}\,\,\text{ and }\,\,\Ad\circ\rho(\gamma^t)(\widetilde\nu(s)_{{\bf r},{\rm nil}})=\widetilde\nu(\gamma^t\cdot s)_{{\bf r},{\rm nil}}.\] 
Since the nilpotent endomorphisms $\widetilde\nu(l_j)_{{\bf r},{\rm nil}}$, $\widetilde\nu(m_j)_{{\bf r},{\rm nil}}$, and $\widetilde{\nu}(s)_{{\bf r},{\rm nil}}$ are all represented by upper triangular matrices if $x=r_1$, and lower triangular matrices if $x=r_2$, so it follows that
\[\|\widetilde\nu(l_{j+tH})_{{\bf r},{\rm nil}}\|\leq K^t\|\widetilde\nu(l_j)_{{\bf r},{\rm nil}}\|,\,\,\|\widetilde\nu(m_{j+tH})_{{\bf r},{\rm nil}}\|\leq K^t\|\widetilde\nu(m_{j})_{{\bf r},{\rm nil}}\|,\]
and
\[\|\widetilde\nu(\gamma^t\cdot s)_{{\bf r},{\rm nil}}-\widetilde\nu(\gamma^{t+1}\cdot s)_{{\bf r},{\rm nil}}\|\leq K^t\|\widetilde\nu(s)_{{\bf r},{\rm nil}}-\widetilde\nu(\gamma\cdot s)_{{\bf r},{\rm nil}}\|\]
for all positive integers $t$, where $\|\cdot\|$ is the Euclidean norm on the space of $n\times n$ matrices. %Similarly, there is some $C'>0$ such that for all $j$,
%\[\|N_{j+H}'\|\leq C'\|N_j'\|\,\,\text{ and }\,\,\|M_{j+H}'\|\leq C'\|M_j'\|.\]
In particular, $\sum_{j=1}^\infty \widetilde{\nu}(l_j)_{{\bf r},{\rm nil}}+\widetilde{\nu}(m_j)_{{\bf r},{\rm nil}}$ and $\sum_{t=0}^\infty\widetilde{\nu}(\gamma^t\cdot s)_{{\bf r},{\rm nil}}-\widetilde{\nu}(\gamma^{t+1}\cdot s)_{{\bf r},{\rm nil}}$ both converge absolutely. Thus, 
\begin{eqnarray*}
\sum_{j=1}^\infty \widetilde{\nu}(l_j)_{{\bf r},{\rm nil}}+\widetilde{\nu}(m_j)_{{\bf r},{\rm nil}}&=&\sum_{t=0}^\infty\Ad\circ\rho(\gamma^t)\cdot\left(\sum_{j=1}^H\widetilde{\nu}(l_j)_{{\bf r},{\rm nil}}+\widetilde{\nu}(m_j)_{{\bf r},{\rm nil}}\right)\\
&=&\sum_{t=0}^\infty\Ad\circ\rho(\gamma^t)\cdot\left(\widetilde{\nu}(s)_{{\bf r},{\rm nil}}-\widetilde{\nu}(\gamma\cdot s)_{{\bf r},{\rm nil}}\right)\\
&=&\sum_{t=0}^\infty\widetilde{\nu}(\gamma^t\cdot s)_{{\bf r},{\rm nil}}-\widetilde{\nu}(\gamma^{t+1}\cdot s)_{{\bf r},{\rm nil}}=\widetilde{\nu}(s)_{{\bf r},{\rm nil}},
\end{eqnarray*}
where the second equality comes from comparing the nilpotent parts of \eqref{eqn: diagonal nilpotent} and \eqref{eqn: diagonal nilpotent'}. This proves (2).
\end{proof}

\begin{notation}\label{not:+_}
Let $\pi_S:\Std\to S$ be the covering map. 
\begin{enumerate}
\item Given an oriented edge $\widehat{\bf e}$ in $\Tmc^o$, let ${\bf e}$ be an oriented edge in $\widetilde{\Tmc}^o$ such that $\pi_S({\bf e})=\widehat{\bf e}$, and denote $\widehat{\bar{\bf e}}:=\pi_S(\bar{\bf e})$. \item Given a non-edge barrier $\widehat{\bf b}$ of $\Mmc$, let ${\bf b}$ be a non-edge barrier in $\widetilde{\Mmc}$ such that $\pi_S({\bf b})=\widehat{\bf b}$, and let $\widehat{\bf b}_+:=\pi_S({\bf b}_+)$ and $\widehat{\bf b}_-:=\pi_S({\bf b}_-)$. 
\end{enumerate}
\end{notation}

Recall that for every ${\bf k}=(k_1,k_2)\in\Amc$, we denote $\bar{\bf k}:=(k_2,k_1)$, and for every ${\bf i}=(i_1,i_2,i_3)\in\Bmc$, we denote ${\bf i}_+:=(i_2,i_3,i_1)$ and ${\bf i}_-:=(i_3,i_1,i_2)$. Also, let $\Bmc_k:=\{(i_1,i_2,i_3)\in\Bmc:i_1=k\}$ for every $k\in\{1,\dots,n-1\}$. Then let
\[\mathscr{W}=\mathscr{W}_{\Tmc}\subset\Rbbb^{\Tmc^o\times\Amc}\times\Rbbb^{\Mmc\times\Bmc}\]
denote the subspace consisting of all the vectors 
\[\alpha=\left(\left(\alpha^{\bf k}_{\widehat{\bf e}}\right)_{(\widehat{\bf e},{\bf k})\in \Tmc^o\times\Amc},\left(\alpha^{\bf i}_{\widehat{\bf b}}\right)_{(\widehat{\bf b},{\bf i})\in \Mmc\times\Bmc}\right)\]
that satisfy the following linear conditions:
\begin{itemize}
\item[(I)] (Symmetry for triples) For each $\widehat{\bf b}\in \Mmc$ and each ${\bf i}=(i_1,i_2,i_3)\in \Bmc$, 
\[\alpha^{\bf i}_{\widehat{\bf b}}=\alpha^{\bf i_+}_{\widehat{\bf b}_+}=\alpha^{\bf i_-}_{\widehat{\bf b}_-},\]  
\item[(II)] (Symmetry for pairs) For each $\widehat{\bf e}\in \Tmc^o$ and each ${\bf k}=(k_1,k_2)\in \Amc$, 
\[\alpha^{\bf k}_{\widehat{\bf e}}=\alpha^{\bar{\bf k}}_{\widehat{\bar{\bf e}}}.\] 
\item[(III)] (Closed leaf equalities) For each non-isolated edge $\widehat{\bf e}\in\Tmc^o$ and each ${\bf k}=(k_1,k_2)\in\Amc$, 
\[\sum_{j=1}^{H}\left(\alpha^{\bar{\mathbf k}}_{\widehat{\mathbf e}_j}+\sum_{\mathbf i\in\Bmc_{k_1}}\alpha^{\mathbf i}_{\widehat{\mathbf b}_j}\right)=\left\{\begin{array}{ll}
\displaystyle\sum_{j=1}^{H'}\left(\alpha^{\mathbf k}_{\widehat{\mathbf e}'_j}+\sum_{\mathbf i\in\Bmc_{k_2}}\alpha^{\mathbf i}_{\widehat{\mathbf b}'_j}\right)&\text{if }x\neq x';\\
\displaystyle-\sum_{j=1}^{H'}\left(\alpha^{\bar{\mathbf k}}_{\widehat{\mathbf e}'_j}+\sum_{\mathbf i\in\Bmc_{k_1}}\alpha^{\mathbf i}_{\widehat{\mathbf b}'_j}\right)&\text{if }x= x'.
\end{array}\right.\]
where $H=H(\widehat{\bf e})$, $H'=H'(\widehat{\bf e})$, $\widehat{\bf e}_j=\widehat{\bf e}_j(\widehat{\bf e})$, $\widehat{\bf b}_j=\widehat{\bf b}_j(\widehat{\bf e})$, $\widehat{\bf e}'_j=\widehat{\bf e}'_j(\widehat{\bf e})$ and $\widehat{\bf b}'_j=\widehat{\bf b}'_j(\widehat{\bf e})$.
\end{itemize}
It is a result of Bonahon and Dreyer \cite[Section 4]{BonahonDreyer1} that 
\[\dim\mathscr{W}=(n^2-1)(2g-2)=\dim\Hit_V(S)\] 
(the equalities specified in (III) above are called the \emph{closed leaf equalities} in \cite{BonahonDreyer1}). 

\begin{prop}\label{prop: Xi injective}
$\Xi$ is a linear embedding whose image is $\mathscr{W}$.
\end{prop}

\begin{proof}
The linearity of $\Xi$ is immediate from its definition. We first prove that $\Xi$ is injective. Suppose that $\nu_1$ and $\nu_2$ are $(\rho,\Tmc,\Jmc)$-tangent cocycles such that $\Xi(\nu_1)=\Xi(\nu_2)$. For both $m=1,2$, let $\widetilde{\nu}_m:C_1(\widetilde{S},\Zbbb)\to\mathfrak{sl}(V)$ denote the $\Ad\circ\rho$-equivariant lift of $\nu_m$. Observe that every $1$-simplex in $\widetilde S$ is homotopic to the concatenation of finitely many $1$-simplices, each of which satisfies one of the following:
\begin{enumerate}
\item[(i)] its image intersects exactly one barrier in $\widetilde{\Dmc}$, which is either an isolated edge or a non-edge barrier,
\item[(ii)] there is a bridge $J$ in $\widetilde{\Jmc}$ across a non-isolated edge $e$ in $\widetilde{\Tmc}$ such that one of the endpoints of its image is $e\cap J$ and the other is an endpoint of $J$,
\item[(iii)] the endpoints of its image lie in the same non-isolated edge in $\widetilde{\Tmc}$.
\end{enumerate}
To prove that $\Xi$ is injective, it thus suffices to verify that $\widetilde{\nu}_1(h)=\widetilde{\nu}_2(h)$ for all $1$-simplices $h:[0,1]\to\widetilde{S}$ that satisfies (i), (ii), or (iii). 

First suppose that $h$ satisfies (i). Then $h$ satisfies Case 1 in Section \ref{sec:tangent cocycle}, so by definition, it suffices to only consider the situation where both endpoints of $h$ do not lie in the barrier that $h$ intersects. If the barrier in $\widetilde{\Dmc}$ that $h$ intersects is a non-edge barrier ${\bf b}$, let ${\bf x}$ be the triple associated to ${\bf b}$. Then either $h$ or $\bar{h}$ passes from the left to the right of ${\bf b}$, so by definition (see Proposition \ref{prop:eruption justification}),
\begin{align*}
\widetilde{\nu}_1(h)&=\left\{\begin{array}{ll}
\displaystyle\sum_{{\bf i}\in\Bmc}\alpha^{\bf i}_{\bf b}(\nu_1)\cdot A^{\mathbf{i}}_{\xi_\rho(\mathbf{x})}&\text{if }h\text{ passes from the left to the right of }{\bf b};\\
\displaystyle-\sum_{{\bf i}\in\Bmc}\alpha^{\bf i}_{\bf b}(\nu_1)\cdot A^{\mathbf{i}}_{\xi_\rho(\mathbf{x})}&\text{if }\bar{h}\text{ passes from the left to the right of }{\bf b}
\end{array}\right.\\
&=\left\{\begin{array}{ll}
\displaystyle\sum_{{\bf i}\in\Bmc}\alpha^{\bf i}_{\bf b}(\nu_2)\cdot A^{\mathbf{i}}_{\xi_\rho(\mathbf{x})}&\text{if }h\text{ passes from the left to the right of }{\bf b};\\
\displaystyle-\sum_{{\bf i}\in\Bmc}\alpha^{\bf i}_{\bf b}(\nu_2)\cdot A^{\mathbf{i}}_{\xi_\rho(\mathbf{x})}&\text{if }\bar{h}\text{ passes from the left to the right of }{\bf b}
\end{array}\right.\\
&=\widetilde{\nu}_2(h).
\end{align*}
On the other hand, if the barrier in $\widetilde{\Dmc}$ that $h$ intersects is an isolated edge $e$ in $\widetilde{\Tmc}$, let ${\bf e}$ be $e$ equipped with the orientation so that $h$ passes from the left to the right of ${\bf e}$. Then let ${\bf r}=(r_1,r_2)$ be the pair associated to ${\bf e}$, i.e. $r_1$ and $r_2$ are the backward and forward endpoints of ${\bf e}$ respectively. With these choices (see Proposition \ref{prop:shearing justification}),
\begin{eqnarray*}
\widetilde{\nu}_1(h)=\sum_{{\bf k}\in\Amc}2\alpha^{\bf k}_{\bf e}(\nu_1)\cdot A^{\mathbf{k}}_{\xi_\rho(\mathbf{r})}=\sum_{{\bf k}\in\Amc}2\alpha^{\bf k}_{\bf e}(\nu_2)\cdot A^{\mathbf{k}}_{\xi_\rho(\mathbf{r})}=\widetilde{\nu}_2(h).
\end{eqnarray*}

Suppose now that $h$ satisfies (ii). Choose an orientation ${\bf e}$ on $e$ so that $h(0)$ lies to the left of ${\bf e}$, and let ${\bf r}:=(r_1,r_2)$ be the pair associated to ${\bf e}$. Then by definition (see Proposition \ref{prop:shearing justification non-isolated}),
\begin{eqnarray*}
\widetilde{\nu}_1(h)_{{\bf r},{\rm diag}}=\sum_{{\bf k}\in\Amc}\alpha^{\bf k}_{\bf e}(\nu_1)\cdot A^{\mathbf{k}}_{\xi_\rho(\mathbf{r})}=\sum_{{\bf k}\in\Amc}\alpha^{\bf k}_{\bf e}(\nu_2)\cdot A^{\mathbf{k}}_{\xi_\rho(\mathbf{r})}=\widetilde{\nu}_2(h)_{{\bf r},{\rm diag}}.
\end{eqnarray*}
Then the fact that $\widetilde{\nu}_1(h)=\widetilde{\nu}_2(h)$ is an immediate consequence of Proposition \ref{prop: infinite box}(2) and the previous paragraph. Finally, if $h$ satisfies (iii), then the fact that $\widetilde{\nu}_1(h)=\widetilde{\nu}_2(h)$ is an immediate consequence of Proposition \ref{prop: infinite box}(1) and the previous paragraph. This finishes the proof that $\Xi$ is injective.

Next, we prove that the image of $\Xi$ is $\mathscr W$. Since $\dim\mathscr{T}=\dim\Hit_V(S)=\dim\mathscr{W}$, it suffices to show that the image of $\Xi$ lies in $\mathscr{W}$, i.e. if $\nu\in\mathscr{T}$, then 
\[\Xi(\nu)=\left(\left(\alpha^{\bf k}_{\widehat{\bf e}}(\nu)\right)_{{\bf k}\in\Amc,\widehat{\bf e}\in \Tmc^o},\left(\alpha^{\bf i}_{\widehat{\bf b}}(\nu)\right)_{{\bf i}\in\Bmc,\widehat{\bf b}\in \Mmc}\right)=\left(\left(\alpha^{\bf k}_{\widehat{\bf e}}\right)_{{\bf k}\in\Amc,\widehat{\bf e}\in \Tmc^o},\left(\alpha^{\bf i}_{\widehat{\bf b}}\right)_{{\bf i}\in\Bmc,\widehat{\bf b}\in \Mmc}\right)\]
satisfy the symmetry for triples, the symmetry of pairs, and the closed leaf equalities (conditions (I), (II), and (III) above). The fact that $\Xi(\nu)$ satisfies the symmetry of triples holds is immediate from Proposition~\ref{prop:eruption justification}(1), and the symmetry of pairs holds by Proposition \ref{prop:shearing justification}(1) and Proposition~\ref{prop:shearing justification non-isolated}(1). 

To see that the closed leaf equalities hold, fixed an oriented non-isolated edge ${\bf e}$ in $\widetilde{\Tmc}^o$ and a bridge $J$ in $\widetilde\Jmc$ across ${\bf e}$. Let ${\bf r}=(r_1,r_2)$ be the pair associated to ${\bf e}$, and let $H$, $x$, $\gamma$, ${\bf b}_j$, ${\bf e}_j$, $\widehat{\bf b}_j$, $\widehat{\bf e}_j$, $l_j$, $m_j$, $m$, $H'$, $x'$, $\gamma'$, ${\bf b}'_j$, ${\bf e}'_j$, $\widehat{\bf b}_j'$, $\widehat{\bf e}_j'$ $l_j'$, $m_j'$, and $m'$ be as defined in Notation \ref{not:admissible}. Then let ${\bf r}_j$ and ${\bf r}'_j$ be the pairs associated to ${\bf e}_j$ and ${\bf e}'_j$ respectively, and let ${\bf x}_j$ and ${\bf x}'_j$ be the triples associated to ${\bf b}_j$ and ${\bf b}'_j$ respectively. Observe that $l_j$ passes from the left to the right of ${\bf b}_j$ if and only if $x=r_2$, and $l_j'$ passes from the left to the right of ${\bf b}'_j$ if and only if $x'=r_1$. Similarly, $m_j$ passes from the left to the right of ${\bf e}_j$ if and only if $x=r_2$, and $m_j'$ passes from the left to the right of ${\bf e}'_j$ if and only if $x'=r_1$. Thus, by definition (see Proposition~\ref{prop:eruption justification}, \eqref{eqn: coefficient 1}, Proposition~\ref{prop:shearing justification}, and \eqref{eqn: coefficient 2}),
\[\widetilde{\nu}(l_j)=\left\{\begin{array}{ll}
\displaystyle-\sum_{{\bf i}\in \Bmc }\alpha^{\mathbf{i}}_{\widehat{\mathbf{b}}_j}\cdot A^{\mathbf{i}}_{\xi_\rho(\mathbf{x}_j)}&\text{if }x=r_1;\\
\displaystyle\sum_{{\bf i}\in \Bmc}\alpha^{\mathbf{i}}_{\widehat{\mathbf{b}}_j}\cdot A^{\mathbf{i}}_{\xi_\rho(\mathbf{x}_j)}&\text{if }x=r_2,
\end{array}\right.\]
\[\widetilde{\nu}(m_j)=\left\{\begin{array}{ll}
\displaystyle\sum_{{\bf k}\in\Amc}2\alpha^{\mathbf{k}}_{\widehat{\mathbf{e}}_j}\cdot D^{\mathbf{k}}_{\xi_\rho(\mathbf{r}_j)}&\text{if }x=r_1;\\
\displaystyle-\sum_{{\bf k}\in\Amc}2\alpha^{\mathbf{k}}_{\widehat{\mathbf{e}}_j}\cdot D^{\mathbf{k}}_{\xi_\rho(\mathbf{r}_j)}&\text{if }x=r_2,
\end{array}\right.\]
\[\widetilde{\nu}(l_j')=\left\{\begin{array}{ll}
\displaystyle\sum_{{\bf i}\in \Bmc}\alpha^{\mathbf{i}}_{\widehat{\mathbf{b}}_j'}\cdot A^{\mathbf{i}}_{\xi_\rho(\mathbf{x}_j')}&\text{if }x'=r_1;\\
\displaystyle-\sum_{{\bf i}\in \Bmc}\alpha^{\mathbf{i}}_{\widehat{\mathbf{b}}_j'}\cdot A^{\mathbf{i}}_{\xi_\rho(\mathbf{x}_j')}&\text{if }x'=r_2,
\end{array}\right.\]
\[\widetilde{\nu}(m_j')=\left\{\begin{array}{ll}
\displaystyle-\sum_{{\bf k}\in\Amc}2\alpha^{\mathbf{k}}_{\widehat{\mathbf{e}}_j'}\cdot D^{\mathbf{k}}_{\xi_\rho(\mathbf{r}_j')}&\text{if }x'=r_1;\\
\displaystyle\sum_{{\bf k}\in\Amc}2\alpha^{\mathbf{k}}_{\widehat{\mathbf{e}}_j'}\cdot D^{\mathbf{k}}_{\xi_\rho(\mathbf{r}_j')}&\text{if }x'=r_2.
\end{array}\right.\]

For any $k\in\{1,\dots,n-1\}$, let $\Bmc_{k}:=\{(i_1,i_2,i_3)\in\Bmc:i_1=k\}$. The next lemma allows us to write down the diagonal part of $\widetilde{\nu}(l_j)$, $\widetilde{\nu}(m_j)$, $\widetilde{\nu}(l_j')$, and $\widetilde{\nu}(m_j')$ with respect to ${\bf r}$.

\begin{lem}\label{lem: diagonal part}
Let ${\bf k}=(k_1,k_2)\in\Amc$. If $B^{\mathbf k}\in\smf\lmf(V)$ is one of the following:
\begin{itemize}
\item $\frac{1}{2}A^{\mathbf i}_{\xi_\rho(\mathbf x_j)}$ for some $\mathbf i\in\Bmc_{k_1}$ and $j=1,\dots,H$, or
\item $-D^{\bar{\mathbf k}}_{\xi_\rho(\mathbf r_j)}$ for some $j=1,\dots,H$, where $\bar{\bf k}:=(k_2,k_1)$ for all ${\bf k}=(k_1,k_2)\in\Amc$, 
\end{itemize}
then $B^{\bf k}\cdot\xi_\rho^{(i)}(x)\subset\xi_\rho^{(i)}(x)$, so we may decompose $B^{\mathbf k}=B^{\mathbf k}_{{\bf r},{\rm diag}}+B^{\mathbf k}_{{\bf r},{\rm nil}}$. Furthermore,
\[B^{\mathbf k}_{{\bf r},{\rm diag}}=\left\{\begin{array}{ll}
D^{\mathbf k}_{\xi_\rho(\mathbf r)}&\text{if }x=r_1,\\ 
D^{\mathbf k}_{\xi_\rho(\bar{\mathbf r})}&\text{if }x=r_2,
\end{array}\right.\]
where $\bar{\bf r}:=(r_2,r_1)$. 
\end{lem}

\begin{proof}
The first statement of the lemma follows from Proposition \ref{prop:linearly independent triple}(1) and Proposition \ref{prop:linearly independent pair}(1). To second statement follows from the observation that $B^{\mathbf k}$ has an eigenbasis whose corresponding eigenvalues are either $\frac{k_2}{2n}$ or $-\frac{k_1}{2n}$, and it scales $\xi_\rho^{(k_1)}(x)$ by $\frac{k_2}{2n}$. 
\end{proof}

Since $D^{\mathbf k}_{\xi_\rho(\mathbf r)}=-D^{\bar{\mathbf k}}_{\xi_\rho(\bar{\mathbf r})}$ for all ${\bf k}\in\Amc$, by Lemma \ref{lem: diagonal part}, 
\begin{align*}
\widetilde{\nu}(l_j)_{{\bf r},{\rm diag}}&=\left\{\begin{array}{ll}
\displaystyle-2\sum_{\mathbf k\in\Amc}\sum_{\mathbf i\in\Bmc_{k_1}}\alpha^{\mathbf i}_{\widehat{\mathbf b}_j}\cdot D^{\mathbf k}_{\xi_\rho(\mathbf r)}&\text{if }x=r_1;\\
\displaystyle 2\sum_{\mathbf k\in\Amc}\sum_{\mathbf i\in\Bmc_{\mathbf k_1}}\alpha^{\mathbf i}_{\widehat{\mathbf b}_j}\cdot D^{\mathbf k}_{\xi_\rho(\bar{\mathbf r})}&\text{if }x=r_2,
\end{array}\right.\\
&=\left\{\begin{array}{ll}
\displaystyle-2\sum_{\mathbf k\in\Amc}\sum_{\mathbf i\in\Bmc_{k_1}}\alpha^{\mathbf i}_{\widehat{\mathbf b}_j}\cdot D^{\mathbf k}_{\xi_\rho(\mathbf r)}&\text{if }x=r_1;\\
\displaystyle-2\sum_{\mathbf k\in\Amc}\sum_{\mathbf i\in\Bmc_{\mathbf k_2}}\alpha^{\mathbf i}_{\widehat{\mathbf b}_j}\cdot D^{\mathbf k}_{\xi_\rho(\mathbf r)}&\text{if }x=r_2,
\end{array}\right.
\end{align*}
and
\begin{align*}
\widetilde{\nu}(m_j)_{{\bf r},{\rm diag}}&=\left\{\begin{array}{ll}
\displaystyle -2\sum_{\mathbf k\in\Amc}\alpha^{\bar{\mathbf k}}_{\widehat{\mathbf e}_j}\cdot D^{\mathbf k}_{\xi_\rho(\mathbf r)}&\text{if }x=r_1;\\
\displaystyle 2\sum_{\mathbf k\in\Amc}\alpha^{\bar{\mathbf k}}_{\widehat{\mathbf e}_j}\cdot D^{\mathbf k}_{\xi_\rho(\bar{\mathbf r})}&\text{if }x=r_2,
\end{array}\right.\\
&=\left\{\begin{array}{ll}
\displaystyle -2\sum_{\mathbf k\in\Amc}\alpha^{\bar{\mathbf k}}_{\widehat{\mathbf e}_j}\cdot D^{\mathbf k}_{\xi_\rho(\mathbf r)}&\text{if }x=r_1;\\
\displaystyle -2\sum_{\mathbf k\in\Amc}\alpha^{\mathbf k}_{\widehat{\mathbf e}_j}\cdot D^{\mathbf k}_{\xi_\rho(\mathbf r)}&\text{if }x=r_2,
\end{array}\right.
\end{align*}
Similarly,
\[\widetilde{\nu}(l_j')_{{\bf r},{\rm diag}}=\left\{\begin{array}{ll}
\displaystyle2\sum_{\mathbf k\in\Amc}\sum_{\mathbf i\in\Bmc_{k_1}}\alpha^{\mathbf i}_{\widehat{\mathbf b}_j'}\cdot D^{\mathbf k}_{\xi_\rho(\mathbf r)}&\text{if }x'=r_1;\\
\displaystyle2\sum_{\mathbf k\in\Amc}\sum_{\mathbf i\in\Bmc_{\mathbf k_2}}\alpha^{\mathbf i}_{\widehat{\mathbf b}_j'}\cdot D^{\mathbf k}_{\xi_\rho(\mathbf r)}&\text{if }x'=r_2,
\end{array}\right.\]
and
\[ \widetilde{\nu}(m_j')_{{\bf r},{\rm diag}}=\left\{\begin{array}{ll}
\displaystyle 2\sum_{\mathbf k\in\Amc}\alpha^{\bar{\mathbf k}}_{\widehat{\mathbf e}_j'}\cdot D^{\mathbf k}_{\xi_\rho(\mathbf r)}&\text{if }x'=r_1;\\
\displaystyle 2\sum_{\mathbf k\in\Amc}\alpha^{\mathbf k}_{\widehat{\mathbf e}_j'}\cdot D^{\mathbf k}_{\xi_\rho(\mathbf r)}&\text{if }x'=r_2.
\end{array}\right.\]

By definition, 
\[\widetilde{\nu}(m)=\left\{\begin{array}{ll}
-\widetilde{\nu}(m')&\text{if }x\neq x';\\
\widetilde{\nu}(m')&\text{if }x=x'.
\end{array}\right.\]
It then follows from Proposition \ref{prop: infinite box} that
\[\sum_{j=1}^H\widetilde{\nu}(l_j)_{{\bf r},{\rm diag}}+\sum_{j=1}^H\widetilde{\nu}(m_j)_{{\bf r},{\rm diag}}=\left\{\begin{array}{ll}
\displaystyle-\sum_{j=1}^{H'}\widetilde{\nu}(l_j')_{{\bf r},{\rm diag}}-\sum_{j=1}^{H'}\widetilde{\nu}(m_j')_{{\bf r},{\rm diag}}&\text{if }x\neq x';\\
\displaystyle\sum_{j=1}^{H'}\widetilde{\nu}(l_j')_{{\bf r},{\rm diag}}+\sum_{j=1}^{H'}\widetilde{\nu}(m_j')_{{\bf r},{\rm diag}}&\text{if }x= x'.
\end{array}\right.\] 
Proposition \ref{prop:linearly independent pair}(2) implies that we may compare coefficients to deduce that
\[\sum_{j=1}^{H}\left(\alpha^{\bar{\mathbf k}}_{\widehat{\mathbf e}_j}+\sum_{\mathbf i\in\Bmc_{k_1}}\alpha^{\mathbf i}_{\widehat{\mathbf b}_{j}}\right)=\left\{\begin{array}{ll}
\displaystyle\sum_{j=1}^{H'}\left(\alpha^{\mathbf k}_{\widehat{\mathbf e}_j'}+\sum_{\mathbf i\in\Bmc_{k_2}}\alpha^{\mathbf i}_{\widehat{\mathbf b}_j'}\right)&\text{if }x\neq x';\\
\displaystyle-\sum_{j=1}^{H'}\left(\alpha^{\bar{\mathbf k}}_{\widehat{\mathbf e}_j'}+\sum_{\mathbf i\in\Bmc_{k_1}}\alpha^{\mathbf i}_{\widehat{\mathbf b}_j'}\right)&\text{if }x= x'.
\end{array}\right.\]
for all ${\bf k}=(k_1,k_2)\in\Amc$. This proves that $\Xi(\nu)$ satisfies the closed leaf equalities.
\end{proof}

%\subsection{The $(\Tmc,\Jmc)$-trivialization is symplectic}
%By Proposition \ref{prop: Xi injective}, we may now define, for every $[\rho]\in\Hit_V(S)$, the linear isomorphism
%\[\Phi_{[\rho],\Tmc,\Jmc}:=\Xi_{\rho,\Tmc,\Jmc}\circ\Psi_{\rho,\Tmc,\Jmc}:T_{[\rho]}\Hit_V(S)\to \mathscr{W}_\Tmc,\]
%where 
%\[\Psi_{\rho,\Tmc,\Jmc}:T_{[\rho]}\Hit_V(S)\to C^1(S,\smf\lmf(V)_{\Ad\circ\rho})\] 
%is the linear embedding defined in Section \ref{sec: tangent cocycles} whose image is $\mathscr{T}(\rho,\Tmc,\Jmc)$.

%\begin{prop}
%The map $\Phi_{[\rho],\Tmc,\Jmc}$ does not depend on the choice of representative $\rho$ of $[\rho]$.
%\end{prop}

%%%%%%%%%%%%%%%%%%%%%%%%%%%%%%%%%%%%%%%%%%%%%%%%%%
\subsection{Description of $\Omega$ and the $(\Tmc,\Jmc)$-parallel vector fields}\label{sec:admissibility linear independence}
%%%%%%%%%%%%%%%%%%%%%%%%%%%%%%%%%%%%%%%%%%%%%%%%%%
Next, we describe the map $\Omega$. First, if ${\bf b}$ is a non-edge barrier in $\widetilde\Dmc$, let ${\bf x}=(x_1,x_2,x_3)$ be the triple associated to ${\bf b}$. Then for any non-edge barrier ${\bf b}$ in $\widetilde{\Dmc}$ and any $\mathbf i\in\Bmc$, define the function 
\[\tau^{\mathbf i}_{\mathbf b}:\Hit_V(S)\to\Rbbb\] 
by $\tau^{\mathbf i}_{\mathbf b}[\rho]:=\log T^{\mathbf i}(\xi_\rho(\mathbf x))$ and ${\bf x}$ is the triple associated to ${\bf b}$. This is well-defined by the projective invariance of the triple ratio and Theorem \ref{thm:Zhang}(1). The projective invariance of the triple ratio also implies that $\tau^{\mathbf i}_{\mathbf b}[\rho]=\tau^{\mathbf i}_{\gamma\cdot\mathbf b}[\rho]$ for all $\gamma\in\Gamma$, so for every non-edge barrier $\widehat{\bf b}$ in $\Dmc$, we may define
\[\tau^{\mathbf i}_{\widehat{\mathbf b}}[\rho]:=\tau^{\mathbf i}_{\mathbf b}[\rho]\]
for some (any) lift ${\bf b}$ in $\widetilde{\Dmc}$ of $\widehat{\bf b}$. 

Similarly, for any oriented edge $\mathbf e\in\widetilde{\Tmc}^o$, let ${\bf r}=(r_1,r_2)$ be the pair associated to ${\bf e}$. If ${\bf e}$ is isolated, let $T$ and $T'$ be the two ideal triangles of $\widetilde{\Tmc}$ that lie to the left and right of ${\bf e}$ respectively, and let $w$ and $w'$ be respectively the vertices of $T$ and $T'$ that are neither $r_1$ nor $r_2$. Then for any $\mathbf k\in\Amc$, define the function 
\[\sigma^{\mathbf k}_{\mathbf e}:\Hit_V(S)\to\Rbbb\] 
by $\sigma^{\mathbf k}_{\mathbf e}[\rho]:=\log\left(-C^\mathbf k\big(\xi_\rho(r_1),\xi_\rho(w),\xi_\rho(r_2),\xi_\rho(w')\big)\right)$. This is well-defined by Theorem \ref{thm:Zhang}(2) and the projective invariance of the cross ratio.  On the other hand, if ${\bf e}$ is non-isolated, let $J$ be a bridge in $\widetilde{\Jmc}$ across $\mathbf e$, and let $T$ and $T'$ be the ideal triangles of $\widetilde{\Tmc}$ that contain the endpoints of $J$, so that $T$ and $T'$ lie to the left and right of $\mathbf e$ respectively. Let ${\bf x}=(x_1,x_2,x_3)$ and ${\bf x}'=(x_1',x_2',x_3')$ be the vertices of $T$ and $T'$ respectively, so that $x_1$ and $x_1'$ are endpoints of ${\bf e}$, while $x_3$ and $x_3'$ are not vertices of the fragmented ideal triangles that contain the endpoints of $J$. Let $q$ and $q'$ be the endpoint of ${\bf e}$ that is not $x_1$ and $x_1'$ respectively. Then for any $\mathbf k\in\Amc$, define the function 
\[\alpha^{\mathbf k}_{\mathbf e}:\Hit_V(S)\to\Rbbb\] 
by $\alpha^{\mathbf k}_{\mathbf e}[\rho]:=\log\left(-C^\mathbf k\big(\xi_\rho(r_1),u\cdot\xi_\rho(x_3),\xi_\rho(r_2),u'\cdot\xi_\rho(x_3')\big)\right)$, where $u\in\PGL(V)$ (resp. $u'\in\PGL(V)$) is the unipotent projective transformations that fixes $\xi(x_1)$ (resp. $\xi(x_1')$) and sends $\xi(x_2)$ to $\xi(q)$ (resp. $\xi(x_2')$ to $\xi(q')$). This is well-defined by \cite[Proposition 4.11]{SunWienhardZhang}. The projective invariance of the cross ratio then implies that regardless of whether ${\bf e}$ is isolated or non-isolated, we have $\sigma^{\mathbf k}_{\mathbf e}[\rho]=\sigma^{\mathbf k}_{\gamma\cdot \mathbf e}[\rho]$ for all $\gamma\in\Gamma$, so for every oriented edge $\widehat{\bf e}\in\Tmc^o$, we may define
\[\sigma^{\mathbf k}_{\widehat{\mathbf e}}[\rho]=\sigma^{\mathbf k}_{\mathbf e}[\rho].\]

With this, we can describe the parameterization map $\Omega$ in Theorem \ref{thm: par}: It is given by
\[\Omega=\Omega_{\Tmc,\Jmc}:[\rho]\mapsto\left(\left(\sigma_{\widehat{\mathbf e}}^{\mathbf k}[\rho]\right)_{(\mathbf k,\widehat{\mathbf e})\in\Amc\times\Tmc^o},\left(\tau^{\mathbf i}_{\widehat{\mathbf b}}[\rho]\right)_{(\mathbf i,\widehat{\mathbf b})\in\Bmc\times\Mmc}\right).\]

\begin{remark}
\begin{enumerate}
\item The convex cone $\mathscr{C}$ is an open cone in $\mathscr{W}$ cut out by a family of explicit inequalites called the \emph{closed leaf inequalities} (see \cite[Section 4]{BonahonDreyer1}). However, the explicit description of $\mathscr{C}$ is not needed in this work, so we do not give it here.
\item Bonahon-Dreyer \cite{BonahonDreyer1} previously constructed a similar parameterization of $\Hit_V(S)$. The only difference between this parameterization and the one given by Bonahon-Dreyer parameterization is that when $\widehat{\bf e}$ is a non-isolated edge in $\Tmc^o$, the invariant $\sigma^{\bf k}_{\widehat{\bf e}}$ give here is different from the ones used in Bonahon-Dreyer. While the invariants used in Bonahon-Dreyer are simpler, Theorem \ref{thm: par} implies that the ones used here are better for describing the $(\Tmc,\Jmc)$-trivialization, and as we will see later, are more suitable for studying the symplectic structure on $\Hit_V(S)$.
\end{enumerate}
\end{remark} 

Using the parameterization $\Omega$, we can define the $(\Tmc,\Jmc)$-parallel vector fields.

\begin{definition}\label{def: parallel vector fields}
For every $\mu\in\mathscr{W}$, the \emph{$(\Tmc,\Jmc)$-parallel vector field on $\Hit_V(S)$ associated to $\mu$} is a vector field $\Xmc_\mu$ on $\Hit_V(S)$ such that ${\rm d}\Omega_{[\rho]}(\Xmc_\mu([\rho]))=\mu$ for all $[\rho]\in\Hit_V(S)$. 
\end{definition}

%Recall that $\Hmc_{\rho,{\bf x}}$ is the real-analytic submanifold of $\widetilde{\Hit}_V(S)$ defined by \eqref{eqn: submanifold}, and $g_{{\bf x},{\bf y}}:\Hmc_{\rho,{\bf x}}\to\PGL(V)$ is the real-analytic map defined by \eqref{eqn: g def}. Also, recall that $\pi:\widetilde{\Hit}_V(S)\to\Hit_V(S)$ is the obvious projection, and $\pi|_{\Hmc_{\rho,{\bf x}}}:\Hmc_{\rho,{\bf x}}\to\Hit_V(S)$ is a real-analytic diffeomorphism. 
Fix 
\[\mu=\left(\left(\mu_{\widehat{\mathbf e}}^{\mathbf k}\right)_{(\mathbf k,\widehat{\mathbf e})\in\Amc\times\Tmc^o},\left(\mu^{\mathbf i}_{\widehat{\mathbf b}}\right)_{(\mathbf i,\widehat{\mathbf b})\in\Bmc\times\Mmc}\right)\in\mathscr{W},\] 
and let $\Xmc_\mu$ be the $(\Tmc,\Jmc)$-parallel vector field on $\Hit_V(S)$ associated to $\mu$. Then for any $\rho\in\Hit_V(S)$, let $\nu_\mu^\rho$ be the $(\rho,\Tmc,\Jmc)$-tangent cocycle associated to $\Xmc_\mu([\rho])$. We will now state two properties of the $\Ad\circ\rho$-equivariant lift $\widetilde\nu_\mu^\rho:C_1(\Std,\Zbbb)\to\smf\lmf(V)$ of $\nu_\mu^\rho$. The first is an immediate consequence of \cite[Proposition 7.12]{SunWienhardZhang} and the second follows easily from \cite[Proposition 7.13]{SunWienhardZhang}. 

\begin{prop}\label{prop:derivative1}
Let $\mathbf x:=(x_1,x_2,x_3)$ be the vertices of an ideal triangle of $\widetilde{\Tmc}$, let $\mu\in\mathscr{W}$, and let $t\mapsto\rho_t$ be the path in $\Hmc_{\rho,{\bf x}}$ such that $[\rho_t]=\Omega^{-1}(\Omega([\rho])+t\mu)$ (since $\mathscr{C}\subset\mathscr{W}$ is open, this is well-defined for sufficiently small $t$). %Then let $g(t)\in\PGL(V)$ be the projective transformation such that 
%\[\left(\xi_t(y_1),\xi_t(y_2),\xi_t^{(1)}(y_3)\right)=g(t)\cdot\left(\xi(y_1),\xi(y_2),\xi^{(1)}(y_3)\right)\]
\begin{enumerate}
\item Let ${\bf y}=(y_1,y_2,y_3)$ be the vertices of an ideal triangle of $\widetilde{\Tmc}$ such that $x_1=y_2$, $x_2=y_1$, and $y_3\neq x_3$. Let ${\bf e}$ be the oriented edge in $\widetilde{\Tmc}^o$ whose backward and forward endpoints are $x_1$ and $x_2$, let $\mathbf r$ be the pair associated to ${\bf e}$ (i.e. ${\bf r}:=(x_1,x_2)$), let $\widehat{\bf e}:=\pi_S({\bf e})$, and let $h$ be a $1$-simplex that passes from the left to the right of ${\bf e}$ (see Figure \ref{fig:shearingbases}). Then 
\[\widetilde\nu_\mu^\rho(h)=\frac{d}{dt}\bigg|_{t=0}g_{{\bf x},{\bf y}}(\rho_t)=-2\sum_{\mathbf k\in\Amc}\mu^{\mathbf k}_{\widehat{\mathbf e}}D^{\mathbf k}_{\xi_\rho(\mathbf r)}.\]
\item Let ${\bf y}=(y_1,y_2,y_3)$ be the vertices of an ideal triangle of $\widetilde{\Tmc}$ such that $x_2=y_1$, $x_3=y_2$ and $x_1=y_3$. Let ${\bf b}$ be the non-edge barrier in $\widetilde{\Dmc}$ such that ${\bf y}$ is associated to ${\bf b}$, let $\widehat{\bf b}:=\pi_S({\bf b})$, and let $h$ be a $1$-simplex that passes from the left to the right of ${\bf b}$ (see Figure~\ref{fig:eruptionbases}). Then
\[\widetilde\nu_\mu^\rho(h)=\frac{d}{dt}\bigg|_{t=0}g_{{\bf x},{\bf y}}(\rho_t)=\sum_{\mathbf i\in\Bmc}\mu^{\mathbf i}_{\widehat{\mathbf b}}A^{\mathbf i}_{\xi_\rho(\mathbf y)}.\]
\end{enumerate}
\end{prop}

To describe the second, recall that if ${\bf e}$ is an oriented, non-isolated edge in $\widetilde{\Tmc}^o$, ${\bf r}=(r_1,r_2)$ is the pair associated to ${\bf e}$, and $J\in\widetilde{\Jmc}$ is a bridge across ${\bf e}$, then $H$, $x$, $\gamma$, ${\bf b}_j$, ${\bf e}_j$, $\widehat{\bf b}_j$, $\widehat{\bf e}_j$, $l_j$, $m_j$, $m$, $H'$, $x'$, $\gamma'$, ${\bf b}'_j$, ${\bf e}'_j$, $\widehat{\bf b}'_j$, $\widehat{\bf e}'_j$, $l_j'$, $m_j'$, and $m'$ were defined in Notation \ref{not:admissible}. Let ${\bf r}_j$ and ${\bf r}'_j$ be the pairs associated to ${\bf e}_j$ and ${\bf e}'_j$ respectively, and let ${\bf x}_j$ and ${\bf x}'_j$ be the triples associated to ${\bf b}_j$ and ${\bf b}'_j$ respectively. Then set
\[A_j(\mu):=\left\{\begin{array}{ll}
\displaystyle-\sum_{{\bf i}\in \Bmc }\mu^{\mathbf{i}}_{\widehat{\mathbf{b}}_j}\cdot A^{\mathbf{i}}_{\xi_\rho(\mathbf{x}_j)}&\text{if }x=r_1;\\
\displaystyle\sum_{{\bf i}\in \Bmc}\mu^{\mathbf{i}}_{\widehat{\mathbf{b}}_j}\cdot A^{\mathbf{i}}_{\xi_\rho(\mathbf{x}_j)}&\text{if }x=r_2,
\end{array}\right.\]
\[D_j(\mu):=\left\{\begin{array}{ll}
\displaystyle-\sum_{{\bf k}\in\Amc}2\mu^{\mathbf{k}}_{\widehat{\mathbf{e}}_j}\cdot D^{\mathbf{k}}_{\xi_\rho(\mathbf{r}_j)}&\text{if }x=r_1;\\
\displaystyle\sum_{{\bf k}\in\Amc}2\mu^{\mathbf{k}}_{\widehat{\mathbf{e}}_j}\cdot D^{\mathbf{k}}_{\xi_\rho(\mathbf{r}_j)}&\text{if }x=r_2,
\end{array}\right.\]
\[A_j'(\mu):=\left\{\begin{array}{ll}
\displaystyle\sum_{{\bf i}\in \Bmc}\mu^{\mathbf{i}}_{\widehat{\mathbf{b}}_j'}\cdot A^{\mathbf{i}}_{\xi_\rho(\mathbf{x}_j')}&\text{if }x'=r_1;\\
\displaystyle-\sum_{{\bf i}\in \Bmc}\mu^{\mathbf{i}}_{\widehat{\mathbf{b}}_j'}\cdot A^{\mathbf{i}}_{\xi_\rho(\mathbf{x}_j')}&\text{if }x'=r_2,
\end{array}\right.\]
\[D_j'(\mu):=\left\{\begin{array}{ll}
\displaystyle\sum_{{\bf k}\in\Amc}2\mu^{\mathbf{k}}_{\widehat{\mathbf{e}}_j'}\cdot D^{\mathbf{k}}_{\xi_\rho(\mathbf{r}_j')}&\text{if }x'=r_1;\\
\displaystyle-\sum_{{\bf k}\in\Amc}2\mu^{\mathbf{k}}_{\widehat{\mathbf{e}}_j'}\cdot D^{\mathbf{k}}_{\xi_\rho(\mathbf{r}_j')}&\text{if }x'=r_2,
\end{array}\right.\]
and set
\begin{align}\label{eqn: N}
N_\infty(\mu,{\bf r},J):=\sum_{j=1}^HA_j(\mu)_{{\bf r},{\rm nil}}+\sum_{j=1}^HD_j(\mu)_{{\bf r},{\rm nil}}
\end{align}
and
\begin{align*}
N_\infty'(\mu,{\bf r},J):=\sum_{j=1}^HA_j'(\mu)_{{\bf r},{\rm nil}}+\sum_{j=1}^HD_j'(\mu)_{{\bf r},{\rm nil}}
\end{align*}

\begin{prop}\label{prop:derivative2}
Let ${\bf e}$ be an oriented, non-isolated edge in $\widetilde{\Tmc}^o$, let $\widehat{\bf e}:=\pi_S({\bf e})$, let ${\bf r}$ be the pair associated to ${\bf e}$, let $J\in\widetilde{\Jmc}$ be a bridge across ${\bf e}$, and let $T$ and $T'$ be the ideal triangles of $\Tmc$ that contain the endpoints of $J$ and lie to the left and right of ${\bf e}$ respectively. Let ${\bf x}=(x_1,x_2,x_3)$ and ${\bf x}'=(x_1',x_2',x_3')$ be the vertices of $T$ and $T'$ respectively, so that $x_1$ and $x_1'$ are endpoints of ${\bf e}$, while $x_3$ and $x_3'$ are not vertices of the fragmented ideal triangles that contain the endpoints of $J$. Let $h:[0,1]\to\widetilde{S}$ be the $1$-simplex such that $h(0)$ and $h(1)$ are the endpoints of $J$ that lie in $T$ and $T'$ respectively. If $t\mapsto\rho_t$ is the path in $\Hmc_{\rho,{\bf x}}$ such that $[\rho_t]=\Omega^{-1}(\Omega([\rho])+t\mu)$, then
\[\widetilde\nu_\mu^\rho(h)=\frac{d}{dt}\bigg|_{t=0}g_{{\bf x},{\bf x}'}(\rho_t)=N_\infty(\mu,\mathbf r,J)-N'_\infty(\mu,\mathbf r,J)-2\sum_{{\bf k}\in\Amc}\mu^{\mathbf k}_{\widehat{\mathbf e}}D^{\mathbf k}_{\xi_\rho(\mathbf r)}.\]
\end{prop}

We may now prove Theorem \ref{thm: symplectic trivialization}.

\begin{proof}[Proof of Theorem \ref{thm: symplectic trivialization}]
By Proposition \ref{prop: Xi injective}, we know that $\Xi$ is an embedding whose image is $\mathscr{W}$. It follows from Propositions \ref{prop:derivative1} and~\ref{prop:derivative2} that the $\Ad\circ\rho$-equivariant lift $\widetilde\nu_\mu^\rho$ of the $(\rho,\Tmc,\Jmc)$-tangent cocycle $\nu_\mu^\rho:=\Psi(\Xmc_\mu([\rho]))$ satisfies the following:
\begin{itemize}
\item If ${\bf b}$ is a non-edge barrier in $\widetilde{\Dmc}$, $\widehat{\bf b}:=\pi_S({\bf b})$, ${\bf x}$ is the triple associated to ${\bf b}$, and $h$ is a $1$-simplex that passes from the left to the right ${\bf b}$, then
\[\widetilde\nu_\mu^\rho(h)=\sum_{\mathbf i\in\Bmc}\mu^{\mathbf i}_{\widehat{\mathbf b}}A^{\mathbf i}_{\xi(\mathbf x)}.\]
\item If ${\bf e}$ is an oriented, isolated edge in $\widetilde{\Tmc}^o$, $\widehat{\bf e}:=\pi_S({\bf e})$, ${\bf r}$ is the pair associated to ${\bf e}$, and $h$ is a $1$-simplex that passes from the left to the right of ${\bf e}$, then
\[\widetilde\nu_\mu^\rho(h)=-2\sum_{\mathbf k\in\Amc}\mu^{\mathbf k}_{\widehat{\mathbf e}}D^{\mathbf k}_{\xi_\rho(\mathbf r)}.\]
\item If ${\bf e}$ is an oriented, non-isolated edge in $\widetilde{\Tmc}^o$, $\widehat{\bf e}:=\pi_S({\bf e})$, ${\bf r}$ is the pair associated to ${\bf e}$, $J$ is a bridge across ${\bf e}$, and $h$ is a $1$-simplex such that $h(0)$ is the endpoint of $J$ that lies to the left of ${\bf e}$ and $h(1)=e\cap J$, then
\[\widetilde\nu_\mu^\rho(h)_{{\bf r},{\rm diag}}=-2\sum_{{\bf k}\in\Amc}\mu^{\mathbf k}_{\widehat{\mathbf e}}D^{\mathbf k}_{\xi_\rho(\mathbf r)}.\]
\end{itemize}
Then from the definition of $\Xi$, it follows that $\Xi\circ\Phi(\Xmc_\mu([\rho]))=\mu$.
\end{proof}

%%%%%%%%%%%%%%%%%%%%%%%%%%%%%%%%%%%%%%%%%%%%%%%%%%
%%%%%%%%%%%%%%%%%%%%%%%%%%%%%%%%%%%%%%%%%%%%%%%%%%
\section{The $(\Tmc,\Jmc)$-trivialization of $\Hit_V(S)$ is symplectic}\label{sec:trivialization}
%%%%%%%%%%%%%%%%%%%%%%%%%%%%%%%%%%%%%%%%%%%%%%%%%%
%%%%%%%%%%%%%%%%%%%%%%%%%%%%%%%%%%%%%%%%%%%%%%%%%%

In this section, we will apply Theorem \ref{thm: symplectic trivialization} to show that the $(\Tmc,\Jmc)$-trivialization of $T\Hit_V(S)$ is symplectic.

\begin{thm}\label{thm:constant}
There is a symplectic bilinear form $\omega_{\mathscr{W}}$ on the vector space $\mathscr{W}$ for which the $(\Tmc,\Jmc)$-trivialization ${\rm d}\Omega:T\Hit_V(S)\to T\mathscr{C}\cong\mathscr{C}\times\mathscr{W}$ is symplectic.

\end{thm}
Theorem \ref{thm:constant} yields the following corollary.

\begin{cor}\label{cor:Hamiltonian}
Every $(\Tmc,\Jmc)$-parallel vector field on $\Hit_V(S)$ is a Hamiltonian vector field.
\end{cor}

\begin{proof}
By definition, every $(\Tmc,\Jmc)$-parallel vector field is identified via ${\rm d}\Omega$ with a constant vector field on $T\mathscr{C}$. Since constant vector fields on an open set in a symplectic vector space are Hamiltonian vector fields, the corollary follows from Theorem \ref{thm:constant}.
\end{proof}

The proof of Theorem \ref{thm:constant} will be given in Section \ref{sec:triangulation} and Section \ref{sec:cup product}. 

%%%%%%%%%%%%%%%%%%%%%%%%%%%%%%%%%%%%%%%%%%%%%%%%%%
\subsection{A triangulation of $S$}\label{sec:triangulation}
%%%%%%%%%%%%%%%%%%%%%%%%%%%%%%%%%%%%%%%%%%%%%%%%%%

In order to prove Theorem \ref{thm:constant}, we need to compute the Goldman symplectic form using (\ref{eqn:formula}), and to do so, we need to choose a finite triangulation $\Tbbb$ of the surface that is well adapted to $\Tmc$ and $\Jmc$. We now describe this choice.

%Informally, $\Tbbb$ is constructed as follows. First, for every ideal triangle of $\Tmc$, we construct an associated \emph{internal triangle} in $\widetilde{\Sigma}$. The internal triangles have the property that if we remove (the closure of) all the internal triangles from $S$, each connected component of the remaining surface is a cylinder that deformation retracts onto a non-isolated edge in $\Tmc$. Then, using the bridge system $\Jmc$, we specify a way to cut each of the remaining cylinders into \emph{external triangles}. The external triangles have the property that they respect the non-isolated edges in $\Tmc$, i.e. the non-isolated edges do not intersect (the interior of) any external triangle.

Choose a point $q_{e}$ on each isolated edge $e\in\widetilde{\Tmc}$, such that $\gamma\cdot q_e=q_{\gamma\cdot e}$ for all $\gamma\in\Gamma$. Then for each ideal triangle $T$ of $\widetilde{\Tmc}$, let $e_1$, $e_2$ and $e_3$ be the edges of $T$, and let $\delta(T)$ be the geodesic triangle in $T$ whose vertices are $q_{e_1},q_{e_2},q_{e_3}$, see Figure \ref{fig:Tbbb}. (Each triangle $\delta(T)$ is an open set in $\widetilde{\Sigma}$.) Observe that $\gamma\cdot\delta(T)=\delta(\gamma\cdot T)$ for all $\gamma\in\Gamma$. Thus, we may denote $\delta(\widehat{T}):=\pi_S(\delta(T))$, where $\pi_S:\widetilde{\Sigma}\to\Sigma$ is the covering map and $\widehat{T}:=\pi_S(T)$.

\begin{figure}[ht]
\centering
\includegraphics[scale=0.7]{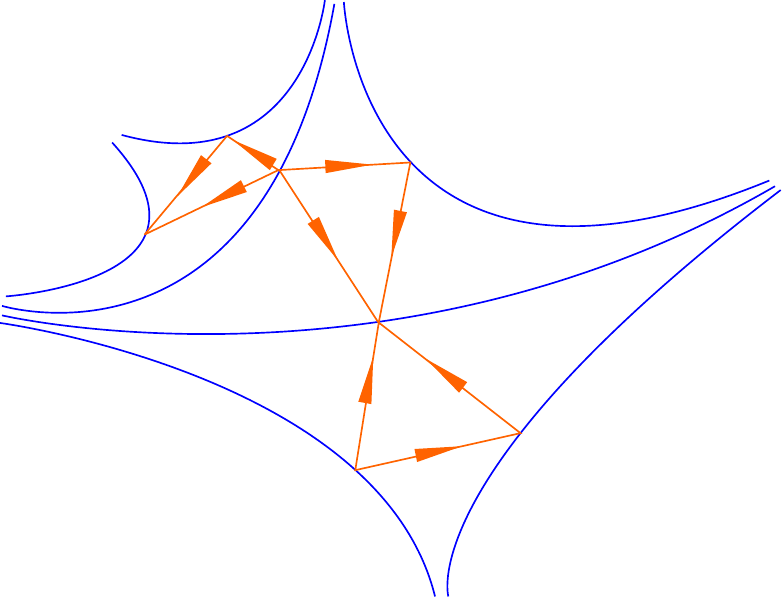}
\footnotesize
\put (-152, 203){$r_1$}
\put (-3, 142){$r_2$}
\put (-269, 97){$p$}
\put (-152, 153){$e_{\delta,1}$}
\put (-128, 122){$e_{\delta,2}$}
\put (-175, 58){$e_1$}
\put (-153, 40){$q_{e_1}$}
\put (-178, 92){$e_2$}
\put (-152, 98){$q_{e_2}$}
\put (-103, 28){$e_3$}
\put (-85, 55){$q_{e_3}$}
\put (-90, 85){$T$}
\put (-129, 63){$\delta(T)$}
\caption{$T$ and $\delta(T)$.}\label{fig:Tbbb}
\end{figure}

Let $\widetilde{\Pmc}$ (resp. $\Pmc$) denote the set of non-isolated edges in $\widetilde\Tmc$ (resp. $\Tmc$), and let $\widetilde{\Theta}$ (resp. $\Theta$) denote the set of ideal triangles of $\widetilde\Tmc$ (resp. $\Tmc$). 

\begin{definition}\
\begin{enumerate}
\item A \emph{boundary cylinder} is a connected component of
\[\Sigma\bigg\backslash\left(\bigcup_{\widehat{T}\in\Theta}\overline{\delta(\widehat{T})}\cup\bigcup_{\widehat{e}\in\Pmc}\widehat{e}\right).\] 
\item A \emph{lifted boundary cylinder} is a connected component of 
\[\widetilde\Sigma\bigg\backslash\left(\bigcup_{T\in\widetilde\Theta}\overline{\delta(T)}\cup\bigcup_{e\in\widetilde\Pmc}e\right).\]
\end{enumerate}
\end{definition}

Observe that the boundary cylinders are topological cylinders. Also, every lifted boundary cylinder $C$ has two boundary components in $\widetilde{\Sigma}$, one of which is a non-isolated edge $e_C$, and the other is a piecewise geodesic, where each of the geodesic pieces are edges of $\delta(T)$ for some ideal triangle $T$ of $\widetilde{\Tmc}$. Furthermore, for every lifted boundary cylinder $C$, $\pi_S|_C:C\to\widehat{C}$ is a covering map from $C$ to some boundary cylinder $\widehat{C}:=\pi_S(C)$. 

\begin{definition}Let $C$ be a lifted boundary cylinder.
\begin{enumerate}
\item For every bridge $J\in\widetilde{\Jmc}$ across $e_C$, let $q_J:=J\cap e_C$. Also, let $T=T_{C,J}$ be the ideal triangle of $\widetilde{\Tmc}$ that contains an endpoint of $J$, such that the (unique) edge $e=e_{C,J}$ of $T$ that intersects $J$ lies in the boundary of $C$. The \emph{bridge parallel edge} $s_{C,J}$ in $C$ corresponding to $J$ is the geodesic segment whose endpoints are $q_J$ and $q_e$, see Figure~\ref{fig:cylinderabstract}. 
\item A \emph{fundamental block} in a lifted boundary cylinder $C$ is a connected component of 
\[C-\left(\bigcup_{\{J\in\widetilde{\Jmc}:J\cap e_C\neq\emptyset\}}s_{C,J}\right).\]
\end{enumerate}
\end{definition}

Note that the boundary $\partial D$ in $\widetilde{S}$ of any fundamental block $D$ is a topological circle that contains two bridge parallel edges, and the complement in $\partial D$ of these two bridge parallel edges are two connected components. One of these components lies in $e_C$, while the other is a piecewise geodesic segment $f_D$ with finitely many pieces, each of which is the edge of $\delta(T)$ for some ideal triangle $T$ of $\widetilde{\Tmc}$.

\begin{figure}[ht]
\centering
\includegraphics[scale=1.4]{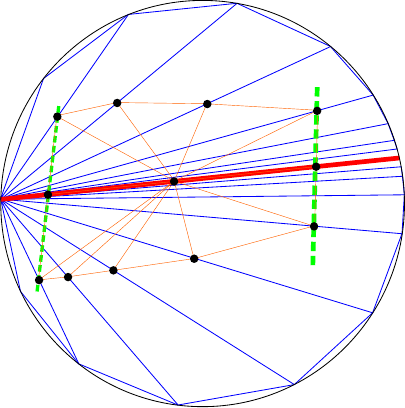}
\footnotesize
\put (-13, 160){$e_C$}
%\put (-228, 47){$e_1$}
\put (-202, 239){$e_{C,J}$}
%\put (-248, 79){$q_{e_1}$}
%\put (-235, 189){$q_{e_2}$}
%\put (-262, 94){$T_1$}
\put (-241, 215){$T_{C,J}$}
\put (-241, 136){$q_J$}
%\put (-57, 157){$a_2$}
%\put (-59, 187){$s_2$}
\put (-158, 142){$b_D$}
\put (-238, 164){$s_{C,J}$}
%\put (-195, 151){$t_1$}
%\put (-110, 160){$t_2$}
\caption{The isolated edges of $\widetilde{\Tmc}$ are drawn in blue, the non-isolated edge $e\in\widetilde{\Tmc}$ is drawn in thick red, and the bridges across $e$ is drawn in dotted green. The thin orange lines are the edges of the triangulation $\widetilde{\Tbbb}$ that lie in two fundamental blocks that share a common boundary segment in $e$.}\label{fig:cylinderabstract}
\end{figure}

For each fundamental block $D$ of a lifted boundary cylinder $C$, choose a point $b_D$ in the interior of the subsegment of $\partial D$ that lies in $e_C$, so that $b_{\gamma\cdot D}=\gamma\cdot b_D$ for all $\gamma\in\Gamma$, and $b_D=b_{D'}$ if $D$ and $D'$ are fundamental blocks whose closures intersect along a subsegment of $e_C$, see Figure~\ref{fig:cylinderabstract}. Then let $\Vbbb_D$ denote the set of endpoints of the geodesic segments in $f_D$, and let $\Fbbb_D$ be the set of geodesic segments in $\widetilde{\Sigma}$ that have $b_D$ and a point in $\Vbbb_D$ as endpoints. Observe that the geodesic segments in $\Fbbb_D$ define a triangulation of $D$. Also, these geodesic segments, together with the bridge parallel edges in $C$, define a triangulation of $C$ that is invariant under the deck group, and thus defines a triangulation of the boundary cylinder $\widehat{C}$. Let $\Tbbb(D)$, $\Tbbb(C)$, and $\Tbbb(\widehat{C})$ denote the set of triangles in this triangulation of $D$, $C$, and $\widehat{C}$ respectively.

Let $\Cmc$ (resp. $\widetilde{\Cmc}$) denote the set of boundary cylinders in $\Sigma$ (resp. lifted boundary cylinders in $\widetilde{\Sigma}$). Recall that $\Theta$ (resp. $\widetilde{\Theta}$) denote the set of ideal triangles of $\Tmc$ (resp. $\widetilde{\Tmc}$). It is clear from our definitions that 
\[\widetilde{\Tbbb}:=\{\delta(T):T\in\widetilde{\Theta}\}\cup\bigcup_{C\in\widetilde{\Cmc}}\Tbbb(C)\]
is a $\Gamma$-invariant triangulation of $\widetilde{\Sigma}$, which descends to the triangulation 
\[\Tbbb:=\{\delta(\widehat{T}):\widehat{T}\in\Theta\}\cup\bigcup_{\widehat{C}\in\Cmc}\Tbbb(\widehat{C})\]
of $\Sigma$. We refer to any triangulation $\Tbbb$ constructed above as a triangulation of $\Sigma$ that is \emph{well adapted} to $(\Tmc,\Jmc)$.

%Note that every bridge parallel edge is a crossing edge. The internal edges are exactly the edges of internal triangles, and the external edges are exactly the ones that lie in non-isolated edges. Note that if a crossing edge lies in the boundary of any fundamental block, then it is necessarily a bridge parallel edge. Also, the boundary $\partial D$ of every fundamental block $D$ contains exactly two bridge parallel edges and two external edges.

%%%%%%%%%%%%%%%%%%%%%%%%%%%%%%%%%%%%%%%%%%%%%%%%%%
\subsection{The cup product of parallel vector fields} \label{sec:cup product}
%%%%%%%%%%%%%%%%%%%%%%%%%%%%%%%%%%%%%%%%%%%%%%%%%%
We will use the following terminology in the proof of Theorem \ref{thm:constant}.

\begin{definition}\label{def: internal and external}
Let $\Tbbb$ be a triangulation of $\Sigma$ that is well adapted to $(\Tmc,\Jmc)$, and let $\widetilde{\Tbbb}$ denote its lift to $\widetilde{\Sigma}$.
\begin{enumerate}
\item An \emph{internal triangle} of $\Tbbb$ (resp. $\widetilde{\Tbbb}$) is a triangle of the form $\delta(\widehat{T})$ (resp. $\delta(T)$) for some ideal triangle $\widehat{T}$ of $\Tmc$ (resp. $T$ of $\widetilde{\Tmc}$).
\item An \emph{external triangle} of $\Tbbb$ (resp. $\widetilde{\Tbbb}$) is a triangle in $\Tbbb(\widehat{C})$ (resp. $\Tbbb(C)$) for some boundary cylinder $\widehat{C}$ in $\Sigma$ (resp. lifted boundary cylinder $C$ in $\widetilde\Sigma$). 
\item
An edge of a triangle in $\Tbbb$ (or $\widetilde{\Tbbb}$) is
\begin{itemize}
\item an \emph{internal edge} if both of its vertices are vertices of internal triangles.
\item an \emph{external edge} if both of its vertices lie in a non-isolated edge.
\item a \emph{crossing edge} if one of its vertices is a vertex of an internal triangle, and the other vertex lies in a non-isolated edge.
\end{itemize}
\end{enumerate}
\end{definition}

Note that every edge of a triangle in $\Tbbb$ is either an internal edge, an external edge, or a crossing edge. Also, all bridge parallel edges are crossing edges, and the boundary of every fundamental block contains two bridge parallel edges. 

Recall that $\Amc$ denotes the set of pairs of positive integers that sum to $n$ and $\Bmc$ denotes the set of triples of positive integers that sum to $n$. Also, recall that for any generic pair of flags $\mathbf E:=(E_1,E_2)$ in $\Fmc(V)$ and $\mathbf k\in\Amc$, we defined the shearing endomorphism $D^\mathbf k_\mathbf E$, see Definition~\ref{def:shearing endomorphism}. Similarly, for any generic triple of flags $\mathbf F:=(F_1,F_2,F_3)$ in $\Fmc(V)$ and any ${\bf i}\in\Bmc$, we defined the eruption endomorphisms $A^\mathbf i_\mathbf F$, see Definition \ref{def:eruption endomorphism}. To prove Theorem \ref{thm:constant}, it is also convenient to use the following notation.

\begin{notation}\label{not:8'} 
Let $\overline{\Bmc}$ denote the set of triples of integers $\mathbf i:=(i_1,i_2,i_3)$ that sum to $n$, and satisfy $0\leq i_1,i_2,i_3\leq n-1$. 
\begin{enumerate}
\item We can extend the eruption endomorphisms from $\Bmc$ to all of $\overline{\Bmc}$ by setting
\begin{itemize}
\item $A^{(i_1,i_2,0)}_\mathbf F$ to be the endomorphism whose eigenspaces are $F_1^{(i_1)}$ and $F_2^{(i_2)}$ corresponding to the eigenvalues $\frac{i_2}{n}$ and $-\frac{i_1}{n}$ respectively, i.e. $A^{(i_1,i_2,0)}_\mathbf F=2D^{(i_1,i_2)}_{(F_1,F_2)}$,
\item $A^{(0,i_2,i_3)}_\mathbf F$ to be the zero endomorphism,
\item $A^{(i_1,0,i_3)}_\mathbf F$ to be the endomorphism whose eigenspaces are $F_1^{(i_1)}$ and $F_3^{(i_3)}$ corresponding to the eigenvalues $\frac{i_3}{n}$ and $-\frac{i_1}{n}$ respectively,  i.e. $A^{(i_1,0,i_3)}_\mathbf F=2D^{(i_1,i_3)}_{(F_1,F_3)}$.
\end{itemize}
\item Given a non-edge barrier $\widehat{\bf b}$ of $\Dmc$, let ${\bf b}$ be a non-edge barrier in $\widetilde{\Dmc}$ such that $\pi_S({\bf b})=\widehat{\bf b}$, and let ${\bf x}=(x_1,x_2,x_3)$ be the triple associated to ${\bf b}$. For all $l=1,2,3$, let ${\bf e}_l$ be the oriented edge in $\widetilde{\Tmc}^o$ whose backward and forward endpoints are $x_{l+1}$ and $x_{l-1}$, and let $\widehat{\bf e}_l:=\pi_S({\bf e}_l)$. For any $\mathbf k:=(k_1,k_2)\in\Amc$, set $\mathbf k_1:=(0,k_1,k_2)$, $\mathbf k_2:=(k_2,0,k_1)$, and $\mathbf k_3:=(k_1,k_2,0)$. Then for any vector
\[\mu=\left(\left(\mu_{\widehat{\mathbf e}}^{\mathbf k}\right)_{(\mathbf k,\widehat{\mathbf e})\in\Amc\times\Tmc^o},\left(\mu^{\mathbf i}_{\widehat{\mathbf b}}\right)_{(\mathbf i,\widehat{\mathbf b})\in\Bmc\times\Mmc}\right)\in\mathscr{W},\]
set 
\[\mu_{\widehat{\mathbf b}}^{\mathbf k_l} := \mu_{\widehat{\mathbf e}_l}^{\mathbf k}\]
for all $l=1,2,3$. 
\end{enumerate}
In this notation, if $h:[0,1]\to\widetilde{S}$ is the $1$-simplex such that $h(0)$ lies in ${\bf e}_2$ and $h(1)$ lies in ${\bf e}_3$, then 
\begin{align}\label{eqn: simplify not}
\widetilde{\nu}_\mu^\rho(h)&=\sum_{\mathbf k\in\Amc}\left(-\mu^\mathbf k_{\widehat{\mathbf e}_2}D^{\mathbf k}_{\xi_\rho(x_3,x_1)}-\mu^\mathbf k_{\widehat{\bar{\mathbf e}}_3}D^{\mathbf k}_{\xi_\rho(x_2,x_1)}\right)+\sum_{\mathbf i\in\Bmc}\mu^\mathbf i_{\widehat{\mathbf b}}A^{\mathbf i}_{\xi_\rho(\mathbf x)}\nonumber\\
&=\sum_{\mathbf k\in\Amc}\left(\mu^{\bar{\mathbf k}}_{\widehat{\mathbf e}_2}D^{\mathbf k}_{\xi_\rho(x_1,x_3)}+\mu^{\mathbf k}_{\widehat{\mathbf e}_3}D^{\mathbf k}_{\xi_\rho(x_1,x_2)}\right)+\sum_{\mathbf i\in\Bmc}\mu^\mathbf i_{\widehat{\mathbf b}}A^{\mathbf i}_{\xi_\rho(\mathbf x)}\\
&=\frac{1}{2}\sum_{\mathbf i\in\overline{\Bmc}\setminus\Bmc}\mu^\mathbf i_{\widehat{\mathbf b}}A^{\mathbf i}_{\xi_\rho(\mathbf x)}+\sum_{\mathbf i\in\Bmc}\mu^\mathbf i_{\widehat{\mathbf b}}A^{\mathbf i}_{\xi_\rho(\mathbf x)},\nonumber
\end{align}
where $\bar{\bf k}:=(k_2,k_1)$ and $\widehat{\bar{\mathbf e}}_3$ is the $\widehat{\mathbf e}_3$ with its orientation reversed. Here, the first equality holds by definition (see \eqref{eqn: Case 1}) and Proposition~\ref{prop:derivative1}, and the second equality holds because $\mu^\mathbf k_{\widehat{\bar{\mathbf e}}_3}=\mu^{\bar{\mathbf k}}_{\widehat{\mathbf e}_3}$ (see Proposition \ref{prop:shearing justification}(1)) $D^{\mathbf k}_{{\bf E}}=-D^{\bar{\mathbf k}}_{\bar{\bf E}}$  for all ${\bf k}=(k_1,k_2)\in\Amc$ and all transverse pairs of flags ${\bf E}=(E_1,E_2)$, and .  and $\bar{\bf E}:=(E_2,E_1)$
\end{notation}

\begin{proof}[Proof of Theorem \ref{thm:constant}]
Fix a triangulation $\Tbbb$ of $\Sigma$ that is well adapted to $(\Tmc,\Jmc)$, and choose an enumeration on the set of vertices $\{v_1,\dots,v_{6g-6+2|\Pmc|}\}$ of $\Tbbb$, where $g$ is the genus of $S$. This induces an orientation on every edge of $\Tbbb$ in the following way: if $v_i$ and $v_j$ are the endpoints of an edge of $\Tbbb$, then $v_i$ is the backward endpoint and $v_j$ is the forward endpoint if and only if $i<j$.

Now, for any triangle $\widehat{\delta}\in\Tbbb$, choose a triangle $\delta\in\widetilde{\Tbbb}$ such that $\pi_S(\delta)=\widehat{\delta}$. Then let $\Tbbb_{\rm lift}$ denote the set of all such triangles in $\widetilde{\Sigma}$, one for each triangle in $\Tbbb$. For any triangle $\delta\in\Tbbb_{\rm lift}$, let $v_a,v_b,v_c$ be the vertices of $\pi_S(\delta)$ such that $a<b<c$, and let $e_{\delta,1}$ and $e_{\delta,2}$ denote the oriented edges of $\delta$ such that the backward and forward endpoints of $\pi_S(e_{\delta,1})$ are $v_a$ and $v_b$ respectively, and the backward and forward endpoints of $\pi_S(e_{\delta,2})$ are $v_b$ and $v_c$ respectively. 

Let $\mu$ and $\mu'$ be a pair of vectors in $\mathscr{W}$, and let $\Xmc_\mu$ and $\Xmc_{\mu'}$ be respectively the corresponding $(\Tmc,\Jmc)$-parallel vector fields on $\Hit_V(S)$. Let $\rho:\Gamma\to\PGL(V)$ is a Hitchin representation. Then let $\nu_\mu^\rho$ and $\nu_{\mu'}^\rho$ be the $(\rho,\Tmc,\Jmc)$-tangent cocycles of the tangent vectors $\Xmc_\mu([\rho])$ and $\Xmc_{\mu'}([\rho])$ in $T_{[\rho]}\Hit_V(S)$ respectively, and let $\widetilde{\nu}_\mu^\rho$ and $\widetilde{\nu}_{\mu'}^{\rho}$ be the $\Ad\circ\rho$-equivariant lift of $\nu_\mu^\rho$ and $\nu_{\mu'}^\rho$ respectively. By (\ref{eqn:formula}),
\begin{equation}\label{eqn:labelling formula}
\omega_{[\rho]}\big(\Xmc_\mu([\rho]),\Xmc_{\mu'}([\rho])\big)=\sum_{\delta\in\widehat{\Tbbb}}\sgn(\delta)\tr\big(\widetilde{\nu}_\mu^\rho(e_{\delta,1})\cdot\widetilde{\nu}_{\mu'}^\rho(e_{\delta,2})\big).
\end{equation}
Thus, it suffices to show that for all $\delta\in\Tbbb_{\rm lift}$ the function $\widetilde\Hit_V(S)\to\Rbbb$ given by 
\[\rho\mapsto\tr\big(\widetilde{\nu}_\mu^\rho(e_{\delta,1})\cdot\widetilde{\nu}_{\mu'}^\rho(e_{\delta,2})\big)\] 
is constant. Indeed, if this were true, then by \eqref{eqn:labelling formula}, the function $\Hit_V(S)\to\Rbbb$ given by 
\[[\rho]\mapsto\omega_{[\rho]}\big(\Xmc_\mu([\rho]),\Xmc_{\mu'}([\rho])\big)\] 
is constant, so we may define the required symplectic bilinear form $\omega_{\mathscr{W}}$ on $\mathscr{W}$ by 
\[\omega_{\mathscr{W}}(\mu,\mu'):=\omega_{[\rho]}\left(\Xmc_\mu([\rho]),\Xmc_{\mu'}([\rho])\right)\]
for some (any) $[\rho]\in\Hit_V(S)$.

The following lemma is the key computation needed to prove that for any $\delta\in\Tbbb_{\rm lift}$, $\rho\mapsto\tr\left(\widetilde{\nu}_\mu^\rho(e_{\delta,1})\cdot\widetilde{\nu}_{\mu'}^\rho(e_{\delta,2})\right)$ is constant. 

\begin{lem}\label{lem:constant0}
Let $\mathbf F:=(F_1,F_2,F_3)$ be a transverse triple of flags in $\Fmc(V)$ and $\mathbf i:=(i_1,i_2,i_3)$ and $\mathbf j:=(j_1,j_2,j_3)$ in $\overline{\Bmc}$. Then the following hold:
\begin{enumerate}
\item $\displaystyle\tr\left(A^{\mathbf i}_{\mathbf F}\cdot A^{\mathbf j_+}_{\mathbf F_+}\right)=[\min\{i_1-j_1,j_2-i_2\}]_+-\frac{i_1j_2}{n}$,
%\item $\displaystyle\tr\left(A^{i,j,k}_{F,G,H}\cdot A^{l,n-l}_{G,F}\right)=[i+l-n]_+-\frac{il}{n}$.
%\item $\displaystyle\tr\left(A^{i,j,k}_{F,G,H}\cdot A^{l,n-l}_{G,H}\right)=[\min\{l-j,i\}]_+-\frac{il}{n}$.
\item $\displaystyle\tr\left(A^{\mathbf i}_{\mathbf F}\cdot A^{\mathbf j}_{\mathbf F}\right)=\min\{i_1,j_1\}-\frac{i_1j_1}{n}$.
%\item $\displaystyle\tr\left(A^{l,n-l}_{F,G}\cdot A^{l',n-l'}_{F,H}\right)=\min\{l,l'\}-\frac{ll'}{n}$.
\end{enumerate}
where $\mathbf j_+:=(j_2,j_3,j_1)$, $\mathbf F_+:=(F_2,F_3,F_1)$, and $[c]_+:=\max\{0,c\}$ for any $c\in\Rbbb$. In particular, $\tr\left(A^{\mathbf i}_{\mathbf F}\cdot A^{\mathbf j_+}_{\mathbf F_+}\right)$ and $\tr\left(A^{\mathbf i}_{\mathbf F}\cdot A^{\mathbf j}_{\mathbf F}\right)$ do not depend on $\mathbf F$.
\end{lem}

%The proof of this lemma is an elementary computation that we perform in Appendix \ref{app:Lemma 5.6}.

\begin{proof}
For $m=1,2,3$, let $\{f_{m,1},\dots,f_{m,n}\}$ be a basis of $V$ such that 
\[F_m^{(j)}=\Span\{f_{m,1},\dots,f_{m,j}\}\] 
for all $j=1,\dots,n-1$.

(1) For any $\mathbf i:=(i_1,i_2,i_3)\in\overline{\Bmc}$, let $E^{\mathbf i}_{\mathbf F}:=A^{\mathbf i}_{\mathbf F}+\frac{i_1}{n}\cdot\id\in\End(V)$, and note that $E^{\mathbf i}_{\mathbf F}$ has $F_1^{(i_1)}$ and $F_2^{(i_2)}+F_3^{(i_3)}$ as eigenspaces corresponding to eigenvalues $1$ and $0$ respectively. This implies that 
\begin{equation}\label{eqn:simplify}
\tr\left(A^{\mathbf i}_{\mathbf F}\cdot A^{\mathbf j_+}_{\mathbf F_+}\right)=\tr\left(E^{\mathbf i}_{\mathbf F}\cdot E^{\mathbf j_+}_{\mathbf F_+}.\right)-\frac{i_1j_2}{n}\end{equation}
for all $\mathbf i:=(i_1,i_2,i_3)$ and $\mathbf j:=(j_1,j_2,j_3)$ in $\overline{\Bmc}$.

The proof of (1) will be divided into the following four cases:
\begin{enumerate}
\item[(i)] $i_1\leq j_1$,
\item[(ii)] $j_2\leq i_2$,
\item[(iii)] $i_1> j_1$, $j_2> i_2$ and $i_3\geq j_3$,
\item[(iv)] $i_1> j_1$, $j_2> i_2$ and $j_3\geq i_3$.
\end{enumerate}
We show that in all four cases, $\tr\left(E^{\mathbf i}_{\mathbf F}\cdot E^{\mathbf j_+}_{\mathbf F_+}.\right)=[\min\{i_1-j_1,j_2-i_2\}]_+$. This, together with (\ref{eqn:simplify}), proves (1).

If (i) holds, then every vector in the basis
\[\{f_{1,1},\dots,f_{1,i_1},f_{2,1},\dots,f_{2,i_2},f_{3,1},\dots,f_{3,i_3}\}\] 
lies in the kernel of $E^{\mathbf j_+}_{\mathbf F_+}\cdot E^{\mathbf i}_{\mathbf F}$, so $E^{\mathbf j_+}_{\mathbf F_+}\cdot E^{\mathbf i}_{\mathbf F}=0$. Similarly, if (ii) holds, then every vector in the basis
\[\{f_{1,1},\dots,f_{1,j_1},f_{2,1},\dots,f_{2,j_2},f_{3,1},\dots,f_{3,j_3}\}\] 
lies in the kernel of $E^{\mathbf i}_{\mathbf F}\cdot E^{\mathbf j_+}_{\mathbf F_+}$, so $E^{\mathbf i}_{\mathbf F}\cdot E^{\mathbf j_+}_{\mathbf F_+}=0$. Thus, if (i) or (ii) holds, then
\[\tr\left(E^{\mathbf i}_{\mathbf F}\cdot E^{\mathbf j_+}_{\mathbf F_+}\right)=\tr\left(E^{\mathbf j_+}_{\mathbf F_+}\cdot E^{\mathbf i}_{\mathbf F}\right)=0=[\min\{i_1-j_1,j_2-i_2\}]_+.\]

Suppose that (iii) holds. In the basis $\{f_{1,1},\dots,f_{1,i_1},f_{2,1},\dots,f_{2,i_2},f_{3,1},\dots,f_{3,i_3}\} $, the endomorphism $E^{\mathbf i}_{\mathbf F}$ can be written as the diagonal matrix
\[E^{\mathbf i}_{\mathbf F}=\left(\begin{array}{cc}
\id_{i_1}&0\\
0&0
\end{array}\right).\]
For any $k$ such that $j_1+1\leq k\leq i_1$, we may write %$f_{1,k}$ in the basis $\{f_{1,1},\dots,f_{1,j_1},f_{2,1},\dots,f_{2,j_2},f_{3,1},\dots,f_{3,j_3}\}$, i.e. 
\[f_{1,k}=\sum_{s=1}^{j_1}a_{1,s}f_{1,s}+\sum_{s=1}^{j_2}a_{2,s}f_{2,s}+\sum_{s=1}^{j_3}a_{3,s}f_{3,s}\]
for some $a_{m,s}\in\Rbbb$. One can compute that
\[E^{\mathbf j_+}_{\mathbf F_+}\cdot f_{1,k}=f_{1,k}-\left(\sum_{s=1}^{j_1}a_{1,s}f_{1,s}+\sum_{s=1}^{j_3}a_{3,s}f_{3,s}\right).\]
Similarly, for any $l$ such that $j_3+1\leq l\leq i_3$, we also have
\[E^{\mathbf j_+}_{\mathbf F_+}\cdot f_{3,l}=f_{3,l}-\left(\sum_{s=1}^{j_1}b_{1,s}f_{1,s}+\sum_{s=1}^{j_3}b_{3,s}f_{3,s}\right)\]
for some $b_{m,s}\in\Rbbb$. This implies that in the basis 
\[\{f_{1,1},\dots,f_{1,i_1},f_{2,1},\dots,f_{2,i_2},f_{3,1},\dots,f_{3,i_3}\},\] 
the endomorphism $E^{\mathbf j_+}_{\mathbf F_+}$ can be written as the matrix
\[E^{\mathbf j_+}_{\mathbf F_+}=\left(\begin{array}{ccccc}
0\cdot\id_{j_1}&*&0&0&*\\
0&\id_{i_1-j_1}&0&0&0\\
0&0&\id_{i_2}&0&0\\
0&*&0&0\cdot\id_{j_3}&*\\
0&0&0&0&\id_{i_3-j_3}
\end{array}\right).\]
It is then a straightforward computation to show that
\[\tr\left(E^{\mathbf i}_{\mathbf F}\cdot E^{\mathbf j_+}_{\mathbf F_+}\right)=i_1-j_1.\]
Furthermore, $i_3\geq j_3$ implies that $i_1-j_1\leq j_2-i_2$, which implies that
\[\tr\left(E^{\mathbf i}_{\mathbf F}\cdot E^{\mathbf j_+}_{\mathbf F_+}\right)=[\min\{i_1-j_1,j_2-i_2\}]_+.\]

By switching the roles of $E^{\mathbf i}_{\mathbf F}$ and $E^{\mathbf j_+}_{\mathbf F_+}$, and using the basis
\[\{f_{1,1},\dots,f_{1,j_1},f_{2,1},\dots,f_{2,j_2},f_{3,1},\dots,f_{3,j_3}\}\] 
in place of 
\[\{f_{1,1},\dots,f_{1,i_1},f_{2,1},\dots,f_{2,i_2},f_{3,1},\dots,f_{3,i_3}\}\] 
in the argument in the previous paragraph, we see that if (iv) holds, then 
\[\tr\left(E^{\mathbf i_1}_{\mathbf F_1}\cdot E^{\mathbf j_2}_{\mathbf F_2}\right)=j_2-i_2=[\min\{i_1-j_1,j_2-i_2\}]_+.\]

(2) By Proposition \ref{prop:linearly independent triple}(1), both $A^{\mathbf i}_{\mathbf F}$ and $A^{\mathbf j}_{\mathbf F}$ fix the flag $F$. Thus, in the basis $\{f_{1,1},\dots,f_{1,n}\}$, they can be represented by the matrices
\[A^{\mathbf i}_{\mathbf F}=\left(\begin{array}{cc}
\frac{i_2+i_3}{n}\id_{i_1}&*\\
0&U_{i_2+i_3}
\end{array}\right)\text{ and }
A^{\mathbf j}_{\mathbf F}=\left(\begin{array}{cc}
\frac{j_2+j_3}{n}\id_{j_1}&*\\
0&U_{j_2+j_3}
\end{array}\right)\]
where $U_{i_2+i_3}$ (resp.  $U_{j_2+j_3}$) is a $(i_2+i_3)\times (i_2+i_3)$ (resp. $(j_2+j_3)\times (j_2+j_3)$) upper triangular matrix whose diagonal entries are all $-\frac{i_1}{n}$ (resp. $-\frac{j_1}{n}$). Then
\begin{align*}
\tr\left(A^{\mathbf i}_{\mathbf F}\cdot A^{\mathbf j}_{\mathbf F}\right)
&=\tr\left(\left[\begin{array}{cc}\frac{i_2+i_3}{n}\id_{i_1}&0\\0&-\frac{i_1}{n}\id_{i_2+i_3}\end{array}\right]\cdot \left[\begin{array}{cc}\frac{j_2+j_3}{n}\id_{j_1}&0\\0&-\frac{j_1}{n}\id_{j_2+j_3}\end{array}\right]\right)\\
&=\tr\left(\left(\left[\begin{array}{cc}\id_{i_1}&0\\0&0\end{array}\right]-\frac{i_1}{n}\id\right)\cdot \left(\left[\begin{array}{cc}\id_{j_1}&0\\0&0\end{array}\right]-\frac{j_1}{n}\id\right)\right)\\
&=\min\{i_1,j_1\}-\frac{i_1j_1}{n}.\qedhere
\end{align*}
\end{proof}

Recall that any $\delta\in\Tbbb_{\rm lift}$ is either an internal triangle or an external triangle. We deal with these two cases separately in the following pair of lemmas.

\begin{lem}\label{lem:constant2}
If $\delta\in\widehat{T}$ is an external triangle, then the function 
\[\rho\mapsto\tr\left(\widetilde{\nu}_\mu^\rho(e_{\delta,1})\cdot\widetilde{\nu}_{\mu'}^\rho(e_{\delta,2})\right)\]
is constant.
\end{lem}

\begin{proof}
Let $D$ denote the fundamental block containing $\delta$, let ${\bf e}$ be the non-isolated edge in $\widetilde{\Tmc}$ that intersects $\partial D$, oriented so that $D$ lies to the left of ${\bf e}$, and let ${\bf r}=(r_1,r_2)$ be the pair associated to ${\bf e}$, i.e. $r_1$ and $r_2$ are respectively the backward and forward endpoints of ${\bf e}$. Let $s$ be the bridge parallel edge in $\partial D$ so that $\gamma\cdot s$ also lies in $\partial D$, where $\gamma$ is as defined in Notation \ref{not:admissible}. Then let $H$, $x$, ${\bf b}_j$, ${\bf e}_j$, $\widehat{\bf b}_j$, $\widehat{\bf e}_j$, $l_j$, $m_j$, and $m$ be as defined in Notation \ref{not:admissible}. Let $t_1$ and $t_2$ be the two external edges in $\partial D$, and let $\mathbb{G}$ denote the set of internal edges in the boundary of $D$. Since all of these edges in $\Tbbb(D)$ are oriented according to the enumeration of the vertices of $\Tbbb$, we may view them as $1$-simplices in $\widetilde{\Sigma}$, in which case they can all  be written as a linear combination of the $1$-simplices in $\Gbbb\cup\big\{t_1,t_2, s\big\}$. Thus, the endomorphism $\widetilde{\nu}_\mu^\rho(e_{\delta,l})$ (resp. $\widetilde{\nu}_{\mu'}^\rho(e_{\delta,l})$) is a linear combination of endomorphisms in 
\[\left\{\widetilde{\nu}_\mu^\rho(h):h\in\Gbbb\cup\big\{t_1,t_2, s\big\}\right\} \quad \left(\text{resp.}\quad \left\{\widetilde{\nu}_{\mu'}^\rho(h):h\in\Gbbb\cup\big\{t_1,t_2, s\big\}\right\}\right),\]
where the coefficients do not depend on $\rho$. It now suffices to show that the map
\[\rho\mapsto \tr\left(\widetilde{\nu}_\mu^\rho(h_1)\cdot\widetilde{\nu}_{\mu'}^\rho(h_2)\right)\]
is constant for all $h_1,h_2\in\Gbbb\cup\big\{t_1,t_2, s\big\}$.

Fix a basis $\{f_1,\dots,f_n\}$ of $\Rbbb^n$ such that $f_i\in \xi^{(i)}(r_1)\cap\xi^{(n-i+1)}(r_2)$ for all $i\in\{1,\dots,n\}$. Observe that by definition (see \eqref{eqn: tangent equation 1}), for all integers $j$, $\widetilde{\nu}_\mu^\rho(m_j)$, $\widetilde{\nu}_\mu^\rho(l_j)$, $\widetilde{\nu}_{\mu'}^\rho(m_j)$ and $\widetilde{\nu}_{\mu'}^\rho(l_j)$ are all represented in this basis by upper triangular matrices if $x=r_1$ and lower triangular matrices if $x=r_2$. Thus, we have decompositions
\[\widetilde{\nu}_\mu^\rho(h)=\widetilde{\nu}_\mu^\rho(h)_{{\bf r},{\rm diag}}+\widetilde{\nu}_\mu^\rho(h)_{{\bf r},{\rm nil}}\,\,\text{ and }\,\,\widetilde{\nu}_{\mu'}^\rho(h)=\widetilde{\nu}_{\mu'}^\rho(h)_{{\bf r},{\rm diag}}+\widetilde{\nu}_{\mu'}^\rho(h)_{{\bf r},{\rm nil}}\]
into the diagonal and nilpotent part with respect to ${\bf r}$ (see \eqref{eqn: decompose}). Then:
\begin{itemize}
\item By definition (see \eqref{eqn: Case 1}), for all $h\in\Gbbb$, there is some integer $j$ such that
\[\widetilde{\nu}_\mu^\rho(h)=\epsilon_h\left(\frac{1}{2}\widetilde{\nu}_\mu^\rho(l_{j-1})+\widetilde{\nu}_\mu^\rho(m_j)+\frac{1}{2}\widetilde{\nu}_{\mu}^\rho(l_j)\right)\]
and
\[\widetilde{\nu}_{\mu'}^\rho(h)=\epsilon_h\left(\frac{1}{2}\widetilde{\nu}_{\mu'}^\rho(l_{j-1})+\widetilde{\nu}_{\mu'}^\rho(m_j)+\frac{1}{2}\widetilde{\nu}_{\mu'}^\rho(l_j)\right),\]
where 
\[\epsilon_h:=\left\{\begin{array}{ll}
1&\text{ if $h$ passes from the left to the right of $m_j$},\\ 
-1&\text{ otherwise}.
\end{array}\right.\]
\item By definition (see \eqref{eqn: tangent equation 2}) and Proposition~\ref{prop: infinite box}(1), for both $l=1,2$,
\[\widetilde{\nu}_\mu^\rho(t_l)=\frac{\epsilon_l}{2}\sum_{j=1}^H\widetilde{\nu}_\mu^\rho(l_j)_{{\bf r},{\rm diag}}+\widetilde{\nu}_\mu^\rho(m_j)_{{\bf r},{\rm diag}}\]
and
\[\widetilde{\nu}_{\mu'}^\rho(t_l)=\frac{\epsilon_l}{2}\sum_{j=1}^H\widetilde{\nu}_{\mu'}^\rho(l_j)_{{\bf r},{\rm diag}}+\widetilde{\nu}_{\mu'}^\rho(m_j)_{{\bf r},{\rm diag}},\]
where 
\[\epsilon_l:=\left\{\begin{array}{ll}
1&\text{ if the orientation on $t_l$ agrees with the orientation on $m$},\\ 
-1&\text{ otherwise}.
\end{array}\right.\]
\item By definition (see \eqref{eqn: Case 1}) and Proposition~\ref{prop: infinite box}(2),
\[\widetilde{\nu}_\mu^\rho(s)_{{\bf r},{\rm nil}}=-\frac{1}{2}\widetilde{\nu}_\mu^\rho(l_1)_{{\bf r},{\rm nil}}+\sum_{j=1}^\infty \left(\widetilde{\nu}_\mu^\rho(l_j)_{{\bf r},{\rm nil}}+\widetilde{\nu}_\mu^\rho(m_j)_{{\bf r},{\rm nil}}\right)\]
and 
\[\widetilde{\nu}_{\mu'}^\rho(s)_{{\bf r},{\rm nil}}=-\frac{1}{2}\widetilde{\nu}_{\mu'}^\rho(l_1)_{{\bf r},{\rm nil}}+\sum_{j=1}^\infty \left(\widetilde{\nu}_{\mu'}^\rho(l_j)_{{\bf r},{\rm nil}}+\widetilde{\nu}_{\mu'}^\rho(m_j)_{{\bf r},{\rm nil}}\right).\]
\end{itemize}
Thus, for all $h\in\Gbbb\cup\big\{t_1,t_2, s\big\}$, $\widetilde{\nu}_\mu^\rho(h)$ is also represented in the basis $\{f_1,\dots,f_n\}$ by upper triangular matrices if $x=r_1$ and lower triangular matrices if $x=r_2$. It thus suffices to show that the map
\[\rho\mapsto \tr\left(\widetilde{\nu}_\mu^\rho(h_1)_{{\bf r},{\rm diag}}\cdot\widetilde{\nu}_{\mu'}^\rho(h_2)_{{\bf r},{\rm diag}}\right)\]
is constant for all $h_1,h_2\in\Gbbb\cup\big\{t_1,t_2, s\big\}$.

By \eqref{eqn: simplify not}, if $h\in\Gbbb$, then there is some $j\in\{1,\dots,H\}$ such that $\widetilde{\nu}_\mu^\rho(h)$ is a linear combination of $\{A^{\mathbf i}_{\xi_\rho(\mathbf x_j)}:{\bf i}\in\overline{\Bmc}\}$ where the coefficients depend only on $\mu$. Then by Lemma \ref{lem: diagonal part}, $\widetilde{\nu}_\mu^\rho(h)_{{\bf r},{\rm diag}}$ is a linear combination of $\{D^{\mathbf{k}}_{\xi_\rho(\mathbf{r})}:{\bf k}\in\Amc\}$, where the coefficients depend only on $\mu$. %Similarly, $\widetilde{\nu}_{\mu'}^\rho(h)_{{\bf r},{\rm diag}}$  is a linear combination of $\{D^{\mathbf{k}}_{\xi_\rho(\mathbf{r})}:{\bf k}\in\Amc\}$, where the coefficients depend only on $\mu'$.
At the same time, since $\widetilde{\nu}_\mu^\rho(t_1)$ and $\widetilde{\nu}_\mu^\rho(t_2)$ are linear combinations of the diagonal parts (with respect to ${\bf r}$) of $\widetilde{\nu}_\mu^\rho(l_j)$ and $\widetilde{\nu}_\mu^\rho(m_j)$ over all $j\in\{1,\dots,H\}$, it follows from Proposition \ref{prop:derivative1} that both $\widetilde{\nu}_\mu^\rho(t_1)$ and $\widetilde{\nu}_\mu^\rho(t_2)$ are linear combinations of $\{D^{\mathbf{k}}_{\xi_\rho(\mathbf{r})}:{\bf k}\in\Amc\}$, where the coefficients depend only on $\mu$. %Similarly, both $\widetilde{\nu}_{\mu'}^\rho(t_1)$ and $\widetilde{\nu}_{\mu'}^\rho(t_2)$ are linear combinations of $\{D^{\mathbf{k}}_{\xi_\rho(\mathbf{r})}:{\bf k}\in\Amc\}$, where the coefficients depend only on $\mu'$.
%
%By Proposition \ref{prop:derivative1}, we have
%\begin{align*}
%\widetilde{\nu}_\mu^\rho(t_l)%&=\frac{\epsilon_l}{2}\sum_{j=1}^H\widetilde{\nu}_\mu^\rho(l_j)_{{\bf r},{\rm diag}}+\widetilde{\nu}_\mu^\rho(m_j)_{{\bf r},{\rm diag}}\\
%&=\left\{\begin{array}{ll}
%\displaystyle-\epsilon_l\sum_{j=1}^H\sum_{{\bf k}\in\Amc}\left(\mu^{\mathbf{k}}_{\widehat{\mathbf{e}}_j}+\sum_{{\bf i}\in\Bmc_{k_2}}\mu^{\mathbf{i}}_{\widehat{\mathbf{b}}_j}\right)\cdot D^{\mathbf{k}}_{\xi_\rho(\mathbf{r})}&\text{if }x=r_2,\\
%\displaystyle-\epsilon_l\sum_{j=1}^H\sum_{{\bf k}\in\Amc}\left(\mu^{\bar{\mathbf{k}}}_{\widehat{\mathbf{e}}_j}+\sum_{{\bf i}\in\Bmc_{k_1}}\mu^{\mathbf{i}}_{\widehat{\mathbf{b}}_j}\right)\cdot D^{\mathbf{k}}_{\xi_\rho(\mathbf{r})}&\text{if }x=r_1,\\
%\end{array}\right.
%\end{align*}
%Similarly,
%\begin{align*}
%\widetilde{\nu}_{\mu'}^\rho(t_l)=\left\{\begin{array}{ll}
%\displaystyle-\epsilon_l\sum_{j=1}^H\sum_{{\bf k}\in\Amc}\left({\mu'}^{\mathbf{k}}_{\widehat{\mathbf{e}}_j}+\sum_{{\bf i}\in\Bmc_{k_2}}{\mu'}^{\mathbf{i}}_{\widehat{\mathbf{b}}_j}\right)\cdot D^{\mathbf{k}}_{\xi_\rho(\mathbf{r})}&\text{if }x=r_2,\\
%\displaystyle-\epsilon_l\sum_{j=1}^H\sum_{{\bf k}\in\Amc}\left({\mu'}^{\bar{\mathbf{k}}}_{\widehat{\mathbf{e}}_j}+\sum_{{\bf i}\in\Bmc_{k_1}}{\mu'}^{\mathbf{i}}_{\widehat{\mathbf{b}}_j}\right)\cdot D^{\mathbf{k}}_{\xi_\rho(\mathbf{r})}&\text{if }x=r_1,\\
%\end{array}\right.
%\end{align*}
Finally, by Proposition \ref{prop:derivative2} and Lemma \ref{lem: diagonal part}, $\widetilde{\nu}_\mu^\rho(s)_{{\bf r},{\rm diag}}$ is also a linear combination of $\{D^{\mathbf{k}}_{\xi_\rho(\mathbf{r})}:{\bf k}\in\Amc\}$, where the coefficients depend only on $\mu$.
%\begin{align}\label{eqn: bridge parallel edge}
%\widetilde{\nu}_\mu^\rho(s)_{{\bf r},{\rm diag}}&=\left\{\begin{array}{ll}
%\displaystyle\sum_{{\bf k}\in\Amc}\mu^\mathbf k_{\widehat{\mathbf e}_1}\cdot D^{\mathbf k}_{\xi_\rho({\bf r}_1)}-\sum_{{\bf k}\in\Amc}\mu^{\mathbf{k}}_{\widehat{\mathbf{e}}}\cdot D^{\mathbf{k}}_{\xi_\rho(\mathbf{r})}&\text{if }x=r_2,\\
%\displaystyle-\sum_{{\bf k}\in\Amc}\mu^\mathbf k_{\widehat{\mathbf e}_1}\cdot D^{\mathbf k}_{\xi_\rho({\bf r}_1)}-\sum_{{\bf k}\in\Amc}\mu^{\mathbf{k}}_{\widehat{\mathbf{e}}}\cdot D^{\mathbf{k}}_{\xi_\rho(\mathbf{r})}&\text{if }x=r_1,\\
%\end{array}\right.\\
%&=\left\{\begin{array}{ll}
%\displaystyle\sum_{{\bf k}\in\Amc}\left(\mu^{{\mathbf k}}_{\widehat{\mathbf e}_1}-\mu^{\mathbf{k}}_{\widehat{\mathbf{e}}}\right)\cdot D^{\mathbf{k}}_{\xi_\rho(\mathbf{r})}&\text{if }x=r_2,\\
%\displaystyle\sum_{{\bf k}\in\Amc}\left(\mu^{\mathbf k}_{\widehat{\bar{\mathbf e}}_1}-\mu^{\mathbf{k}}_{\widehat{\mathbf{e}}}\right)\cdot D^{\mathbf{k}}_{\xi_\rho(\mathbf{r})}&\text{if }x=r_1,
%\end{array}\right.\nonumber
%\end{align}
%and 
%\[\widetilde{\nu}_{\mu'}^\rho(s)_{{\bf r},{\rm diag}}=\left\{\begin{array}{ll}
%\displaystyle\sum_{{\bf k}\in\Amc}\left({\mu'}^{{\mathbf k}}_{\widehat{\mathbf e}_1}-{\mu'}^{\mathbf{k}}_{\widehat{\mathbf{e}}}\right)\cdot D^{\mathbf{k}}_{\xi_\rho(\mathbf{r})}&\text{if }x=r_2,\\
%\displaystyle\sum_{{\bf k}\in\Amc}\left({\mu'}^{\mathbf k}_{\widehat{\bar{\mathbf e}}_1}-{\mu'}^{\mathbf{k}}_{\widehat{\mathbf{e}}}\right)\cdot D^{\mathbf{k}}_{\xi_\rho(\mathbf{r})}&\text{if }x=r_1.
%\end{array}\right.\\\]

To finish the proof, it now suffices to verify that the map
\[\rho\mapsto \tr\left(D^{\mathbf{k}}_{\xi_\rho(\mathbf{r})}\cdot D^{\mathbf{l}}_{\xi_\rho(\mathbf{r})}\right)\]
is constant for all ${\bf k}=(k_1,k_2)$ and ${\bf l}=(l_1,l_2)$ in $\Amc$. This is immediate from Lemma~\ref{lem:constant0}(2).
\end{proof}

\begin{lem}\label{lem:constant1}
If $\delta\in\widehat{T}$ is an internal triangle, then the function 
\[[\rho]\mapsto\tr\left(\widetilde{\nu}_\mu^\rho(e_{\delta,1})\cdot\widetilde{\nu}_{\mu'}^\rho(e_{\delta,2})\right)\]
is constant.
\end{lem}

\begin{proof}
Let $T$ be the ideal triangle of $\widetilde\Tmc$ such that contains $\delta$, and let $x_1$, $x_2$, and $x_3$ be the vertices of $T$ such that $(x_1,x_2,x_3)$ is cyclically ordered and the non-edge barrier ${\bf b}$ in $T$ with $x_1$ as its terminal endpoint intersects $e_{\delta,1}$. Then for $l=1,2,3$, let $e_l$ be the geodesic whose endpoints are $e_{l+1}$ and $e_{l-1}$. 

Recall that ${\bf b}_+$, and ${\bf b}_-$ denote the non-edge barriers in $T$ whose terminal endpoints are $x_2$, and $x_3$ respectively, and ${\bf x}_+$ and ${\bf x}_-$ denote the triples associated to ${\bf b}_+$, and ${\bf b}_-$ respectively. For both $k=1,2$, let $e_{\delta,k}^+$ and $e_{\delta,k}^-$ denote the forward and backward endpoints of $e_{\delta,k}$ respectively. Then by \eqref{eqn: simplify not}
\[\widetilde{\nu}_\mu^\rho(e_{\delta,1})=\left\{\begin{array}{ll}
\displaystyle\frac{1}{2}\sum_{\mathbf i\in\overline{\Bmc}\setminus\Bmc}\mu^\mathbf i_{\widehat{\mathbf b}}A^{\mathbf i}_{\xi_\rho(\mathbf x)}+\sum_{\mathbf i\in\Bmc}\mu^\mathbf i_{\widehat{\mathbf b}}A^{\mathbf i}_{\xi_\rho(\mathbf x)}&\text{if }e_{\delta,1}^+\in e_3\\
\displaystyle-\frac{1}{2}\sum_{\mathbf i\in\overline{\Bmc}\setminus\Bmc}\mu^\mathbf i_{\widehat{\mathbf b}}A^{\mathbf i}_{\xi_\rho(\mathbf x)}-\sum_{\mathbf i\in\Bmc}\mu^\mathbf i_{\widehat{\mathbf b}}A^{\mathbf i}_{\xi_\rho(\mathbf x)}&\text{if }e_{\delta,1}^+\in e_2
\end{array}\right.\]
and
\[\widetilde{\nu}_{\mu'}^\rho(e_{\delta,2})=\left\{\begin{array}{ll}
\displaystyle\frac{1}{2}\sum_{\mathbf i\in\overline{\Bmc}\setminus\Bmc}{\mu'}^\mathbf i_{\widehat{\mathbf b}_+}A^{\mathbf i}_{\xi_\rho(\mathbf x_+)}+\sum_{\mathbf i\in\Bmc}{\mu'}^\mathbf i_{\widehat{\mathbf b}_+}A^{\mathbf i}_{\xi_\rho(\mathbf x_+)}&\text{if }e_{\delta,2}^-\in e_3\\
\displaystyle-\frac{1}{2}\sum_{\mathbf i\in\overline{\Bmc}\setminus\Bmc}{\mu'}^\mathbf i_{\widehat{\mathbf b}_-}A^{\mathbf i}_{\xi_\rho(\mathbf x_-)}-\sum_{\mathbf i\in\Bmc}{\mu'}^\mathbf i_{\widehat{\mathbf b}_-}A^{\mathbf i}_{\xi_\rho(\mathbf x_-)}&\text{if }e_{\delta,2}^-\in e_2
\end{array}\right.\]
It thus follows from Lemma \ref{lem:constant0}(1) that $\tr\left(\widetilde{\nu}_\mu^\rho(e_{\delta,1})\cdot\widetilde{\nu}_{\mu'}^\rho(e_{\delta,2})\right)$ does not depend on $\rho$.
\end{proof}

Together, Lemma \ref{lem:constant1}, Lemma \ref{lem:constant2} finishes the proof of Theorem \ref{thm:constant}.
\end{proof}

%%%%%%%%%%%%%%%%%%%%%%%%%%%%%%%%%%%%%%%%%%%%%%%%%%
%%%%%%%%%%%%%%%%%%%%%%%%%%%%%%%%%%%%%%%%%%%%%%%%%%
\section{A symplectic basis of vector fields on $\Hit_V(S)$}\label{sec:symplectic basis}
%%%%%%%%%%%%%%%%%%%%%%%%%%%%%%%%%%%%%%%%%%%%%%%%%%%%%%%%%%%%%%%%%%%%%%%%%%%%%%%%%%%%%%%%%%%%%%%%%%%%

In this section, we consider a particular ideal triangulation and compatible bridge system $(\Tmc,\Jmc)$ on $S$ that is subordinate to a given pants decomposition on $S$, see Section \ref{sec:pants}. In the companion paper \cite{SunWienhardZhang} the authors used such a pair $(\Tmc,\Jmc)$ to construct a family of $(n^2-1)(2g-2)$ \emph{special} $(\Tmc,\Jmc)$-parallel vector fields on $\Hit_V(S)$ (see \cite[Definition 8.12]{SunWienhardZhang} and Remark \ref{rem: correction} below). They then used these special $(\Tmc,\Jmc)$-parallel fields to construct an associated global coordinate system on $\Hit_V(S)$ (see \cite[Corollary 8.15]{SunWienhardZhang}).

The goal of this section is to explicitly compute the symplectic pairing between the tangent vector fields of any pair of special $(\Tmc,\Jmc)$-parallel flows. As a consequence, we prove that the global coordinate system on $\Hit_V(S)$ constructed from the special $(\Tmc,\Jmc)$-vector fields is a global Darboux coordinate system on $\Hit_V(S)$.

%%%%%%%%%%%%%%%%%%%%%%%%%%%%%%%%%%%%%%%%%%%%%%%%%%
\subsection{Choice of ideal triangulation and compatible bridge system}\label{sec:pants}
%%%%%%%%%%%%%%%%%%%%%%%%%%%%%%%%%%%%%%%%%%%%%%%%%%
We begin by recalling what it means for an ideal triangulation $\Tmc$ and a compatible bridge system $\Jmc$ to be subordinate to a pants decomposition. 

Recall that we have fixed a hyperbolic metric on $S$. Choose a pants decomposition on $S$, i.e. a maximal collection of pairwise disjoint, simple, closed geodesics on $S$. These cut the surface into $2g-2$ pairs of pants $\mathbb{P}=\{P_1,\cdots, P_{2g-2}\}$. For each $P\in\Pbbb$, choose peripheral group elements $\alpha_{P,1},\alpha_{P,2},\alpha_{P,3}\in\pi_1(P)$ such that $\alpha_{P,1} \alpha_{P,2} \alpha_{P,3}=\id$, and $P$ lies to the right of its boundary components, oriented according to $\alpha_{P,1}$, $\alpha_{P,2}$ and $\alpha_{P,3}$ (see Figure \ref{fig:pantsidealtriangulation}).

By choosing base points, the inclusion $P\subset S$ induces an inclusion $\pi_1(P)\subset\Gamma$, so we can view $\alpha_{P,1},\alpha_{P,2},\alpha_{P,3}$ as group elements in $\Gamma$. For any $\gamma\in\Gamma$, denote the repellor and attractor of $\gamma$ in $\partial\Gamma$ by $\gamma^-$ and $\gamma^+$ respectively. Then let $e_{P,1}$, $e_{P,2}$, $e_{P,3}$, $c_{P,1}$, $c_{P,2}$, and $c_{P,3}$ be geodesics in $\widetilde{S}$ such that for all $l=1,2,3$, the endpoints of $e_{P,l}$ are $\alpha_{P,l-1}^-$ and $\alpha_{P,l+1}^-$, and the endpoints of $c_{P,l}$ are $\alpha_{P,l}^-$ and $\alpha_{P,l}^+$. Observe that 
\[\widetilde{\Tmc}:=\bigcup_{P\in\Pbbb}\Gamma\cdot\big\{e_{P,1},e_{P,2},e_{P,3},c_{P,1},c_{P,2},c_{P,3}\big\}.\]
is an ideal triangulation, does not depend on the choice of base points, and is $\Gamma$-invariant. Define $\Tmc:=\widetilde{\Tmc}/\Gamma$, and observe that the set of non-isolated edges $\Pmc$ in $\Tmc$ is canonically identified with the chosen pants decomposition on $S$.  Also, note that the set $\widetilde{\Theta}$ of ideal triangles of $\widetilde{\Tmc}$ can be described as
\[\widetilde{\Theta}=\bigcup_{P\in\Pbbb}\Gamma\cdot\big\{T_P,T_P'\big\},\]
where $T_P$ is the ideal triangle in $\widetilde{S}$ whose vertices are $\alpha_{P,1}^-$, $\alpha_{P,2}^-$, and $\alpha_{P,3}^-$, while $T_P'$ is the ideal triangle in $\widetilde{S}$ whose vertices are $\alpha_{P,1}^-$, $\alpha_{P,3}^-$, and $\alpha_{P,3}\cdot\alpha_{P,2}^-$. We also let $\widehat{T}_P:=\pi_S(T_P)$ and $\widehat{T}_P':=\pi_S(T_P')$, where $\pi_S:\widetilde{S}\to S$ is the covering map.

\begin{figure}[ht]
\centering
\includegraphics[scale=0.8]{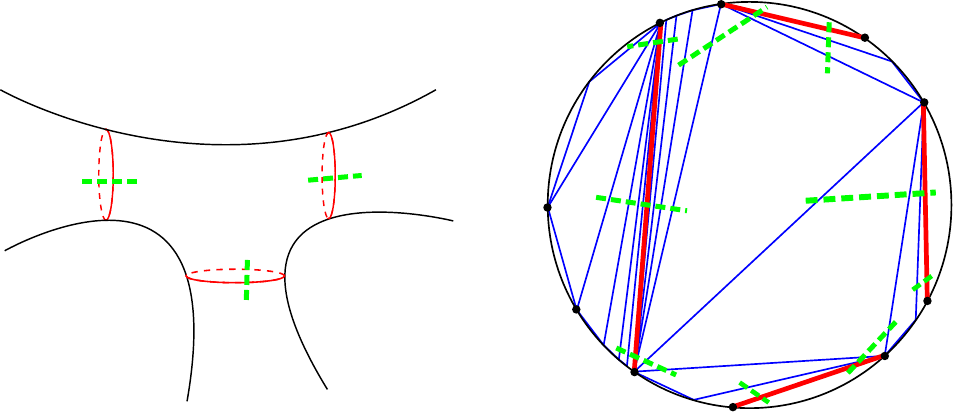}
\footnotesize
\put (-164, 7){$\alpha_{P_1,1}^-=\alpha_{P_2,1}^+$}
\put (-155, 157){$\alpha_{P_1,1}^+=\alpha_{P_2,1}^-$}
\put (-95, 164){$\alpha_{P_1,2}^-$}
\put (-33, 148){$\alpha_{P_1,2}^+$}
\put (-10, 122){$\alpha_{P_1,3}^-$}
\put (-8, 40){$\alpha_{P_1,3}^+$}
\put (-110, -7){$\alpha_{P_1,3}\cdot \alpha_{P_1,2}^+$}
\put (-30, 15){$\alpha_{P_1,3}\cdot\alpha_{P_1,2}^-$}
\put (-166, 37){$\alpha_{P_2,2}^-$}
\put (-178, 77){$\alpha_{P_2,3}^-$}
\put (-80, 100){$T_{P_1}$}
\put (-60, 50){$T_{P_1}'$}
\put (-148, 80){$T_{P_2}$}
\put (-285, 80){$P_1$}
\put (-360, 90){$P_2$}
\put (-328, 113){$[\alpha_{P_1,1}]=[\alpha_{P_2,1}^{-1}]$}
\put (-253, 113){$[\alpha_{P_1,2}]$}
\put (-257, 53){$[\alpha_{P_1,3}]$}
\caption{The two pairs of pants $P_1$ and $P_2$ drawn on the left share a common boundary edge in $\Pmc$. On the right, the non-isolated and isolated edges of $\widetilde{\Tmc}$ are drawn in thick red and blue respectively, while the bridges in $\widetilde{\Jmc}$ are drawn in dotted green. %$T_{P_1}$, $T_{P_1}'$ and $T_{P_2}$ are lifts of $\widehat{T}_{P_1}$, $\widehat{T}_{P_1}'$ and $\widehat{T}_{P_2}$ respectively.
}\label{fig:pantsidealtriangulation}
\end{figure}

Next, fix a bridge system $\widetilde{\Jmc}$ on $\widetilde{S}$ compatible with $\widetilde{\Tmc}$, such that every endpoint of every bridge in $\widetilde{\Jmc}$ lies in an ideal triangle of the form $\gamma\cdot T_P$ for some $\gamma\in\Gamma$ and some $P\in\mathbb{P}$. Then let $\Jmc:=\widetilde{\Jmc}/\Gamma$.

\begin{definition}
Any pair $(\Tmc,\Jmc)$ as described above is said to be \emph{subordinate to the pants decomposition $\Pmc$}.
\end{definition}

%%%%%%%%%%%%%%%%%%%%%%%%%%%%%%%%%%%%%%%%%%%%%%%%%%
\subsection{Special parallel vector fields}\label{sec:special}
%%%%%%%%%%%%%%%%%%%%%%%%%%%%%%%%%%%%%%%%%%%%%%%%%%

Let $(\Tmc,\Jmc)$ be an ideal triangulation and compatible bridge system on $S$ that is subordinate to a pants decomposition of $S$. Recall that $\Mmc$ denotes the set of non-edge barriers associated to $\Tmc$, $\Pmc\subset\Tmc$ denotes the set of non-isolated edges (which in this case is a pants decomposition of $S$), and $\Pmc^o\subset\Tmc^o$ denote the set of oriented, non-isolated edges. Also, recall that $\Amc$ denotes the set of pairs of positive integers that sum to $n$, $\Bmc$ denotes the set of triples of positive integers that sum to $n$, and $\overline{\Bmc}$ denotes the set of triples of integers $(i_1,i_2,i_3)$ that sum to $n$, such that $0\leq i_1,i_2,i_3\leq n-1$. 

In the companion paper \cite{SunWienhardZhang}, the special $(\Tmc,\Jmc)$-parallel vector fields were given by explicitly specifying the vector $\mu\in\mathscr{W}=\mathscr{W}_\Tmc$ that corresponds to each of them via the isomorphism ${\rm d\Omega}$, see Theorem \ref{thm: par}. We now recall these vectors $\mu$. Using Notation \ref{not:8'} to specify the vector $\mu$, we need only to specify
\[\left(\mu^{\mathbf k}_{\widehat{\bf c}}\right)_{(\mathbf k,\widehat{\bf c})\in\Amc\times\Pmc^o}\text{ and }\left(\mu^{\mathbf i}_{\widehat{\mathbf b}}\right)_{(\mathbf i,\widehat{\mathbf b})\in\overline{\Bmc}\times\Mmc}.\]
Of course, one needs to ensure that if $\widehat{\mathbf b}$ and $\widehat{\mathbf b}'$ are non-edge barriers whose associated triples $\mathbf x:=(x_1,x_2,x_3)$ and $\mathbf x':=(x_1',x_2',x_3')$ satisfy $x_1=x_1'$ and $x_2=x_3'$, then $\mu^{(k_1,k_2,0)}_{\widehat{\mathbf b}}=\mu^{(k_1,0,k_2)}_{\widehat{\mathbf b}'}$ for all ${\bf k}=(k_1,k_2)\in\Amc$.

We first describe the vectors $\mu$ associated to two types of special $(\Tmc,\Jmc)$-parallel vector fields, called the \emph{eruption vector fields} and \emph{hexagon vector fields}. These are vector fields associated to each pair of pants $P\in\Pbbb$. To do so, it is convenient to introduce the following notions.%, and their associated $\Tmc$-admissible labellings at $\rho$. Since these flows are each associated to a pair of pants $P$, we first need to label the two ideal triangles that we have cut $P$ into. 

%For each pair of pants $P$ given by $\Pmc$, the ideal triangulation $\Tmc$ cuts $P$ into two ideal triangles. Let $[\{x_1,x_2,x_3\}]$ be the bridge ending ideal triangle of $P$, and let $[\{y_1,y_2,y_3\}]$ be the other triangle in $P$. Choose representatives $\{x_1,x_2,x_3\}$ of $[\{x_1,x_2,x_3\}]$ and $\{y_1,y_2,y_3\}$ of $[\{y_1,y_2,y_3\}]$ such that $x_1=y_1$ and $x_3=y_2$, and note that $x_1=y_1<x_2<x_3=y_2<y_3<x_1=y_1$. Observe that the vertices $x_1$, $x_2$, $x_3$ of $\{x_1,x_2,x_3\}$ are repelling fixed points of three peripheral elements in $\pi_1(P)$, each of which corresponds to a boundary component of $P$. As such, the three non-edge barriers in the ideal triangle $\{x_1,x_2,x_3\}$, when viewed as ordered triples, induces three different orderings on the boundary components of $P$.

\begin{definition}
A \emph{marked pair of pants given by $\Pmc$} is a pair ${\bf P}=(P,\widehat{c})$ such that $P$ is a pair of pants in $\Pbbb$ and $\widehat{c}$ is a boundary component of $P$. The \emph{pair $(\widehat{\mathbf{b}},\widehat{\mathbf{b}}')$ of non-edge barriers associated to $\mathbf P$} are non-edge barriers $\widehat{\mathbf b}$ and $\widehat{\mathbf{b}}'$ in $\Mmc$ such that $\widehat{\bf b}$ lies in $\widehat{T}_P$, $\widehat{\bf b}'$ lies in $\widehat{T}_P'$, and $\widehat{c}$ lies in the closure of both $\widehat{\bf b}$ and $\widehat{\bf b}'$. 
\end{definition}

%Recall that $\Psi=\Psi_{\rho,\Tmc,\Jmc}:T_{[\rho]}\Hit_V(S)\to\mathscr{T}$ is the linear bijection that assigns to every tangent vector in $T_{[\rho]}\Hit_V(S)$ its associated $(\rho,\Sigma,\Tmc,\Jmc)$-tangent cocycle, and $\Xi:\mathscr{T}\to\mathscr{W}$ is the linear bijection that assigns to every $(\rho,\Tmc,\Jmc)$-tangent cocycle its associated $\Tmc$-admissible labelling at $\rho$.
To simplify notation, for any oriented edge $\widehat{\bf e}\in\Tmc^o$ and any ${\bf k}=(k_1,k_2)\in\Amc$, we denote
\[[{\bf k},\widehat{\bf e}]:=\{({\bf k},\widehat{\bf e}),(\bar{\bf k},\widehat{\bar{\bf e}})\},\]
where recall that $\bar{\bf k}:=(k_2,k_1)$ and $\widehat{\bar{\bf e}}$ was defined by Notation \ref{not:+_}(1). Similarly, for every non-edge barrier $\widehat{\bf b}\in \Mmc$ and any ${\bf i}=(i_1,i_2,i_3)\in\Bmc$, we denote
\[[{\bf i},\widehat{\bf b}]:=\{({\bf i},\widehat{\bf b}),({\bf i}_+,\widehat{\bf b}_+),({\bf i}_-,\widehat{\bf b}_-)\},\]
where ${\bf i}_+:=(i_2,i_3,i_1)$, ${\bf i}_-:=(i_3,i_1,i_2)$, and $\widehat{\bf b}_+$ and $\widehat{\bf b}_-$ were defined by Notation \ref{not:+_}(2).

\begin{definition} \label{def:eruption}
For any marked pair of pants $\mathbf P$ given by the pants decomposition $\Pmc$ and any $\mathbf i\in\Bmc$, let $(\widehat{\mathbf{b}},\widehat{\mathbf{b}}')$ be the pair of non-edge barriers associated to $\mathbf P$. Then define $\mu\in \mathscr{W}$ to be the vector such that
\begin{itemize}
\item $\mu^{\mathbf l}_{\widehat{\mathbf f}}=0$ for all $(\mathbf l,\widehat{\bf f})\in\Amc\times\Pmc^o$,
 \item $\mu^{\mathbf j}_{\widehat{\mathbf d}}=\left\{\begin{array}{ll}
 \frac{1}{2}&\text{if }(\mathbf j,\widehat{\mathbf d})\in[\mathbf i,\widehat{\mathbf b}];\\
 - \frac{1}{2}&\text{if }(\mathbf j,\widehat{\mathbf d})\in[\bar{\mathbf i},\widehat{\mathbf b}'];\\
  0&\text{otherwise}\\
 \end{array} \right.$ \\
 for all $(\mathbf j,\widehat{\mathbf d})\in\overline{\Bmc}\times\Mmc$.
 \end{itemize}
The $\mathbf i$-{\em eruption vector field} associated to $\mathbf P$, denoted $\Emc_\mathbf P^\mathbf i$, is the $(\Tmc,\Jmc)$-parallel vector field associated to $\mu$, i.e. ${\rm d}\Omega_{[\rho]}\left(\Emc_\mathbf P^\mathbf i([\rho])\right)=\mu$ for all $[\rho]\in\Hit_V(S)$, see Figure~\ref{fig:specialadmissiblelabellings}. The $\mathbf i$-{\em eruption cocycle} associated to $\mathbf P$ is the $(\rho,\Tmc,\Jmc)$-tangent cocycle $\epsilon^\mathbf i_\mathbf P:=\Psi\left(\Emc_\mathbf P^\mathbf i([\rho])\right)$.
\end{definition}

%\begin{equation}\label{eqn:eruption}
%E_{\mathbf P}^{\mathbf i}:=\frac{1}{2}\left(L_{\mathbf x}^{\mathbf i} - L_{\mathbf y}^{\overline{\mathbf i}}\right),
%\end{equation}

For any $\mathbf i=(i_1,i_2,i_3)\in\Bmc$, denote 
\[\mathbf i':=(i_1,i_3,i_2).\] 
Also, for any triple of integers $(a_1,a_3,a_3)$, denote 
\[\mathbf i(a_1,a_2,a_3):=(i_1+a_1,i_2+a_2,i_3+a_3).\]

\begin{definition} \label{def:hexagon}
For any marked pair of pants $\mathbf P$ given by the pants decomposition $\Pmc$ and any $\mathbf i\in\Bmc$, let $(\widehat{\mathbf{b}},\widehat{\mathbf{b}}')$ be the pair of non-edge barriers associated to $\mathbf P$. Then define $\upsilon,\upsilon'\in \mathscr{W}$ by
\begin{itemize}
\item $\upsilon^{\mathbf l}_{\widehat{\mathbf f}}={\upsilon'}^{\mathbf l}_{\widehat{\mathbf f}}=0$ for all $(\mathbf l,\widehat{\bf f})\in\Amc\times\Pmc^o$,
 \item $\upsilon^{\mathbf j}_{\widehat{\mathbf d}}=\left\{\begin{array}{ll}
1&\text{if }(\mathbf j,\widehat{\mathbf d})\in[\mathbf i(0,1,-1),\widehat{\mathbf b}]\cup[\mathbf i(-1,0,1),\widehat{\mathbf b}]\cup[\mathbf i(1,-1,0),\widehat{\mathbf b}]\\
-1&\text{if }(\mathbf j,\widehat{\mathbf d})\in[\mathbf i(0,-1,1),\widehat{\mathbf b}]\cup[\mathbf i(1,0,-1),\widehat{\mathbf b}]\cup[\mathbf i(-1,1,0),\widehat{\mathbf b}]\\
  0&\text{otherwise}\\
 \end{array} \right.$ \\
 for all $(\mathbf j,\widehat{\mathbf d})\in\overline{\Bmc}\times\Mmc$,
 \item ${\upsilon'}^{\mathbf j}_{\widehat{\mathbf d}}=\left\{\begin{array}{ll}
1&\text{if }(\mathbf j,\widehat{\mathbf d})\in[\mathbf i'(0,-1,1),\widehat{\mathbf b}']\cup[\mathbf i'(1,0,-1),\widehat{\mathbf b}']\cup[\mathbf i'(-1,1,0),\widehat{\mathbf b}']\\
-1&\text{if }(\mathbf j,\widehat{\mathbf d})\in[\mathbf i'(0,1,-1),\widehat{\mathbf b}']\cup[\mathbf i'(-1,0,1),\widehat{\mathbf b}']\cup[\mathbf i'(1,-1,0),\widehat{\mathbf b}']\\
 0&\text{otherwise}\\
 \end{array} \right.$ \\
 for all $(\mathbf j,\widehat{\mathbf d})\in\overline{\Bmc}\times\Mmc$,
 \end{itemize}
and define $\mu:=\upsilon+\upsilon'$. The $\mathbf i$-{\em hexagon vector field} associated to $\mathbf P$, denoted $\Hmc^\mathbf i_\mathbf P$, is the $(\Tmc,\Jmc)$-parallel vector field associated to $\mu$, i.e. ${\rm d}\Omega_{[\rho]}\left(\Hmc_\mathbf P^\mathbf i([\rho])\right)=\mu$ for all $[\rho]\in\Hit_V(S)$, see Figure~\ref{fig:specialadmissiblelabellings}. The $\mathbf i$-{\em hexagon cocycle} associated to $\mathbf P$ is the $(\rho,\Tmc,\Jmc)$-tangent cocycle $\eta^\mathbf i_\mathbf P:=\Psi\left(\Hmc_\mathbf P^\mathbf i([\rho])\right)$.
\end{definition}

%Using Lemma \ref{lem:explicit}, $H_{\mathbf P}^{\mathbf i}$ can be explicitly written as
%\begin{eqnarray}
%H_{\mathbf P}^{\mathbf i}&:=L_{\mathbf x}^{\mathbf i(0,1,-1)} - L_{\mathbf x}^{\mathbf i(-1,1,0)} + L_{\mathbf x}^{\mathbf i(-1,0,1)} - L_{\mathbf x}^{\mathbf i(0,-1,1)} + L_{\mathbf x}^{\mathbf i(1,-1,0)} - L_{\mathbf x}^{\mathbf i(1,0,-1)}\nonumber\\
%&+ L_{\mathbf y}^{\overline{\mathbf i}(0,-1,1)} - L_{\mathbf y}^{\overline{\mathbf i}(-1,0,1)} + L_{\mathbf y}^{\overline{\mathbf i}(-1,1,0)}- L_{\mathbf y}^{\overline{\mathbf i}(0,1,-1)} + L_{\mathbf y}^{\overline{\mathbf i}(1,0,-1)} - L_{\mathbf y}^{\overline{\mathbf i}(1,-1,0)},
 %\end{eqnarray}
%see Figure \ref{fig:specialadmissiblelabellings}. 

Given a marked pair of pants ${\bf P}=(P,\widehat{c})$, let $\widehat{c}_+$ and $\widehat{c}_-$ be the other two boundary components of $P$, so that both $\widehat{\bf b}_+$ and $\widehat{\bf b}'_-$ contain $\widehat{c}_+$ in their closures, and both $\widehat{\bf b}_-$ and $\widehat{\bf b}'_+$ contain $\widehat{c}_-$ in their closures. Then let ${\bf P}_+:=(P,\widehat{c}_+)$ and  ${\bf P}_-:=(P,\widehat{c}_-)$, and note that 
\[\Emc_{\mathbf P}^{\mathbf i}=\Emc_{\mathbf P_+}^{\mathbf i_+}=\Emc_{\mathbf P_-}^{\mathbf i_-}\,\,\text{ and }\,\,\Hmc_{\mathbf P}^{\mathbf i}=\Hmc_{\mathbf P_+}^{\mathbf i_+}=\Hmc_{\mathbf P_-}^{\mathbf i_-}.\] 
for all ${\bf i}\in\Amc$. Thus, every $P\in\Pbbb$ has $\frac{(n-1)(n-2)}{2}$ eruption labellings and $\frac{(n-1)(n-2)}{2}$ hexagon labellings associated to it.

\begin{figure}[ht]
\centering
\includegraphics[scale=0.5]{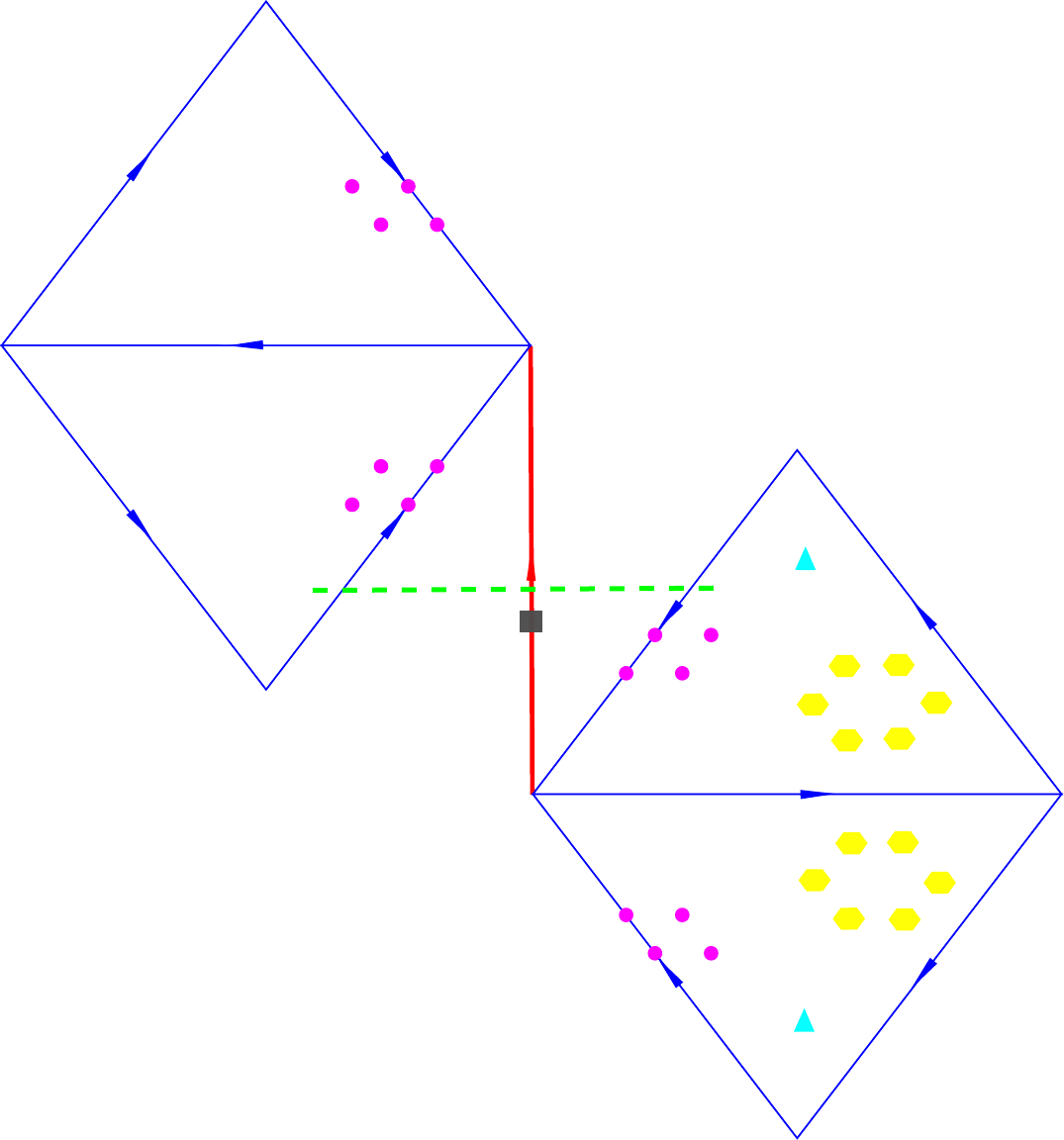}
\footnotesize
\put (-180, 84){$x_1(1)=y_1(1)$}
\put (-72, 170){$x_2(1)$}
\put (0, 84){$x_3(1)=y_2(1)$}
\put (-69, -5){$y_3(1)$}
\put (-128, 192){$x_1(2)=y_1(2)$}
\put (-200, 105){$x_2(2)$}
\put (-310, 192){$x_3(2)=y_2(2)$}
\put (-202, 280){$y_3(2)$}
\put (-66, 146){$\frac{1}{2}$}
\put (-70, 20){$-\frac{1}{2}$}
\put (-55, 119){$1$}
\put (-46, 119){$-1$}
\put (-55, 88){$1$}
\put (-46, 88){$-1$}
\put (-31, 103){$1$}
\put (-72, 103){$-1$}
\put (-55, 76){$1$}
\put (-46, 76){$-1$}
\put (-55, 46){$1$}
\put (-46, 46){$-1$}
\put (-31, 61){$1$}
\put (-72, 61){$-1$}
\put (-146, 121){$-\frac{1}{2}$}
\put (-105, 122){$1$}
\put (-84, 122){$-1$}
\put (-119, 112){$-1$}
\put (-92, 112){$1$}
\put (-105, 44){$1$}
\put (-84, 44){$-1$}
\put (-119, 53){$-1$}
\put (-92, 53){$1$}
\put (-172, 163){$1$}
\put (-151, 163){$-1$}
\put (-185, 153){$-1$}
\put (-158, 153){$1$}
\put (-172, 221){$1$}
\put (-151, 221){$-1$}
\put (-185, 231){$-1$}
\put (-158, 231){$1$}
\caption{A non-isolated edge of $\widetilde{\Tmc}$ is drawn in red, a bridge in $\widetilde{\Jmc}$ across the non-isolated edge is drawn in green, and isolated edges in $\widetilde{\Tmc}$ are draw in blue. Each dot in an ideal triangle represents a triple $\mathbf i\in\Bmc$, and each dot along an edge represents a pair $\mathbf k\in\Amc$. In the picture above, we have drawn a diagramatic representation for the eruption, hexagon, twist and lozenge labellings in turquoise (triangles), yellow (hexagons), grey (squares), and purple (discs) respectively.}\label{fig:specialadmissiblelabellings}
\end{figure}

Next, we describe the vectors $\mu$ associated to the remaining two types of special $(\Tmc,\Jmc)$-parallel vector fields, called the \emph{twist vector fields} and \emph{length vector fields}.  

\begin{definition} \label{def:twist}
For any oriented, non-isolated edge $\widehat{\bf c}$ in $\Pmc^o$ and any $\mathbf k\in\Amc$, let $\mu\in \mathscr{W}$ be the vector defined by
\begin{itemize}
\item $\mu^{\mathbf l}_{\widehat{\mathbf f}}=\left\{\begin{array}{ll}
-\frac{1}{2}&\text{if }(\mathbf l,\widehat{\mathbf f})\in[\mathbf k,\widehat{\bf c}];\\
 0&\text{otherwise}
 \end{array} \right.$\\
   for all $(\mathbf l,\widehat{\bf f})\in\Amc\times\Pmc^o$,
 \item $\mu^{\mathbf j}_{\widehat{\mathbf d}}=0$ for all $(\mathbf j,\widehat{\mathbf d})\in\overline{\Bmc}\times\Mmc$.
 \end{itemize}
The $\mathbf k$-{\em twist vector field} associated to $\widehat{\mathbf c}$, denoted $\Smc_{\widehat{\mathbf c}}^\mathbf k$, is the $(\Tmc,\Jmc)$-parallel vector field associated to $\mu$, i.e. ${\rm d}\Omega_{[\rho]}\left(\Smc_{\widehat{\mathbf c}}^\mathbf k([\rho])\right)=\mu$ for all $[\rho]\in\Hit_V(S)$, see Figure~\ref{fig:specialadmissiblelabellings}. The $\mathbf i$-{\em twist cocycle} associated to $\mathbf P$ is the $(\rho,\Tmc,\Jmc)$-tangent cocycle $\sigma_{\widehat{\mathbf c}}^\mathbf k:=\Psi\left(\Smc_{\widehat{\mathbf c}}^\mathbf k([\rho])\right)$.
\end{definition}

Recall that for any $\mathbf k=(k_1,k_2)\in\Amc$, we denote $\mathbf k_1:=(0,k_1,k_2)$, $\mathbf k_2:=(k_2,0,k_1)$, $\mathbf k_3:=(k_1,k_2,0)$, $\bar{\mathbf k}_1:=(0,k_2,k_1)$, $\bar{\mathbf k}_2:=(k_1,0,k_2)$, and $\bar{\mathbf k}_3:=(k_2,k_1,0)$. 

\begin{definition} \label{def:length}
For any oriented, non-isolated edge $\widehat{\bf c}$ in $\Pmc^o$ and any $\mathbf k\in\Amc$, let $P(1)$ and $P(2)$ be the pairs of pants in $\Pbbb$ that share $\widehat{\bf c}$ as a common boundary component (it is possible that $P(1)=P(2)$), such that $P(1)$ and $P(2)$ lie to the right and left of $\widehat{\bf c}$ respectively. Also, for both $m=1,2$, let ${\bf P}(m):=(P(m),\widehat{c})$, and let $(\widehat{\mathbf b}(m),\widehat{\mathbf b}'(m))$ be the pair of non-edge barriers associated to ${\bf P}(m)$. Then define $\upsilon,\upsilon'\in \mathscr{W}$ by
\begin{itemize}
\item $\upsilon^{\mathbf l}_{\widehat{\mathbf f}}={\upsilon'}^{\mathbf l}_{\widehat{\mathbf f}}=0$ for all $(\mathbf l,\widehat{\bf f})\in\Amc\times\Pmc^o$,
 \item $\upsilon^{\mathbf j}_{\widehat{\mathbf d}}=\left\{\begin{array}{ll}
1&\text{if }(\widehat{\bf d},{\bf j})\in[\widehat{\bf b}(1),{\bf k}_3]\cup[\widehat{\bf b}(1),{\bf k}_3(0,-1,1)];\\
1&\text{if }(\widehat{\bf d},{\bf j})\in[\widehat{\bf b}(2),\bar{\bf k}_3]\cup[\widehat{\bf b}(2),\bar{\bf k}_3(0,-1,1)];\\
-1&\text{if }(\widehat{\bf d},{\bf j})\in[\widehat{\bf b}(1),{\bf k}_3(1,-1,0)]\cup[\widehat{\bf b}(1),{\bf k}_3(-1,0,1)];\\
-1&\text{if }(\widehat{\bf d},{\bf j})\in[\widehat{\bf b}(2),\bar{\bf k}_3(1,-1,0)]\cup[\widehat{\bf b}(2),\bar{\bf k}_3(-1,0,1)];\\
 0&\text{otherwise}
 \end{array} \right.$\\
for all $(\mathbf j,\widehat{\mathbf d})\in\overline{\Bmc}\times\Mmc$,
\item ${\upsilon'}^{\mathbf j}_{\widehat{\mathbf d}}=\left\{\begin{array}{ll}
1&\text{if }(\widehat{\bf d},{\bf j})\in[\widehat{\bf b}'(1),\bar{\bf k}_2]\cup[\widehat{\bf b}'(1),\bar{\bf k}_2(0,1,-1)];\\
1&\text{if }(\widehat{\bf d},{\bf j})\in[\widehat{\bf b}'(2),{\bf k}_2]\cup[\widehat{\bf b}'(2),{\bf k}_2(0,1,-1)];\\
-1&\text{if }(\widehat{\bf d},{\bf j})\in[\widehat{\bf b}'(1),\bar{\bf k}_2(1,0,-1)]\cup[\widehat{\bf b}'(1),\bar{\bf k}_2(-1,1,0)];\\
-1&\text{if }(\widehat{\bf d},{\bf j})\in[\widehat{\bf b}'(2),{\bf k}_2(1,0,-1)]\cup[\widehat{\bf b}'(2),{\bf k}_2(-1,1,0)];\\
 0&\text{otherwise}
\end{array} \right.$\\
for all $(\mathbf j,\widehat{\mathbf d})\in\overline{\Bmc}\times\Mmc$,
 \end{itemize}
 and define $\mu:=\upsilon+\upsilon'$.
\begin{enumerate}
\item The $\mathbf k$-{\em lozenge vector field} associated to $\widehat{\mathbf c}$, denoted $\Zmc_{\widehat{\mathbf c}}^\mathbf k$, is the $(\Tmc,\Jmc)$-parallel vector field associated to $\mu$, i.e. ${\rm d}\Omega_{[\rho]}\left(\Zmc_{\widehat{\mathbf c}}^\mathbf k([\rho])\right)=\mu$ for all $[\rho]\in\Hit_V(S)$, see Figure~\ref{fig:specialadmissiblelabellings}. The $\mathbf k$-{\em lozenge cocycle} associated to $\mathbf P$ is the $(\rho,\Tmc,\Jmc)$-tangent cocycle $\zeta_{\widehat{\mathbf c}}^\mathbf k:=\Psi\left(\Zmc_{\widehat{\mathbf c}}^\mathbf k([\rho])\right)$.
\item The $\mathbf k$-{\em length vector field} associated to $\widehat{\mathbf c}$, denoted $\Ymc_{\widehat{\mathbf c}}^\mathbf k$, is the $(\Tmc,\Jmc)$-parallel vector field defined by
\[\Ymc_{\widehat{\mathbf c}}^\mathbf k:=\Zmc^\mathbf k_{\widehat{\mathbf c}}+2\Emc^{\mathbf k_3(0,-1,1)}_{\mathbf P(1)}-2\Emc^{\mathbf k_3(-1,0,1)}_{\mathbf P(1)}+2\Emc^{ \bar{\mathbf k}_3(0,-1,1)}_{\mathbf P(2)}-2\Emc^{\bar{\mathbf k}_3(-1,0,1)}_{\mathbf P(2)},\]
where the vector field $\Emc^{\mathbf k_3(0,-1,1)}_{\mathbf P(1)}$ (resp. $\Emc^{\mathbf k_3(-1,0,1)}_{\mathbf P(1)}$, $\Emc^{ \bar{\mathbf k}_3(0,-1,1)}_{\mathbf P(2)}$, and $\Emc^{\bar{\mathbf k}_3(-1,0,1)}_{\mathbf P(2)}$) is defined to be the zero vector field if $\mathbf k_3(0,-1,1)$ (resp. $\mathbf k_3(-1,0,1)$, $\bar{\mathbf k}_3(0,-1,1)$, $\bar{\mathbf k}_3(-1,0,1)$) does not lie in $\Bmc$.
The $\mathbf k$-{\em length cocycle} associated to $\widehat{\mathbf c}$ is the $(\rho,\Tmc,\Jmc)$-tangent cocycle $\gamma_{\widehat{\mathbf c}}^\mathbf k:=\Psi\left(\Ymc_{\widehat{\mathbf c}}^\mathbf k([\rho])\right)$. 
\end{enumerate}
\end{definition}

\begin{remark}\label{rem: correction}
In \cite{SunWienhardZhang}, the $\mathbf k$-length vector field associated to $\widehat{\mathbf c}$ was defined as
\[\Ymc_{\widehat{\mathbf c}}^\mathbf k:=\Zmc^\mathbf k_{\widehat{\mathbf c}}+\Emc^{\mathbf k_3(0,-1,1)}_{\mathbf P(1)}-\Emc^{\mathbf k_3(-1,0,1)}_{\mathbf P(1)}+\Emc^{ \bar{\mathbf k}_3(0,-1,1)}_{\mathbf P(2)}-\Emc^{\bar{\mathbf k}_3(-1,0,1)}_{\mathbf P(2)}.\]
Notice that this is not the same as the definition given in Definition \ref{def:length}. The corrected definition (given in Definition \ref{def:length}) is required for Theorem \ref{thm:symplectic basis} to hold.
\end{remark}

%We refer to 
%\[\left(\widehat{\mathbf b}(1),\widehat{\mathbf b}'(1),\widehat{\mathbf b}(2),\widehat{\mathbf b}'(2)\right)\] 
%as the \emph{quadruple of non-edge barriers associated to $\widehat{\bf c}$}. 

Observe that for any oriented, non-isolated edge $\widehat{\mathbf c}\in\Tmc^o$ and any $\mathbf k\in\Amc$, we have
\[\Smc_{\widehat{\mathbf c}}^{\mathbf k}=\Smc_{\widehat{\bar{\mathbf c}}}^{\bar{\mathbf k}},\quad \Zmc_{\widehat{\mathbf c}}^{\mathbf k}=\Zmc_{\widehat{\bar{\mathbf c}}}^{\bar{\mathbf k}},\,\,\text{ and }\,\,\Ymc_{\widehat{\mathbf c}}^{\mathbf k}=\Ymc_{\widehat{\bar{\mathbf c}}}^{\bar{\mathbf k}}.\] 
Thus, every $\widehat{c}\in\Pmc$ has $n-1$ twist, and length vector fields associated to it.

%By Lemma \ref{lem:explicit}, the $\mathbf k$-lozenge labelling associated to $\mathbf c$ can be written as
%\begin{eqnarray}\label{eqn:lozenge}
%Z_{\mathbf c}^\mathbf k&=&L_{\mathbf x(1)}^{\mathbf k_1^3} + L_{\mathbf x(1)}^{\mathbf k_1^3(0,-1,1)} -L_{\mathbf x(1)}^{\mathbf k_1^3(1,-1,0)} - L_{\mathbf x(1)}^{\mathbf k_1^3(-1,0,1)}\\\nonumber
%&&L_{\mathbf y(1)}^{\mathbf k_1^2} + L_{\mathbf y(1)}^{\mathbf k_1^2(0,1,-1)} -L_{\mathbf y(1)}^{\mathbf k_1^2(1,0,-1)} - L_{\mathbf y(1)}^{\mathbf k_1^2(-1,1,0)}\\\nonumber
%&& L_{\mathbf x(2)}^{\mathbf k_2^3} + L_{\mathbf x(2)}^{\mathbf k_2^3(0,-1,1)} -L_{\mathbf x(2)}^{\mathbf k_2^3(1,-1,0)}- L_{\mathbf x(2)}^{\mathbf k_2^3(-1,0,1)}\\\nonumber
%&& L_{\mathbf y(2)}^{\mathbf k_2^2} + L_{\mathbf y(2)}^{\mathbf k_2^2(0,1,-1)} -L_{\mathbf y(2)}^{\mathbf k_2^2(1,0,-1)} - L_{\mathbf y(2)}^{\mathbf k_2^2(-1,1,0)},\nonumber
%\end{eqnarray}
%see Figure \ref{fig:specialadmissiblelabellings}, and the $\mathbf k$-length labelling associated to $\mathbf c$ can be written as
%\[Y_\mathbf c^\mathbf k:=Z^\mathbf k_\mathbf c+E^{\mathbf k^3_1(0,-1,1)}_{\mathbf P(1)}-E^{\mathbf k^3_1(-1,0,1)}_{\mathbf P(1)}+E^{\mathbf k^3_2(0,-1,1)}_{\mathbf P(2)}-E^{\mathbf k^3_2(-1,0,1)}_{\mathbf P(2)}.\]

\begin{definition} \label{def:special}
Any eruption, hexagon, twist, or length vector field is a \emph{special} $(\Tmc,\Jmc)$-parallel vector field. Also, any eruption, hexagon, twist, or length cocycle is a \emph{special} $(\rho,\Tmc,\Jmc)$-tangent cocycle.
\end{definition}

There are $(n^2-1)(2g-2)$ special admissible labellings in total. Our next and final goal is to prove the following theorem. 

\begin{thm}\label{thm:symplectic basis}
Let $\Xmc$ and $\Xmc'$ be any pair of special $(\Tmc,\Jmc)$-parallel vector fields and let $[\rho]\in\Hit_V(S)$ be any point. Let $\widehat{\bf c}\in\Pmc^o$, and let $\mathbf P$ be a marked pair of pants given by $\Pmc$.
\begin{itemize}
\item If $\Xmc'=\Smc^\mathbf k_{\widehat{\mathbf c}}$, then 
\[\omega_{[\rho]}\left(\Xmc'([\rho]),\Xmc([\rho])\right)=\left\{\begin{array}{ll}
1&\text{if }\Xmc=\Ymc^\mathbf k_{\widehat{\mathbf c}};\\
0&\text{otherwise.}
\end{array}\right.\] 
\item If $\Xmc'=\Ymc^{\mathbf k}_{\widehat{\mathbf c}}$, then 
\[\omega_{[\rho]}\left(\Xmc'([\rho]),\Xmc([\rho])\right)=\left\{\begin{array}{ll}
-1&\text{if }\Xmc=\Smc^\mathbf k_{\widehat{\mathbf c}};\\
0&\text{otherwise.}
\end{array}\right.\] 
\item If $\Xmc'=\Emc^{\mathbf i}_{\mathbf P}$, then 
\[\omega_{[\rho]}\left(\Xmc'([\rho]),\Xmc([\rho])\right)=\left\{\begin{array}{ll}
1&\text{if }\Xmc=\Hmc^{\mathbf i}_{\mathbf P};\\
0&\text{otherwise.}
\end{array}\right.\] 
\item If $\Xmc'=\Hmc^{\mathbf i}_{\mathbf P}$, then 
\[\omega_{[\rho]}\left(\Xmc'([\rho]),\Xmc([\rho])\right)=\left\{\begin{array}{ll}
-1&\text{if }\Xmc=\Emc^{\mathbf i}_{\mathbf P};\\
0&\text{otherwise.}
\end{array}\right.\] 
\end{itemize}
\end{thm}

Theorem \ref{thm:symplectic basis} and Theorem \ref{thm:constant} together imply the following corollary.

\begin{cor}\label{cor:symplectic}
The set of special $(\Tmc,\Jmc)$-parallel vector fields, when evaluated at any $[\rho]\in\Hit_V(S)$, give a symplectic basis of $T_{[\rho]}\Hit_V(S)$.
\end{cor}

For any special $(\Tmc,\Jmc)$-parallel vector field $\Xmc$, let $\Xmc^*$ be the special $(\Tmc,\Jmc)$-parallel vector field that is dual to $\Xmc$, i.e. $\omega(\Xmc^*,\Xmc)=1$. By Theorem \ref{thm:symplectic basis}, $\Xmc^*$ can be described explicitly as
\[\Xmc^*=\left\{\begin{array}{ll}
\Smc^\mathbf k_{\widehat{\mathbf c}}&\text{if }\Xmc=\Ymc^\mathbf k_{\widehat{\mathbf c}};\\
-\Ymc^\mathbf k_{\widehat{\mathbf c}}&\text{if }\Xmc=\Smc^\mathbf k_{\widehat{\mathbf c}};\\
\Emc^\mathbf i_\mathbf P&\text{if }\Xmc=\Hmc^\mathbf i_\mathbf P;\\
-\Hmc^\mathbf i_\mathbf P&\text{if }\Xmc=\Emc^\mathbf i_\mathbf P;\\
\end{array}\right.\]

In the companion paper \cite[Theorem 8.18 and Theorem 8.22]{SunWienhardZhang}, the authors computed, for every special $(\Tmc,\Jmc)$-parallel vector field $\Xmc$, a real-analytic function $H(\Xmc):\Hit_V(S)\to\Rbbb$ with the defining property that the derivative of $H(\Xmc)$ is $1$ in the direction of $\Xmc^*$, and $0$ in the direction of any other special $(\Tmc,\Jmc)$-parallel vector field. These functions were written in terms of triple ratios and cross ratios of certain flags along the flag maps of Hitchin representations. As a consequence of Corollary \ref{cor:symplectic}, we have the following corollary.

\begin{cor}
For any special $(\Tmc,\Jmc)$-parallel vector field $\Xmc$, the function $H(\Xmc)$ is the Hamiltonian function of $\Xmc$. Furthermore, the collection of functions
\[\{H(\Xmc):\Xmc\text{ is a special }(\Tmc,\Jmc)\text{-parallel vector field}\}\]
is a global Darboux coordinate system on $\Hit_V(S)$, i.e.
\[\omega=\sum_{\mathbf i\in\Bmc}\sum_{\mathbf P}{\rm d}(H(\Emc^\mathbf i_\mathbf P))\wedge {\rm d}(H(\Hmc^\mathbf i_\mathbf P))+\sum_{\mathbf k\in\Amc}\sum_{\mathbf c}{\rm d}(H(\Smc^\mathbf k_\mathbf P))\wedge {\rm d}(H(\Ymc^\mathbf k_\mathbf P)).\]
\end{cor}

\subsection{Internal and external parts}\label{sec:Tbbb}
%%%%%%%%%%%%%%%%%%%%%%%%%%%%%%%%%%%%%%%%%%%%%%%%%%

Since $\omega_{[\rho]}$ is skew-symmetric, to prove Theorem \ref{thm:symplectic basis}, it is sufficient to compute $\omega_{[\rho]}\left(\Xmc'([\rho]),\Xmc([\rho])\right)$ in the following cases:
\begin{enumerate}
\item \label{case:s} $\Xmc'$ is a twist vector field and $\Xmc$ is any special $(\Tmc,\Jmc)$-parallel vector field;
\item \label{case:e} $\Xmc'$ is an eruption cocycle, $\Xmc$ is an eruption, hexagon or length vector field;
\item \label{case:h} $\Xmc'$ is a hexagon cocycle, $\Xmc$ is a hexagon or length vector field;
\item \label{case:z} $\Xmc'$ and $\Xmc$ are both length vector fields.
\end{enumerate}
%We say that a pair $(\Xmc',\Xmc)$ of special $(\Tmc,\Jmc)$-parallel vector fields is a \emph{considered pair} if it is one of the pairs described by the cases (1) to (4) above. The number (1) to (4) above is the \emph{type} of the considered pair.

The computation of $\omega_{[\rho]}\big(\Xmc',\Xmc\big)$ for the above four cases is elementary but long. The goal of the rest of this paper is to perform this computation. As such, we fix once and for all the following data:
\begin{itemize}
\item a Hitchin representation $\rho:\Gamma\to\PGL(V)$ that represents $[\rho]$,
\item a triangulation $\Tbbb$ of $\Sigma$ that is well-adapted to $(\Tmc,\Jmc)$, see Section \ref{sec:triangulation}.
\end{itemize}
We will also choose an enumeration of the vertices of $\Tbbb$ and choose a representative in $\widetilde{\Tbbb}$ for each triangle in $\Tbbb$ as follows.

Recall that $\Tbbb$ consists of two types of triangles, the internal triangles and the external triangles (see Definition \ref{def: internal and external}). Each pair of pants $P$ of the pants decomposition $\Pmc$ contains two internal triangles and twelve external triangles. Let $\widehat{\delta}_P$ be the internal triangle in $P$ such that the ideal triangle of $\Tmc$ that contains $\widehat{\delta}_P$ contains the endpoints of all the bridges that intersect $P$, and let $\widehat{\delta}_P'$ be the other internal triangle in $P$. Enumerate the vertices of $\Tbbb$ by $\{v_1,\dots,v_{12g-12}\}$ such that the following holds:
\begin{itemize}
\item If $v_a$, $v_b$, and $v_c$ are the vertices of $\widehat{\delta}_P$ with $a<b<c$, then the edges from $v_a$ to $v_b$ and  from $v_b$ to $v_c$ are in the anti-clockwise direction about $\widehat{\delta}_P$.
\item If $v_a$ is a vertex in the interior of a pair of pants in $\Pbbb$ and $v_b$ is a vertex in a non-isolated edge in $\Pmc$, then $a<b$.
\item If $v_a$ and $v_b$ are the two vertices on the same non-isolated edge in $\Pmc$ and $v_a$ is the vertex that lies in a bridge parallel edge, then $a<b$.
\end{itemize}
%This enumeration induces an orientation on the edges of $\Tbbb$ as described in Section \ref{sec:triangulation}. 

Recall that $\pi_S:\Std\to S$ denotes the covering map. Choose triangles $\delta_P$ and $\delta_P'$ of $\widetilde{\Tbbb}$ such that $\pi_S(\delta_P)=\widehat{\delta}_P$ and $\pi_S(\delta_P')=\widetilde{\delta}_P'$, so that $\delta_P$ and $\delta_P'$ share a common vertex $v$ with the property that if $v_a$, $v_b$, and $v_c$ are the vertices of $\widehat{\delta}_P$ with $a<b<c$, then the endpoints of $\pi_S(v)=v_c$. The union of these twelve external triangles in $P$ are the three boundary cylinders of $P$, each of which is a union of four external triangles. Choose a lift to $\Std$ of each of these external triangles, such that 
\begin{itemize}
\item the lifts of the four external triangles in the same boundary cylinder of $P$ lie in the same fundamental block, see Section \ref{sec:triangulation}. 
\item for any oriented, non-isolated edge $\widehat{\bf c}\in\Pmc^o$, there is a lift ${\bf c}$ of $\widehat{\bf c}$ such that if $D(1)$ is the fundamental block that contains the lifts of the four external triangles to the right of $\widehat{\bf c}$ and $D(2)$ is the fundamental block that contains the lifts of the four external triangles to the left of $\widehat{\bf c}$, then $\partial D(1)\cap{\bf c}$ and $\partial D(2)\cap{\bf c}$ are non-empty and coincide.
\end{itemize}
We have thus chosen a representative in $\widetilde{\Tbbb}$ for each triangle in $\Tbbb$. Let $\Tbbb_{\rm lift}$ be the collection of these choices, and let $\Tbbb^{\mathrm ext}$ and $\Tbbb^{\mathrm int}$ denote the set of external and internal triangles in $\Tbbb_{\rm lift}$ respectively. 

For each triangle $\delta\in\Tbbb_{\rm lift}$, the enumeration we chose on the vertices of $\Tbbb$ picks out a pair of edges $e_{\delta,1}$ and $e_{\delta,2}$ of $\delta$ that are equipped with an induced orientation described in Section \ref{sec:triangulation}. Then by (\ref{eqn:formula}) we may decompose
\begin{align*}
\omega_{[\rho]}\left(\Xmc'([\rho]),\Xmc([\rho])\right)=\sum_{\delta\in\mathbb{T}^{\mathrm{ext}}}\sgn(\delta)&\tr\big(\widetilde{\nu}'(e_{\delta,1})\cdot\widetilde{\nu}(e_{\delta,2})\big)\\
&+\sum_{\delta\in\mathbb{T}^{\mathrm{int}}}\sgn(\delta)\tr\big(\widetilde{\nu}'(e_{\delta,1})\cdot\widetilde{\nu}(e_{\delta,2})\big),
\end{align*}
where $\widetilde{\nu}'$ and $\widetilde{\nu}$ are the $\Ad\circ\rho$-equivariant lift of the $(\rho,\Tmc,\Jmc)$-tangent cocycle $\nu':=\Psi(\Xmc'([\rho]))$ and $\nu:=\Psi(\Xmc([\rho]))$ respectively.
We refer to the first and second sum on the right hand side above as the \emph{internal} and \emph{external} parts of $\omega_{[\rho]}\left(\Xmc'([\rho]),\Xmc([\rho])\right)$ respectively. We will compute these two parts separately in the following two sections.

\begin{figure}[ht]
\centering
\includegraphics[scale=0.9]{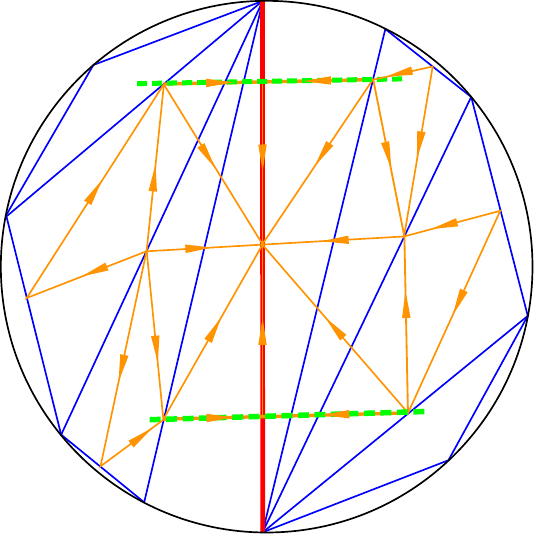}
\small
\caption{A non-isolated edge in $\widetilde{\Tmc}$ is drawn in thick red, the isolated edges in $\widetilde{\Tmc}$ are drawn in blue, and the bridges in $\widetilde{\Jmc}$ are drawn in dotted green. The edges of the $(\Tmc,\Jmc)$-triangulation $\Tbbb$ are drawn in orange. The bridge parallel edges of $\Tbbb$ are the thicker orange lines.}\label{fig:cylinder1}
\end{figure}

%%%%%%%%%%%%%%%%%%%%%%%%%%%%%%%%%%%%%%%%%%%%%%%%%%
\subsection{The internal part of $\omega_{[\rho]}\left(\Xmc'([\rho]),\Xmc([\rho])\right)$}\label{sec:computational tools}
%%%%%%%%%%%%%%%%%%%%%%%%%%%%%%%%%%%%%%%%%%%%%%%%%%
The following notion is useful to calculate internal part of $\omega_{[\rho]}\left(\Xmc'([\rho]),\Xmc([\rho])\right)$.

\begin{definition}
Let $P\in\Pbbb$ be a pair of pants. Choose a boundary curve $c$ of $P$, let ${\bf P}:=(P,c)$, and let $(\widehat{\mathbf b},\widehat{\mathbf b}')$ be the pair of non-edge barriers in $\Dmc$ associated to $\mathbf P$. A vector $\mu\in\mathscr{W}$ is {\em symmetric (resp. skew-symmetric)} in $P$ if $\mu^{\mathbf i}_{\widehat{\mathbf b}}=\mu^{\mathbf i'}_{\widehat{\mathbf b}'}$ (resp. $\mu^{\mathbf i}_{\widehat{\mathbf b}}=-\mu^{\mathbf i'}_{\widehat{\mathbf b}'}$) for all $\mathbf i\in\overline{\Bmc}$. A $(\Tmc,\Jmc)$-parallel vector field $\Xmc$ on $\Hit_V(S)$ is {\em symmetric (resp. skew-symmetric)} in $P$ if the vector $\mu:={\rm d}\Omega_{[\rho]}(\Xmc([\rho]))$ for some (any) $[\rho]\in\Hit_V(S)$ is symmetric (resp. skew-symmetric) in $P$. %$(\rho,\Sigma,\Tmc,\Jmc)$-tangent cocycle $\nu$ is {\em symmetric (resp. skew-symmetric)} in $P$ if the $\Tmc$-admissible labelling $\Xi(\nu)$ is symmetric (resp. skew-symmetric).
\end{definition}
%%draw picture

By the symmetries of $\mu$, note that the notion of being symmetric or skew-symmetic in $P$ does not depend on the choice of $\mathbf P$. Observe that the eruption and twist vector fields are skew-symmetric in every pair of pants in $\Pbbb$, and the twist, hexagon and lozenge vector fields are symmetric in every pair of pants in $\Pbbb$.

The next lemma gives an easily verified condition that is useful for calculating the internal part of $\omega_{[\rho]}\left(\Xmc'([\rho]),\Xmc([\rho])\right)$ in many cases. 

\begin{lem}\label{lem:internal}
Let $P$ be a pair of pants in $\Pbbb$, let $\Xmc'$ and $\Xmc$ be $(\Tmc,\Jmc)$-parallel vector fields, and let $\widetilde{\nu}'$ and $\widetilde{\nu}$ be the $\Ad\circ\rho$-equivariant lift of the $(\rho,\Tmc,\Jmc)$-tangent cocycles $\nu':=\Psi(\Xmc'([\rho]))$ and $\nu:=\Psi(\Xmc([\rho]))$ respectively. 
\begin{enumerate}
\item If $\Xmc'$ and $\Xmc$ are either both symmetric or both skew-symmetric in $P$, then
\[\tr\big(\widetilde{\nu}'(e_{\delta_P,1})\cdot\widetilde{\nu}(e_{\delta_P,2})\big)=\tr\big(\widetilde{\nu}'(e_{\delta_P',1})\cdot\widetilde{\nu}(e_{\delta_P',2})\big).\]
In particular, if $\Xmc'$ and $\Xmc$ are either both symmetric or both skew-symmetric in every pair of pants in $\Pbbb$, then the internal part of $\omega_{[\rho]}\left(\Xmc'([\rho]),\Xmc([\rho])\right)$ is zero.
\item If one of $\Xmc'$ or $\Xmc$ is symmetric in $P$, while the other is skew-symmetric in $P$, then
\[\tr\big(\widetilde{\nu}'(e_{\delta_P,1})\cdot\widetilde{\nu}(e_{\delta_P,2})\big)=-\tr\big(\widetilde{\nu}'(e_{\delta_P',1})\cdot\widetilde{\nu}(e_{\delta_P',2})\big)\]
\end{enumerate}
\end{lem} 

\begin{proof}
Let $c$ be the boundary component of $P$ such that if $(\widehat{\bf b},\widehat{\bf b}')$ is the pair of non-edge barriers associated to ${\bf P}:=(P,c)$, then there are (unique) lifts ${\bf b}$ and ${\bf b}'$ in $\widetilde{\Mmc}$ of $\widehat{\bf b}$ and $\widehat{\bf b}'$ respectively, such that ${\bf b}$ intersects the interior of $e_{\delta_P,1}$ and ${\bf b}'$ intersects the interior of $e_{\delta_P,2}$. Let ${\bf x}$ be the triple associated to ${\bf b}$ and let ${\bf x}'$ be the triple associated to ${\bf b}'$.
By \eqref{eqn: simplify not},
\begin{eqnarray*}
\widetilde{\nu}'(e_{\delta_P,1})&=&-\frac{1}{2}\sum_{\mathbf i\in\overline{\Bmc}\setminus\Bmc}{\mu'}^\mathbf i_{\widehat{\mathbf b}}A^{\mathbf i}_{\xi_\rho(\mathbf x)}-\sum_{\mathbf i\in\Bmc}{\mu'}^\mathbf i_{\widehat{\mathbf b}}A^{\mathbf i}_{\xi_\rho(\mathbf x)},\\
\widetilde{\nu}(e_{\delta_P,2})&=&-\frac{1}{2}\sum_{\mathbf i\in\overline{\Bmc}\setminus\Bmc}\mu^\mathbf i_{\widehat{\mathbf b}_+}A^{\mathbf i}_{\xi_\rho(\mathbf x_+)}-\sum_{\mathbf i\in\Bmc}\mu^\mathbf i_{\widehat{\mathbf b}_+}A^{\mathbf i}_{\xi_\rho(\mathbf x_+)},\\
\widetilde{\nu}'(e_{\delta_P',1})&=&\frac{1}{2}\sum_{\mathbf i\in\overline{\Bmc}\setminus\Bmc}{\mu'}^\mathbf i_{\widehat{\mathbf b}'}A^{\mathbf i}_{\xi_\rho(\mathbf x')}+\sum_{\mathbf i\in\Bmc}{\mu'}^\mathbf i_{\widehat{\mathbf b}'}A^{\mathbf i}_{\xi_\rho(\mathbf x')},\\
\widetilde{\nu}(e_{\delta_P',2})&=&\frac{1}{2}\sum_{\mathbf i\in\overline{\Bmc}\setminus\Bmc}\mu^\mathbf i_{\widehat{\mathbf b}'_-}A^{\mathbf i}_{\xi_\rho(\mathbf x'_-)}+\sum_{\mathbf i\in\Bmc}\mu^\mathbf i_{\widehat{\mathbf b}'_-}A^{\mathbf i}_{\xi_\rho(\mathbf x'_-)}.\\
\end{eqnarray*}

Lemma \ref{lem:constant0}(1) implies that for any $\mathbf i,\mathbf j\in\overline{\Bmc}$, 
\[\tr\left(A^{\mathbf i}_{\xi_\rho(\mathbf x)}\cdot A^{\mathbf j}_{\xi_\rho(\mathbf x_+)}\right)=[\min\{i_1-j_3,j_1-i_2\}]_+-\frac{i_1j_1}{n}=\tr\left(A^{\mathbf i'}_{\xi_\rho(\mathbf x')}\cdot A^{\mathbf j'}_{\xi_\rho(\mathbf x'_-)}\right).\]
By definition, if $\Xmc'$ and $\Xmc$ are either both symmetric in $P$ or skew-symmetric in $P$, then 
\[{\mu'}^{\mathbf i}_{\widehat{\mathbf b}}\,\mu^{\mathbf j}_{\widehat{\mathbf b}_+}={\mu'}^{\mathbf i}_{\widehat{\mathbf b}}\,\mu^{\mathbf j_-}_{\widehat{\mathbf b}}={\mu'}^{\mathbf i'}_{\widehat{\mathbf b}'}\,\mu^{(\mathbf j_-)'}_{\widehat{\mathbf b}'}={\mu'}^{\mathbf i'}_{\widehat{\mathbf b}'}\,\mu^{\mathbf j'}_{\widehat{\mathbf b}'_-}\]
for all $\mathbf i,\mathbf j\in\overline{\Bmc}$, so (1) holds. On the other hand, when one of $\Xmc'$ or $\Xmc$ is symmetric in $P$ while the other is skew-symmetric in $P$, then
\[{\mu'}^{\mathbf i}_{\widehat{\mathbf b}}\,\mu^{\mathbf j}_{\widehat{\mathbf b}_+}=-{\mu'}^{\mathbf i'}_{\widehat{\mathbf b}'}\,\mu^{\mathbf j'}_{\widehat{\mathbf b}'_-}\]
for all $\mathbf i,\mathbf j\in\overline{\Bmc}$, so (2) holds.
\end{proof}

\subsection{The external part of $\omega_{[\rho]}\left(\Xmc'([\rho]),\Xmc([\rho])\right)$}\label{sec:computational tools2}
%%%%%%%%%%%%%%%%%%%%%%%%%%%%%%%%%%%%%%%%%%%%%%%%%%

To calculate the external part of $\omega_{[\rho]}\left(\Xmc'([\rho]),\Xmc([\rho])\right)$, we consider each fundamental block separately.

For any fundamental block $D$, recall that $\Tbbb(D)$ denote the four triangles in $\widetilde{\Tbbb}$ that lie in $D$. By the way we enumerate the vertices of $\Tbbb$, note that all the crossing edges of $\Tbbb(D)$ are oriented so that their forward endpoints lie in a non-isolated edge of $\widetilde{\Tmc}$, and all the external edges of $\Tbbb(D)$ are oriented so that their backward endpoints lie in a bridge in $\widetilde{\Jmc}$. However, the orientation of the internal edges of $\Tbbb(D)$ depends on the position of the boundary cylinder $\pi(D)$ relative to the enumeration of the vertices of $\Tbbb$ in the pair of pants containing $\pi(D)$. This motivates the following definition.

\begin{definition}
Let $D$ be a fundamental block and let $P$ be the pair of pants  in $\Pbbb$ containing the boundary cylinder $\pi(D)$. We say $D$ is \emph{consistent} if the boundary of $\pi(D)$ contains edges of the form $\pi(e_{\delta_P,l})$ or $\pi(e_{\delta_P',l})$, where $l=1,2$. 
\end{definition}

If we have three fundamental blocks in $\widetilde\Sigma$ that project to the three different boundary cylinders of a pair of pants  in $\Pbbb$, then the enumeration of the vertices of $\Tbbb$ imply that of these three fundamental blocks, exactly two of them are consistent, and exactly one of them is inconsistent. When a fundamental block is consistent (resp. inconsistent), its internal edges are oriented so that their backward (resp. forward) endpoints lie in a bridge crossing edge. 

To write down formulas, we use the following notation.
\begin{notation}\label{not:block}
Let $D$ be a fundamental block. 
\begin{itemize} 
\item Let $u_1$ and $u_5$ be the bridge parallel edges in $D$, oriented so that their forward endpoints lie in a non-isolated edge, and $D$ lies to the left of $u_1$ and right of $u_5$. 
\item Let $\delta_1,\dots,\delta_4$ be the four triangles whose union is $D$, such that $\delta_1$ has $u_1$ as an edge, $\delta_4$ has $u_5$ as an edge, and $\delta_l$ and $\delta_{l+1}$ share a common crossing edge for all $l=1,2,3$. 
\item For $l=1,2,3$, let $u_{l+1}$ denote the common edge shared by $\delta_l$ and $\delta_{l+1}$, oriented so that its forward endpoint lie in a non-isolated edge. 
\item For $l=1,\dots,4$, let $v_l$ be the unique edge of $\delta_l$ that is not a crossing edge, oriented so that their backward endpoints lie in the bridge crossing edge.
\end{itemize}
\end{notation} 

\begin{figure}[ht]
\centering
\includegraphics[scale=0.8]{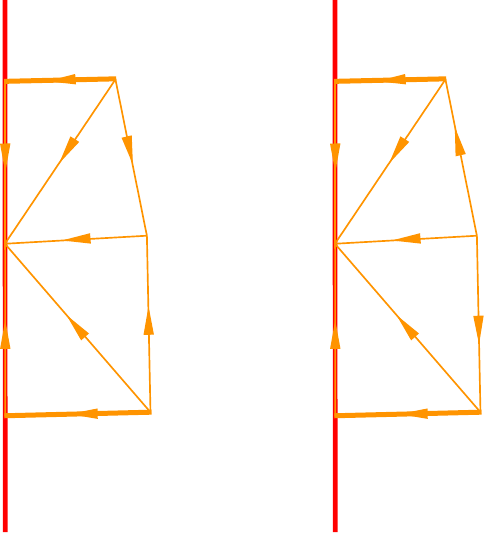}
\put (-173, 155){$\delta_1$}
%\put (-175, 142){$u_2$}
\put (-150, 134){$\delta_2$}
%\put (-153, 116){$u_3$}
\put (-160, 95){$\delta_3$}
%\put (-152, 75){$u_4$}
\put (-173, 65){$\delta_4$}
\put (-46, 155){$\delta_1$}
\put (-33, 127){$\delta_2$}
\put (-33, 95){$\delta_3$}
\put (-46, 65){$\delta_4$}
\put (-165, 182){$u_1$}
\put (-197, 145){$v_1$}
\put (-134, 145){$v_2$}
\put (-127, 75){$v_3$}
\put (-197, 75){$v_4$}
\put (-180, 36){$\gamma^{-1}\cdot u_1=u_5$}
\put (-38, 182){$u_1$}
\put (-70, 145){$v_1$}
\put (-7, 145){$-v_2$}
\put (0, 75){$-v_3$}
\put (-70, 75){$v_4$}
\put (-53, 36){$\gamma^{-1}\cdot u_1=u_5$}
\caption{On the left and right are consistent and inconsistent fundamental blocks respectively.}\label{fig:cylinder2}
\end{figure}

In this notation $u_1,\dots,u_5, v_1$, and $v_4$ are oriented according to the enumeration of the vertices of $\Tbbb$. However, $v_2$ and $v_3$ are oriented according to the enumeration of the vertices of $\Tbbb$ if and only if $D$ is consistent. From this, we see that 
\[e_{\delta_1,1}=u_1,\,\,\, e_{\delta_1,2}=v_1,\,\,\,  e_{\delta_4,1}=u_5,\,\,\,\text{ and }\,\,\,e_{\delta_4,2}=v_4\] 
for all fundamental blocks $D$. On the other hand, 
\[e_{\delta_2,1}=v_2,\,\,\, e_{\delta_3,1}=v_3,\,\,\,\text{ and }\,\,\,e_{\delta_2,2}=e_{\delta_3,2}=u_3\] 
if $D$ is a consistent fundamental block, but
\[e_{\delta_2,1}=-v_2,\,\,\, e_{\delta_2,2}=u_2,\,\,\, e_{\delta_3,1}=-v_3,\,\,\,\text{ and }\,\,\,e_{\delta_3,2}=u_4\]
if $D$ is an inconsistent fundamental block, see Figure \ref{fig:cylinder2}.

Let $\Xmc'$ and $\Xmc$ be $(\Tmc,\Jmc)$-parallel vector fields, and let $\widetilde{\nu}'$ and $\widetilde{\nu}$ be the $\Ad\circ\rho$-equivariant lift of the $(\rho,\Tmc,\Jmc)$-tangent cocycles $\nu':=\Psi(\Xmc'([\rho]))$ and $\nu:=\Psi(\Xmc([\rho]))$ respectively. Let ${\bf c}$ be the oriented, non-isolated edge in $\widetilde{\Pmc}^o$ such that ${\bf c}$ intersects $\partial D$ and $D$ lies to the right of ${\bf c}$. Note that if $\mathbf r:=(r_1,r_2)$ is the pair associated to ${\bf c}$, and $\{f_1,\dots,f_n\}$ is a basis of $\Rbbb^n$ such that $f_i\in\xi^{(i)}(r_1)\cap\xi^{(n-i+1)}(r_2)$ for all $i\in\{1,\dots,n\}$, then for all $h\in\{u_1,\dots,u_4\}\cup\{v_1,\dots,v_d\}$, both $\widetilde{\nu}(h)$ and $\widetilde{\nu}'(h)$ are represented in this basis by an upper triangular matrix. It is now an elementary computation to check that if $D$ is consistent, then
\begin{eqnarray}\label{eqn:alphacylinder}
&&\sum_{\delta\in\Tbbb(D)}\sgn(\delta)\tr\big(\widetilde{\nu}'(e_{\delta,1})\cdot\widetilde{\nu}(e_{\delta,2})\big)\nonumber\\
&=&\sum_{\delta\in\Tbbb(D)}\sgn(\delta)\tr\big(\widetilde{\nu}'(e_{\delta,1})_{{\bf r},{\rm diag}}\cdot\widetilde{\nu}(e_{\delta,2})_{{\bf r},{\rm diag}}\big)\nonumber\\
%&=&\tr\big(\widetilde{\nu}'(u_1)\cdot\widetilde{\nu}(v_1)\big)-\tr\big(\widetilde{\nu}'(u_5)\cdot\widetilde{\nu}(v_4)\big)+\tr\big(\widetilde{\nu}'(v_3)\cdot\widetilde{\nu}(u_3)\big)-\tr\big(\widetilde{\nu}'(v_2)\cdot\widetilde{\nu}(u_3)\big)\nonumber\\
&=&\tr\big(\widetilde{\nu}'(u_1)_{{\bf r},{\rm diag}}\cdot\widetilde{\nu}(v_1-v_4)\big)+\tr\big(\widetilde{\nu}'(v_3-v_2)_{{\bf r},{\rm diag}}\cdot\widetilde{\nu}(u_3)_{{\bf r},{\rm diag}}\big)\\
&=&\tr\big(\widetilde{\nu}'(u_1)_{{\bf r},{\rm diag}}\cdot\widetilde{\nu}(v_2-v_3)_{{\bf r},{\rm diag}}\big)+\tr\big(\widetilde{\nu}'(v_3-v_2)_{{\bf r},{\rm diag}}\cdot\widetilde{\nu}(-v_2+u_1+v_1)_{{\bf r},{\rm diag}}\big),\nonumber
\end{eqnarray}
Similarly, if $D$ is inconsistent, then
\begin{eqnarray}\label{eqn:gammacylinder}
&&\sum_{\delta\in\Tbbb(D)}\sgn(\delta)\tr\big(\widetilde{\nu}^1(e_{\delta,1})\cdot\widetilde{\nu}^2(e_{\delta,2})\big)\nonumber\\
&=&\sum_{\delta\in\Tbbb(D)}\sgn(\delta)\tr\big(\widetilde{\nu}^1(e_{\delta,1})_{{\bf r},{\rm diag}}\cdot\widetilde{\nu}^2(e_{\delta,2})_{{\bf r},{\rm diag}}\big)\nonumber\\
%&=&\tr\big(\widetilde{\nu}'(u_1)\cdot\widetilde{\nu}(v_1)\big)-\tr\big(\widetilde{\nu}'(u_5)\cdot\widetilde{\nu}(v_4)\big)+\tr\big(\widetilde{\nu}'(v_3)\cdot\widetilde{\nu}(u_4)\big)-\tr\big(\widetilde{\nu}'(v_2)\cdot\widetilde{\nu}(u_2)\big)\nonumber\\
&=&\tr\big(\widetilde{\nu}'(u_1)_{{\bf r},{\rm diag}}\cdot\widetilde{\nu}(v_1-v_4)\big)+\tr\big(\widetilde{\nu}'(v_3-v_2)_{{\bf r},{\rm diag}}\cdot\widetilde{\nu}(u_3)_{{\bf r},{\rm diag}}\big)\\
&&+\tr\big(\widetilde{\nu}'(v_3)_{{\bf r},{\rm diag}}\cdot\widetilde{\nu}(v_3)_{{\bf r},{\rm diag}}\big)-\tr\big(\widetilde{\nu}'(v_2)_{{\bf r},{\rm diag}}\cdot\widetilde{\nu}(v_2)_{{\bf r},{\rm diag}}\big)\nonumber\\
&=&\tr\big(\widetilde{\nu}'(u_1)_{{\bf r},{\rm diag}}\cdot\widetilde{\nu}(v_2-v_3)_{{\bf r},{\rm diag}}\big)+\tr\big(\widetilde{\nu}'(v_3-v_2)_{{\bf r},{\rm diag}}\cdot\widetilde{\nu}(-v_2+u_1+v_1)_{{\bf r},{\rm diag}}\big)\nonumber\\
&&+\tr\big(\widetilde{\nu}'(v_3)_{{\bf r},{\rm diag}}\cdot\widetilde{\nu}(v_3)_{{\bf r},{\rm diag}}\big)-\tr\big(\widetilde{\nu}'(v_2)_{{\bf r},{\rm diag}}\cdot\widetilde{\nu}(v_2)_{{\bf r},{\rm diag}}\big).\nonumber
\end{eqnarray}
 
 As a consequence, we have the following two lemmas.
 
 \begin{lem}\label{lem:alphacylinder}
If $D$ is a consistent fundamental block, and one of the following holds:
\begin{enumerate}
\item $\widetilde{\nu}'(v_3-v_2)$ and $\widetilde{\nu}(v_3-v_2)$ are both nilpotent endomorphisms, or
\item $\widetilde{\nu}'(v_3-v_2)$ and $\widetilde{\nu}'(u_1)$ are both nilpotent endomorphisms,
\end{enumerate}
then
\[\sum_{\delta\in\Tbbb(D)}\sgn(\delta)\tr\big(\widetilde{\nu}'(e_{\delta,1})\cdot\widetilde{\nu}(e_{\delta,2})\big)=0.\]
\end{lem} 

\begin{proof}
This is obvious from (\ref{eqn:alphacylinder}). 
\end{proof}

\begin{lem}\label{lem:gammacylinder}
If $D$ is an inconsistent fundamental block, and one of the following holds:
\begin{enumerate}
\item $\widetilde{\nu}'(v_3-v_2)$ and $\widetilde{\nu}(v_3-v_2)$ are both nilpotent endomorphisms, or
\item $\tr\big(\widetilde{\nu}'(v_3)\cdot\widetilde{\nu}(v_3)\big)=\tr\big(\widetilde{\nu}'(v_2)\cdot\widetilde{\nu}(v_2)\big)$, and the endomorphisms $\widetilde{\nu}'(v_3-v_2)$ and $\widetilde{\nu}'(u_1)$ are both nilpotent.
\end{enumerate}
 Then
\[\sum_{\delta\in\Tbbb(D)}\sgn(\delta)\tr\big(\widetilde{\nu}^1(e_{\delta,1})\cdot\widetilde{\nu}^2(e_{\delta,2})\big)=0.\]
\end{lem} 

\begin{proof}
If both $\widetilde{\nu}'(v_3-v_2)$ and $\widetilde{\nu}(v_3-v_2)$ are nilpotent, then the diagonal parts of $\widetilde{\nu}'(v_3)$ and $\widetilde{\nu}'(v_2)$ (with respect to ${\bf r}$) agree, and the diagonal parts of $\widetilde{\nu}(v_3)$ and $\widetilde{\nu}(v_2)$ agree, so (1) implies that $\tr\big(\widetilde{\nu}'(v_3)_{{\bf r},{\rm diag}}\cdot\widetilde{\nu}(v_3)_{{\bf r},{\rm diag}}\big)=\tr\big(\widetilde{\nu}'(v_2)_{{\bf r},{\rm diag}}\cdot\widetilde{\nu}(v_2)_{{\bf r},{\rm diag}}\big)$ as well. The lemma then follows from (\ref{eqn:gammacylinder}). 
\end{proof}
 
In order to apply Lemmas \ref{lem:alphacylinder} and \ref{lem:gammacylinder} to the special $(\Tmc,\Jmc)$-parallel vector fields, we use the following proposition. We omit its proof since it is a straightforward verification using Lemma~\ref{lem: diagonal part}.

\begin{prop}\label{prop:nil}
Let $D$ be a fundamental block, let $c\in\widetilde{\Pmc}$ be the non-isolated edge that contains the external edges in $D$ and let $\widehat{c}:=\pi(c)$.
\begin{enumerate}
\item If $\nu$ is a twist cocycle, then $\widetilde{\nu}(v_2)=0=\widetilde{\nu}(v_3)$. Further, if $\nu$ is not associated to $\widehat c$, then $\widetilde{\nu}(u_1)=0$ as well.
\item If $\nu$ is an eruption cocycle, then $\widetilde{\nu}(v_2-v_3)$ and $\widetilde{\nu}(u_1)$ are nilpotent endomorphisms.
\item If $\nu$ is a hexagon cocycle, then $\widetilde{\nu}(v_2-v_3)$ is a nilpotent endomorphism. Moreover, if $\mathbf P:=(P,\widehat{c})$ is a marked pair of pants, and $\nu=\eta^\mathbf i_\mathbf P$ for some $\mathbf i:=(i_1,i_2,i_3)\in\Bmc$ with $i_3>1$, then $\widetilde{\nu}(u_1)$ is also a nilpotent endomorphism.
\item If $\nu$ is a lozenge cocycle not associated to $\widehat c$, then $\widetilde{\nu}(v_2)$, $\widetilde{\nu}(v_3)$, $\widetilde{\nu}(v_2-v_3)$, and $\widetilde{\nu}(u_1)$ are nilpotent endomorphisms.
\end{enumerate}
\end{prop}

%\begin{prop}
%Let $D$ be a fundamental block and let $P$ be the pair of pants of $\Pmc$ that contains the boundary cylinder $\pi(D)$. If $\nu$ is a $(\rho,\Sigma,\Tmc,\Jmc)$-tangent cocycle that is symmetric in $P$, and $\widetilde{\nu}(v_2-v_3)$ is nilpotent, then $\widetilde{\nu}(v_2)$ and $\widetilde{\nu}(v_3)$ are both nilpotent.
%\end{prop}

%%%%%%%%%%%%%%%%%%%%%%%%%%%%%%%%%%%%%%%%%%%%%%%%%%
\subsection{Explicit computations for the proof of Theorem \ref{thm:symplectic basis}}\label{sec:explicit}
%%%%%%%%%%%%%%%%%%%%%%%%%%%%%%%%%%%%%%%%%%%%%%%%%%
Using the tools we developed in Sections \ref{sec:computational tools} and \ref{sec:computational tools2}, we will now compute $\omega_{[\rho]}\left(\Xmc'([\rho]),\Xmc([\rho])\right)$ when $(\Xmc',\Xmc)$ satisfies of one of the four cases described at the start of Section \ref{sec:Tbbb}. %To do so, recall one last time that 
%\begin{itemize}
%\item $\Amc$ denotes the set of pairs of positive integers that sum to $n$, $\Bmc$ denotes the set of triples of positive integers that sum to $n$, and $\overline{\Bmc}$ denotes the set of triples of integers $(i_1,i_2,i_3)$ that sum to $n$, such that $0\leq i_1,i_2,i_3\leq n-1$. 
%\item $\Xi:\mathscr{T}\to\mathscr{A}$ is the linear isomorphism that assigns to every $(\rho,\Sigma,\Tmc,\Jmc)$-tangent cocycle its associated $\Tmc$-admissible labelling at $\rho$. 
%\item if $\mathbf c$ is an oriented closed curve in $\Pmc^o$ and $\mathbf k\in\Amc$, then $\sigma^\mathbf k_\mathbf c$, $\zeta^\mathbf k_\mathbf c$ and $\gamma^\mathbf k_\mathbf c$ are respectively the $\mathbf k$-twist, $\mathbf k$-lozenge, and $\mathbf k$-length cocycle associated to $\mathbf c$.
%\item if $\mathbf P$ is a cyclically oriented pair of pants given by $\Pmc$ and $\mathbf i\in\Bmc$, then $\epsilon^\mathbf i_\mathbf P$ and $\eta^\mathbf i_\mathbf P$ are respectively the $\mathbf i$-eruption and $\mathbf i$-hexagon cocycle associated to $\mathbf P$.
%\item for all triples of flags $\mathbf F:=(F_1,F_2,F_3)$ in $\Fmc(V)$ and for all $\mathbf i\in\Bmc$, $A^\mathbf i_\mathbf F$ is the $\mathbf i$-eruption endomorphism associated to $\mathbf F$.
%\item for all triples of flags $\mathbf E:=(E_1,E_2)$ in $\Fmc(V)$ and for all $\mathbf k\in\Amc$, $D^\mathbf k_\mathbf E$ is the $\mathbf k$-shearing endomorphism associated to $\mathbf E$.
%\end{itemize}
%
The outcome of the computations in these four cases are described in the following four propositions. Together, they prove Theorem \ref{thm:symplectic basis}.

\begin{prop} \label{prop:(1)} Let $[\rho]\in\Hit_V(S)$, let $\Xmc'=\Smc^\mathbf k_{\widehat{\mathbf c}}$ for some oriented closed curve $\widehat{\mathbf c}\in\Pmc^o$ and some $\mathbf k:=(k_1,k_2)\in\Amc$ and $\Xmc$ is any special $(\Tmc,\Jmc)$-parallel vector field. Then
\[\omega_{[\rho]}(\Xmc'([\rho]),\Xmc([\rho]))=\left\{\begin{array}{ll}
1&\text{if }\Xmc=\Ymc^\mathbf k_{\widehat{\mathbf c}};\\
0&\text{otherwise.}
\end{array}\right.\]
\end{prop}

In this proof, we will use the following notation.

\begin{notation} \label{not:(1)}
Let $\mathbf l:=(l_1,l_2)\in\Amc$. For any generic pair of flags $\mathbf F:=(F_1,F_2)$, let $M^\mathbf l_\mathbf F:V\to V$ be the endomorphism that fixes the line $F_1^{(l_1)}\cap F_2^{(l_2+1)}$, and whose kernel is $F_1^{(l_1-1)}+F_2^{(l_2)}$.
\end{notation}

\begin{proof}[Proof of Proposition \ref{prop:(1)}]
Observe that $\Xmc'$ is both symmetric and skew-symmetric in every pair of pants  in $\Pbbb$. Since every special $(\Tmc,\Jmc)$-parallel vector field is either symmetric or skew-symmetric in every pair of pants  in $\Pbbb$, Lemma \ref{lem:internal}(1) implies that the internal part of $\omega_{[\rho]}\left(\Xmc'([\rho]),\Xmc([\rho])\right)$ is zero. It is now sufficient to show that the external part of $\omega_{[\rho]}\left(\Xmc'([\rho]),\Xmc([\rho])\right)$ is $1$ when $\Xmc=\Ymc^\mathbf k_{\widehat{\mathbf c}}$, and is $0$ otherwise. To do so, let $D$ be a fundamental block, and let $u_1,\dots,u_5,v_1,\dots,v_4$ be as defined in Notation \ref{not:block}. The proof proceeds in the following cases.

{\bf Case 1: $\Xmc$ is a twist, eruption or hexagon vector field.} 
By Proposition \ref{prop:nil}(1)--(3), we know that both $\widetilde{\nu}'(v_2-v_3)$ and $\widetilde{\nu}(v_2-v_3)$ are nilpotent, so Lemma \ref{lem:alphacylinder}(1) and Lemma \ref{lem:gammacylinder}(1) imply that 
\[\sum_{\delta\in\Tbbb(D)}\sgn(\delta)\tr\big(\widetilde{\nu}'(e_{\delta,1})\cdot\widetilde{\nu}(e_{\delta,2})\big)=0.\]
Since $D$ is arbitrary, the external part of $\omega_{[\rho]}\left(\Xmc'([\rho]),\Xmc([\rho])\right)$ is zero. 

{\bf Case 2: $\Xmc$ is a length vector field not associated to $\widehat{\mathbf c}$.} If the boundary cylinder $\pi(D)$ does not have $\widehat{\mathbf c}$ as a boundary component, then Proposition \ref{prop:nil}(1) implies that $\widetilde{\nu}'(u_1)=\widetilde{\nu}'(v_3)=\widetilde{\nu}'(v_2)=0$. Then Lemma \ref{lem:alphacylinder}(2) and Lemma \ref{lem:gammacylinder}(2) imply that
\[\sum_{\delta\in\Tbbb(D)}\sgn(\delta)\tr\big(\widetilde{\nu}'(e_{\delta,1})\cdot\widetilde{\nu}(e_{\delta,2})\big)=0.\]
Suppose that the boundary cylinder $\pi(D)$ has $\widehat{\mathbf c}$ as a boundary component. Since $\Xmc$ is a length vector field that is not associated to $\widehat{\mathbf c}$, Proposition \ref{prop:nil}(1),(4) imply that $\widetilde{\nu}'(v_2-v_3)$ and $\widetilde{\nu}(v_2-v_3)$ are both nilpotent, so Lemma \ref{lem:alphacylinder}(1) and Lemma \ref{lem:gammacylinder}(1) gives
\[\sum_{\delta\in\Tbbb(D)}\sgn(\delta)\tr\big(\widetilde{\nu}'(e_{\delta,1})\cdot\widetilde{\nu}(e_{\delta,2})\big)=0.\]
Thus, the external part of $\omega_{[\rho]}\left(\Xmc'([\rho]),\Xmc([\rho])\right)$ is zero.

{\bf Case 3: $\Xmc=\Ymc^{\mathbf l}_{\widehat{\mathbf c}}$ for some $\mathbf l\in\Amc$.} If the boundary cylinder $\pi(D)$ does not have $\widehat{\mathbf c}$ as a boundary component, the same argument as we used in Case 2 proves
\[\sum_{\delta\in\Tbbb(D)}\sgn(\delta)\tr\big(\widetilde{\nu}'(e_{\delta,1})\cdot\widetilde{\nu}(e_{\delta,2})\big)=0.\]
Suppose that the boundary cylinder $\pi(D)$ has $\widehat{\mathbf c}$ as a boundary component. Recall that $\Ymc^{\mathbf l}_{\widehat{\mathbf c}}$ is the sum of the lozenge vector field $\Zmc^\mathbf l_{\widehat{\mathbf c}}$ with some eruption cocycles. We have already established in Case 1 that 
\[\sum_{\delta\in\Tbbb(D)}\sgn(\delta)\tr\big(\widetilde{\nu}'(e_{\delta,1})\cdot\widetilde{\nu}''(e_{\delta,2})\big)=0.\]
 for any eruption cocycle $\nu''$, so
\[\sum_{\delta\in\Tbbb(D)}\sgn(\delta)\tr\big(\widetilde{\nu}'(e_{\delta,1})\cdot\widetilde{\nu}(e_{\delta,2})\big)=\sum_{\delta\in\Tbbb(D)}\sgn(\delta)\tr\big(\widetilde{\nu}'(e_{\delta,1})\cdot\widetilde{\zeta}^\mathbf l_{\widehat{\mathbf c}}(e_{\delta,2})\big),\]
where recall that $\zeta^\mathbf l_{\widehat{\mathbf c}}$ is the $(\rho,\Tmc,\Jmc)$-tangent cocycle associated to $\Zmc^\mathbf l_{\widehat{\mathbf c}}([\rho])$. Since $\widetilde{\nu}'(v_2)=0=\widetilde{\nu}'(v_3)$ by Proposition \ref{prop:nil}(1), (\ref{eqn:alphacylinder}) and (\ref{eqn:gammacylinder}) imply that regardless of whether $D$ is consistent or not, we have
\[\sum_{\delta\in\Tbbb(D)}\sgn(\delta)\tr\big(\widetilde{\nu}^1(e_{\delta,1})\cdot\widetilde{\nu}^2(e_{\delta,2})\big)=\tr\big(\widetilde{\sigma}^\mathbf k_{\widehat{\mathbf c}}(u_1)_{{\bf r},{\rm diag}}\cdot\widetilde{\zeta}^\mathbf l_{\widehat{\mathbf c}}(v_2-v_3)_{{\bf r},{\rm diag}}\big),\]
where recall that $\sigma^\mathbf k_{\widehat{\mathbf c}}$ is the $(\rho,\Tmc,\Jmc)$-tangent cocycle associated to $\Smc^\mathbf k_{\widehat{\mathbf c}}([\rho])$.

Let $P(1)$ and $P(2)$ be the pairs of pants of $\Pmc$ that lie to the right and left of $\widehat{\mathbf c}$ respectively. Then for both $m=1,2$, let ${\bf P}(m):=(P(m),\widehat{c})$, and let $(\widehat{\mathbf b}(m),\widehat{\mathbf b}'(m))$ be the pair of non-edge barriers associated to ${\bf P}(m)$. Let ${\bf c}$, ${\bf b}(m)$ and ${\bf b}'(m)$ be lifts to $\widetilde{S}$ of $\widehat{\bf c}$, $\widehat{\bf b}(m)$ and $\widehat{\bf b}'(m)$ respectively, so that ${\bf b}(m)$ intersects $\delta_{P(m)}$, ${\bf b}'(m)$ intersects $\delta'_{P(m)}$, and the three barriers ${\bf c}$, ${\bf b}(m)$, and ${\bf b}'(m)$ share a common endpoint. Then let ${\bf x}(m)$ and ${\bf x}'(m)$ be the triples associated to $\widehat{\bf b}(m)$ and $\widehat{\bf b}'(m)$ respectively, and let ${\bf r}$ be the pair associated to ${\bf c}$. 

Suppose that $\pi(D)$ lies in $P(1)$. Then by Proposition \ref{prop:derivative2},
\[\widetilde{\sigma}^\mathbf k_{\widehat{\mathbf c}}(u_1)_{{\mathbf r},{\rm diag}}=-\frac{1}{2}D^{\mathbf k}_{\xi_\rho({\mathbf r})}.\]
Also, by Proposition \ref{prop:derivative1} and \eqref{eqn: simplify not}
\begin{align*}
\widetilde{\zeta}^\mathbf l_{\widehat{\mathbf c}}(v_2-v_3)=&\frac{1}{2}\left(-A^{\mathbf l_3}_{\xi_\rho(\mathbf x(1))}+A^{\mathbf l_3(1,-1,0)}_{\xi_\rho(\mathbf x(1))}-A^{\bar{\mathbf l}_2}_{\xi_\rho(\mathbf x'(1))}+A^{\bar{\mathbf l}_2(1,0,-1)}_{\xi_\rho(\mathbf x'(1))}\right)\\
&-A^{\mathbf l_3(0,-1,1)}_{\xi_\rho(\mathbf x(1))}+A^{\mathbf l_3(-1,0,1)}_{\xi_\rho(\mathbf x(1))}\\
&+\frac{1}{2}\left(-A^{\mathbf l_3}_{\xi_\rho(\mathbf x(1))}+A^{\mathbf l_3(1,-1,0)}_{\xi_\rho(\mathbf x(1))}-A^{\bar{\mathbf l}_2}_{\xi_\rho(\mathbf x'(1))}+A^{\bar{\mathbf l}_2(1,0,-1)}_{\xi_\rho(\mathbf x'(1))}\right)\\
&-A^{\bar{\mathbf l}_2(0,1,-1)}_{\xi_\rho(\mathbf x'(1))}+A^{\bar{\mathbf l}_2(-1,1,0)}_{\xi_\rho(\mathbf x'(1))}\\
=&-A^{\mathbf l_3}_{\xi_\rho(\mathbf x(1))}+A^{\mathbf l_3(1,-1,0)}_{\xi_\rho(\mathbf x(1))}-A^{\bar{\mathbf l}_2}_{\xi_\rho(\mathbf x'(1))}+A^{\bar{\mathbf l}_2(1,0,-1)}_{\xi_\rho(\mathbf x'(1))}\\
&-A^{\mathbf l_3(0,-1,1)}_{\xi_\rho(\mathbf x(1))}+A^{\mathbf l_3(-1,0,1)}_{\xi_\rho(\mathbf x(1))}-A^{\bar{\mathbf l}_2(0,1,-1)}_{\xi_\rho(\mathbf x'(1))}+A^{\bar{\mathbf l}_2(-1,1,0)}_{\xi_\rho(\mathbf x'(1))},
\end{align*}
where we recall that 
\begin{itemize}
\item for any $\mathbf k=(k_1,k_2)\in\Amc$, we denote $\mathbf k_1:=(0,k_1,k_2)$, $\mathbf k_2:=(k_2,0,k_1)$, $\mathbf k_3:=(k_1,k_2,0)$, $\bar{\mathbf k}_1:=(0,k_2,k_1)$, $\bar{\mathbf k}_2:=(k_1,0,k_2)$, and $\bar{\mathbf k}_3:=(k_2,k_1,0)$, and  
\item for any ${\bf i}=(i_1,i_2,i_3)\in\Bmc$ and any triple of integers $(a_1,a_3,a_3)$, we denote $\mathbf i(a_1,a_2,a_3):=(i_1+a_1,i_2+a_2,i_3+a_3)$.
\end{itemize}
So, Lemma \ref{lem: diagonal part} implies that if for any ${\bf k}=(k_1,k_2)\in\Amc$ and any pair of integers $(a_1,a_2)$, we denote ${\bf k}(a_1,a_2):=(k_1+a_1,k_2+a_2)$, then 
\begin{eqnarray}\label{eqn:lozenge}
\widetilde{\zeta}^\mathbf l_{\widehat{\mathbf c}}(v_2-v_3)_{{\mathbf r},{\rm diag}}&=&4D^{\mathbf l(-1,1)}_{\xi_\rho(\mathbf r)}-8D^{\mathbf l}_{\xi_\rho(\mathbf r(1))}+4D^{\mathbf l(1,-1)}_{\xi_\rho(\mathbf r)}\\
&=&-2M^{\mathbf l}_{\xi_\rho(\mathbf r)}+2M^{\mathbf l(1,-1)}_{\xi_\rho(\mathbf r)}.\nonumber
\end{eqnarray}
%where $M^\mathbf l_\mathbf E$ is the endomorphism defined in Notation \ref{not:(1)}, and $N$ is a nilpotent endomorphism that preserves $\xi(d_m)$. 
Similarly, if $\pi(D)$ lies in $P(2)$, then
\[\widetilde{\sigma}^\mathbf k_{\widehat{\mathbf c}}(u_1)_{{\mathbf r},{\rm diag}}=\widetilde{\sigma}^{\bar{\mathbf k}}_{\widehat{\bar{\mathbf c}}}(u_1)_{{\mathbf r},{\rm diag}}=-\frac{1}{2}D^{\bar{\mathbf k}}_{\xi_\rho({\mathbf r})}\]
and
\[\widetilde{\zeta}^\mathbf l_{\widehat{\mathbf c}}(v_2-v_3)_{{\mathbf r},{\rm diag}}=\widetilde{\zeta}^{\bar{\mathbf l}}_{\widehat{\bar{\mathbf c}}}(v_2-v_3)_{{\mathbf r},{\rm diag}}=-2M^{\bar{\mathbf l}}_{\xi_\rho(\mathbf r)}+2M^{\bar{\mathbf l}(1,-1)}_{\xi_\rho(\mathbf r)}.\]
Thus, regardless of whether $\pi(D)$ lies in $P(1)$ or $P(2)$, we have
\begin{eqnarray*}
\sum_{\delta\in\Tbbb(D)}\sgn(\delta)\tr\big(\widetilde{\nu}'(e_{\delta,1})\cdot\widetilde{\nu}(e_{\delta,2})\big)%&=&\tr\left(D^{\mathbf k_m}_{\xi(\mathbf d_2)}\cdot \left(-M^{\mathbf l_m}_{\xi(\mathbf d_m)}+M^{\mathbf l_m(1,-1)}_{\xi(\mathbf d_m)}\right)\right)\\
&=&\left\{\begin{array}{ll}
\frac{1}{2}&\text{if }\mathbf l=\mathbf k;\\
0&\text{otherwise.}
\end{array}\right.
\end{eqnarray*}

This shows that if $\mathbf l\neq\mathbf k$, then the external part of $\omega_{[\rho]}\left(\Xmc'([\rho]),\Xmc([\rho])\right)$ is zero. On the other hand, when $\mathbf l=\mathbf k$, then of all the fundamental blocks $D$ that is the union of external triangles in $\widehat{\Tbbb}$, only two of them have the property that 
\[\sum_{\delta\in\Tbbb(D)}\sgn(\delta)\tr\big(\widetilde{\nu}'(e_{\delta,1})\cdot\widetilde{\nu}(e_{\delta,2})\big)=\frac{1}{2},\]
while the rest of them have the property that
\[\sum_{\delta\in\Tbbb(D)}\sgn(\delta)\tr\big(\widetilde{\nu}'(e_{\delta,1})\cdot\widetilde{\nu}(e_{\delta,2})\big)=0,\]
and so the external part of $\omega_{[\rho]}\left(\Xmc'([\rho]),\Xmc([\rho])\right)$ is $1$.
\end{proof}

\begin{prop}\label{prop:(2)}
Let $\Xmc'=\Emc^{\mathbf i}_{\mathbf P}$ for some marked pair of pants $\mathbf P$ and some $\mathbf i:=(i_1,i_2,i_3)\in\Bmc$, and let $\Xmc$ be any eruption, hexagon or length vector field. Then
\[\omega_{[\rho]}\left(\Xmc'([\rho]),\Xmc([\rho])\right)=\left\{\begin{array}{ll}
1&\text{if }\Xmc=\Hmc^{\mathbf i}_{\mathbf P};\\
0&\text{otherwise.}
\end{array}\right.\]
\end{prop}

\begin{proof}
The proof is divided into the following cases.

{\bf Case 1: $\Xmc$ is an eruption vector field.} Observe that $\Xmc'$ and $\Xmc$ are both skew-symmetric, so Lemma \ref{lem:internal}(1) implies that the internal part of $\omega_{[\rho]}\left(\Xmc'([\rho]),\Xmc([\rho])\right)$ is zero. Let $D$ be a fundamental block, and let $u_1,\dots,u_5,v_1,\dots,v_4$ be as defined in Notation \ref{not:block}. By Proposition~\ref{prop:nil}(2), $\widetilde{\nu}'(v_2-v_3)$ and $\widetilde{\nu}(v_2-v_3)$ are both nilpotent, so Lemma \ref{lem:alphacylinder}(1) and Lemma \ref{lem:gammacylinder}(1) imply that 
\[\sum_{\delta\in\Tbbb(D)}\sgn(\delta)\tr\big(\widetilde{\nu}^1(e_{\delta,1})\cdot\widetilde{\nu}^2(e_{\delta,2})\big)=0.\]
Since $D$ is arbitrary, the external part of $\omega_{[\rho]}\left(\Xmc'([\rho]),\Xmc([\rho])\right)$ is also zero.

{\bf Case 2: $\Xmc$ is a length vector field.} 
Let $\Xmc=\Ymc^\mathbf k_{\widehat{\mathbf c}}$, where $\widehat{\mathbf c}\in\Pmc^o$ and $\mathbf k\in\Amc$. By definition, $\Ymc^\mathbf k_{\widehat{\mathbf c}}$ is the sum of the lozenge vector field $\Zmc^\mathbf k_{\widehat{\mathbf c}}$ with some eruption cocycles, so Case 1 implies that
\[\omega_{[\rho]}\left(\Xmc'([\rho]),\Xmc([\rho])\right)=\omega_{[\rho]}\left(\Xmc'([\rho]),\Zmc^\mathbf k_{\widehat{\mathbf c}}([\rho])\right).\]
We thus need to show that $\omega_{[\rho]}\left(\Xmc'([\rho]),\Zmc^\mathbf k_{\widehat{\mathbf c}}([\rho])\right)=0$. 

Let $P(1)$ and $P(2)$ be the pairs of pants of $\Pmc$ that share $\widehat{\bf c}$ as a common boundary component, such that $P(1)$ and $P(2)$ lie to the right and left of $\widehat{\mathbf c}$ respectively. %Let $P$ be the pair of pants of $\Pmc$ such that ${\bf P}=(P,\widehat{c}')$ for some non-isolated edge $\widehat{c}'\in\Pmc$.

{\bf Case 2.1: $\widehat{\bf c}$ is not a boundary component of $P$.} For any triangle $\delta\in\Tbbb_{\rm lift}$, note that either $\widetilde{\nu}'(e_{\delta,1})=0=\widetilde{\nu}'(e_{\delta,2})$, or $\widetilde{\zeta}^\mathbf k_{\widehat{\mathbf c}}(e_{\delta,1})=0=\widetilde{\zeta}^\mathbf k_{\widehat{\mathbf c}}(e_{\delta,2})$. Thus, $\omega_{[\rho]}\left(\Xmc'([\rho]),\Zmc^\mathbf k_{\widehat{\mathbf c}}([\rho])\right)=0$.

{\bf Case 2.2: $\widehat{\bf c}$ is a boundary component of $P$.} By the symmetries of $\Emc^{\mathbf i}_{\mathbf P}$ and $\Ymc^\mathbf k_{\widehat{\mathbf c}}$, we may assume that ${\bf P}=(P(1),\widehat{c})$. 

If $\delta\in\Tbbb_{\rm lift}$ is a triangle such that $\pi(\delta)$ does not lie in $P(1)$, then by definition, $\widetilde{\nu}'(e_{\delta,1})=0=\widetilde{\nu}'(e_{\delta,2})$, so $\tr\big(\widetilde{\nu}'(e_{\delta,1})\cdot\widetilde{\zeta}^\mathbf k_{\widehat{\mathbf c}}(e_{\delta,2})\big)=0$. Also, if $D$ is a fundamental block and $u_1,\dots,u_5,v_1,\dots,v_4$ are the edges of $\Tbbb(D)$ as defined in Notation \ref{not:block}, then Proposition \ref{prop:nil}(2) states that $\widetilde{\nu}'(v_2-v_3)$ is a nilpotent endomorphism and $\widetilde{\nu}'(u_1)=0$. Hence, Lemma \ref{lem:alphacylinder}(2) implies that if $D$ is consistent, then
\[\sum_{\delta\in\Tbbb(D)}\sgn(\delta)\tr\big(\widetilde{\nu}'(e_{\delta,1})\cdot\widetilde{\zeta}^\mathbf k_{\widehat{\mathbf c}}(e_{\delta,2})\big)=0.\]
However, if $D$ is inconsistent, (\ref{eqn:gammacylinder}) implies that
\[\sum_{\delta\in\Tbbb(D)}\sgn(\delta)\tr\big(\widetilde{\nu}'(e_{\delta,1})\cdot\widetilde{\zeta}^\mathbf k_{\widehat{\mathbf c}}(e_{\delta,2})\big)=\tr\big(\widetilde{\nu}'(v_3)\cdot\widetilde{\zeta}^\mathbf k_{\widehat{\mathbf c}}(v_3)\big)-\tr\big(\widetilde{\nu}'(v_2)\cdot\widetilde{\zeta}^\mathbf k_{\widehat{\mathbf c}}(v_2)\big).\]
Thus, if $v_2$ and $v_3$ are the internal edges in the fundamental block $D$ such that $\pi(D)$ is the inconsistent boundary cylinder in $P=P(1)$, then 
\begin{eqnarray*}
\omega_{[\rho]}\left(\Xmc'([\rho]),\Zmc^\mathbf k_{\widehat{\mathbf c}}([\rho])\right)&=&\tr\big(\widetilde{\nu}'(e_{\delta_P,1})\cdot\widetilde{\zeta}^\mathbf k_{\widehat{\mathbf c}}(e_{\delta_P,2})\big)-\tr\big(\widetilde{\nu}'(v_2)\cdot\widetilde{\zeta}^\mathbf k_{\widehat{\mathbf c}}(v_2)\big)\\
&&-\left(\tr\big(\widetilde{\nu}'(e_{\delta_P',1})\cdot\widetilde{\zeta}^\mathbf k_{\widehat{\mathbf c}}(e_{\delta_P',2})\big)-\tr\big(\widetilde{\nu}'(v_3)\cdot\widetilde{\zeta}^\mathbf k_{\widehat{\mathbf c}}(v_3)\big)\right)\\
&=&2\tr\big(\widetilde{\nu}'(e_{\delta_P,1})\cdot\widetilde{\zeta}^\mathbf k_{\widehat{\mathbf c}}(e_{\delta_P,2})\big)-2\tr\big(\widetilde{\nu}'(v_2)\cdot\widetilde{\zeta}^\mathbf k_{\widehat{\mathbf c}}(v_2)\big),
\end{eqnarray*}
where the second equality is a consequence of Lemma \ref{lem:internal}(2) and the observation that $\nu'$ is skew-symmetric in $P$ while $\zeta^\mathbf k_\mathbf c$ is symmetric in $P$. To finish the proof in this case, it is now sufficient to show that 
\[\tr\big(\widetilde{\nu}'(e_{\delta_P,1})\cdot\widetilde{\zeta}^\mathbf k_{\widehat{\mathbf c}}(e_{\delta_P,2})\big)-\tr\big(\widetilde{\nu}'(v_2)\cdot\widetilde{\zeta}^\mathbf k_{\widehat{\mathbf c}}(v_2)\big)=0.\]

Let $(\widehat{\mathbf b},\widehat{\mathbf b}')$ be the pair of non-edge barriers associated to ${\bf P}$, let ${\bf b}$ and ${\bf b}'$ be lifts to $\widetilde{S}$ of $\widehat{\bf b}$ and $\widehat{\bf b}'$ respectively, so that ${\bf b}$ intersects $\delta_P$, ${\bf b}'$ intersects $\delta'_{P}$, and let ${\bf x}$ and ${\bf x}'$ be the triples associated to $\widehat{\bf b}$ and $\widehat{\bf b}'$ respectively. By cyclically permuting the enumeration of the three vertices of $\Tbbb$ in the interior of $P$ if necessary, we may ensure that $e_{\delta_P,1}$ intersects ${\bf b}$.  Then by \eqref{eqn: simplify not}
\begin{eqnarray*}
\widetilde{\nu}'(e_{\delta_P,1})&=&-\frac{1}{2}A^{\mathbf i}_{\xi_\rho(\mathbf x)},\\
\widetilde{\nu}'(v_2)&=&-\frac{1}{2}A^{\mathbf i_+}_{\xi_\rho(\mathbf x_+)},\\
\widetilde{\zeta}^\mathbf k_\mathbf c(e_{\delta_P,2})&=& - A_{\xi_\rho(\mathbf x_-)}^{\mathbf k_1(1,0,-1)}  + A_{\xi_\rho(\mathbf x_-)}^{\mathbf k_1(1,-1,0)},\quad\text{and}\\
\widetilde{\zeta}^\mathbf k_\mathbf c(v_2)%&=&\frac{1}{2}\left(-A^{\mathbf k_2}_{\xi_\rho(\mathbf x_+)}+A^{\mathbf k_2(-1,0,1)}_{\xi_\rho(\mathbf x_+)}\right)-A^{\mathbf k_2(-1,1,0)}_{\xi_\rho(\mathbf x_+)}+A^{\mathbf k_2(0,1,-1)}_{\xi_\rho(\mathbf x_+)}+\frac{1}{2}\left(-A^{\mathbf k_2}_{\xi_\rho(\mathbf x_+)}+A^{\mathbf k_2(-1,0,1)}_{\xi_\rho(\mathbf x_+)}\right)\\
&=&-A^{\mathbf k_2}_{\xi_\rho(\mathbf x_+)}+A^{\mathbf k_2(-1,0,1)}_{\xi_\rho(\mathbf x_+)}-A^{\mathbf k_2(-1,1,0)}_{\xi_\rho(\mathbf x_+)}+A^{\mathbf k_2(0,1,-1)}_{\xi_\rho(\mathbf x_+)}.
\end{eqnarray*}
Applying Lemma \ref{lem:constant0} gives
\begin{align*}
\tr\big(\widetilde{\nu}'(e_{\delta_P,1})\cdot\widetilde{\zeta}^\mathbf k_{\widehat{\mathbf c}}(e_{\delta_P,2})\big)=\frac{1}{2}\bigg([\min\{i_1-k_2+1,1-i_2\}]_+-[\min\{i_1-k_2,1-i_2\}]_+\bigg)=0
\end{align*}
and
\begin{align*}
&\tr\big(\widetilde{\nu}'(v_2)\cdot\widetilde{\zeta}^\mathbf k_{\widehat{\mathbf c}}(v_2)\big)\\
&=\frac{1}{2}\bigg(\min\{i_2,k_2\}-\min\{i_2,k_2-1\}+\min\{i_2,k_2-1\}-\min\{i_2,k_2\}\bigg)\\
&=0.
\end{align*}
%so $\tr\big(\widetilde{\nu}^1(e_{\delta_P,1})\cdot\widetilde{\zeta}^\mathbf k_\mathbf c(e_{\delta_P,2})\big)-\tr\big(\widetilde{\nu}^1(v_2)\cdot\widetilde{\zeta}^\mathbf k_\mathbf c(v_2)\big)=0$.
%\[\begin{array}{lll}
%&&\tr\left(\widetilde{\epsilon}^{\mathbf i}_{\mathbf P}(e_{\delta_P,1})\cdot\widetilde{\zeta}^\mathbf k_\mathbf c(e_{\delta_P,2})\right)\\
%&=&\displaystyle\frac{1}{2}\tr\left(\left(A_{\xi(\mathbf x_1)}^{\mathbf k^3} -A_{\xi(\mathbf x_1)}^{\mathbf k^3(1,-1,0)} + A_{\xi(\mathbf x_1)}^{\mathbf k^3(0,-1,1)}  - A_{\xi(\mathbf x_1)}^{\mathbf k^3(-1,0,1)}\right)\cdot A^{\mathbf i_2}_{\xi(\mathbf x_2)}\right)\vspace{.1cm}\\
%&=&\displaystyle\frac{1}{2}\left([\min\{k_1-i_1,i_2-k_2\}]_+-[\min\{k_1+1-i_1,i_2-k_2+1\}]_+\right.\\
%&&\left.+[\min\{k_1-i_1,i_2-k_2+1\}]_+-[\min\{k_1-1-i_1,i_2-k_2\}]_+\right)\vspace{.1cm}\\
%&=&\displaystyle\frac{1}{2}\left([i_2-k_2]_+-[i_2-k_2+1]_++[i_2-k_2+1]_+-[i_2-k_2]_+\right)\\
%&=&0.
%\end{array}\]

%\[\tr\big(\widetilde{\epsilon}^{\mathbf i}_{\mathbf P}(k_2)\cdot\widetilde{\zeta}^\mathbf k_\mathbf c(k_2)\big)=\frac{1}{2}\tr\left(A^{\mathbf i_3}_{\xi(\mathbf x_3)}\cdot\left(A_{\xi(\mathbf x_3)}^{\mathbf k^1(1,0,-1)}  - A_{\xi(\mathbf x_3)}^{\mathbf k^1(1,-1,0)}\right)\right)=0\]

{\bf Case 3: $\Xmc$ is a hexagon vector field.} Let $\Xmc=\Hmc^{\bf j}_{\bf Q}$, where ${\bf Q}$ be a marked pair of pants given by $\Pmc$ and ${\bf j}\in\Bmc$. By Proposition \ref{prop:nil}(2),(3), both $\widetilde{\nu}'(k_2-k_3)$ and $\widetilde{\nu}(k_2-k_3)$ are nilpotent, so Lemma \ref{lem:alphacylinder}(1) and Lemma \ref{lem:gammacylinder}(1) imply that 
\[\sum_{\delta\in\Tbbb(D)}\sgn(\delta)\tr\left(\widetilde{\nu}'(e_{\delta,1})\cdot\widetilde{\nu}(e_{\delta,2})\right)=0.\]
for any fundamental block $D$. In particular, the external part of $\omega_{[\rho]}\left(\Xmc'([\rho]),\Xmc([\rho])\right)$ is zero. Thus, it is sufficient to show that the internal part of $\omega_{[\rho]}\left(\Xmc'([\rho]),\Xmc([\rho])\right)$ is $1$ if $\mathbf Q=\mathbf P$ and $\mathbf j=\mathbf i$, and is zero otherwise. Let $P$ and $Q$ be the pairs of pants underlying $\mathbf P$ and $\mathbf Q$ respectively. 

{\bf Case 3.1: $P\neq Q$.} Observe that for any pair of pants $R$ of $\Pmc$, either $\widetilde{\nu}'(e_{\delta_{R}',1})=0=\widetilde{\nu}'(e_{\delta_{R}',2})$ or $\widetilde{\nu}(e_{\delta_{R}',1})=0=\widetilde{\nu}(e_{\delta_{R}',2})$, so the internal part of $\omega_{[\rho]}\left(\Xmc'([\rho]),\Xmc([\rho])\right)$ is zero when $P\neq Q$. 

{\bf Case 3.2: $P=Q$.} 
By the symmetries of $\Emc^{\bf i}_{\bf P}$ and $\Hmc^{\bf j}_{\bf Q}$, we may assume that $\mathbf P=\mathbf Q$, and that the pair $(\widehat{\bf b},\widehat{\bf b}')$ of non-edge barriers associated to ${\bf P}$ satisfy the property that there is a lift ${\bf b}$ of $\widehat{\bf b}$ that intersects $e_{\delta_P,1}$ and a lift ${\bf b}'$ of $\widehat{\bf b}'$ that intersects $e_{\delta_P',1}$. Let ${\bf x}$ and ${\bf x}'$ be the triples associated to ${\bf b}$ and ${\bf b}'$ respectively. 

By \eqref{eqn: simplify not},
\begin{align*}
\widetilde{\nu}'(e_{\delta_P,1})=&-\frac{1}{2}A^{\mathbf i}_{\xi_\rho(\mathbf x)}\quad\text{and}\\
%\widetilde{\nu}'(v_2)&=&-\frac{1}{2}A^{\mathbf i_+}_{\xi(\mathbf x_+)},\\
\widetilde{\nu}(e_{\delta_P,2})=&-A^{\mathbf j_-(-1,0,1)}_{\xi_\rho(\mathbf x_-)}+A^{\mathbf j_-(0,-1,1)}_{\xi_\rho(\mathbf x_-)}-A^{\mathbf j_-(1,-1,0)}_{\xi_\rho(\mathbf x_-)}\\
&+A^{\mathbf j_-(1,0,-1)}_{\xi_\rho(\mathbf x_-)}-A^{\mathbf j_-(0,1,-1)}_{\xi_\rho(\mathbf x_-)}+A^{\mathbf j_-(-1,1,0)}_{\xi_\rho(\mathbf x_-)}.
\end{align*}
%The previous paragraph proves that the internal part of $\omega_\rho([\nu^1],[\nu^2])$ can be written as
%\[\tr\big(\widetilde{\nu}^1(e_{\delta_P,1})\cdot\widetilde{\nu}^2(e_{\delta_P,2})\big)-\tr\big(\widetilde{\nu}^1(e_{\delta_P',1})\cdot\widetilde{\nu}^2(e_{\delta_P',2})\big).\]
Thus, by Lemma \ref{lem:constant0}(1)
\begin{align*}
&\tr\big(\widetilde{\nu}'(e_{\delta_P,1})\cdot\widetilde{\nu}(e_{\delta_P,2})\big)\\
&=\frac{1}{2}\bigg([\min\{j_3-1-i_3,i_1-j_1\}]_+-[\min\{j_3-i_3,i_1-j_1+1\}]_+\\
&+[\min\{j_3+1-i_3,i_1-j_1+1\}]_+-[\min\{j_3+1-i_3,i_1-j_1\}]_+\\
&+[\min\{j_3-i_3,i_1-j_1-1\}]_+-[\min\{j_3-1-i_3,i_1-j_1-1\}]_+\bigg)
%&=&\frac{1}{2}\tr\left( A^{\mathbf i_2}_{\xi(\mathbf x_2)}\cdot\left(A^{\mathbf j_1(0,1,-1)}_{\xi(\mathbf x_1)}-A^{\mathbf j_1(-1,1,0)}_{\xi(\mathbf x_1)}+A^{\mathbf j_1(-1,0,1)}_{\xi(\mathbf x_1)}-A^{\mathbf j_1(0,-1,1)}_{\xi(\mathbf x_1)}\right.\right. \\
%&& \left.\left.+A^{\mathbf j_1(1,-1,0)}_{\xi(\mathbf x_1)}-A^{\mathbf j_1(1,0,-1)}_{\xi(\mathbf x_1)}\right)\right)\\
%&=&\frac{1}{2}\left([\min(j_1-i_1,i_2-j_2-1)]_+ - [\min(j_1-1-i_1,i_2-j_2-1)]_+ \right.\\
%&&+[\min(j_1-1-i_1,i_2-j_2)]_+- [\min(j_1-i_1,i_2-j_2+1)]_+\\
%&&\left. +[\min(j_1+1-i_1,i_2-j_2+1)]_+ - [\min(j_1+1-i_1,i_2-j_2)]_+\right)\\
%&=&\begin{cases}
%0 & \text{if }j_1-i_1\geq i_2-j_2+1;
%\cr0 & \text{if }j_1-i_1\leq i_2-j_2-1;
%\cr\displaystyle\frac{1}{2}\left([j_1-1-i_1]_+- [j_1-i_1]_++[j_1+1-i_1]_+ - [j_1-i_1]_+ \right)& \text{if }j_1-i_1= i_2-j_2;
%\end{cases}\\
\end{align*}
From this, it is straightforward to observe that if $j_3-i_3\geq i_1-j_1+1$ and $j_3-i_3\leq i_1-j_1-1$, then $\tr\big(\widetilde{\nu}'(e_{\delta_P,1})\cdot\widetilde{\nu}(e_{\delta_P,2})\big)=0$.
On the other hand, if $j_3-i_3=i_1-j_1$, then the above formula specializes to
\begin{align*}
&\tr\big(\widetilde{\nu}'(e_{\delta_P,1})\cdot\widetilde{\nu}(e_{\delta_P,2})\big)\\
&=\frac{1}{2}\bigg([i_1-j_1-1]_+-[i_1-j_1]_++[i_1-j_1+1]_+-[i_1-j_1]_+\bigg)\\
&=\begin{cases}\frac{1}{2}&\text{if }i_1=j_1;\cr0 & \text{if }i_1\neq j_1\end{cases}
\end{align*}
We have thus proven that 
\[\tr\big(\widetilde{\nu}'(e_{\delta_P,1})\cdot\widetilde{\nu}(e_{\delta_P,2})\big)=\begin{cases}\frac{1}{2}&\text{if }\mathbf i=\mathbf j;\cr0 & \text{otherwise}\end{cases}\] 
Since $\Emc^{\bf i}_{\bf P}$ is skew-symmetric in $P$ and $\Hmc^{\bf j}_{\bf P}$ is symmetric in $P$, by Lemma \ref{lem:internal}(2), 
\[\tr\big(\widetilde{\nu}'(e_{\delta_P',1})\cdot\widetilde{\nu}(e_{\delta_P',2})\big)=\begin{cases}-\frac{1}{2}&\text{if }\mathbf i=\mathbf j;\cr0 & \text{otherwise}\end{cases}\] 
since the internal part of $\omega_{[\rho]}\left(\Xmc'([\rho]),\Xmc([\rho])\right)$ is 
\[\tr\big(\widetilde{\nu}'(e_{\delta_P,1})\cdot\widetilde{\nu}(e_{\delta_P,2})\big) -  \tr\big(\widetilde{\nu}'(e_{\delta_P',1})\cdot\widetilde{\nu}(e_{\delta_P',2})\big),\]
it follows that the internal part of $\omega_{[\rho]}\left(\Xmc'([\rho]),\Xmc([\rho])\right)$ is $1$ if $\mathbf i=\mathbf j$, and is zero otherwise.
\end{proof}

%\begin{remark}\label{rem:eruptionlozenge}
%Note that in Case 2 of the proof of Proposition \ref{prop:(2)}, we proved that 
%\[\omega_{[\rho]}\left(\Xmc'([\rho]),\Xmc([\rho])\right)=0\] 
%when $\Xmc'$ is an eruption cocycle and $\Xmc$ is a lozenge cocycle.
%\end{remark}

\begin{prop}\label{prop:(3)}
If $\Xmc'$ is a hexagon vector field and let $\Xmc$ is a hexagon or length vector field, then
\[\omega_{[\rho]}\left(\Xmc'([\rho]),\Xmc([\rho])\right)=0.\]
\end{prop}

\begin{proof}
Let $\Xmc'=\Hmc^{\mathbf i}_{\mathbf P}$ for some marked pair of pants $\mathbf P$ and some $\mathbf i=(i_1,i_2,i_3)\in\Bmc$. We prove the proposition for the cases when $\nu$ is a hexagon vector field and a length vector field separately.

{\bf Case 1: $\Xmc$ is a hexagon vector field.} Observe that $\Xmc'$ and $\Xmc$ are both symmetric in every pair of pants of $\Pmc$, so Lemma \ref{lem:internal}(1) tells us that the internal part of $\omega_{[\rho]}\left(\Xmc'([\rho]),\Xmc([\rho])\right)$ is zero. On the other hand, for any fundamental block $D$, if we let $u_1,\dots,u_5,v_1,\dots,v_4$ be as defined in Notation \ref{not:block}, then Proposition \ref{prop:nil}(3) states that both $\widetilde{\nu}'(v_3-v_2)$ and $\widetilde{\nu}(v_3-v_2)$ are nilpotent. Lemma \ref{lem:alphacylinder}(1) and Lemma \ref{lem:gammacylinder}(1) then imply that 
\[\sum_{\delta\in\Tbbb(D)}\sgn(\delta)\tr\big(\widetilde{\nu}'(e_{\delta,1})\cdot\widetilde{\nu}(e_{\delta,2})\big)=0,\]
so the external part of $\omega_{[\rho]}\left(\Xmc'([\rho]),\Xmc([\rho])\right)$ is also zero. 

{\bf Case 2: $\nu$ is a length vector field.} Let $\Xmc=\Ymc^\mathbf k_{\widehat{\mathbf c}}$ for some oriented, non-isolated edge $\widehat{\mathbf c}\in\Pmc^o$ and some $\mathbf k:=(k_1,k_2)\in\Amc$. Let $P$ be the pair of pants underlying $\mathbf P$, and let $\widehat{c}$ be the non-isolated edge in $\Pmc$ underlying $\widehat{\bf c}$.

{\bf Case 2.1: $\widehat{c}$ is not a boundary component of $P$.} In this case, observe that for any triangle $\delta\in\Tbbb_{\rm lift}$, either $\widetilde{\nu}'(e_{\delta,1})=0=\widetilde{\nu}'(e_{\delta,2})$ or $\widetilde{\nu}(e_{\delta,1})=0=\widetilde{\nu}(e_{\delta,2})$. Thus, $\omega_{[\rho]}\left(\Xmc'([\rho]),\Xmc([\rho])\right)=0$.

{\bf Case 2.2: $\widehat{c}$ is a boundary component of $P$.} Let $P(1)$ and $P(2)$ be the two pairs of pants of $\Pmc$ that share $\bf c$ as a common boundary component, and lie to the right  and left of $\widehat{\mathbf c}$ respectively. For both $m=1,2$, let ${\bf P}(m):=(P(m),\widehat{c})$. By the symmetries of $\Hmc^{\mathbf i}_{\mathbf P}$ and $\Ymc^\mathbf k_{\widehat{\mathbf c}}$, we may assume without loss of generality that ${\bf P}={\bf P}(1)$. Let $(\widehat{\mathbf b},\widehat{\mathbf b}')$ be the pair of non-edge barriers associated to ${\bf P}$, let ${\bf b}$ and ${\bf b}'$ be lifts to $\widetilde{S}$ of $\widehat{\bf b}$ and $\widehat{\bf b}'$ respectively, so that ${\bf b}$ intersects $\delta_P$, ${\bf b}'$ intersects $\delta'_{P}$. Let ${\bf x}$ and ${\bf x}'$ be the triples associated to $\widehat{\bf b}$ and $\widehat{\bf b}'$ respectively. By cyclically permuting the enumeration of the three vertices of $\Tbbb$ in the interior of $P$ if necessary, we may ensure that $e_{\delta_P,1}$ intersects ${\bf b}$ and $e_{\delta'_P,1}$ intersects ${\bf b}'$. 

By definition, 
\[\Ymc_{\widehat{\mathbf c}}^\mathbf k-\Zmc^\mathbf k_{\widehat{\mathbf c}}=2\Emc^{\mathbf k_3(0,-1,1)}_{\mathbf P(1)}-2\Emc^{\mathbf k_3(-1,0,1)}_{\mathbf P(1)}+2\Emc^{ \bar{\mathbf k}_3(0,-1,1)}_{\mathbf P(2)}-2\Emc^{\bar{\mathbf k}_3(-1,0,1)}_{\mathbf P(2)},\]
Hence, by Proposition \ref{prop:(2)}, we know that 
\[\omega_{[\rho]}(\Xmc',\Ymc^{\mathbf k}_{\widehat{\mathbf c}}-\Zmc^{\mathbf k}_{\widehat{\mathbf c}})=\left\{\begin{array}{ll}
2&\text{if }\mathbf k_3(-1,0,1)=\mathbf i;\\
-2&\text{if }\mathbf k_3(0,-1,1)=\mathbf i;\\
0&\text{otherwise}.
\end{array}\right.\]
To finish the proof, it is thus sufficient to show that
\[\omega_{[\rho]}(\Xmc',\Zmc^{\mathbf k}_{\widehat{\mathbf c}})=\left\{\begin{array}{ll}
-2&\text{if }\mathbf k_3=\mathbf i(1,0,-1);\\
2&\text{if }\mathbf k_3=\mathbf i(0,1,-1);\\
0&\text{otherwise}.
\end{array}\right.\]

Observe from their definitions that both $\Xmc'$ and $\Zmc^{\mathbf k}_{\widehat{\mathbf c}}$ are symmetric in every pair of pants given by $\Pmc$. Thus, Lemma \ref{lem:internal} implies that the internal part of $\omega_{[\rho]}(\Xmc',\Zmc^{\mathbf k}_{\widehat{\mathbf c}})$ is zero. To compute the external part of $\omega_{[\rho]}(\Xmc',\Zmc^{\mathbf k}_{\widehat{\mathbf c}})$, let $D$ be a fundamental block, and let $u_1,\dots,u_5,v_1,\dots,v_4$ be as defined in Notation \ref{not:block}. If $\pi(D)$ does not lie in $P$, then observe that $\widetilde{\nu}'(e_{\delta,1})=0$ for all $\delta\in\Tbbb(D)$, which implies
\[\sum_{\delta\in\Tbbb(D)}\sgn(\delta)\tr\left(\widetilde{\nu}'(e_{\delta,1})\cdot\widetilde{\zeta}^\mathbf k_{\widehat{\mathbf c}}(e_{\delta,2})\right)=0.\]

If $\pi(D)$ lies in $P$ but does not have $\widehat{c}$ as a boundary component, then Proposition \ref{prop:nil}(3),(4) tell us that $\widetilde{\nu}'(v_2-v_3)$ and $\widetilde{\zeta}^{\mathbf k}_{\widehat{\mathbf c}}(v_2-v_3)$ are both nilpotent. Thus Lemma \ref{lem:alphacylinder}(1) and Lemma~\ref{lem:gammacylinder}(1) imply
\[\sum_{\delta\in\Tbbb(D)}\sgn(\delta)\tr\big(\widetilde{\nu}'(e_{\delta,1})\cdot\widetilde{\zeta}^\mathbf k_{\widehat{\mathbf c}}(e_{\delta,2})\big)=0.\]

Now, suppose that $\pi(D)$ lies in $P$ and has $\widehat{c}$ as a boundary component. Let ${\bf c}$ be the lift in $\Std$ of $\widehat{\bf c}$ that intersects $\partial D$, and let ${\bf r}=(r_1,r_2)$ be the pair associated to $\mathbf c$. Then $D$ is necessarily a consistent fundamental block. By Proposition \ref{prop:derivative2} and Lemma \ref{lem: diagonal part},
\begin{align*}
\widetilde{\nu}'(u_1)_{{\bf r},{\rm diag}}&=\left\{\begin{array}{ll}
0&\text{if }i_3>1;\\
-\left(A^{\mathbf i(0,1,-1)}_{\xi_\rho(\mathbf x)}\right)_{{\bf r},{\rm diag}}+\left(A^{\mathbf i(1,0,-1)}_{\xi_\rho(\mathbf x)}\right)_{{\bf r},{\rm diag}}&\text{if }i_3=1.
\end{array}\right.\\
%&=\left\{\begin{array}{ll}
%0&\text{if }i_3>1;\\
%-\left(A^{\mathbf i'(0,-1,1)}_{\xi_\rho(\mathbf x')}\right)_{{\bf r},{\rm diag}}+\left(A^{\mathbf i'(1,-1,0)}_{\xi_\rho(\mathbf x')}\right)_{{\bf r},{\rm diag}}&\text{if }i_3=1.
%\end{array}\right.\\
&=\left\{\begin{array}{ll}
0&\text{if }i_3>1;\\
-2D^{(i_1,i_2+1)}_{\xi_\rho(\mathbf r)}+2D^{(i_1+1,i_2)}_{\xi_\rho(\mathbf r)}&\text{if }i_3=1.
\end{array}\right.
\end{align*}
Also, we previously computed, see \eqref{eqn:lozenge}, that
\[\widetilde{\zeta}^\mathbf k_{\widehat{\mathbf c}}(v_2-v_3)_{{\bf r},{\rm diag}}=-2M^{\mathbf k}_{\xi_\rho(\mathbf r)}+2M^{\mathbf k(1,-1)}_{\xi_\rho(\mathbf r)},\] 
where $M^\mathbf k_{\xi_\rho(\mathbf r)}$ is the endomorphism that is the identity on $\xi^{(k_1)}(r_1)\cap\xi^{(k_2+1)}(r_2)$, and whose kernel is $\xi^{(k_1-1)}(r_1)+\xi^{(k_2)}(r_2)$. 

Since Proposition \ref{prop:nil}(3) states that $\widetilde{\nu}'(v_2-v_3)$ is nilpotent, (\ref{eqn:alphacylinder}) implies
\begin{align*}
\sum_{\delta\in\Tbbb(D)}\sgn(\delta)\tr\big(\widetilde{\nu}'(e_{\delta,1})\cdot\widetilde{\zeta}^\mathbf k_{\widehat{\mathbf c}}(e_{\delta,2})\big)&=\tr\big(\widetilde{\nu}'(u_1)_{{\bf r},{\rm diag}}\cdot\widetilde{\zeta}^\mathbf k_{\widehat{\mathbf c}}(v_2-v_3)_{{\bf r},{\rm diag}}\big)\\
%&=&\tr\left(\left(D^{(i_1,i_2+1)}_{\xi(\mathbf d)}-D^{(i_1+1,i_2)}_{\xi(\mathbf d)}\right)\cdot\left(2M^{\mathbf k}_{\xi(\mathbf d)}-2M^{\mathbf k(1,-1)}_{\xi(\mathbf d)}\right)\right)\\
&=\left\{\begin{array}{ll}
-\frac{4i_2}{2n}+\frac{4(-i_1-1)}{2n}&\text{if } k_1=i_1+1\text{ and }i_3=1;\\
\frac{4(i_2+1)}{2n}-\frac{4(-i_1)}{2n}&\text{if } k_1=i_1\text{ and }i_3=1;\\
0&\text{otherwise.}\\
\end{array}\right.\\
&=\left\{\begin{array}{ll}
-2&\text{if }\mathbf k_3=\mathbf i(1,0,-1);\\
2&\text{if }\mathbf k_3=\mathbf i(0,1,-1);\\
0&\text{otherwise.}\\
\end{array}\right.
\end{align*}

We have thus shown that of all the fundamental blocks $D$ such that $\Tbbb(D)\subset\widehat{\Tbbb}$, the only one where the sum
\[\sum_{\delta\in\Tbbb(D)}\sgn(\delta)\tr\big(\widetilde{\nu}'(e_{\delta,1})\cdot\widetilde{\zeta}^\mathbf k_{\widehat{\mathbf c}}(e_{\delta,2})\big)\]
possibly does not vanish is the fundamental block $D$ such that $\pi(D)$ is a boundary cylinder of $P$ that has $\widehat{c}$ as a boundary component. Even then, 
\[\sum_{\delta\in\Tbbb(D)}\sgn(\delta)\tr\big(\widetilde{\nu}'(e_{\delta,1})\cdot\widetilde{\zeta}^\mathbf k_{\widehat{\mathbf c}}(e_{\delta,2})\big)=\left\{\begin{array}{ll}
-2&\text{if }\mathbf k_3=\mathbf i(1,0,-1);\\
2&\text{if }\mathbf k_3=\mathbf i(0,1,-1);\\
0&\text{otherwise.}
\end{array}\right.\]
Since we have also established that the internal part of $\omega_{[\rho]}(\Xmc',\Zmc^{\mathbf k}_{\widehat{\mathbf c}})$ is always zero, we have thus proven, as required, that 
\[\omega_{[\rho]}(\Xmc',\Zmc^{\mathbf k}_{\widehat{\mathbf c}})=\left\{\begin{array}{ll}
-2&\text{if }\mathbf k_3=\mathbf i(1,0,-1);\\
2&\text{if }\mathbf k_3=\mathbf i(0,1,-1);\\
0&\text{otherwise}.
\end{array}\right.\]
\end{proof}

\begin{prop}
If $\Xmc'$ and $\Xmc$ are both length vector fields, then
\[\omega_{[\rho]}\left(\Xmc'([\rho]),\Xmc([\rho])\right)=0.\]
\end{prop}

\begin{proof}
By Proposition \ref{prop:(2)}, we know that if $\Wmc'$ is an eruption vector field and $\Wmc$ is either an eruption or lozenge vector field, then $\omega_{[\rho]}\left(\Wmc'([\rho]),\Wmc([\rho])\right)=0$. Since the length vector fields are a sum of a lozenge vector field with some eruption vector fields, it suffices to show that
\[\omega_{[\rho]}\left(\Zmc^{\bf k}_{\widehat{\bf c}'}([\rho]),\Zmc^{\bf l}_{\widehat{\bf c}}([\rho])\right)=0\]
for some $\mathbf k,\mathbf l\in\Amc$ and some $\widehat{\mathbf c},\widehat{\mathbf c}'\in\Pmc^o$. Let $\widehat{c}$ and $\widehat{c}'$ denote the non-isolated edge in $\Pmc$ underlying $\widehat{\mathbf c}$ and $\widehat{\mathbf c}'$ respectively. 

Observe that lozenge vector fields are symmetric in every pair of pants of $\Pmc$, so Lemma \ref{lem:internal}(1) implies that the internal part of $\omega_{[\rho]}\left(\Zmc^{\bf k}_{\widehat{\bf c}'}([\rho]),\Zmc^{\bf l}_{\widehat{\bf c}}([\rho])\right)$ is zero. Thus, it is sufficient to prove that the external part of $\omega_{[\rho]}\left(\Zmc^{\bf k}_{\widehat{\bf c}'}([\rho]),\Zmc^{\bf l}_{\widehat{\bf c}}([\rho])\right)$ is also zero. We prove this in two cases.

{\bf Case 1: $\widehat{c}\neq \widehat{c}'$.}  For every fundamental block $D$, let $u_1,\dots,u_5,v_1,\dots,v_4$ be the edges in $D$ as defined in Notation \ref{not:block}. Then either $\widehat{c}'$ is not a boundary component of $\pi(D)$, $\widehat{c}$ is not a boundary component of $\pi(D)$. Proposition \ref{prop:nil}(4) then implies that $\widetilde{\zeta}^{\bf k}_{\widehat{\bf c}'}(v_3-v_2)$ and $\widetilde{\zeta}^{\bf k}_{\widehat{\bf c}'}(u_1)$ are nilpotent in the former case, while $\widetilde{\zeta}^{\bf l}_{\widehat{\bf c}}(v_3-v_2)$ and $\widetilde{\zeta}^{\bf l}_{\widehat{\bf c}}(u_1)$ are nilpotent in the latter case. Lemma \ref{lem:alphacylinder}(2) and Lemma \ref{lem:gammacylinder}(2) then implies that 
\[\sum_{\delta\in\Tbbb(D)}\sgn(\delta)\tr\big(\widetilde{\zeta}^{\bf k}_{\widehat{\bf c}'}(e_{\delta,1})\cdot\widetilde{\zeta}^{\bf l}_{\widehat{\bf c}}(e_{\delta,2})\big)=0.\]
Since $D$ is arbitrary, this implies that the external part of $\omega_{[\rho]}\left(\Zmc^{\bf k}_{\widehat{\bf c}'}([\rho]),\Zmc^{\bf l}_{\widehat{\bf c}}([\rho])\right)$ is zero.

{\bf Case 2: $\widehat{c}=\widehat{c}'$.} By the symmetry of lozenge vector fields, we may assume that $\widehat{\bf c}=\widehat{\bf c}'$. Observe that the same arguments used in Case 1 proves that if $D$ is a fundamental block such that $\pi(D)$ does not have $\widehat{c}$ as a boundary component, then
\[\sum_{\delta\in\Tbbb(D)}\sgn(\delta)\tr\big(\widetilde{\zeta}^{\bf k}_{\widehat{\bf c}'}(e_{\delta,1})\cdot\widetilde{\zeta}^{\bf l}_{\widehat{\bf c}}(e_{\delta,2})\big)=0.\]
Let $D(1)$ and $D(2)$ be the fundamental blocks such that $\Tbbb(D(m))\subset\Tbbb_{\rm lift}$, and $\pi(D(1))$ and $\pi(D(2))$ are the boundary cylinders that share $\widehat{\bf c}$ as a boundary component, and lie to the right and left of ${\widehat{\bf c}}$ respectively. To show that the external part of $\omega_{[\rho]}\left(\Zmc^{\bf k}_{\widehat{\bf c}}([\rho]),\Zmc^{\bf l}_{\widehat{\bf c}'}([\rho])\right)$ is zero, we now need only to prove that
\[\sum_{\delta\in\Tbbb(D(1))}\sgn(\delta)\tr\big(\widetilde{\zeta}^{\bf k}_{\widehat{\bf c}'}(e_{\delta,1})\cdot\widetilde{\zeta}^{\bf l}_{\widehat{\bf c}}(e_{\delta,2})\big)+\sum_{\delta\in\Tbbb(D(2))}\sgn(\delta)\tr\big(\widetilde{\zeta}^{\bf k}_{\widehat{\bf c}'}(e_{\delta,1})\cdot\widetilde{\zeta}^{\bf l}_{\widehat{\bf c}}(e_{\delta,2})\big)=0.\]

For both $m=1,2$, let $P(m)$ be the pair of pants of $\Pbbb$ that contains $\pi(D(m))$, let ${\bf P}(m):=(P(m),\widehat{c})$, and let $(\widehat{\mathbf b}(m),\widehat{\mathbf b}'(m))$ be the pair of non-edge barriers associated to ${\bf P}(m)$. Then let ${\bf b}(m)$ and ${\bf b}'(m)$ be lifts to $\widetilde{S}$ of $\widehat{\bf b}(m)$ and $\widehat{\bf b}'(m)$ respectively so that ${\bf b}(m)$ intersects $\delta_{P(m)}$ and ${\bf b}'(m)$ intersects $\delta'_{P(m)}$, and let ${\bf x}(m)$ and ${\bf x}'(m)$ be the triples associated to ${\bf b}(m)$ and ${\bf b}'(m)$ respectively. By cyclically permuting the enumeration of the three vertices of $\Tbbb$ in the interior of $P$ if necessary, we may ensure that $e_{\delta_{P(m)},1}$ intersects ${\bf b}(m)$ and $e_{\delta'_{P(m)},1}$ intersects ${\bf b}'(m)$. Let $u_1(m),\dots,u_5(m),v_1(m),\dots,v_4(m)$ be the edges in $D(m)$ as defined in Notation \ref{not:block}, let ${\bf c}$ be the lift of $\widehat{\bf c}$ such that $D(1)\cap{\bf c}$ and $D(2)\cap{\bf c}$ are non-empty and coincide, and let $\mathbf r=(r_1,r_2)$ be the pair associated to ${\bf c}$. 

Since lozenge vector fields are symmetric, 
\[\widetilde{\zeta}^{\bf k}_{\widehat{\bf c}'}(v_2(m))_{{\bf r},{\rm diag}}=-\widetilde{\zeta}^{\bf k}_{\widehat{\bf c}'}(v_3(m))_{{\bf r},{\rm diag}}\quad\text{and}\quad\widetilde{\zeta}^{\bf l}_{\widehat{\bf c}}(v_2(m))_{{\bf r},{\rm diag}}=-\widetilde{\zeta}^{\bf l}_{\widehat{\bf c}}(v_3(m))_{{\bf r},{\rm diag}}\] 
for both $m=1,2$, so
\[\tr\big(\widetilde{\zeta}^{\bf k}_{\widehat{\bf c}'}(v_3(m))_{{\bf r},{\rm diag}}\cdot\widetilde{\zeta}^{\bf l}_{\widehat{\bf c}}(v_3(m))_{{\bf r},{\rm diag}}\big)-\tr\big(\widetilde{\zeta}^{\bf k}_{\widehat{\bf c}'}(v_2(m))_{{\bf r},{\rm diag}}\cdot\widetilde{\zeta}^{\bf l}_{\widehat{\bf c}}(v_2(m))_{{\bf r},{\rm diag}}\big)=0.\] 
Also, since $\widetilde{\zeta}^{\bf k}_{\widehat{\bf c}'}(v_1(m))=-\widetilde{\zeta}^{\bf k}_{\widehat{\bf c}'}(v_4(m))$, we see that 
\[2\widetilde{\zeta}^{\bf k}_{\widehat{\bf c}'}(v_1(m))=\widetilde{\zeta}^{\bf k}_{\widehat{\bf c}'}(v_1(m)-v_4(m))=\widetilde{\zeta}^{\bf k}_{\widehat{\bf c}'}(v_2(m)-v_3(m))_{{\bf r},{\rm diag}}=2\widetilde{\zeta}^{\bf k}_{\widehat{\bf c}'}(v_2(m))_{{\bf r},{\rm diag}}.\] 
Similarly, $\widetilde{\zeta}^{\bf l}_{\widehat{\bf c}}(v_1(m))=-\widetilde{\zeta}^{\bf l}_{\widehat{\bf c}}(v_4(m))$,
\[2\widetilde{\zeta}^{\bf l}_{\widehat{\bf c}}(v_1(m))=\widetilde{\zeta}^{\bf l}_{\widehat{\bf c}}(v_2(m)-v_3(m))_{{\bf r},{\rm diag}}=2\widetilde{\zeta}^{\bf l}_{\widehat{\bf c}}(v_2(m))_{{\bf r},{\rm diag}}.\] 
Thus, (\ref{eqn:alphacylinder}) and (\ref{eqn:gammacylinder}) imply
\begin{align*}
\sum_{\delta\in\Tbbb(D(m))}\sgn(\delta)\tr\big(\widetilde{\zeta}^{\bf k}_{\widehat{\bf c}'}(e_{\delta,1})\cdot\widetilde{\zeta}^{\bf l}_{\widehat{\bf c}}(e_{\delta,2})\big)&=2\tr\big(\widetilde{\zeta}^{\bf k}_{\widehat{\bf c}'}(u_1(m))_{{\bf r},{\rm diag}}\cdot\widetilde{\zeta}^{\bf l}_{\widehat{\bf c}}(v_1(m))\big)\\
&-2\tr\big(\widetilde{\zeta}^{\bf k}_{\widehat{\bf c}'}(v_1(m))\cdot\widetilde{\zeta}^{\bf l}_{\widehat{\bf c}}(u_1(m))_{{\bf r},{\rm diag}}\big).
\end{align*}

Next, observe that $v_1(2)=v_4(1)$, so $\widetilde{\zeta}^{\bf k}_{\widehat{\bf c}'}(v_1(2))=-\widetilde{\zeta}^{\bf k}_{\widehat{\bf c}'}(v_1(1))$ and $\widetilde{\zeta}^{\bf l}_{\widehat{\bf c}}(v_1(2))=-\widetilde{\zeta}^{\bf l}_{\widehat{\bf c}}(v_1(1))$, so
\begin{eqnarray*}
&&\sum_{\delta\in\Tbbb(D(1))}\sgn(\delta)\tr\big(\widetilde{\zeta}^{\bf k}_{\widehat{\bf c}'}(e_{\delta,1})\cdot\widetilde{\zeta}^{\bf l}_{\widehat{\bf c}}(e_{\delta,2})\big)+\sum_{\delta\in\Tbbb(D(2))}\sgn(\delta)\tr\big(\widetilde{\zeta}^{\bf k}_{\widehat{\bf c}'}(e_{\delta,1})\cdot\widetilde{\zeta}^{\bf l}_{\widehat{\bf c}}(e_{\delta,2})\big)\\
&=&2\tr\left(\left(\widetilde{\zeta}^{\bf k}_{\widehat{\bf c}'}(u_1(1))_{{\bf r},{\rm diag}}-\widetilde{\zeta}^{\bf k}_{\widehat{\bf c}'}(u_1(2))_{{\bf r},{\rm diag}}\right)\cdot\widetilde{\zeta}^{\bf l}_{\widehat{\bf c}}(v_1(1))\right)\\
&&+2\tr\left(\widetilde{\zeta}^{\bf k}_{\widehat{\bf c}'}(v_1(1))\cdot\left(\widetilde{\zeta}^{\bf l}_{\widehat{\bf c}}(u_1(2))_{{\bf r},{\rm diag}}-\widetilde{\zeta}^{\bf l}_{\widehat{\bf c}}(u_1(1))_{{\bf r},{\rm diag}}\right)\right).
\end{eqnarray*}
By definition and Lemma \ref{lem: diagonal part},
\begin{align*}
\widetilde{\zeta}^{\bf k}_{\widehat{\bf c}'}(u_1(1))_{{\bf r},{\rm diag}}&=-\left(A^{{\mathbf k}_3}_{\xi_\rho(\mathbf x(1))}\right)_{{\bf r},{\rm diag}}+\left(A^{{\mathbf k}_3(1,-1,0)}_{\xi_\rho(\mathbf x(1))}\right)_{{\bf r},{\rm diag}}\\
&=-2D^{\bf k}_{\xi_\rho(\mathbf r)}+2D^{(k_1+1,k_2-1)}_{\xi_\rho(\mathbf r)}
\end{align*}
and
\begin{align*}
\widetilde{\zeta}^{\bf k}_{\widehat{\bf c}'}(u_1(2))_{{\bf r},{\rm diag}}&=-\left(A^{\bar{\mathbf k}_3}_{\xi_\rho(\mathbf x(2))}\right)_{{\bf r},{\rm diag}}+\left(A^{\bar{\mathbf k}_3(1,-1,0)}_{\xi_\rho(\mathbf x(2))}\right)_{{\bf r},{\rm diag}}\\
&=-2D^{\bf k}_{\xi_\rho(\mathbf r)}+2D^{(k_1+1,k_2-1)}_{\xi_\rho(\mathbf r)},
\end{align*}
so 
\[\widetilde{\zeta}^{\bf k}_{\widehat{\bf c}'}(u_1(1))_{{\bf r},{\rm diag}}-\widetilde{\zeta}^{\bf k}_{\widehat{\bf c}'}(u_1(2))_{{\bf r},{\rm diag}}=0.\]
Similarly,
\[\widetilde{\zeta}^{\bf l}_{\widehat{\bf c}}(u_1(2))_{{\bf r},{\rm diag}}-\widetilde{\zeta}^{\bf l}_{\widehat{\bf c}}(u_1(1))_{{\bf r},{\rm diag}}=0.\]
It follows that 
\[\sum_{\delta\in\Tbbb(D(1))}\sgn(\delta)\tr\big(\widetilde{\nu}^1(e_{\delta,1})\cdot\widetilde{\nu}^2(e_{\delta,2})\big)+\sum_{\delta\in\Tbbb(D(2))}\sgn(\delta)\tr\big(\widetilde{\nu}^1(e_{\delta,1})\cdot\widetilde{\nu}^2(e_{\delta,2})\big)=0.\qedhere\]
\end{proof}

\bibliographystyle{amsalpha}
\bibliography{ref}

\end{document}